\documentclass[11pt]{article}

\usepackage[margin=1in]{geometry}
\usepackage{amssymb, color}
\usepackage{hyperref}
\usepackage{indentfirst}
\usepackage{tikz}
\usepackage{caption}
\usepackage[T1]{fontenc}
\usepackage{slashed}
\usepackage{soul}
\usepackage{qtree}

 \usepackage{pgfplots}
\usepackage{graphicx,xcolor}
\usepackage{amsthm,amsfonts,euscript,amsmath,comment,slashed,here,todonotes}
\usepackage[affil-it]{authblk}
\usepackage{mathrsfs}
\usepackage{stmaryrd}
\usepackage[cp1251]{inputenc}

\newtheorem{theorem}{Theorem}[section]
\newtheorem{lemma}[theorem]{Lemma}

\newtheorem{proposition}{Proposition}[section]
\newtheorem{corollary}{Corollary}[section]

\theoremstyle{definition}
\newtheorem{definition}[theorem]{Definition}

\theoremstyle{remark}
\newtheorem{remark}[theorem]{Remark}

\numberwithin{equation}{section}

\newcommand{\ud}{\,\mathrm{d}}
\newcommand{\p}{\ensuremath{\partial}}
\newcommand{\n}{\ensuremath{\nonumber}}
\newcommand{\eps}{\ensuremath{\varepsilon}}

\newcommand\be{\begin{equation}}
\newcommand\ee{\end{equation}}
\newcommand\bea{\begin{eqnarray}}
\newcommand\eea{\end{eqnarray}}
\newcommand\bi{\begin{itemize}}
\newcommand\ei{\end{itemize}}
\newcommand\ben{\begin{enumerate}}
\newcommand\bena{\begin{enumerate}[(a)]}
\newcommand\een{\end{enumerate}}

\newcommand\bp{\begin{proof}}
\newcommand\ep{\end{proof}}

\renewcommand\Re{\operatorname{Re}}
\renewcommand\Im{\operatorname{Im}}

\allowdisplaybreaks

\title{Reversal in the Stationary Prandtl Equations}
\author{Sameer Iyer \footnote{Department of Mathematics, University of California, Davis, Davis, CA 95616, \url{sameer@math.ucdavis.edu}} \qquad Nader Masmoudi \footnote{NYUAD Research Institute, New York University Abu Dhabi, PO Box 129188, Abu Dhabi, United Arab Emirates. \\ \footnotesize \normalfont
 Courant Institute of Mathematical Sciences, New York University, 251 Mercer Street, New York, NY 10012, USA.  \url{masmoudi@cims.nyu.edu}}}

\usetikzlibrary{arrows,positioning} 
\tikzset{
    >=stealth',
    punkt/.style={
           rectangle,
           rounded corners,
           draw=black, very thick,
           text width=6.5em,
           minimum height=2em,
           text centered},
    pil/.style={
           ->,
           thick,
           shorten <=2pt,
           shorten >=2pt,}
}

\begin{document}

\maketitle

\begin{abstract} We demonstrate the existence of a class of data which exhibits \textit{reversal} and \textit{recirculation} for the stationary Prandtl equations (data is taken in an appropriately defined product space due to the simultaneous forward and backward causality in the problem). Reversal describes the development of the solution beyond the Goldstein singularity, and is characterized by the presence of (spatio-temporal) regions in which $u > 0$ and $u < 0$. The classical point of view of regarding the system as an evolution in the tangential direction completely breaks down past the Goldstein singularity. Instead, to describe the development, we view the problem as a new \textit{mixed-type,  non-local, quasilinear, free-boundary} problem across the curve $\{ u = 0 \}$. In a well-chosen nonlinear and self-similar coordinate system, we extract a coupled system for the bulk solution and several modulation variables describing the free boundary. Our work introduces new techniques from mixed-type problems, free-boundary problems, modulation theory, and spectral theory combined with several new points of view on the stationary Prandtl system. 
\end{abstract}

\setcounter{tocdepth}{2}
\tableofcontents

\section{Introduction}

\subsection{The Setting}

We are interested in the 2D, stationary Prandtl equations 
\begin{align}
\begin{aligned} \label{eq:PR:0}
&u_P \p_x u_P + v_P \p_y u_P - \p_y^2 u_{P} = -\p_x p_E(x), \\
&\p_x u_P + \p_y v_P = 0, \qquad (x,y) \in \Omega_L := (1, 1+L) \times (0, \infty),
\end{aligned}
\end{align}
This equation is supplemented with boundary data in the vertical direction, 
\begin{align} \label{numEul}
u_P|_{y = 0} = v_P|_{y = 0} = 0, \qquad \lim_{y \rightarrow \infty} u_P(x, y) = u_E(x), 
\end{align}
where $u_E(x)$ is related to $p_E(x)$ through Bernoulli's law
\begin{align} \label{uEdef}
u_E(x) \p_x u_E(x) = - \p_x p_E(x). 
\end{align}
From the perspective of the Prandtl equations, \eqref{eq:PR:0}, the ``outer" Euler flow, $u_E(x)$, and the corresponding pressure, $\p_x p_E(x)$ is an input appearing as a source term in \eqref{eq:PR:0}. We will discuss below in \eqref{choice:uEx} the particular choices we make for the functions $u_E(x)$ and thus also $\p_x p_E(x)$. These choices are motivated by the existence of a family of classical self-similar profiles, known as the Falkner-Skan profiles for reversed flow. 

It is well-known that due to the parabolic scaling exhibited by \eqref{eq:PR:0}, the stationary Prandtl equations are, in fact, an \textit{evolution} equation in the $x$ direction. Considering the left-hand side of \eqref{eq:PR:0}, we can formally identify $u_P \p_x$ scales like $\p_y^2$, and therefore \eqref{uEdef} behaves like a forward evolution in $x$. This point of view persists so long as $u_P \ge 0$, since $u_P(x, y)$ is the coefficient in front of the transport term $\p_x u_P$. 

To our knowledge, all mathematical results to date regarding the stationary Prandtl system, \eqref{eq:PR:0}, are under the assumption that $u_P$ remains nonnegative, and therefore the point of view of regarding \eqref{eq:PR:0} as a forward evolution in $x$ is essentially always adopted. \textit{Our paper is the first to study the stationary Prandtl system, \eqref{eq:PR:0}, that allows for a sign change for $u_P$.} These works will be surveyed below in Subsection \ref{existing:Lit}, but we discuss now the most relevant results for our theorem.

The existence of local solutions ($0 < x < L << 1$) was established by Oleinik in \cite{Oleinik}, who also established global in $x$ solutions under the assumption that $\p_x p_E(x) \le 0$ (a favorable pressure gradient). In the case of an adverse pressure gradient, $\p_x p_E(x) > 0$, the physics literature has well-documented the possibility of \textit{boundary layer separation}, starting in fact with Prandtl's seminal 1904 paper, \cite{Prandtl}, and in many other works, for instance \cite{Goldstein}. 

Separation is a physical phenomena in which the boundary layer detaches from the wall and enters the flow, as can be seen below in Figure \ref{Fig:1}. 
\begin{figure}[h] 
\hspace{28 mm} \includegraphics[scale=0.2]{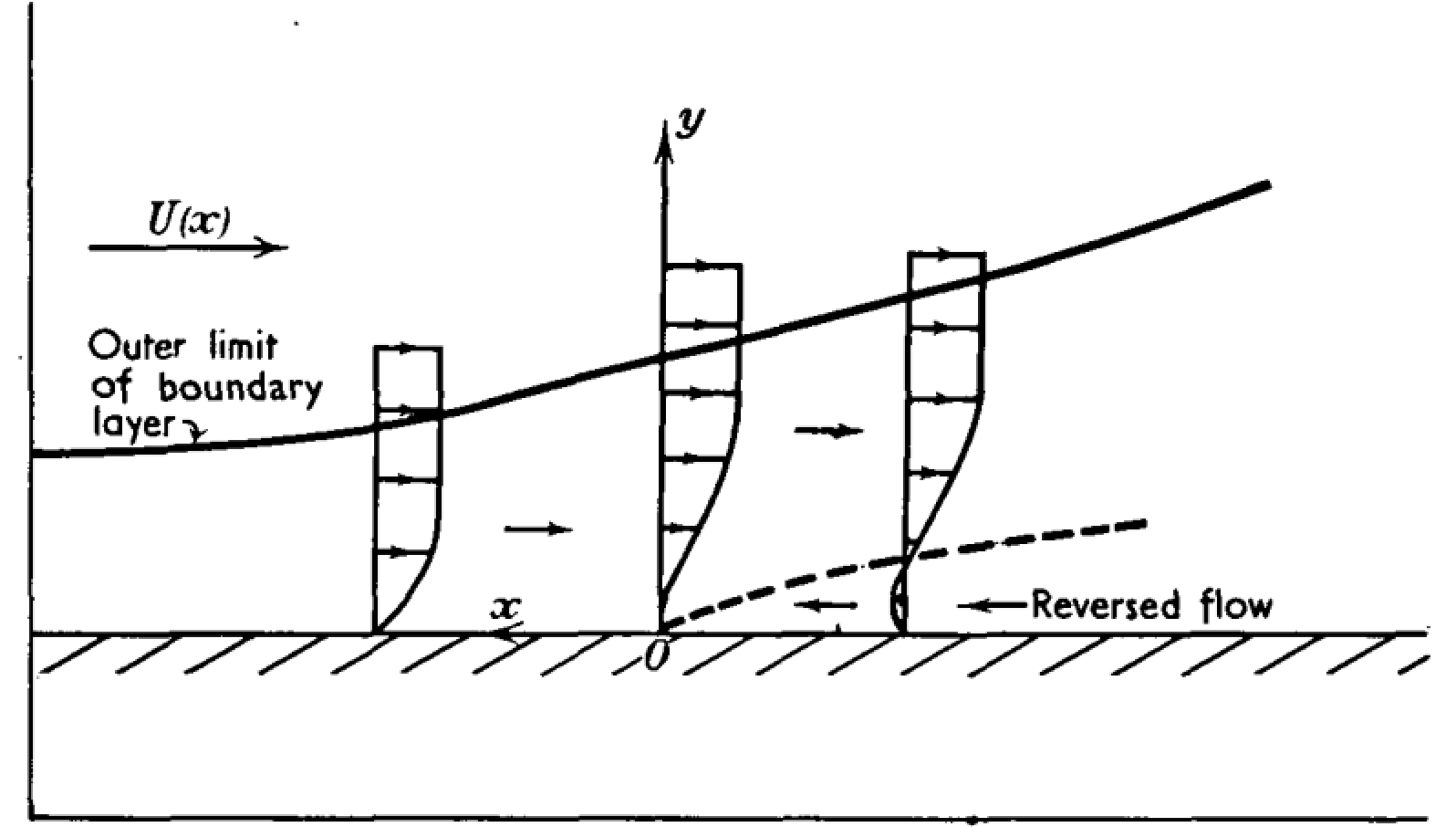}
\caption{Onset and continuation past separation, \cite{Stew0}} \label{Fig:1}
\end{figure}
Mathematically, boundary layer separation has been characterized in the work of \cite{MD}. We point the reader also to the works \cite{Zhangsep} as well as the paper \cite{E} in which a separation result is announced. Separation occurs at a point $x_\ast$ (which in Figure \ref{Fig:1} corresponds to $x_\ast= 0$) if $\p_y u_P(x_{\ast}, 0) = 0$, while $\p_y u_P(x, 0) > 0$ for $x < x_{\ast}$.

\textit{In the present work, we are concerned with describing the solution after separation}, that is, to the right of $x = 0$ in Figure \ref{Fig:1}. Physically, the flow after separation undergoes a new phenomenon of reversal and recirculation: the tangential velocity can be negative, compared to the ``pre-separation" regime. Mathematically, the fact that $u_P$ changes sign completely destroys the pre-existing approaches to the stationary Prandtl equations (for example \cite{MD}, \cite{Oleinik}, \cite{Serrin}, \cite{IyerBlasius}, \cite{Zhangsep}), which, as we have described above, is to regard the equations as an evolution. In this work, we introduce an entirely new framework to study the Prandtl system in this reversal regime. 

The problem of continuing the solution to \eqref{eq:PR:0} after the Goldstein singularity has historically garnered significant attention from many fluid-dynamicists: 
\begin{quote}
``A famous paper by Goldstein asserts that... the solution of the boundary layer equations can, in general, not be continued beyond the point of separation. Subsequent attempts by many authors to overcome the difficulty of continuation have failed." (P. Lagerstrom, 1975, \cite{Lagerstrom75})
\end{quote}

Several features become clear upon inspection: since the flow after separation exhibits \textit{reversal} flow, there exist regions where $u_P < 0$ (below the dotted curve in Figure \ref{Fig:1}), and also regions where $u_P > 0$ (above the dotted curve in Figure \ref{Fig:1}). In particular, the assumption that $u_P \ge 0$ is not true, and the classical point of view of regarding \eqref{eq:PR:0} as a forward evolution in $x$ completely breaks down. 

We develop a new framework, and introduce a new point of view to study the solution in this reversal regime. We devote Section \ref{section:ideas} to describing this new point of view, the steps of our proof, and the various difficulties that are encountered. As a result, we only briefly mention here some aspects of our point of view in a relatively informal manner. The dotted line in Figure \ref{Fig:1} which emerges immediately after the Goldstein singularity occurs depicts the set $\{ u_P = 0\}$. We identify this curve as a \textit{free-boundary} which needs to be characterized. Above this free boundary, we have $u_P > 0$, and therefore \eqref{eq:PR:0}, in principle, should behave like a forward evolution equation in $x$. Below this free boundary, $u_P < 0$, and therefore \eqref{eq:PR:0}, in principle, should behave as a \textit{backwards} evolution in $x$. For this reason, we will regard the reversal problem as a \textit{mixed-type} problem: the equation changes type across $\{u_P = 0\}$ from forward to backward parabolic. It is useful to keep the following picture, Figure \ref{fig:org:mix}, in mind.
\begin{figure}[h]
\centering
\begin{tikzpicture}
\draw[ultra thick, <-] (-.5,3) node[above]{$y$} -- (-.5,-0.5);
\draw[ultra thick, -] (4,1.5) -- (4,-2);
\draw[ultra thick, -] (-.5,-2) -- (4,-2);
\node [below] at (1.75, -2) {$u_P = 0, \psi_P = 0$};
\draw[ultra thick,blue, -] (-.5, -.5) to[bend left] (4,1.5);
\node [below] at (1.75, 1) {$\textcolor{blue}{u_P =0}$};
\node [below] at (1.75, -.2) {$u_P < 0$};
\node [below] at (1.75, -.6) {\scriptsize $u_P \p_x u_P + \textcolor{cyan}{v_P \p_y u_P} - \p_y^2 u_P = F$};
\node [below] at (1.75, -1.2) {(reversed flow)};
\node [below] at (1.75, 3.2) {$u_P > 0$};
\node [below] at (1.75, 2.7) {\scriptsize $u_P \p_x u_P + \textcolor{cyan}{ v_P \p_y u_P} - \p_y^2 u_P = F$};
\node [below] at (1.75, 2.1) {(forward flow)};
\draw[ultra thick, dashed, -] (-.5,-2) -- (-.5,-.5);
\draw[ultra thick, dashed, ->] (4,1.5) -- (4,3);
\draw[ultra thick, dashed, ->] (4,-2) -- (5,-2);
\draw[ultra thick, dashed, ->] (-.5,-2) -- (-1.5,-2);
\node[right] at (5,-2) {$x$};
\end{tikzpicture}
\caption{Prandtl as a Free Boundary Mixed-Type Problem} \label{fig:org:mix}
\end{figure}

In general, mixed-type problems are those in which the equation undergoes a change of type across a particular hypersurface. Some classical examples of mixed-type equations  include the Tricomi equation, $u_{xx} + xu_{yy} = 0$, and the Keldysh equation, $xu_{xx} + u_{yy} = 0$. Evidently, these two aforementioned equations change type from elliptic (when $x > 0$) to hyperbolic (when $x < 0$) across the hypersurface $\{x = 0\}$. \textit{Linear} mixed-type problems have been studied in the past, though certainly not as intensively as their elliptic/ parabolic/ hyperbolic counterparts. The methods surrounding these problems are largely dependent on deriving explicit solutions, see for instance the book of Bitsadze, \cite{Bitsadze}, for a relatively comprehensive overview of linear methods.

In contrast, \textit{quasilinear} mixed-type problems pose substantial difficulties, and there are exceedingly few results. The two equations mentioned above have become increasingly important in gas dynamics recently due to their appearance in problems related to transonic shock formation and the shock reflection problem. We refer the reader to some excellent works in this direction: the pioneering works of Morawetz, for example \cite{MorTransonic}, as well as the remarkable works of Chen-Feldman: \cite{Chen-Feldman-1}, \cite{Chen-Feldman-2} (again, to name just a few). 

It appears to be a general principle that \textit{quasilinear, mixed-type} problems may be regarded as a \textit{free-boundary} problem. Indeed, if we return to the Tricomi equation and replaced it with $u_{xx} + f(u) u_{yy} = 0$, then $f(u)$ is playing the role of $x$, and we would need to determine the zero set $f(u) = 0$ in order to determine the hypersurface across which the equation changes type. This turns out to be extremely challenging in general, which is one reason why there are so few results regarding quasilinear, mixed-type problems. 

In the present work, due to the sign change of the coefficient of $u(x, y)$ in front of the $\p_x u$ term, our situation is a quasilinear, mixed type problem which results in the formation of a free boundary, $\{u = 0\}$, whose characterization is a main goal. This free-boundary point of view is novel for the stationary Prandtl system. A further complexity is that the term $v \p_y u$ in \eqref{eq:PR:0} is \textit{nonlocal} in the vertical direction due to the incompressibility, which has the potential of spoiling this mixed-type characterization. 

We urge the reader to consult Figure \ref{fig:flow:of:info} which depicts these \textbf{three (3) flows of information}. Above the free boundary, $\{ u = 0\}$, the coefficient $u > 0$, and we expect forward parabolic behavior (information flows left-to-right). Below the free boundary, the coefficient $u < 0$ and we therefore expect backward parabolic behavior (information flows right-to-left). The nonlocal term, $v \p_y u$, \textit{sends information from bottom-to-top} through the incompressibility $v = - \int_0^y \p_x u$. Obtaining efficient ways to disentangle these three (3) types of information flow is extremely challenging, delicate, and is a central achievement of this paper.  
\begin{figure}[h]
\centering
\begin{tikzpicture}
\draw[ultra thick, <-] (-.5,3) node[above]{$y$} -- (-.5,-0.5);
\draw[ultra thick, -] (4,1.5) -- (4,-2);
\draw[ultra thick, -] (-.5,-2) -- (4,-2);
\node [below] at (1.75, -2) {$u_P = 0, \psi_P = 0$};
\draw[ultra thick,blue, -] (-.5, -.5) to[bend left] (4,1.5);
\node [below] at (3, -.75) {$\textcolor{magenta}{u_P < 0}$};
\node [below] at (3, -1.2) {\textcolor{magenta}{(reversed flow)}};
\node [below] at (0.75, 3.2) {$\textcolor{orange}{u_P > 0}$};
\node [below] at (0.75, 2.8) {\textcolor{orange}{(forward flow)}};
\draw[ultra thick, orange, dashed, ->] (-.5,1.9) -- (1.5,1.9);
\draw[ultra thick, magenta, dashed, ->] (4,-.5) -- (2,-.5);
\node [below] at (2, 1.2) {\textcolor{cyan}{(incompressibility)}};
\node [below] at (2, .7) {\textcolor{cyan}{$v_P = - \int_0^y \p_x u_P$}};
\draw[ultra thick, cyan, dashed, ->] (2,-2) -- (2,3);
\draw[ultra thick, dashed, -] (-.5,-2) -- (-.5,-.5);
\draw[ultra thick, dashed, ->] (4,1.5) -- (4,3);
\draw[ultra thick, dashed, ->] (4,-2) -- (5,-2);
\draw[ultra thick, dashed, ->] (-.5,-2) -- (-1.5,-2);
\node[right] at (5,-2) {$x$};
\end{tikzpicture}
\caption{Flow of Information in the Prandtl system} \label{fig:flow:of:info}
\end{figure}

It is worth mentioning that a second reason for the lack of works contending with quasilinear mixed-type problems is simply that, at least as far as we are aware, it is relatively rare to find a physically natural context in which these arise, at least compared to their classical hyperbolic, parabolic, elliptic counterparts (again, the well-studied transonic shock problem from gas dynamics is a notable exception). A noteworthy aspect of our present work even at the outset is that we encounter a natural setting, from the point of view of fluid dynamics, which presents as a quasilinear, mixed-type problem. 

We note that our setup, Figure \ref{fig:org:mix}, is also reminiscent of the classical Obstacle problem or the Stefan problem, but with some key differences. In the Obstacle problem or Stefan problem, one regards the set $\p \{u = 0\}$ as a ``free boundary", whose regularity is the crucial question of study. On the other hand, the setup differs from ours as typically an elliptic or parabolic equation (for instance, $\Delta u = 0$ or $\p_t u - \Delta u = 0$) is prescribed only in the region $\{u > 0\}$, and the equation is not allowed to become of mixed-type. We refer the reader to the many classical works in this area, for instance \cite{CafaralliFB}, \cite{FRS1}, \cite{FigSer}  just to name a few. There is also a large literature on the related porous medium equation, which also features a degeneracy in the diffusion. Nearly all results we know of regarding the free-boundary porous medium equation require the degeneracy coefficient to remain nonnegative; see for example \cite{RosOton1}, and the many references therein.

\subsection{Self-Similarity and Reversal}

For concreteness, we shall fix the Euler flow in \eqref{uEdef} to be the following
\begin{align} \label{choice:uEx}
u_E(x) := x^n, \qquad n = \frac{\beta}{2-\beta}, \qquad - 0.2 < \beta < 0.
\end{align}
This choice of $u_E(x)$ and corresponding forcing $- \p_x p_E(x)$ into \eqref{eq:PR:0} is known to produce a self-similar \textit{reversed} solution to the Prandtl equations. Indeed, consider the ansatz:
\begin{align} \label{uFS}
u_{FS}(x, y) := u_E(x) f'(\eta), \qquad \eta := \frac{y}{x^{\frac{1-n}{2}}}. 
\end{align}
The self-similar profile, $f(\cdot)$, satisfies the following equation:
\begin{align} \label{FS:beta}
f''' + f f'' + \beta (1 - (f')^2) = 0, \qquad f(0) = f'(0) = 0, \qquad f'(\infty) = 1,
\end{align}
which is often referred to as a Falkner-Skan self-similar profile to the Prandtl equations. Numerical simulations for negative values of $\beta$ yield the following right-most profile shown below, which undergoes reversed flow (see for instance, \cite{Lagree}, P. 35):

\begin{figure}[h]
\hspace{40 mm} \includegraphics[scale=0.4]{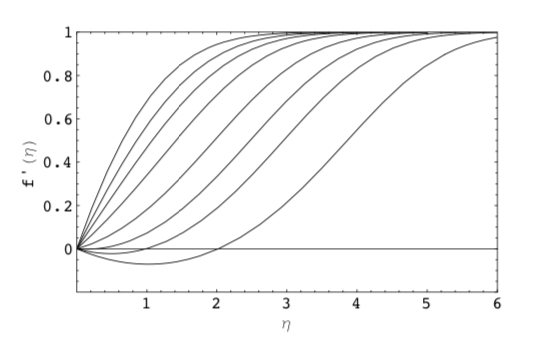}
\caption{Self Similar Falkner-Skan Profiles} \label{Fig:2}
\end{figure}

As far as mathematical treatments of these self-similar flows, we state the following classical result: 
\begin{lemma}[Stewartson, \cite{Stewartson}] For $\beta_{\ast} < \beta < 0$, where $\beta_{\ast} < 0$, there exists a solution $f(\eta)$ whose graph of $f'$ undergoes ``reversal", that is, $f''(0) < 0$, $f'(\eta) \le 0$ for $0 \le \eta \le \eta_{\ast}$, and $f'(\eta) > 0$ for $\eta_{\ast} < \eta < \infty$. 
\end{lemma}
After the work \cite{Stewartson}, the quantification of the largest negative value $\beta_\ast$ for which one can find reversed self-similar Falkner-Skan (``FS") profiles has been a subject of study, \cite{BrownStewartson}, \cite{Hastings}, \cite{YangLan}, for instance. As discussed in the works \cite{Hartree}, \cite{Stewartson},  \cite{BrownStewartson}, these ``reversed" self-similar profiles are classical in the boundary layer theory due to their importance in connection with boundary layer separation, and also their importance in various applied settings (such as airfoil design in aerodynamics, etc...). 

\textit{Our main objectives in this work are to establish the existence of a set of data that results in reversed (non self-similar) flows, and to provide stability theorems: if data is prescribed ``near" these reversed Falkner-Skan profiles, they remain ``near" them on an order one tangential scale.}

\begin{remark} Our choice of the ``FS" profiles is exclusively for the sake of concreteness: they serve as a particular solution to the Prandtl equations \eqref{eq:PR:0} (of course, with adverse pressure gradient), which exhibit reversed flow. One could repeat our results for a much wider class of reversed flows. 
\end{remark}

More formally, consider
\begin{align} \label{psir}
\psi_P = \psi_{FS} + \eps \psi_{R}, \qquad u_P = u_{FS} + \eps u_R, \qquad v_P = v_{FS} + \eps v_R, \qquad 0 < \eps << 1, 
\end{align}
where the subscript $``R"$ stands for ``Remainder," and $\eps$ is the size of the perturbation. Above, we introduce the stream function $\psi_P$ such that $\p_y \psi_P = u_P$, $- \p_x \psi_P = v_P$, and $\psi_P|_{y = 0} = 0$. Then, the remainder satisfies the following equation  
\begin{align} \label{exp:eq:intro:1}
(u_{FS} + \eps u_R) \p_x u_R + u_R \p_x u_{FS} + (v_{FS} + \eps v_R) \p_y u_R + v_R \p_y u_{FS} - \p_y^2 u_R = 0. 
\end{align}

\subsection{Type of Equation \& Leading Order Role of Incompressibility}

We now want to describe the data that we pose. Of course, it is always the case that posing the correct data requires understanding the \textit{type} of the equation (for parabolic equations we prescribe Cauchy data on $u$, for hyperbolic we prescribe $u$ and $\p_t u$, etc...). For mixed-type problems (even linear mixed-type) it can be a delicate matter to pose the correct data because the type is non-uniform (see for instance, \cite{Bitsadze}). For quasilinear, mixed-type problems the issue of posing the correct data becomes even more pronounced (see for instance discussions in Morawetz's early work on transonic shocks, \cite{MorTransonic}). In our case, we have not only a quasilinear mixed type problem, but we also have the nonlocal role of incompressibility which enters at leading order (we discuss this below). Therefore, we need to introduce a new conceptual understanding of the type of equation we have in order to even pose the correct data. We will now discuss this at length, and in the next section we will pose the data given the outcome of the present section.  

To our knowledge, all prior mathematical results on \eqref{eq:PR:0} fall into one of two categories: (1) construction and analysis of specific self-similar solutions (for instance  \cite{Goldstein}, \cite{Stewartson}, \cite{YangLan}), or (2) studying the wellposedness/ finite-$x$-blowup of \eqref{eq:PR:0} as an initial value problem, with data prescribed at $\{x = 0\}$ (for instance, \cite{Oleinik}, \cite{GI2}, \cite{Serrin}, \cite{IyerBlasius}, \cite{MD}, \cite{Zhangsep}, \cite{Zhifei:smooth}). The present work is distinguished since, as described earlier, \eqref{eq:PR:0} cannot be regarded as a traditional Cauchy problem in the presence of flow reversal. 

\vspace{2 mm}

\noindent \textit{Reduction to a Model Equation:} First, to fix a simpler setting, we consider the linearized Prandtl equations around a reversed Falkner-Skan profile, $[u_{FS}, v_{FS}]$. The linearized equation, obtained by retaining the leading order terms in $\eps$ from \eqref{exp:eq:intro:1} reads:
\begin{align}
u_{FS} \p_x u_R + u_R \p_x u_{FS} + v_{FS} \p_y u_R + v_R \p_y u_{FS} - \p_y^2 u_R = 0. 
\end{align}
Since the ``type" of the equation is determined by the leading order terms, we may drop the second and third term from the left-hand side, and study the simplified model problem: 
\begin{align} \label{model:pb:00}
u_{FS} \p_x u_R + \p_y u_{FS} v_R - \p_y^2 u_R = 0. 
\end{align}

\vspace{2 mm}

\noindent \textit{Scaling Near the Interface: } It is tempting to remove the middle, nonlocal, term $\p_y u_{FS}v_R$ from \eqref{model:pb:00}, in which case we obtain $u_{FS} \p_x u_R  - \p_y^2 u_R = 0$. Recalling that $u_{FS}$ changes sign, we can parametrize the curve $\{u_{FS} = 0\}$ by $y = \Lambda_G(x)$. Noticing that $u_{FS}$ is nearly linear across $y = \Lambda_G(x)$, we can reduce the study further to the explicit equation $(y - \Lambda_G(x)) \p_x u_R - \p_y^2 u_R = 0$. Of course, at this level, the shape of $\Lambda_G(x)$ plays little role, and so we may flatten it by introducing, $z = \frac{y}{ \Lambda_G(x)}$, after which $y = \Lambda_G(x)$ gets mapped to $\{z = 1\}$. We are therefore led to consider the toy problem  
\begin{align}
((z-1) \p_x - \p_z^2) u_R = 0, \qquad (x, z) \in (1,1+ L) \times \mathbb{R}. 
\end{align}
Above, $L > 0$ is the tangential length-scale. For the toy-problem of this discussion, we also consider the simpler scenario of $z \in \mathbb{R}$ as opposed to $\mathbb{R}_+$, so as to isolate the discussion to the effect of the change in sign, as opposed to also involving effects arising from the physical boundary, $\{y = 0\}$ (which we eventually have to contend with). 

For this equation, $u_R$ exhibits \textit{forward-backward dynamics:} for $z < 1$ the evolution is backwards in $x$ whereas for $z > 1$ the evolution is forwards in $x$. The study of this type of linear problem has been performed in Pagani's works, \cite{Pagani1}, \cite{Pagani2}. The natural data to supplement such a forward-backward mixed-type problem is $u_{R}|_{x = 1, z > 1}$ to initiate the forward evolution for $z > 1$ and $u_R|_{x =1+ L, z < 1}$ to initiate the backward evolution on $z < 1$. 

Performing a formal derivative count (for this discussion we are interested in very close to the interface $\{z = 1\}$), one counts the weight of $z-1$ as $\p_z^{-1}$, and therefore this equation exhibits the formal scaling $\p_z^{-1} \p_x = \p_z^2$, or equivalently $\p_z = \p_x^{\frac13}$. 

We now restore the term $\p_y u_{FS}v_R$ as in \eqref{model:pb:00} and repeat the formal derivative count. Notice that since $u_{FS}$ is linear across the interface, $\p_y u_{FS}$ should act like $1$ near the interface. Therefore, we may simply think of this term as $v_R = - \int_0^y \p_x u_R$. Changing from $y$ to $z = \frac{y}{ \Lambda_G(x)}$ is just a shift, so this term counts as $\p_z^{-1} \p_x$, which also scales like the $((z-1) \p_x - \p_z^2)$ operator. Therefore, when we consider \eqref{model:pb:00}, the incompressibility enters at the leading order. 

Even from a heuristic standpoint, such a nonlocal, leading order effect has the potential to completely spoil the forward-backwards dynamics on $u_R$ (and it does!). Indeed, the term $(z-1) \p_x u_R$ dictates the forward evolution for $z > 1$ and backwards evolution for $z < 1$ essentially due to local considerations: we look at the sign of the coefficient, $z-1$, which is a pointwise feature. However, if $- \p_z^{-1} \p_x u_R$ is included in the equation and contributes at top order, such an operator is nonlocal vertically; it becomes completely unclear what is the sign, if any, of the operator in front of the $\p_x$ (very formally, $(z-1) - \p_z^{-1}$). Therefore, we can no longer conclude that $z > 1$ exhibits forward evolution and $z < 1$ exhibits backwards evolution (which is even a precursor to prescribing the correct data for the problem). 

In light of the above discussion, in order to proceed, we need to identify a mechanism to separate out the forward-backwards dynamics from the nonlocal effect of the incompressibility. \ul{Perhaps the central conceptual achievement of this work is (1) to identify these two effects as leading order, competing effects, and (2) to identify a mechanism to distinguish the role of these effects}. We insist on the word ``conceptual" above for two reasons. First, the idea we introduce, to be described below, marks a major departure from all prior works on the stationary Prandtl system: in a sense, the discussion above indicates that $u_R$ is not the correct quantity to study, at least near the interface. Second, the analytic task of closing estimates, to which the majority of the paper is of course devoted to, is highly nontrivial; these aspects will be discussed in Section \ref{section:ideas}.

\vspace{2 mm}

\noindent \textit{The Good Unknown, $\omega_g$: } Our first goal is to extract leading order forward-backward dynamics by exploiting cancellations in the Prandtl system. First of all, by localizing near the interface, $\{u_{FS} = 0\}$, we notice that $\p_y u_{FS} > 0$. We go to the vorticity formulation of \eqref{model:pb:00} by taking $\p_y$ of the equation. Such a procedure results in a cancellation of the ``cross terms" and yields 
\begin{align} \label{model:pb:vort}
u_{FS} \p_{x} \p_y u_R + \p_y^2 u_{FS} v_R - \p_y^3 u_R = 0
\end{align}
Now, the trick is to take the linear combination $\eqref{model:pb:vort} - \frac{\p_y^2 u_{FS}}{\p_y u_{FS}} \eqref{model:pb:00}$ which cancels out the $v_R$ term, and specifically leaves an equation of the type 
\begin{align} \label{omegaGeq}
u_{FS} \p_x \omega_g - \p_y^2 \omega_g + \mathcal{K}[u_R, \omega_g] = 0, \qquad \omega_g :=  \p_y u_R - \frac{\p_y^2 u_{FS}}{\p_y u_{FS}} u_R. 
\end{align}
Above, the operator $\mathcal{K}$ is a lower order operator on $\omega_g$ and also on the velocity itself $u_R$. Such an idea is motivated by those of Masmoudi-Wong, \cite{MW}, and the Crocco transform from the dynamical setting, but is completely novel for the stationary setting. We specifically need to use it because it allows us to extract a leading order forward-backward equation on the good unknown, $\omega_g$. 

\vspace{2 mm}

\noindent \textit{Three Regions \& the Role of the Critical Point(s): } It is clear that the use of the good unknown, $\omega_g$, relies on the monotonicity of the background profile, $\p_y u_{FS}> 0$, which in turn is only valid near the interface, $\{u_{FS} = 0\}$. Indeed, an examination of the profile from Figure \ref{Fig:2} actually shows that there exists a \textit{critical point} in the background, $\p_y u_{FS} = 0$. Formally, there is also a critical point at $y = \infty$ due to the decay of $\p_y u_{FS}$. We therefore have the picture shown in Figure \ref{Fig:CritPt}. In particular, we decompose the domain into three, mutually overlapping regions:
\begin{quote}

\textit{Crocco Region, $R_{\mathrm{Crocco}}$:} This region is above the curve $\p_y u_{FS} = 0$ and is where $\p_y u_{FS} > 0$ is strictly signed. We call this the ``Crocco Region." Here, due to the cancellation on the good unknown shown above, the forward-backwards dynamics is dominant (on $\omega_g$).

\textit{Lower von-Mise Region, $R_{\mathrm{VM}_-}$:} This region is below the curve $u_{FS} = 0$, and is where $u_{FS} < 0$ is signed negative. In particular, this region contains the critical line $u_{FS}' = 0$. In this region, we shall use the sign of $u_{FS}$ itself (as opposed to $u_{FS}'$) through the use of the so-called von-Mise transform in order to study a backwards evolution equation on a variant of the velocity, $U$, (see the definition \eqref{abs:4:bljk}). 

\textit{Upper von-Mise Region, $R_{\mathrm{VM}_+}$:} This region is (strictly) above the curve $u_{FS} = 0$. Here, $u_{FS} > 0$ and we have access to the von-Mise transform again in order to study a forward evolution equation on a variant of the velocity. However, the treatment of this region is not symmetric with the lower von-Mise region, again due to incompressibility. 

\end{quote}
Clearly, there is some choice to be made regarding exactly where these regions start and stop. For this discussion, we will remain vague regarding this. However, the important aspect that we capitalize on is that they are overlapping, which allows us to to transition from one type of analysis to the next. 

\begin{figure}[h]
\centering
\begin{tikzpicture}
\draw[ultra thick, <-] (0,3) node[above]{$y$} -- (0,0);
\node [left] at (0, 0) {$\textcolor{red}{(1, y_1)}$};
\filldraw [red] (0, 0) circle (2pt);
\draw[ultra thick, -] (3,1.5) -- (3,-2);
\node [right] at (3.2, 1.5) {$\textcolor{red}{(1 + L, y_{1 + L})}$};
\filldraw [red] (3, 1.5) circle (2pt);
\draw[ultra thick, -] (0,-2) -- (3,-2);
\node [below] at (1.5, -2) {$u_R = 0, \psi_R = 0$};
\draw[ultra thick,red, -] (0, 0) to[bend left] (3, 1.5);
\draw[ultra thick,red, -] (0, -1.5) to[bend left] (3, 0);
\node [above] at (1.5, 1.4) {$\textcolor{red}{\{u_{FS} = 0\}}$};
\node [above] at (1.5, -1.5) {$\textcolor{red}{\{\p_y u_{FS} = 0\}}$};
\draw[ultra thick, dashed, -] (0,-2) -- (0,0);
\draw[ultra thick, dashed, ->] (3,-1.5) -- (3,3);
\draw[ultra thick, dashed, ->] (3,-2) -- (4.5,-2);
\draw[ultra thick, dashed, ->] (0,-2) -- (-1.5,-2);
\node[right] at (4.5,-2) {$x$};
\end{tikzpicture}
\caption{$(x, y)$ Domain} \label{Fig:CritPt}
\end{figure}
\vspace{2 mm}

\noindent \textit{Codimension One Forward-Backward Dynamics for $u_R$: } Let us now re-interpret the dynamics in terms of the velocity, $u_R$. For this forthcoming discussion, $x$ plays the role of a parameter (meaning the formula \eqref{formula:1} below applies $x$ by $x$). Fix a value $y_{0} \in R_{\mathrm{VM}_-} \cap R_{\mathrm{Crocco}}$, and any $y \in R_{\mathrm{Crocco}}$. By inverting the formula \eqref{omegaGeq} for $u_R$ in terms of $\omega_g$, we obtain 
\begin{align} \label{formula:1}
u_R(y) = \underbrace{u_R(y_0) \frac{u_{FS}'(y)}{u_{FS}'(y_0)}}_{\text{Backwards von-Mise Dynamics}} + \underbrace{ u_{FS}'(y) \int_{y_0}^y \frac{\omega_g(\bar{y})}{u_{FS'}(\bar{y})} \ud \bar{y}.}_{\text{Forward-Backwards Dynamics}}
\end{align}
Due to the forward-backwards equation on $\omega_g$, the second, integral term on the right-hand side above can be thought of as ``being determined by the forward-backwards dynamics". However, to recover $u_R$, we also need to know the first term. This term represents information coming from beneath: it is determined by the von-Mise backwards evolution. 

It is important to note that we cannot choose $y_0 = 0$ in \eqref{formula:1}, which would eliminate the first term since $u_R|_{y = 0} = 0$, exactly due to the critical point of $u_{FS}$: $y_0$ must necessarily lie about the critical point. This shows the fundamental role of the critical point of $u_{FS}$: it necessarily implies that the von-Mise evolution is coupled with the Crocco evolution.

Nevertheless, this expression provides an interpretation to the original question raised at the beginning of the section: 
\begin{quote}
Q: Because $u_{FS} \p_x u_R$ and $\p_y u_{FS} v_R$ appear to have leading order scaling, how can we distinguish the forward-backwards dynamics from the role of the incompressibility?

A: The decomposition \eqref{formula:1} allows us to think of $u_R$ as having a decomposition into a codimension one component, $\omega_g$, which exhibits the forward backwards dynamics to leading order, and a one dimensional component which has to be determined through a separate analysis (it does not exhibit forward backwards dynamics). 
\end{quote}

\begin{remark} When we refer to the stream function, $\psi_R$, we further integrate in $y$ \eqref{formula:1}. In that case, the forward-backwards aspect is co-dimension two.
\end{remark}

\subsection{Prescribed Data}

Motivated by the understanding of the type of the problem, we are now prepared to specify the data we pose. More specifically, we want to prescribe data which is consistent with the three regions described above. For the Crocco region, we want to prescribe data on $\omega_g$ on the left and right, in a similar manner to Pagani for the forward backward problem $(z-1) \p_x - \p_z^2$. For the lower von-Mise region, we want to prescribe data on $u$ itself, but on the right at $\{x = 1 + L\}$ in order to initiate the backwards evolution. For the upper von-Mise region, we want to prescribe data again for $u$ itself, but on the left at $\{x = 1\}$ in order to initiate the forwards evolution. Of course, these need to be done in a consistent matter as one transitions from the lower von-Mise region to the Crocco region at $\{x = 1 + L\}$, and similarly from Crocco to upper von-Mise at $\{x = 1\}$. 

We highlight that even at this level, our analysis represents a significant departure from the prior works on \eqref{eq:PR:0}. To our knowledge, all prior mathematical results on \eqref{eq:PR:0} fall into one of two categories: (1) construction and analysis of specific self-similar solutions (for instance  \cite{Goldstein}, \cite{Stewartson}, \cite{YangLan}), or (2) studying the wellposedness/ finite-$x$-blowup of \eqref{eq:PR:0} as an initial value problem, with data prescribed at $\{x = 0\}$ (for instance, \cite{Oleinik}, \cite{GI2}, \cite{Serrin}, \cite{IyerBlasius}, \cite{MD}, \cite{Zhangsep}, \cite{Zhifei:smooth}). 

\vspace{2 mm}

\noindent \textit{Preliminaries: } We first fix a parameter $L > 0$ which describes our tangential length scales. Upon doing so, we define $y_1, y_2$  associated to the Falkner-Skan solution, \eqref{uFS}:
\begin{align}
u_{FS}(1, y_1)  = f'(y_1) = 0, \qquad u_{FS}(1+L, y_2) = (1+L)^n f'( \frac{y_2}{(1+L)^{\frac{1-n}{2}}} ) = 0
\end{align}
We will be looking for solutions perturbed around $u_{FS}$ as in \eqref{psir}. 
We also define the curve $\Lambda_G(x)$ to be the zero set of the background flow, $u_{FS}$, from \eqref{uFS}, and the free boundary $\Lambda(x)$ to be the perturbed zero set of the full solution, $u_P$:
\begin{align}
u_{FS}(x, \Lambda_G(x)) = 0, \qquad u_P(x, \Lambda(x)) =  u_{FS}(x, \Lambda(x)) + \eps u_R(x, \Lambda(x)) = 0 
\end{align}
It is useful to have the image below, Figure \ref{figure:dom}, in mind (which is a slightly more detailed version of Figure \ref{fig:org:mix}). 
\begin{figure}[h]
\centering
\begin{tikzpicture}
\draw[ultra thick, magenta, <-] (0,3) node[above]{\textcolor{black}{$y$}} -- (0,1.6);
\draw[ultra thick, -] (0,-0.5) -- (0,1.6);
\node [left] at (0, 0) {$\textcolor{red}{(1, y_1)}$};
\filldraw [red] (0, 0) circle (2pt);
\draw[ultra thick, -] (3,1.5) -- (3,-.5);
\draw[ultra thick, cyan, -] (3,-0.5) -- (3,-2);
\node [right] at (3.2, 1.5) {$\textcolor{red}{(1 + L, y_{1 + L})}$};
\filldraw [red] (3, 1.5) circle (2pt);
\node [right] at (3, -0.5) {$\textcolor{cyan}{(1+L, y_\ast)}$};
\filldraw [cyan] (3, -0.5) circle (2pt);
\draw[ultra thick, -] (0,-2) -- (3,-2);
\node [below] at (1.5, -2) {$u_R = 0, \psi_R = 0$};
\draw[ultra thick,red, -] (0, 0) to[bend left] (3, 1.5);
\node [left] at (0, -.5) {$\textcolor{blue}{(1, y_1(\eps))}$};
\filldraw [blue] (0, -.5) circle (2pt);
\node [left] at (0, 1.6) {$\textcolor{magenta}{(1, y^\ast)}$};
\filldraw [magenta] (0, 1.6) circle (2pt);
\draw[ultra thick,blue, -] (0, -.5) to[bend left] (3,1.5);
\node [below] at (1.5, .8) {$\textcolor{blue}{y = \Lambda(x)}$};
\node [above] at (1.5, 1.4) {$\textcolor{red}{y = \Lambda_G(x)}$};
\node [right] at (3, 0.7) {$\omega_g|_{x = 1 + L}$ given};
\node [right] at (3, 0.2) {in terms of $F_{\mathrm{Right}}$};
\node [right] at (3, -1.2) {\textcolor{cyan}{$u_R|_{x = 1 + L}$ given}};
\node [right] at (3, -1.8) {\textcolor{cyan}{in terms of $G_{\mathrm{Right}}$}};
\node [left] at (0, 1.1) {$\omega_g|_{x = 1}$ given};
\node [left] at (0, 0.6) {in terms of $F_{\mathrm{Left}}$};
\node [left] at (0, 2.7) {\textcolor{magenta}{$u_R|_{x = 1}$ given}};
\node [left] at (0, 2.1) {\textcolor{magenta}{in terms of $G_{\mathrm{Left}}$}};
\draw[ultra thick, dashed, -] (0,-2) -- (0,-.5);
\draw[ultra thick, dashed, ->] (3,1.5) -- (3,3);
\draw[ultra thick, dashed, ->] (3,-2) -- (4.5,-2);
\draw[ultra thick, dashed, ->] (0,-2) -- (-1.5,-2);
\node[right] at (4.5,-2) {$x$};
\end{tikzpicture}
\caption{$(x, y)$ Domain} \label{figure:dom}
\end{figure}
As shown in Figure \ref{figure:dom}, $y_1(\eps), y_{1 + L}$ are defined as follows:
\begin{align}
y_1(\eps)= \Lambda(1), \qquad y_{1 + L} = \Lambda(1 + L),
\end{align}
and we correspondingly define 
\begin{align} \label{gudef}
\gamma_u := u_R(1, y_1(\eps)), \qquad \gamma_\psi = \psi_R(1, y_1(\eps)). 
\end{align}
The role of these two numbers will become clear shortly. 

To help declutter notations, we use $'$ to denote $\p_y$ in the forthcoming discussion. We choose $y_\ast$ to be such that $u_{FS}'(1 + L, y_\ast) > \frac14$ and $u_{FS}(1 + L, y_\ast) < 0$ (so $y_\ast$ is in the overlap region between lower von-Mise and Crocco at the right boundary, $\{x = 1 + L\}$). Similarly, we choose $y^\ast$ such that $u_{FS}'(1, y^\ast) > \frac14$ and $u_{FS}(y^\ast) > 0$ (so $y^\ast$ is in the overlap region between Crocco and upper von-Mise at the right boundary, $\{x = 1\}$). We now define four cutoff functions, which are used to transition between the various regions. We define $\chi_{\mathrm{Cont}, +}$ to be a smooth, decreasing cutoff function satisfying 
\begin{align}
\chi_{\mathrm{Top}} := 1 - \chi_{\mathrm{Cont}, +}, \qquad \chi_{\mathrm{Cont}, +}(y) = \begin{cases} 1, \qquad y < y^\ast +1 \\ 0 \qquad y > y^\ast + 2 \end{cases} 
\end{align}
and $\chi_{\mathrm{Cont}, -}(y)$ will be a smooth, increasing cutoff function satisfying
\begin{align}
\chi_{\mathrm{Bot}} := 1 - \chi_{\mathrm{Cont}, -}, \qquad \chi_{\mathrm{Cont}, -}(y) = \begin{cases} 0, \qquad y < \frac{y_\ast}{4} \\ 1, \qquad y > \frac{y_\ast}{2}.  \end{cases}
\end{align}

\vspace{2 mm}

\noindent \textit{Structure of Prescribed Data:} We may now give the precise formulation of the prescribed data. First, we fix functions 
\begin{align*}
&F_{\mathrm{Left}}: \mathbb{R}_+ \rightarrow \mathbb{R}, && F_{\mathrm{Right}}: \mathbb{R}_- \rightarrow \mathbb{R}, \\
&G_{\mathrm{Left}}: [y^\ast, \infty) \rightarrow \mathbb{R}, && G_{\mathrm{Right}}: [0, y_\ast) \rightarrow \mathbb{R}.  
\end{align*}
Then, we prescribe the following:
\begin{align} \n
u_R|_{x = 1}(y) := & \chi_{\mathrm{Cont, +}}(y)(\gamma_{u} \frac{u_{FS}'(y)}{u_{FS}'(y_1(\eps))} + u_{FS}'(y) \int_{y_1(\eps)}^y \frac{F_{\mathrm{Left}}(\frac{1}{L^{\frac13}}(\frac{\bar{y}}{y_1(\eps)} - 1) )}{u_{FS}'} \ud \bar{y} )  \\ \label{PhiIota:1}
& + \chi_{\mathrm{Top}}(y) G_{\mathrm{Left}}(y)  , \qquad y > y^\ast\\ \label{PhiIota:2}
\omega_g|_{x = 1}(y) := & F_{\mathrm{Left}}(\frac{1}{L^{\frac13}}( \frac{y}{y_1(\eps)} - 1) ) , \qquad y_1(\eps) < y < y^\ast, \\ \label{PhiIota:3}
\omega_g|_{x = 1 + L}(y) := & F_{\mathrm{Right}}(\frac{1}{L^{\frac13}}(\frac{y}{y_{1 + L}} - 1) ), \qquad y_\ast < y < y_{1 + L}, \\ \n
u_R|_{x = 1 + L}(y) := & \chi_{\mathrm{Cont}, -}(y) ( - u_{FS}'(y) \int_{y_\ast}^{y_{1 + L}} \frac{F_{\mathrm{Right}}(\frac{1}{L^{\frac13}}(\frac{\bar{y}}{y_{1 + L}} - 1) )}{u_{FS}'} \ud \bar{y} \\  \n
&+ u_{FS}'(y) \int_{y_\ast}^y \frac{F_{\mathrm{Right}}(\frac{1}{L^{\frac13}}(\frac{\bar{y}}{y_{1 + L}} - 1) )}{u_{FS}'}  \ud \bar{y}) \\ \label{PhiIota:4}
& + \chi_{\mathrm{Bot}}(y) G_{\mathrm{Right}}(y) , \qquad 0 < y < y_{\ast}.
\end{align}


Several points are noteworthy about the setup: 

\begin{itemize}
\item[(1)] \textit{Nonresonant tangential scales:} The tangential length-scale we choose will need to be non-resonant. As we state in the main theorem below, there exist discretely many values $0 < L_1 < L_2 < \dots < L_{Max}$ which we need to avoid in choosing our tangential domain $x \in (0, L)$. These discrete resonant values correspond to the presence of zero eigenvalues for a particular elliptic operator we derive and study, and depend only on the choice of background flow, $u_{FS}$. Understanding the physical interpretation of these resonances is an interesting question for further investigation. 

\item[(2)] \textit{Data in a product space \& three regions:} As we have emphasized, there are three regions where different dynamics dominate. The prescription of four data elements reflects this aspect of the problem. The functions $[F_{\mathrm{Left}}, F_{\mathrm{Right}}]$ initiate the forward-backwards dynamics for $\omega_g$, whereas $G_{\mathrm{Left}}$ initiates the forward evolution for $u_R$, and $G_{\mathrm{Right}}$ initiates the backwards evolution for $u_R$. We have made a choice of using cutoff functions in \eqref{PhiIota:1} -- \eqref{PhiIota:4}. This is the reason $F_{\mathrm{Left}}$ enters in the expression for $u_R|_{x = 1}$ and similarly $F_{\mathrm{Right}}$ enters the expression for $u_R|_{x = 1 + L}$. We could have alternatively simply prescribed $G_{\mathrm{Left}}, G_{\mathrm{Right}}$, but then enforced compatibility conditions at $y = y_\ast, y^\ast$ by hand.  We chose to use smooth cutoffs, as they are more convenient for performing estimates, and we thus avoid these extra compatibility conditions at $y = y_\ast, y^\ast$ ($F_{\mathrm{Left}}, F_{\mathrm{Right}}, G_{\mathrm{Left}}, G_{\mathrm{Right}}$ will require their own compatibility conditions at $y = 0, y_1(\eps), y_{1 + L}$ as described below in \eqref{compat:1} -- \eqref{these:numbers}).

\item[(3)] \textit{The velocity $u$ at $\{x = 1\}$:} We apply the integral formula \eqref{formula:1} at $x = 1$, with $y_0 = y_1(\eps)$ and $y > y_1(\eps)$. This gives 
\begin{align} \label{gammauint}
u_R(1, y) = \underbrace{\gamma_u \frac{u_{FS}'(y)}{u_{FS}'(y_1(\eps))}}_{\text{Not prescribed}} + \underbrace{ u_{FS}'(y) \int_{y_1(\eps)}^y \frac{F_{\mathrm{Left}}(\frac{1}{L^{\frac13}}( \frac{\bar{y}}{y_1(\eps)} - 1) )}{u_{FS}'(\bar{y})} \ud \bar{y}.}_{\text{Prescribed}}
\end{align} 

The interpretation of this calculation is that we \textbf{cannot} prescribe $u$ on both sides. Instead we must leave the projection $\gamma_u$ onto the one-dimensional subspace spanned by $u_{FS}'(1, \cdot)$ on the left as free. This projection is determined through the global behavior of the solution. \textit{This is a conceptual departure from all prior works on the stationary Prandtl equation}, and displays the extremely subtle nature of studying the Prandtl system in the reversed regime. 

This is in contrast with the velocity $u$ at $\{x = 1 + L\} \cap \{0 < y < y_{1 + L}\}$. In this case, we see from \eqref{PhiIota:4} that for $y < y_\ast$, the velocity $u_R$ itself is prescribed. We can then apply the integral formula \eqref{formula:1} at $x = 1 + L$, with $y_0 = y_\ast$ for $y_\ast < y < y_{1 + L}$ to obtain 
\begin{align} \label{gammauint:yebud}
u_R(1 + L, y) = \underbrace{u_R(1 + L, y_\ast) \frac{u_{FS}'(y)}{u_{FS}'(y_\ast)}}_{\text{Prescribed}} +  \underbrace{u_{FS}'(y) \int_{y_\ast}^y \frac{F_{\mathrm{Right}}(\frac{1}{L^{\frac13}}( \frac{\bar{y}}{y_{1 + L}} - 1) )}{u_{FS}'(\bar{y})} \ud \bar{y}}_{\text{Prescribed}}.
\end{align} 

\item[(4)] \textit{Nonlinearly shifted vertical scale:} We have defined $y_1(\eps)$ to be a nonlinear correction to the known $y_1$: it is given by $y_1(\eps) = \Lambda(1)$. One subtlety, therefore, is that we can prescribe the profile function $F_{\mathrm{Left}}(\cdot)$ as data, but then we need to shift the argument. Note that this subtlety can be seen from the expression \eqref{gammauint:yebud}. Evaluating \eqref{gammauint:yebud} at $y = y_{1 + L}$, we see that $u_R(1 + L, y_{1 + L}) = 0$. Therefore, we may choose $y_{1 + L}$ to simultaneously satisfy 
\begin{align}
u_{FS}(1 + L, y_{1 + L}) = 0, \qquad u_R(1 +L, y_{1 + L}) = 0, \qquad u_P(1 + L, y_{1+ L}) =0.  
\end{align}
This explains why there is no discrepancy between $\Lambda(1 + L)$ and $\Lambda_G(1 + L)$ in Figure \ref{figure:dom}. On the other hand, this cannot be enforced on the left, as we do not prescribe data for $u_R(1, y)$ for $0 < y < y_1(\eps)$. 

\item[(5)] \textit{The $L^{-\frac13}$ rescaling of $F_{\mathrm{Left}}, F_{\mathrm{Right}}$: } We have introduced the tangential scale, $L$, into the data for $\omega_g|_{x = 1}, \omega_g|_{x = 1 + L}$ through \eqref{PhiIota:2}, \eqref{PhiIota:3}. While this is not strictly necessary, it is a natural choice for the following reason, motivated by even the toy problem considered above. If one is working with the equation $((z-1) \p_x - \p_z^2)u = G$ on the domain $(x, z) \in (1, 1 + L) \times \mathbb{R}$, and one wants estimates uniform in $L$, a natural way to proceed is to introduce the rescalings $t := \frac{x-1}{L}, Z = \frac{z-1}{L^{\frac13}}$. By doing so, data of the form \eqref{PhiIota:2} -- \eqref{PhiIota:3} essentially gets rescaled to the profile functions $(F_{\mathrm{Left}}(\cdot), F_{\mathrm{Right}}(\cdot))$. This is seen, for instance, in our Airy analysis, (as one specific example out of many, the form of \eqref{F:plus}). 
\end{itemize}


\vspace{2 mm}

\noindent \textit{Spaces for $[F_{\mathrm{Left}}, F_{\mathrm{Right}}, G_{\mathrm{Left}}, G_{\mathrm{Right}}]$: } We will fix an order of tangential regularity, $k_\ast$, which will be selected as any number $k_\ast \ge 7$. We fix the weight parameter $N_0 := 3 k_\ast + 5$. Then, we assume 
\begin{align}
\| F_{\mathrm{Left}} \langle \eta \rangle^3 \|_{H^{2 + 3k_\ast}_\eta} + \| F_{\mathrm{Right}} \langle \eta \rangle^3 \|_{H^{2 + 3k_\ast}_\eta}  
+ \|  G_{\mathrm{Right}} \|_{H^{3 + 3k_\ast}_y} + \|  G_{\mathrm{Left}} \langle y \rangle^{N_0} \|_{H^{3 + 2k_\ast}_y} < \infty.
\end{align} 
Above, the variable $\eta$ is a dummy variable. Here are the noteworthy points: 

\begin{itemize}

\item[(1)] \textit{Scalings:} Let us motivate the scalings above. Formally, $F_{\mathrm{Left}}, F_{\mathrm{Right}}$ measure the vorticity, $\omega_g$, which is at one order of vertical derivative more than the velocity, where the $G_{\mathrm{Left}}, G_{\mathrm{Right}}$ come in. This explains the $+2$ versus $+3$ discrepancy between the regularities. Next, very close to the interface, we have the scalings $\p_x \sim \p_y^3$ as motivated in the discussion above. Hence, $\p_x^{k_\ast}$ should translate to $3k_\ast$ vertical derivatives for $F_{\mathrm{Left}}, F_{\mathrm{Right}}$. The same scaling is valid for $G_{\mathrm{Right}}$ due to the physical boundary at $\{y = 0\}$. On the other hand, the $G_{\mathrm{Left}}$ enters in the purely parabolic regime, hence $\p_x^{k_\ast}$ costs $2k_\ast$ vertical derivatives. 

\item[(2)] \textit{Regularity and Choice of $k_\ast$:} We need to work in a higher regularity space $k_\ast \ge 7$ to control nonlinear contributions. At a high level, the restriction of $k_\ast \ge 7$ appears due to numerology with Sobolev embeddings. At a more specific level, there are essentially three contributors to the restriction: numerology for quasilinear estimates (seen, for instance, in \eqref{mr:b:1}), numerology for semilinear estimates (seen, for instance, in \eqref{mr:b:2}), and numerology for maximal regularity estimates (seen for instance in Lemma \ref{lmrsem}).

\item[(3)] \textit{Decay and Weights:} The choice of $N_0 = 3k_\ast + 5$ is made to make some convenient numerology work out in the analysis. The choice has to be tied to the eventual regularity, $k_\ast$ we propagate due to a \textit{downward weight cascade}, \eqref{def:NsubK}, where the weight depends on the regularity. In turn, this creates some numerology with the maximal regularity analysis; the exponent $N_k - J_k$ in \eqref{MR:norm} will be required to be $\ge 0$, which is guaranteed by the choice $N_0 = 3k_\ast + 5$. 

In a similar spirit, the fact that we use $\langle \eta \rangle^3$ is again due to convenience of some numerology. Specifically, one can examine \eqref{hm:1:R} -- \eqref{hm:2:R} for two instances. To help simplify the $\max\{ K, 3 \}$ power appearing, we simply choose the exponent of $3$. It is likely that one can decrease this weight at the expense of introducing more complicated estimates, though this is not the goal of the current work. 
\end{itemize}

\vspace{2 mm}

\noindent \textit{Compatibility Conditions on $[F_{\mathrm{Left}}, F_{\mathrm{Right}}, G_{\mathrm{Left}}, G_{\mathrm{Right}}]$: } In our setting, we will have three types of compatibility conditions on the data. 

\begin{itemize}

\item[(1)] \textit{Compatibility Conditions at $(x, y) = (1 + L, 0)$:} These are standard parabolic compatibility conditions for the backwards evolution where the Cauchy surface $\{x = 1 + L\}$ meets the boundary $\{y = 0\}$: 
\begin{align} \label{compat:1}
&\p_y^jG_{\mathrm{Right}}(0) = 0, \qquad j = 2 + k, 3 + k, \qquad 0 \le k \le k_\ast - 1.
\end{align}

This type of compatibility condition appears from the purely parabolic (albeit backwards) part of the problem, and therefore is completely standard, and can be found in essentially any works which require differentiating the Prandtl system (for example, \cite{GI2}, \cite{IyerBlasius}).  

\item[(2)] \textit{Compatibility Conditions at $(x, y) = \{ (1, y_1(\eps)), (1 + L, y_{1 + L}) \}$: } We will assume the decomposition 
\begin{align}
(F_{\mathrm{Left}}, F_{\mathrm{Right}}) = (\mathring{F}_{\mathrm{Left}}, \mathring{F}_{\mathrm{Right}}) + (Q_{\mathrm{Left}}, Q_{\mathrm{Right}}),
\end{align}
where $(\mathring{F}_{\mathrm{Left}}, \mathring{F}_{\mathrm{Right}})$ satisfies homogeneous conditions at $0$: 
\begin{align} \label{compat:3}
\p_{\eta}^{2 + 3k} \mathring{F}_{\mathrm{Left}}(0) =  \p_{\eta}^{2 + 3k} \mathring{F}_{\mathrm{Right}}(0) = 0. 
\end{align}
On the other hand, $(Q_{\mathrm{Left}}, Q_{\mathrm{Right}})$ will be taken to be compactly supported, $C^\infty$ functions that satisfy the following conditions at $0$:
\begin{align} \label{these:numbers}
\p_{\eta}^{2 + 3k} Q_{\mathrm{Left}}(0) =  q_{\mathrm{Left}, 2 + 3(k-1)},  \qquad \p_{\eta}^{2 + 3k} Q_{\mathrm{Right}}(0) =  q_{\mathrm{Right}, 2 + 3(k-1)}
\end{align}
The numbers $q_{\mathrm{Left}, 2 + 3(k-1)}, q_{\mathrm{Right}, 2 + 3(k-1)})$ are difficult to write explicitly, but are represented in \eqref{yeayea:1} -- \eqref{cc:1}. 

\begin{remark} The need for these compatibility conditions is as follows: if we consider an equation of the type $((z-1) \p_x - \p_z^2) \omega_g = G$. If we want to propagate one $\p_x$, then the data will be of the form $\p_x \omega_g|_{x = 1} = \frac{\p_z^2 \omega_g|_{x = 1} + G|_{x = 1}}{z-1}$. Therefore, due to the division by $z-1$, we should enforce the compatibility condition $\p_z^2 \omega_g|_{x = 1}(1) = - G|_{x = 1}(1)$. An analogous condition should hold at $x = 1 + L$. This is precisely what is being achieved by the $Q_{\mathrm{Left}}, Q_{\mathrm{Right}}$ ($z = 1$ in this analogy is playing the role of $y = y_1(\eps), y_{1 + L}$). 
\end{remark}

\begin{remark} Normally, even if $G$ contains lower order terms in $\omega_g$, $G|_{x = 1}(0)$ is also prescribed data and therefore the numbers \eqref{these:numbers} are explicitly determinable in terms of the given data. In our setting, $G$ will contain terms that depend on $u, \psi$, not just $\omega_g$. Due to the condition \eqref{gammauint}, this implies that $G$ will contain quantities that actually depend on the solution itself (that isn't prescribed). Nevertheless, we will be able to control these contributions precisely (see for example, \eqref{QbdLR}). 
\end{remark}

\item[(3)] \textit{Compatibility Conditions of Orthogonality Type: } It will turn out that we will need to further decompose $(\mathring{F}_{\mathrm{Left}}, \mathring{F}_{\mathrm{Right}})$ as follows. First, define $H^{K}_\eta(M_1) \times H^{K}_\eta(M_2)$ to be the following weighted Hilbert space:
\begin{align}
\| (F_{\mathrm{Left}}, F_{\mathrm{Right}}) \|_{H^{K}_\eta(M_1) \times H^{K}_\eta(M_2)} := \| F_{\mathrm{Left}} \langle \eta \rangle^{M_1} \|_{H^{K}_\eta} + \| F_{\mathrm{Right}} \langle \eta \rangle^{M_2} \|_{H^{K}_\eta \times H^{K}_\eta}.
\end{align}
Then, it will turn out that we want our data $(\mathring{F}_{\mathrm{Left}}, \mathring{F}_{\mathrm{Right}})$ to lie in a $2k_\ast$ co-dimension subspace of $H^{2 + 3k_\ast}_\eta(3) \times H^{2+3k_\ast}_\eta(3)$. To define this subspace,  there are $2k_\ast$ elements $(\bold{e}^{(k)}_0, \bold{e}^{(k)}_1)_{k = 1}^{k_\ast}$ where $\bold{e}^{(k)}_i \in H^{K}_\eta \times H^{K}_\eta$ and $2k_\ast$ linearly independent functionals $(\ell^{(k)}_0, \ell^{(k)}_1)_{k =1}^{k_\ast}$ such that
\begin{align*}
\ell^{(k)}_i(\bold{e}^{(k')}_{i'}) = \delta_{i = i'} \delta_{k =k'}.
\end{align*}
While these functionals are not explicit, they are fixed and universal. For the purposes of this paper, these do not play any role. We therefore invite the reader to consult our companion paper, \cite{IM22b}, for a more detailed description. We then define 
\begin{align} \n
\overline{\mathcal{P}}^{K}_{k_\ast}(M) := \{ (F_{\mathrm{Left}}, F_{\mathrm{Right}}) \in H^{K}_\eta(M) \times H^{K}_\eta(M) :& \ell_i^{(k)}(F_{\mathrm{Left}}, F_{\mathrm{Right}}) = 0, \\
& 0 \le i \le 1, 0 \le k \le k_\ast \}.
\end{align}
Our data will be of the form
\begin{align} \label{uanL1}
(\mathring{F}_{\mathrm{Left}}, \mathring{F}_{\mathrm{Right}}) = ( \bar{F}_{\mathrm{Left}}, \bar{F}_{\mathrm{Right}}) + \sum_{k = 1}^{k_\ast} (c_0^{(k)} \bold{e}^{(k)}_0 + c^{(k)}_1 \bold{e}^{(k)}_1),
\end{align}
where $( \bar{F}_{\mathrm{Left}}, \bar{F}_{\mathrm{Right}}) \in \overline{\mathcal{P}}^{2 + 3k_\ast}_{k_\ast}(3)$, and the $2k_\ast$ numbers $c_i^{(k)}$ are projections which, while non-explicit, are fixed. 

\begin{remark} These final type of compatibility conditions are further restrictions on $(\mathring{F}_{\mathrm{Left}}, \mathring{F}_{\mathrm{Right}})$ that are required in order to guarantee smoothness of the solution, as introduced by \cite{DMR}. In this paper, we are proving \textit{a-priori} estimates given any $(\mathring{F}_{\mathrm{Left}}, \mathring{F}_{\mathrm{Right}})$, and in our companion paper, \cite{IM22b}, we apply these estimates to the restricted class given by \eqref{uanL1}. Therefore, this type of further restriction on the data will not appear in this paper. 
\end{remark}

\end{itemize}

\vspace{2 mm}

%

\subsection{Main Theorem}

In order to state our theorems, we need to use the $\mathcal{Z}$ norm, which, while precisely defined in \eqref{Z:norm:ult}, will be unnecessary to fully introduce here. Instead, we inform the reader that this $\mathcal{Z}$ norm is stronger than the following: 
\begin{align}
\| \psi_R \|_{\mathcal{Z}} \gtrsim \sum_{k = 0}^{k_\ast} ( \| \psi_{R,k} \|_{L^\infty} +\| u_{R,k} \|_{L^\infty} + \| \p_y u_{R,k} \|_{L^\infty_x L^2_y} +\| \p_y^2 u_{R,k} \|_{L^2_{xy}} ),
\end{align}
where in turn, $\psi_{R, k}, u_{R,k}$ are normalized tangential derivatives defined in \eqref{Rknot}. In the true definition \eqref{Z:norm:ult}, one sees weights in $y$ appearing and much more refined information. 

We first state a special case of our main result (which we choose to state as a theorem in its own right so that we can refer to it in the body of the paper). 
\begin{theorem}[Small $L$ Version] \label{main:thm} Let $0 < L << 1$ relative to universal constants, and let $0 < \eps << L$. Assume $(\mathring{F}_{\mathrm{Left}}, \mathring{F}_{\mathrm{Right}}, G_{\mathrm{Left}}, G_{\mathrm{Right}})$ are given functions satisfying the compatibility conditions \eqref{compat:1}, \eqref{compat:3}. Fix $k_\ast \ge 7$, and let $N_0 = 3k_\ast + 5$. Assume 
\begin{align}
\| \mathring{F}_{\mathrm{Left}} \langle \eta \rangle^3 \|_{H^{2 + 3k_\ast}_\eta} + \| \mathring{F}_{\mathrm{Right}} \langle \eta \rangle^3 \|_{H^{2 + 3k_\ast}_\eta}  
+ \|  G_{\mathrm{Right}} \|_{H^{3 + 3k_\ast}_y} + \|  G_{\mathrm{Left}} \langle y \rangle^{N_0} \|_{H^{3 + 2k_\ast}_y} < \infty.
\end{align} 
Assume there exists a solution, $\psi_R$ to \eqref{exp:eq:intro:1} with data \eqref{PhiIota:1}  --\eqref{PhiIota:4} satisfying the bound 
\begin{align} \label{bootstrap:thm:1}
\| \psi_R \|_{\mathcal{Z}} < \eps^{-\frac14} \qquad \text{(Bootstrap)}
\end{align} 
Then $\psi_R$ satisfies the \textit{a-priori} stability estimate 
\begin{align}
\| \psi_R \|_{\mathcal{Z}} \lesssim L^{\frac16}( \| \mathring{F}_{\mathrm{Left}} \langle \eta \rangle^3 \|_{H^{2 + 3k_\ast}_\eta} + \| \mathring{F}_{\mathrm{Right}} \langle \eta \rangle^3 \|_{H^{2 + 3k_\ast}_\eta} ) 
+ \|  G_{\mathrm{Right}} \|_{H^{3 + 3k_\ast}_y} + \|  G_{\mathrm{Left}} \langle y \rangle^{N_0} \|_{H^{3 + 2k_\ast}_y}.
\end{align}
\end{theorem}

 The most general result we have is the following, in which the smallness $0 < L << 1$ is replaced with a more general condition:
\begin{theorem}[Nonresonant $L$ Version] \label{thm:res} Let $0 < L_{\mathrm{Max}} < \infty$ be any given tangential length scale. There exist a discrete set of resonant lengths, $\{L_k\}_{k = 1}^N$, so that the following holds. Let $L \in (0, L_{\mathrm{Max}}) \setminus \cup_{k = 1}^N L_k$. There exists an $\eps_\ast > 0$, small relative to $L_{\mathrm{Max}}$ and relative to $\min |L - L_k|$ such that for $0 < \eps < \eps_\ast$, the same result of Theorem \ref{main:thm} holds, namely, assuming there exists a solution, $\psi_R$ to \eqref{exp:eq:intro:1} with data \eqref{PhiIota:1}  --\eqref{PhiIota:4} satisfying the bootstrap bound \eqref{bootstrap:thm:1}, then $\psi_R$ satisfies the \textit{a-priori} stability estimate 
\begin{align}
\| \psi_R \|_{\mathcal{Z}} \lesssim  \| \mathring{F}_{\mathrm{Left}} \langle \eta \rangle^3 \|_{H^{2 + 3k_\ast}_\eta} + \| \mathring{F}_{\mathrm{Right}} \langle \eta \rangle^3 \|_{H^{2 + 3k_\ast}_\eta} 
+ \|  G_{\mathrm{Right}} \|_{H^{3 + 3k_\ast}_y} + \|  G_{\mathrm{Left}} \langle y \rangle^{N_0} \|_{H^{3 + 2k_\ast}_y},
\end{align}
where the constant in the above inequality depends on $L_{\mathrm{Max}}$ and on small values of $\min_{1 \le k \le N} |L - L_k|$. 
\end{theorem}
We refer the reader to Section \ref{section:ideas} in which we provide a detailed overview of the ideas we introduce and the overall strategy we follow to obtain Theorems \ref{main:thm} and \ref{thm:res}.  

\begin{remark} As stated, the above results are \textit{a-priori} stability estimates. A good analogue to keep in mind to understand the presentation of this theorem is to the (clearly much simpler) quasilinear elliptic equation $(1 + \eps u) \Delta u = F$. In order to prove existence for such a problem, given the appropriate data, one needs to couple \textit{a-priori} estimates (in this example quite straightforward) with an iteration scheme of the sort $(1 + \eps u^{(n)}) \Delta u^{(n+1)} = F$ through which a contraction mapping argument is performed. The role of the bootstrap \eqref{bootstrap:thm:1} will essentially be replaced by uniform bounds on the iterates along the iteration scheme, which is carried out in our companion paper, \cite{IM22b}. In the current setting of the reversed Prandtl system, the most challenging part of the argument is the obtention of the \textit{a-priori} estimates, which we perform in this paper. We also note that in order to make our estimates as general as possible, we develop them in an abstract setting first. More precisely, one can see from \eqref{genbd:1} that we will consider \textit{general backgrounds, $u_B$} which will serve as a placeholder for $u^{(n)}$ in the iteration scheme.  
\end{remark}

The argument leading to the construction of the solution is given in a companion paper, \cite{IM22b}, which essentially performs an iteration which uses as a ``black-box" the stability analysis of this paper (Theorems \ref{main:thm}, \ref{thm:res}). An iteration argument is needed for two main reasons. First, as we emphasized above, for \textit{any} quasilinear equation (as a typical example $(1 + \eps u) \Delta u = F$), we would need to iterate the coefficient $(1 + \eps u^{(n)}) \Delta u^{(n+1)} = F$ for the sake of a contraction mapping type argument. Second, in the spirit of the recent work \cite{DMR} on a related mixed-type problem, we need to add further compatibility conditions on $(\mathring{F}_{\mathrm{Left}}, \mathring{F}_{\mathrm{Right}})$ as in \eqref{uanL1}.

We can now state the main result that combines the stability theorems, Theorems \ref{main:thm}, \ref{thm:res}, with the existence result of \cite{IM22b}:
\begin{theorem}[Existence and Stability] \label{thm:twomain} Let $k_\ast \ge 7$, and $N_0 = 3k_\ast + 5$. Fix an arbitrary $( \overline{F}_{\mathrm{Left}}, \overline{F}_{\mathrm{Right}}) \in \overline{\mathcal{P}}^{2 + 3k_\ast}_{k_\ast}(3)$, and $( G_{\mathrm{Left}}, G_{\mathrm{Right}}) \in H_y^{3 + 2k_\ast}(N_0) \times H_y^{3 + 3k_\ast}(0)$. There exists an operator 
\begin{align}
\mathcal{N}_{side}: \overline{\mathcal{P}}^{2 + 3k_\ast}_{k_\ast}(3) \times H_y^{3 + 2k_\ast}(N_0) \times H_y^{3 + 3k_\ast}(0) \mapsto H^{2 + 3k_\ast}_\eta(3) \times H^{2+3k_\ast}_\eta(3),
\end{align}
which is a compact perturbation of the identity (when restricted to $\overline{\mathcal{P}}^{2 + 3k_\ast}_{k_\ast}(3) \subset  H^{2 + 3k_\ast}_\eta(3) \times H^{2+3k_\ast}_\eta(3)$), such that the following holds. Let $(\mathring{F}_{\mathrm{Left}}, \mathring{F}_{\mathrm{Right}}) := \mathcal{N}_{side}[\overline{F}_{\mathrm{Left}}, \overline{F}_{\mathrm{Right}}, G_{\mathrm{Left}}, G_{\mathrm{Right}}]$. Let $L \in (0, L_{Max}) \setminus \cup_{k = 1}^N L_k$. There exists an $\eps_\ast > 0$, small relative to $L_{\mathrm{Max}}$ and relative to $\min |L - L_k|$ such that for $0 < \eps < \eps_\ast$, 

\begin{itemize}

\item There exists a unique solution $\psi_R$  such that $\omega_g$ achieves the data as shown in \eqref{PhiIota:1} -- \eqref{PhiIota:4}. Moreover, this solution has the regularity $\p_x^k u_R \in L^2_x H^1_y$ for each $k$, $0 \le k \le k_{\ast}$.

\item The solution obeys the stability estimates 
\begin{align}  \label{stable:2:ult}  
\| \psi_R \|_{\mathcal{Z}} \lesssim  \| \overline{F}_{\mathrm{Left}} \langle \eta \rangle^3 \|_{H^{2 + 3k_\ast}_\eta} + \| \overline{F}_{\mathrm{Right}} \langle \eta \rangle^3 \|_{H^{2 + 3k_\ast}_\eta} 
+ \|  G_{\mathrm{Right}} \|_{H^{3 + 3k_\ast}_y} + \|  G_{\mathrm{Left}} \langle y \rangle^{N_0} \|_{H^{3 + 2k_\ast}_y}
\end{align}

\item The numbers $(c_i^{(k)})$ from \eqref{uanL1} are determined by the prescribed functions $(\overline{F}_{\mathrm{Left}}, \overline{F}_{\mathrm{Right}})$ and $(G_{\mathrm{Left}}, G_{\mathrm{Right}})$, and obey the estimate 
\begin{align} \label{A:bd:thm}
|c_i^{(k)} | \lesssim ( \| \overline{F}_{\mathrm{Left}} \langle \eta \rangle^3 \|_{H^{2 + 3k_\ast}_\eta} + \| \overline{F}_{\mathrm{Right}} \langle \eta \rangle^3 \|_{H^{2 + 3k_\ast}_\eta} 
+ \|  G_{\mathrm{Right}} \|_{H^{3 + 3k_\ast}_y} + \|  G_{\mathrm{Left}} \langle y \rangle^{N_0} \|_{H^{3 + 2k_\ast}_y}),
\end{align}
where the constants above depend on large $L_{\mathrm{Max}}$ and on small values of $\min_{1 \le k \le N} |L - L_k|$.
\end{itemize}
\end{theorem}

\subsection{Related Literature} \label{existing:Lit}

The boundary layer theory originated with Prandtl's seminal 1904 paper, \cite{Prandtl}, which developed the theory in precisely the present setting: for 2D, steady flows over a plate. Since its inception, the boundary layer theory has had monumental impacts in various domains of physics and engineering, perhaps most notably in aerodynamics. Despite it being a classical physical theory, mathematical results are less widespread. In fact, in Prandtl's original work, he states the following: 

\begin{quote}
``The most important practical result of these investigations is that, in certain cases, the flow separates from the surface at a point [$x_\ast$] entirely determined by external conditions... As shown by closer consideration, the necessary condition for the separation of the flow is that there should be a pressure increase along the surface in the direction of the flow." (L. Prandtl, 1904, \cite{Prandtl})
\end{quote}

 The stationary Prandtl equations are a classical system, whose wellposedness theory was initiated by Oleinik in the classical works, \cite{Oleinik}. Indeed, the dichotomy between favorable and unfavorable pressure gradient (which was actually pointed out in the original work of Prandtl) appears in Oleinik's results: she obtains local in $x$ wellposedness results in general, and, in the case of favorable pressure gradients, global wellposedness. 

The local wellposedness results of Oleinik rely upon a nonlinear change of variables, which faces difficulty upon further differentiation and therefore cannot be employed to obtain higher regularity estimates. Higher regularity was obtained in \cite{GI1} through energetic arguments, and in \cite{Zhifei:smooth} using maximum principle type arguments. 

In the case of $\p_x p_E(x) \le 0$, the result of Serrin, \cite{Serrin} characterizes the self-similar Falkner-Skan profiles as asymptotic in $x$ ``attractors": general solutions to the Prandtl system approach self-similarity as $x \rightarrow \infty$. This result was revisited using different techniques in \cite{IyerBlasius}. While the methods employed by Serrin are maximum principle based, \cite{IyerBlasius} relies more on weighted energy/ virial type estimates. 

In the case of $\p_x p_E(x) > 0$, separation can occur, and we refer to the recent work \cite{MD} which shows the existence of an open set of datum that undergoes the Goldstein singularity. The methods from \cite{MD} rely upon making a self-similar change of variables, and subsequently using modulation theory techniques. We refer also to the work \cite{Zhangsep}, which provides a different, maximum principle based approach to separation. 

Regarding flow reversal in the stationary Prandtl equations, several authors have investigated the structure of the self-similar Falkner-Skan profiles, \eqref{FS:beta}. This began with classical investigations by Hartree, \cite{Hartree}, Stewartson, \cite{Stewartson}, Brown and Stewartson, \cite{BrownStewartson}, Hastings, \cite{Hastings}, and there also are some more recent works in this direction: \cite{YangLan}. As far as we know, all known results on reversal are regarding the particular self-similar profiles \eqref{FS:beta}, which are extremely important from the point of view of applications. Indeed, the reader can consult the discussions in the sources \cite{Hartree}, \cite{Stewartson}, \cite{BrownStewartson} which discuss the various applied settings in which the FS profiles are used (airfoil design, etc...). The present paper is the first to consider \textit{stability} of these profiles.  

As stated earlier, we note that due to the mixed-type nature of the problem, it is a nontrivial task to go from \textit{a-priori} stability estimates to the construction of the solution. This is achieved in our companion work, \cite{IM22b}. We refer the reader also to the recent work of \cite{DMR} which addresses the existence of strong solutions to the mixed-type problem $u \p_x u - \p_y^2 u = 0$. 

There is a separate but very important question of \textit{validity} of the Prandtl ansatz in describing the inviscid limit of Navier-Stokes. Here, there are very few works in the stationary setting, which are all relatively recent. The first results were local in the tangential variable, for instance \cite{GI1}, \cite{GI3}, and also \cite{Varet-Maekawa}. Recently, the works \cite{IM2}, \cite{IM21} established stability of the Prandtl ansatz globally in the tangential variable with asymptotics in $x$. See also the related works, \cite{Gao-Zhang}, \cite{GN}.

For unsteady flows, the picture is rather different than the stationary setting, and for which there is a substantial body of works studying both (1) wellposedness results as well as (2) stability results for the inviscid limit of Navier-Stokes. We focus the discussion here only wellposedness results, since the present paper is not concerned with stability under the inviscid limit. For wellposedness results, we refer the reader to \cite{Oleinik}, \cite{AL}, \cite{MW},  \cite{KMVW} assuming either monotonicity or partial monotonicity in Sobolev spaces. If one removes the hypothesis of monotonicity, wellposedness results are in smoother spaces (either analytic or Gevrey spaces):  \cite{GVDie},  \cite{GVM}, \cite{Kuka},\cite{LMY}, \cite{Caflisch1}. In Sobolev spaces, the unsteady Prandtl equations are illposed: \cite{GVD}. We refer the reader also to finite time blowup results: \cite{EE}, as well as unsteady boundary layer separation: \cite{Collot1}, \cite{Collot2}, \cite{Collot3} using blowup arguments.

The above discussion is not comprehensive, and we have elected to provide a more in-depth of the steady theory due to its relevance to the present paper. We refer to the articles, \cite{BT1}, \cite{E}, and references therein for a more complete review of other aspects of the boundary layer theory.

\subsection{Notational Conventions}

\noindent \underline{Coordinate Systems:} We will be working with four coordinate systems in our analysis. The $(x, y)$ coordinates are the original coordinates in which the system \eqref{eq:PR:0} is set. The $(s, z)$ coordinates are self-similar and adapted to the free-boundary; they are defined in \eqref{z:variable}. The $(s, Y)$ and $(s, Z)$ coordinate systems are local to the ``interface", $\{y = \Lambda(x)\}$, and are defined in \eqref{change:1} and \eqref{philz1}, respectively. We denote by $\mathcal{F}$ the usual Fourier transform defined for functions, $f \in L^2(\mathbb{R})$: 
\begin{align}
\mathcal{F}[f](\xi) = \widehat{f}(\xi) := \int_{\mathbb{R}} e^{- i \xi s} f(s) \ud s,
\end{align}
and $\mathcal{F}^{-1}$ the inverse Fourier transform. We sometimes use the notation $'$ for vertical differentiation, either in $\p_y$, $\p_z$, or $\p_Z$, made clear by context. 

\vspace{2 mm}

\noindent \underline{Norms:} We often have to use norms in different coordinate systems. We choose to emphasize this by using subscripts for the variable (for example, $\| f \|_{L^2_{xy}}$ versus $\| f \|_{L^2_{sz}}$). If there is no possible confusion from the context, we omit the subscript.

\vspace{2 mm}

\noindent \underline{Parameters:} The main small parameters that we have are $L$ (only in the course of proving Theorem \ref{main:thm}), the tangential length scale, and $\eps$, the size of the perturbation (and hence the pre-factor in front of the quadratic terms). There is also the parameter $\bar{L}$, which will eventually shown to be equivalent to $L$ via \eqref{LbarL}. 

\vspace{2 mm}

\noindent \underline{Data on the Sides:} Due to the ``fast rescaling" appearing in \eqref{PhiIota:1} -- \eqref{PhiIota:4}, we need to introduce two more variables: 
\begin{align*}
\eta := \frac{y-\Lambda(x)}{\Lambda(x) L^{\frac13}} = \frac{z-1}{L^{\frac13}}, \qquad \rho = \frac{Z}{\bar{L}^{\frac13}}.
\end{align*}
These will be primarily used to describe data elements associated to $(F_{\mathrm{Left}}, F_{\mathrm{Right}})$. We use the subscripts ``Left" and ``Right" to refer to functions on the prescribed boundaries, $\{x = 1\} \cap \{y > y_1(\eps)\}$ and $\{x = 1 + L\} \cap \{0 < y < y_{1 + L}\}$, respectively. 

\vspace{2 mm}

\noindent \underline{Cutoff Functions:} Throughout the paper, $\chi$ will stand for a smooth, normalized cutoff function satisfying 
\begin{align} \label{chidef}
\chi(p) := \begin{cases} 0, \qquad p < -\frac{1}{5} \\ 1, \qquad - \frac{1}{10}< p < \frac{1}{10} \\ 0, \qquad p > \frac{1}{5}. \end{cases}
\end{align}
While the choice of $\frac{1}{5}, \frac{1}{10}$ above may seem artificial, it allows for some convenience, for example in definitions of \eqref{ext:gamma:1} it allows us to put two non-overlapping cutoffs side by side. Several other cutoff functions will be introduced throughout the paper, but as far as possible will be written in terms of this $\chi(\cdot)$. A common type of modification is to apply these cutoffs at different scales. Given $\ell > 0$, define $\chi_{\ell}(p) := \chi(\frac{p}{\ell})$.

\section{Main New Ideas \& Overview of the Proof} \label{section:ideas}

In this section, our aim is two-fold. First, we want to provide an outline of the proof and organize our strategy into the main steps. Second, by doing so, we will be able to highlight the difficulties that arise and the several main ingredients we need to employ in order to obtain our result. 

\subsection{Formulations of the Prandtl System, Section \ref{SectionCH}}

We begin by displaying the main system that we analyze. We seek solutions that are small, $O(\eps)$, perturbations of the background profile $[u_{FS}, v_{FS}]$, \eqref{psir}, which results in \eqref{exp:eq:intro:1}. The coefficient $u_P = u_{FS} + \eps u_R$ in front of the $\p_x$ term above determines the change of sign. This coefficient is now nonlinear as opposed to the linearized discussion above. An essential aspect of our approach is a use of the ``tame principle" from quasilinear equations: if we take higher tangential derivatives of the equation, the governing coefficient $u_P = u_{FS} + \eps u_R$ remains the same. Working in such a higher order space enables us to think of $u_R$, despite being nonlinear, as a lower order quantity. 

\subsubsection{The $(x, y)$ Coordinates}

Our strategy will be to commute a rescaled $\p_x$ into \eqref{exp:eq:intro:1}. Even at the outset, we notice a choice. The level set $\{u_{FS} + \eps u_R = 0\}$ clearly does not coincide with any level set of $y$ itself. Therefore, one may decide to take tangential derivatives to this level set, which essentially corresponds to taking $\p_s$ derivatives, \eqref{chain:rule:1}. Normally in free boundary problems, one has a boundary condition on the free boundary $\{u_{FS} + \eps u_R = 0\}$ that is only respected with derivatives tangent to the free boundary, which requires tangential derivatives. In our case, there is no boundary condition on this free boundary, and hence we may equally well elect to take flat $\p_x$ derivatives. It turns out to be natural to commute rescaled $L \p_x$ derivatives to \eqref{exp:eq:intro:1}. This is essentially due to the fact that $x \in (1, 1 + L)$, and hence $L\p_x$ respects the tangential length scale of the domain. Doing so produces an equation of the form 
\begin{align} \label{urk}
&u_P \p_x u_{R, k} + u_{Py} v_{R,k} + \bold{C}_1  \p_y u_{R,k} + \bold{C}_2 u_{R,k}  +\bold{C}_3 \psi_{R,k} - \p_y^2 u_{R, k}= \overline{ \mathcal{R}}_{k}, 
\end{align}
for $u_{R,k} := (L\p_x)^k u_R$, $\bold{C}_i$ (eventually nonlinear) coefficients, and $\overline{\mathcal{R}}_k$ to be lower order source terms. This form of the equation will be our first formulation, and is thought of as the ``parent formulation" (many offspring arise from this upon performing transformations, etc...). In fact, we are careful to develop estimates for $\bold{C}_i, \overline{R}_k$ as abstract coefficients and source terms, and even with $u_P = u_{FS} + \eps u_R$ being replaced by a more general background $\bar{U} := u_{FS} + \eps u_B$. This is because we want to eventually apply our analysis not just to the full nonlinear equations, but also on a more general linearized version of the equations for an iteration scheme (in our companion existence paper, \cite{IM22b}). 

Going forward in this section, we will drop the subscript $k$ in \eqref{urk}, which allows us to treat all $k$ derivatives in an essentially unified way, and also to de-clutter the notations. Therefore, we will now consider 
\begin{align} \label{urk:abs}
&\bar{U} \p_x u_{R} + \bar{U}_y v_{R} + \bold{C}_1  \p_y u_{R} + \bold{C}_2 u_{R}  +\bold{C}_3 \psi_{R} - \p_y^2 u_{R}=  \mathcal{R}, 
\end{align}
with the understanding that the ideas we develop can be applied ``$k$-by-$k$" to \eqref{urk}.

\subsubsection{The $(s, z)$ Modulated Coordinates}

One observes that the level set of interest is $\{u_P = 0\}$. We parametrize this (nonlinear) level set by $\Lambda(x)$, and subsequently introduce
\begin{align} \label{z:variable}
\frac{\ud s}{\ud x} = \frac{1}{\Lambda(x)^2}, \qquad z := \frac{y}{\Lambda(x)} = \frac{y}{\Lambda_G(x) + \eps \Xi(x)}, \qquad  \bar{U}(x, \Lambda(x)) = 0
\end{align}
This transformation has the effect of straightening out the level set $y = \Lambda(x)$ into $z = 1$, and is motivated by corresponding strategies in free boundary problems. On the other hand, the rescaling of the $x$ variable is done in order to preserve the scaling of the equation. Such a transformation converts \eqref{urk:abs} into 
\begin{align} \label{forme}
\bar{W} \p_s u + v \p_z \bar{W}  - \p_z^2 u + \bold{c}_1 \p_z u + \bold{c}_2 u + \bold{c}_3 \psi   = \bold{R},
\end{align}
for newly defined coefficients $\bar{W}, \bold{c}_i$ and source term $\bold{R}$ (these will be defined precisely in \eqref{dubbar:1}, \eqref{defn:bold:c1}, and \eqref{res:in:2}).

\subsubsection{The $(s, Z)$ Coordinates: A New Perspective on Crocco and Masmoudi-Wong}

One of the steps of the analysis will be through exploiting specific features of the Airy functions, solutions to the ODE \eqref{AIRY:ODE}, which can only be done if our system of study is \textit{exactly} of the form 
\begin{align} \label{offo}
Z \p_s \Omega_I - \p_Z^2 \Omega_I = F_I,
\end{align}
for appropriately defined variables, $Z, \Omega_I$, and where $F_I$ contains lower order linear terms, that are small with respect to small $L$, or nonlinear terms that contain a prefactor of $\eps$ so that the terms in $F_I$ can be regarded as small perturbations. The motivation for invoking these Airy functions will be explained in the next subsection, Section \ref{AIRY::1}. 

For now, we turn to the task of transforming \eqref{forme} into this precise form. The conceptual underpinnings of these transformations have been motivated earlier, surrounding the discussion of \eqref{model:pb:vort}. Here, we develop this in a manner more similar to the actual proof which relies on a change of variables. In order to proceed at a conceptual level, we drop the $\bold{c}_1 \p_z u + \bold{c}_2 u + \bold{c}_3 \psi$ terms from \eqref{forme} for the rest of this discussion, as these are lower order terms that eventually become perturbative.

Localizing to near the zero set of $\bar{W}$, we can introduce the change of variables and change of function: 
\begin{align} \label{Y}
Y = \bar{W}, \qquad \phi(s, Y) = \psi(s, z), \qquad V(s, Y) := \p_Y^2 \phi(s, Y). 
\end{align}
Considering now the equation on $\phi$, we see a substantial cancellation which effectively eliminates the nonlocal term $v \p_z \bar{W}$ from \eqref{forme}, and produces 
\begin{align} \label{eq:om:1}
Y \p_s V - |\bar{W}_z|^2 \p_Y^2 V = \text{lower order terms}. 
\end{align} 
Thus, the non-locality is, in a sense, eliminated through the change of variables, \eqref{Y}. 

From \eqref{eq:om:1}, we would now like to eliminate the term $\bar{W}_z^2$ in front of the diffusion. We achieve this by splitting the diffusion coefficient into $\bar{W}_z^2 - \bar{W}_z^2(s, 1)$, which we move to the right-hand side and subsequently use the vanishing, and $\bar{W}_z^2(s, 1)$, which is purely dependent on $s$. To do this, we introduce another transformation, which is a simple rescaling, $Z = p(s, Y) = \frac{Y}{\bar{W}_z(s, 1)^{\frac23}}$. This process is carried out in Section \ref{subsection:local}. We then set $\Omega(s, Z) = V(s, Y)$.

To appreciate further the introduction of $V(s, Y)$ and $\Omega(s, Z)$ in \eqref{Y}, we can compute an explicit representation of $\Omega$ which turns out to read 
\begin{align} \label{om:sy}
\Omega(s, Z) = \frac{1}{\bar{W}_z^2} (\p_z u - \frac{\bar{W}_{zz}}{\bar{W}_z}u),
\end{align}
which is clearly reminiscent the good unknown, $\omega_g$, \eqref{omegaGeq} of Masmoudi-Wong, \cite{MW}. The $\Omega_I$ appearing in \eqref{offo} is technically a localized version of $\Omega$ near the interface, thereby avoiding the critical points where $\bar{W}_z = 0$, which is required in order that the change of variables $Y = \bar{W}$ is well-defined. 

Our strategy of a change of variables $Y = \bar{W}$ is, on a technical level, reminiscent of the Crocco transform (which relies on the monotonicity of the background $\bar{W}$). To our knowledge, \textit{neither} of these ideas have been used in the study of the stationary Prandtl equation. In this work, we will need to use these ideas extensively: the monotonicity $\bar{W}_z > 0$ allows a mechanism to transition between regions of signed background, $\bar{W} < 0$ and $\bar{W} > 0$ (as we have emphasized in the discussion of the three regions surrounding Figure \ref{Fig:CritPt}).

\subsubsection{von-Mise type good unknowns}

The $(s, Y)$ coordinates introduced in \eqref{Y} (and by implication, the $(s, Z)$ coordinate system introduced thereafter) are only defined locally near the interface, $\{z = 1\}$. This is due to the definition $Y = \bar{W}$, where $\p_z \bar{W}$ degenerates as $z \rightarrow \infty$ and also in the interior of $0 < z < 1$. Since our analysis is not localized to just near $\{z = 1\}$, but needs to eventually extend globally to $0 < z < \infty$, we need a different set of transformations which apply globally in $z$ (in particular in the ``outer region"). 

In this region, the formulation we use is in the self-similar variables, $(s, z)$, \eqref{forme}. We observe the presence of the nonlocal term $v \p_z \bar{W}$. By noticing that $v = - \p_s \psi$, we identify this term as being transported, and therefore losing one $s$-derivative if moved to the right-hand side and treated perturbatively. To handle this, we introduce the quantity 
\begin{align} \label{newcalU}
U := \bar{W}^2 \p_z \Big( \frac{\psi}{\bar{W}} \Big).
\end{align}
The quantity $U$ is reminiscent of the classical von-Mise type unknown, which has been used in various works prior to this one (see, for instance, \cite{IyerBlasius}, \cite{IM2}, \cite{IM21}, \cite{MD} as just a few examples). 

\subsubsection{Three Regions: ``lower von-Mise" $\mapsto$ ``Crocco" $\mapsto$ ``upper von-Mise"} \label{3regs}

The discussion above indicates that there are three regions: a ``Crocco" region localized to the interface, where we use the good unknown $\Omega$ (reminiscent of the Crocco transform), and two ``von-Mise" regions where we use the good unknown $U$: one which is localized below the interface and the other above the interface. Therefore, we have the following diagram, Figure \ref{three:reg}, which is useful to keep in mind. 
\begin{figure}[h]
\centering
\begin{tikzpicture}
\draw[ultra thick, <-] (0,2) node[above]{$z$} -- (0,0);
\node [left] at (0, 0) {$\textcolor{blue}{(1, 1)}$};
\filldraw [blue] (0, 0) circle (2pt);
\draw[ultra thick, -] (3,0) -- (3,-2);
\node [right] at (3.2, 0) {$\textcolor{blue}{(1 + \overline{L}, 1)}$};
\filldraw [blue] (3, 0) circle (2pt);
\draw[ultra thick, -] (0,-2) -- (3,-2);
\node [below] at (1.5, -2) {$u = 0, \psi = 0$};
\draw[ultra thick,blue, -] (0, 0) to (3,0);
\node [below] at (1.5, .8) {\textcolor{blue}{Crocco Region}};
\draw[ultra thick, dashed, -] (0,-2) -- (0,0);
\draw[ultra thick, dashed, ->] (3,0) -- (3,2);
\draw[ultra thick, dashed, ->] (3,-2) -- (4.5,-2);
\draw[ultra thick, dashed, ->] (0,-2) -- (-1.5,-2);
\node[right] at (4.5,-2) {$s$};
\draw[thick, dashed, red, ->] (-1,1.5) -- (1,1.5);
\draw[thick, dashed, red, <-] (2,-1.5) -- (4,-1.5);
\node [above] at (1.5, 1.5) {\textcolor{red}{upper von-Mise Region}};
\node [below] at (1.5, -1.5) {\textcolor{red}{lower von-Mise Region}};
\end{tikzpicture}
\caption{Three Regions} \label{three:reg}
\end{figure}

Prior works on the stationary Prandtl system, \cite{MD}, \cite{Oleinik}, \cite{IyerBlasius}, \cite{Serrin} essentially relied exclusively on a globally defined ``von-Mise" region, where one can work with $U$ and go back and forth to $u$ and $\psi$. This is due to the sign definite, $u_P \ge 0$, property in all these prior works. Here, such a strategy will fail due to the sign change of the background flow. It is for this reason that we need to decompose and work with both the von-Mise and Crocco formulations, depending on the value of $z$. 

When performing estimates, we need to proceed ``upwards". Indeed, near the physical boundary $y = 0$, one has a sign condition $\bar{u}_{FS} \le 0$, and therefore, we can work with the von-Mise unknown $U$ (and invert to $u$ and $\psi$). However, as we approach the interface, we need to switch to the Crocco variable ($\omega_g$), which due to the sign condition $\bar{W}_z \gtrsim 1$ in the interface region, is well-defined and invertible. As we then leave the interface towards $z \rightarrow \infty$, the Crocco transform starts to fail (due to the rapid vanishing of $\bar{W}_z$ as $z \rightarrow \infty$) and the von-Mise formulation becomes valid again. To perform these arguments, we use crucially the overlapping cutoffs $\chi_I$ (which localizes to the Crocco region, defined in \eqref{defn:Chi:I}), $\chi_{O,-}$ (which localizes to the lower von-Mise region, defined in \eqref{sightglass:1}) and $\chi_{O,+}$ (which localizes to the upper von-Mise region, also defined in \eqref{sightglass:1}).

\subsubsection{Abstract Operator Formulation \& Architecture of the Argument}

As we have emphasized above, we develop a lot of the analysis in a $k$-independent manner, and then afterwards apply our results to the $k$-by-$k$ system, \eqref{urk}. The overall structure of the paper matches this. For the sake of organization, we introduce (without precise definitions) the following organizational chart: 
\begin{figure}[h] 
\hspace{50 mm} \includegraphics[scale=0.3]{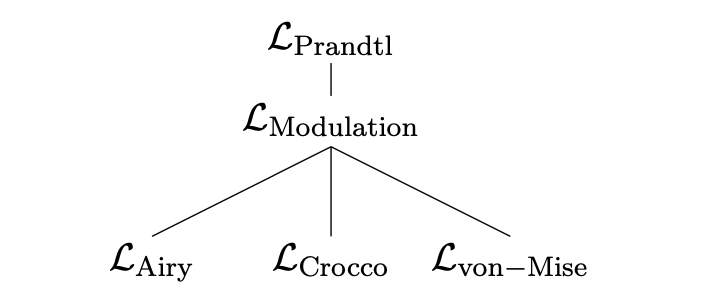}
\caption{Organizational Structure of Abstract Operators} \label{org:tree}
\end{figure}

The operator $\mathcal{L}_{\mathrm{Prandtl}}$ describes the operator in the original $(x, y)$ coordinates, \eqref{urk:abs}. After transforming to $(s, z)$ coordinates, we obtain $\mathcal{L}_{\mathrm{Modulation}}$, which essentially describes \eqref{forme}. The operators $\mathcal{L}_{\mathrm{Airy}}$, $\mathcal{L}_{\mathrm{Crocco}}$  describe formulations near the interface, essentially studying the unknown $\Omega$ (more precisely a localized version, $\Omega_I$), in the $(s, Z)$ variables introduced above. The operator $\mathcal{L}_{\mathrm{von-Mise}}$ is used to study and estimate the von-Mise unknown, $U$, in $(s, z)$ variables. 

The paper is structured by first studying the ``children" operators, $\mathcal{L}_{\mathrm{Airy}}$ (Section \ref{Section:Airy}), $\mathcal{L}_{\mathrm{Crocco}}$ (Section \ref{Section:Interior}), and $\mathcal{L}_{\mathrm{von-Mise}}$ (Section \ref{Section:Outer}). Then, in Section \ref{LpLqestsec}, we glue these three components together to study the ``parent" operator $\mathcal{L}_{\mathrm{Modulation}}$, and in Section \ref{yikes} we transition back from $(s,z)$ to $(x ,y)$ to reach the ``grandparent" operator, $\mathcal{L}_{\mathrm{Prandtl}}$, thereby concluding the study of the abstract operator formulation. 

We subsequently restore the $k$-dependence in Sections \ref{ksec} -- \ref{bdsdata:Sec}, and close all our bounds to prove Theorem \ref{main:thm} in Section \ref{final:sec}. Finally, in Section \ref{sec:spec}, we include a generalization of the result to address the case of $L$ not small, thereby obtaining Theorem \ref{thm:res}.

In the forthcoming subsections, we highlight the main mathematical components that appear in each section of the paper.  
 
\subsection{Airy Analysis, Section \ref{Section:Airy}} \label{AIRY::1}

In Section \ref{Section:Airy}, our main objective is to study forward-backward systems of the type \eqref{offo}, supplemented with appropriate data, which is written precisely in \eqref{abs:2:b:sec3}. The data elements $(\Xi_{\mathrm{Left}}, \Xi_{\mathrm{Right}})$ from \eqref{abs:2:b:sec3} will be thought of as abstract for this discussion, but will eventually be related to the given data \eqref{PhiIota:1} -- \eqref{PhiIota:4} in a $k$-by-$k$ manner. 

\subsubsection{Dirichlet-Neumann Matching}

It will turn out that the main quantity of interest in solving \eqref{offo} is the one-dimensional quantity 
\begin{align}
\Omega_I|_{Z = 0}(s) =: \gamma^{(0)}_{\Omega}(s), \qquad 1 < s < 1 + \bar{L}.
\end{align}
First of all, to see that $\gamma^{(0)}_\Omega(s)$ is, in a sense, the ``main unknown", notice that \textit{if we knew} $\gamma_\Omega^{(0)}(s)$, then \eqref{offo} would totally decouple into an ``upper" problem $(Z > 0)$ in which a forward parabolic evolution would take place and into a ``lower" problem $(Z < 0)$ in which a backwards parabolic evolution would take place.  

In order to identify the unknown trace $\gamma_\Omega^{(0)}(s)$, our strategy is a ``Dirichlet-to-Neumann" matching procedure. Indeed, consider the following operators:
\begin{align*}
\gamma_{\Omega,\pm}^{(1)}(s)[\gamma_\Omega^{(0)}](s) := \p_Z \Omega_I^{(\pm)}(s, 0),
\end{align*}
where we think of the Neumann condition as a functional on $\gamma_\Omega^{(0)}$ as well as the data on the sides. Above, we have $(+)$ corresponding to the Neumann condition arising from the ``upper" solution ($Z > 0$) and similarly the $(-)$ corresponding to the Neumann condition arising from the ``lower" solution ($Z < 0$). We now can realize $\gamma_\Omega^{(1)}(s)$ as the unique solution to the equation: 
\begin{align} \label{gamma2gamma1}
\gamma_{\Omega, +}^{(1)}[\gamma_\Omega^{(0)}] = \gamma_{\Omega, -}^{(1)}[\gamma_\Omega^{(0)}]. 
\end{align}

\subsubsection{The maps $\gamma_{\Omega, \pm}^{(1)}[\gamma_\Omega^{(0)}]$}

Now that we have identified \eqref{gamma2gamma1} as a central object of study, we need to characterize the maps $\gamma_2^{(\pm)}[\cdot]$. Let us first deal with $\gamma_{\Omega, +}^{(1)}[\cdot]$.  One way to understand these Dirichlet-to-Neumann maps is to (1) take extensions in $s$ to $s \in \mathbb{R}$ (there is some care to be taken in doing this), and subsequently (2) taking the Fourier transform in the $s$ variable. Doing so results essentially in: 
 \begin{subequations} 
\begin{align} \label{sun3Fourier}
&i \xi Z \widehat{\Omega}_I - \p_Z^2 \widehat{\Omega}_I = \widehat{F}_I, \qquad (\xi, Z) \in \mathbb{R} \times \mathbb{R}_+, \\ \label{sun4}
&\widehat{\Omega}_I|_{Z = 0}= \widehat{\gamma}_{\Omega}^{(0)}(\xi).
\end{align} 
\end{subequations} 
where $f_I$ is a source term depending on the original source, $F_I$, as well as the data elements $\Xi_{\mathrm{Left}}, \Xi_{\mathrm{Right}}$ which appears through the extension procedure. We can now explicitly describe solutions to \eqref{sun3Fourier} -- \eqref{sun4} through the Airy functions (and the phase-shifted Airy functions): 
\begin{subequations}
\begin{align} \label{us:n:1}
\widehat{\Omega}_I^{(+)}(\xi, Z) = &\frac{ai( (i\xi)^{\frac13}Z )}{ai(0)} \widehat{\gamma}_\Omega^{(0)}(\xi) + \mathcal{G}^{(+)}[f_I],  \\ \label{us:n:2}
\widehat{\Omega}_I^{(-)}(\xi, Z) = & \frac{B_{-sgn(\xi)}( (i\xi)^{\frac13}Z )}{B_{-sgn(\xi)}(0)} \widehat{\gamma}_\Omega^{(0)}(\xi) + \mathcal{G}^{(-)}[f_I]
\end{align}
\end{subequations}
Above, the quantity $\mathcal{G}^{(\pm)}[\cdot]$ is a functional on the source term and should be regarded temporarily as being on the right-hand side. 

The solution in the region $Z < 0$ needs to be described in two cases, due to the asymptotic behavior of the Airy functions in different subsets of $\mathbb{C}$. The quantities $ai(\cdot), B_{\pm}(\cdot)$ are the classical Airy function and appropriate phase shifted Airy functions, as described in Lemma \ref{lemma:fund:solutions}. To explain this more, we refer the reader to Figure \ref{FigC}.
\begin{figure}[h]
\centering
\begin{tikzpicture}
\draw[<->] (-3,0) -- (3,0);
\node[right] at (3,0) {$\Re(w)$};
\draw[<->] (0,-3) -- (0,3);
\node[above] at (0,3) {$\Im(w)$};
\draw[ultra thick, dashed, -, magenta] (0,0) -- (2,3.464);
\draw[ultra thick, dashed, -, magenta] (0,0) -- (2,-3.464);
\draw[ultra thick, dashed,  -, magenta] (0,0) -- (-4, 0);
\node [right] at (2.2, 3) {$\textcolor{magenta}{\Sigma_0}$};
\node [right] at (0.8, -3) {$\textcolor{magenta}{\Sigma_{-1}}$};
\node [right] at (-4, 0.5) {$\textcolor{magenta}{\Sigma_{+1}}$};
\draw[ultra thick, -, cyan] (0,0) --  (3.464, 2);
\draw[ultra thick, -, cyan] (0,0) -- (3.464, -2);
\node [right] at (3.5, 1.5) {$\textcolor{cyan}{\arg(w) = \frac{\pi}{6}}$};
\node [right] at (3.5, -1.5) {$\textcolor{cyan}{\arg(w) = - \frac{\pi}{6}}$};
\draw[ultra thick, -, violet] (0,0) -- (-3.464, -2);
\draw[ultra thick, -, teal] (0,0) -- (-3.464, 2);
\node [right] at (-3.6, 2.2) {$\textcolor{teal}{\arg(w) = \frac{5\pi}{6}}$};
\node [right] at (-3.6, -2.2) {$\textcolor{violet}{\arg(w) = - \frac{5\pi}{6}}$};
\end{tikzpicture}
\caption{Dirichlet-Neumann Matching on $\mathbb{C}$} \label{FigC}
\end{figure}

First of all, the three sectors $\Sigma_0, \Sigma_{+1}, \Sigma_{-1}$ are defined so that there is a consistent choice of decaying Airy basis function on each sector. However, as the sector changes, the decaying basis function also changes. Lemma \ref{lemma:fund:solutions} details this, but in summary: the decaying basis function in $\Sigma_0$ is $ai(\cdot)$, in $\Sigma_{+1}$ is the phase shifted $B_{+1}(\cdot) := ai(e^{- \frac{2}{3}\pi i}\cdot)$, and in $\Sigma_{-1}$ is the phase shifted $B_{-1}(\cdot) := ai(e^{\frac{2}{3}\pi i}\cdot)$.

Let us first turn to the case when $Z > 0$, where we write the solution as \eqref{us:n:1}. The arguments into the Airy functions are of the form $(i\xi)^{\frac13}Z$, which for $Z > 0$, $\xi > 0$, are elements on the ray $\arg = \frac{\pi}{6}$. For reasons to be described shortly, we think of these rays as parametrized by $Z$, for fixed $\xi$. Similarly, complex numbers $(i\xi)^{\frac13}Z$ when $Z > 0$ and $\xi < 0$ are elements of the ray $\arg = - \frac{\pi}{6}$. Both of these rays remain within $\Sigma_0$, and therefore we choose $ai((\xi)^{\frac13}Z)$ independently of $sgn(\xi)$, as in \eqref{us:n:1}. 

When $Z < 0$, the situation is different. Indeed, when $Z < 0$ and $\xi > 0$, we have $\arg((i\xi)^{\frac13}Z) = - \frac{5 \pi}{6}$, and this ray (parametrized by $Z$ for fixed $\xi$ is contained in $\Sigma_{-1}$).  On the other hand, when $Z < 0$ and $\xi < 0$, we have $\arg((i\xi)^{\frac13}Z) =  \frac{5 \pi}{6}$, and this ray is contained in $\Sigma_{+1}$. Therefore, we need to distinguish the two cases by $sgn(\xi)$ in \eqref{us:n:2}. We will see that these are crucial observations: they lead to the ``hidden ellipticity" in the problem, as we describe now. 

\subsubsection{Hidden Ellipticity on the Boundary and the Fractional Poisson Problem}

We now turn to \eqref{gamma2gamma1}. Given Figure \ref{FigC}, we can reinterpret the Dirichlet-Neumann matching of \eqref{gamma2gamma1} as matching the derivative of $ai(\cdot)$ along the ray $\arg = \frac{\pi}{6}$ with the derivative of $B_{-1}(\cdot)$ along the ray $\arg = - \frac{5\pi}{6}$ at the origin, for $\xi > 0$. Similarly, for $\xi < 0$, we can think of matching the derivative of $ai(\cdot)$ along the ray $\arg = - \frac{\pi}{6}$ with the derivative of $B_{+1}(\cdot)$ along the ray $\arg = \frac{5 \pi}{6}$ again at the origin (since we are interested in the Neumann condition at $Z = 0$). 

Performing this calculation with \eqref{us:n:1} -- \eqref{us:n:2} gives roughly the identity: 
\begin{align}\n
(i\xi)^{\frac13} \Big( \frac{ai'(0)}{ai(0)} - \frac{B_{-sgn(\xi)}'(0)}{B_{-sgn(\xi)}(0)} \Big) \widehat{\gamma}_{\Omega}^{(0)}(\xi) =& \text{Source Terms depending only on $F_I$ } + \\  \label{od2}
&\text{Data Terms depending on $(\Xi_{\mathrm{Left}}, \Xi_{\mathrm{Right}})$}. 
\end{align}
Examining now the structure of the ``Wronskian"-type Fourier multiplier appearing above, we discover the explicit identity (this is done in Lemma \ref{prop:Fourier:mult}):
\begin{align} \label{pd3}
\frac{ai'(0)}{ai(0)} - \frac{B_{-sgn(\xi)}'(0)}{B_{-sgn(\xi)}(0)} =  \frac{1}{ai(0)^2} \frac{1}{2\pi} e^{- i sgn(\xi) \frac{\pi}{6} }.
\end{align}
The fact that we distinguished the dependance of the basis functions on $sgn(\xi)$ in \eqref{us:n:2} produces the $\xi$-dependent factor above in \eqref{pd3}. Pairing now with the prefactor of $(i\xi)^{\frac13}$ appearing in \eqref{od2}, we discover a ``hidden-ellipticity":
\begin{align} \label{HE1}
(i\xi)^{\frac13} \Big( \frac{ai'(0)}{ai(0)} - \frac{B_{-sgn(\xi)}'(0)}{B_{-sgn(\xi)}(0)} \Big) \widehat{\gamma}_\Omega^{(0)}(\xi) = \frac{1}{ai(0)^2} \frac{1}{2\pi}  |\xi|^{\frac13} \widehat{\gamma}_\Omega^{(0)}(\xi).
\end{align}
The multiplier being applied to $\gamma_\Omega^{(0)}(\xi)$ is a fractional Laplacian. This ``hidden-ellipticity" reflects the fact that the free boundary, described by $\widehat{\gamma}_\Omega^{(0)}(\xi)$ should by symmetric between the left and the right side. Multipliers of the type $(i\xi)^{\frac13}$ are directional, whereas multipliers of $|\xi|^{\frac13}$ are symmetric.

We thus naturally arrive at a fractional elliptic problem: 
\begin{subequations}
\begin{align} \label{frac:pb:1}
&(-\Delta_D)^{\frac16} \gamma_\Omega^{(0)} =   \mathcal{G}^{(+)}[F_I] - \mathcal{G}^{(-)}[F_I] + \text{$(\Xi_{\mathrm{Left}}, \Xi_{\mathrm{Right}})$ contributions}, \qquad s \in (1, 1 + \bar{L}), \\ \label{frac:pb:3}
&\mathcal{G}^{(+)}[F_I] \approx \int_0^{\infty} ai( (i\xi)^{\frac13}\bar{Z}) \widehat{F}_I(\xi,\bar{Z}) \ud \bar{Z}. 
\end{align}
\end{subequations}
Above, the fractional Laplacian is the classical Riesz fractional Laplacian of power $1/6$:
\begin{align}
(-\Delta_D)^{\frac16} u(s) := \mathcal{R}_{(1, 1 + \bar{L})} \mathcal{F}^{-1} |\xi|^{\frac13} \mathcal{F} u.
\end{align}
Several remarks are in order regarding \eqref{frac:pb:1}:

First, the scaling of \eqref{offo} reflects the fact that, formally, $\p_s^{\frac13} \approx \p_Z$, and the Neumann trace should thus be $1/3$ $s$-derivative weaker than the Dirichlet trace. Second, one can see formally that the operator $\mathcal{G}^{(\pm)}[F]$, defined in \eqref{frac:pb:3} are smoothing of order $1/3$ in $s$-derivative. Hence, \eqref{frac:pb:1} continues to reflect the fact that the left-hand side of \eqref{offo} formally gains $\p_Z^2 \approx \p_s^{\frac23}$ over the right-hand side. 

Finally, we find it interesting (but also, \textit{a-posteriori} natural) that there is some hidden ellipticity in the problem that we need to find and use. Intuitively, this answers the question: if we give data on the left at ($s = 1$) and on the right $(s = 1 + \bar{L})$, how do they ``know of, and adapt to, each other's presence" dynamically? Clearly, the mechanism cannot be purely through independent forward and backwards evolutions. Ellipticity in the tangential, $s$, direction perfectly ``factors in" the left and right in a symmetric fashion. Returning to the analogue of Figure \ref{figure:dom}, we have in the $(s, z)$ coordinates the figure shown below, Figure \ref{figure:domsz}. 
\begin{figure}[h]
\centering
\begin{tikzpicture}
\draw[ultra thick, <-] (0,2) node[above]{$z$} -- (0,0);
\node [left] at (0, 0) {$\textcolor{blue}{(1, 1)}$};
\filldraw [blue] (0, 0) circle (2pt);
\draw[ultra thick, -] (3,0) -- (3,-2);
\node [right] at (3.2, 0) {$\textcolor{blue}{(1 + \overline{L}, 1)}$};
\filldraw [blue] (3, 0) circle (2pt);
\draw[ultra thick, -] (0,-2) -- (3,-2);
\node [below] at (1.5, -2) {$u = 0, \psi = 0$};
\draw[ultra thick,blue, -] (0, 0) to (3,0);
\node [below] at (1.5, .8) {$\textcolor{blue}{(-\Delta_D)^{\frac16} \gamma_\Omega^{(0)} = f}$};
\node [right] at (3, -1) {$\Omega_I = \Xi_{\mathrm{Right}}$};
\node [left] at (0, 1) {$\Omega_I = \Xi_{\mathrm{Left}}$};
\draw[ultra thick, dashed, -] (0,-2) -- (0,0);
\draw[ultra thick, dashed, ->] (3,0) -- (3,2);
\draw[ultra thick, dashed, ->] (3,-2) -- (4.5,-2);
\draw[ultra thick, dashed, ->] (0,-2) -- (-1.5,-2);
\node[right] at (4.5,-2) {$s$};
\draw[thick, dashed, red, ->] (-1,1.5) -- (1,1.5);
\draw[thick, dashed, red, <-] (2,-1.5) -- (4,-1.5);
\node [above] at (1.5, 1.5) {\textcolor{red}{+ parabolicity}};
\node [below] at (1.5, -1.5) {\textcolor{red}{- parabolicity}};
\end{tikzpicture}
\caption{$(s, z)$ Domain} \label{figure:domsz}
\end{figure}

Above, the source term $f = f(s)$ is a nonlocal integral operator that depends on the given data $\Xi_{\mathrm{Left}}, \Xi_{\mathrm{Right}}$, as well as lower order terms on the interface itself, $\gamma_\Omega^{(0)}$. Therefore, we can write $f = f[\Xi_{\mathrm{Left}}, \Xi_{\mathrm{Right}}, \gamma_\Omega^{(1)}]$. These lower order terms are what require us to remove certain ``resonances" in the parameter, $L$ in the large $L$ case.

The main outcome of Section \ref{Section:Airy} through the study of \eqref{frac:pb:1} -- \eqref{frac:pb:3} are bounds which control the Neumann trace, $\gamma_{\Omega}^{(1)}(s) := \p_Z \Omega(s, 0)$ in terms of the right-hand side, heuristically bounds of the form 
\begin{align} \label{yeahyeahbh:1}
\| \gamma_\Omega^{(1)} \|_{L^2_s} \lesssim \text{Contributions of Source Terms, $F_I$} + \text{Contributions of Data, $(\Xi_{\mathrm{Left}}, \Xi_{\mathrm{Right}})$}
\end{align}

\subsection{Crocco Analysis, Section \ref{Section:Interior}}\label{secIEE}

The type of bounds that appear from the previous section, \eqref{yeahyeahbh:1} are of the type $L^\infty_Z L^2_s$. Our objective in Section \ref{Section:Interior} is to obtain $L^\infty_s L^2_Z$ type bounds for appropriate quantities, because it turns out $L^\infty_s$ type norms are important to control in order to control our full set of norms (essentially we need to control $\| u \|_{L^\infty}$ to close the entire loop of estimates, which by Sobolev embedding is controlled by $\| \Omega \|_{L^\infty_s L^2_Z}$). In order to proceed, we can essentially think of $\gamma_\Omega^{(1)}(s)$ as a given function, since it has been controlled by the analysis of the previous section. Then, we may perform Crocco-type energy estimates. 

In particular, we define the norm $\| \cdot \|_{\mathrm{Crocco}}$ in \eqref{def:Crocco:norm}, and Section \ref{Section:Interior} is devoted to performing delicate energy estimates that control this norm. In particular, from a formal standpoint, we notice that $\| \Omega_I \|_{L^\infty_s L^2_Z}$ has the same scaling as $\p_s^{\frac12} \Omega_I \sim \p_Z^{\frac32} \Omega_I \in L^2_{sZ}$ which has the same scaling as $\p_Z \Omega_I \in L^\infty_Z L^2_s$. Moreover, the norm contains a hierarchy of $(Z\p_Z)$ commuted to the basic energy-dissipation components of the norm. Therefore, the Crocco norm contains norms that are all of ``equal strength" from a scaling point-of-view to the $\|\gamma_\Omega^{(1)} \|_{L^2_s}$ controlled from the previous section. This is captured in the eventual main estimate, \eqref{pro:cro:1}, by the appearance of $\| \gamma_\Omega^{(1)} \|_{L^2_s}^2$ on the right-hand side, with no small $L$ prefactors.   

A nuanced point is that for this portion of the analysis, we actually will retain all the lower order in $\Omega_I$ terms instead of moving them to the right-hand side (as in $F_I$ in \eqref{offo}). Therefore, we actually consider the operator shown in \eqref{abs:3}. The reason is essentially because we are performing energy estimates, we want to keep the structure of these terms which allows extra integrations by parts. For this reason, we give the operator a different name from the previous section, $\mathcal{L}_{\mathrm{Crocco}}$.

\subsection{von-Mise Analysis, Section \ref{Section:Outer}}

The von-Mise analysis takes place in the exterior region (which we loosely refer to as being away from the interface). This is quantified by the introduction of a sequence of nested cutoff functions, defined in \eqref{para:1} -- \eqref{parahyun}. In particular, due to the critical point $\bar{W}_z = 0$, we lose access to the $(s, Z)$ coordinate system, and we need to work with the less explicit formulation \eqref{forme}. 

The starting point of our strategy to perform energy estimates in the exterior region is to introduce the new good unknown $U$ from \eqref{newcalU}. Given this, we need to design and control several norms which are tailored to the variable $U$, which we label our von-Mise norms (see \eqref{yahL1} -- \eqref{vonMisenorm} for these definitions). 

We emphasize that, although the von-Mise type good unknown has been used in previous works, the overall perspective and technical components are quite different in this paper. Due to the zero set $\bar{W} = 0$, we cannot use a global von-Mise formulation; as we have been emphasizing we need to switch to the Crocco framework and rely on the monotonicity $\bar{W}_z > 0$ to transition from the backwards von-Mise to the forwards von-Mise formulation. To appreciate the complexity that this adds, we turn to the \textit{inversion} formulae: how to recover the original variable stream function, $\psi$, from the good unknown, $U$. Solving the ODE \eqref{newcalU} gives two cases
\begin{align} \label{jo1}
\psi(s, z) = &\bar{W}(s, z) \int_0^z \frac{U(s, z')}{\bar{W}(s, z')^2} \ud z', && \text{$z$ in Lower von-Mise Region}, \\ \label{jo2}
\psi(s, z) = & \frac{ \psi(s, z_0)}{\bar{W}(s, z_0)}\bar{W}(s, z) + \bar{W}(s, z) \int_{z_0}^z \frac{U(s, z')}{\bar{W}(s, z')^2} \ud z', && \text{$z$ in Upper von-Mise Region}
\end{align}
where $z_0 > 1$ is fixed. We think of the ``Lower von-Mise Region" as the set where $\chi_{O,-} = 1$, and the upper von-Mise region as the set where $\chi_{O,+} = 1$ (these cutoffs are defined in \eqref{sightglass:1}). 

The interpretation is that in the Lower von-Mise region due to the boundary condition $\psi|_{z = 0}$, the good unknown $U$ can be inverted in a stand-alone manner to recover $\psi$. The main difficulty in analyzing the inversion formula is the degeneracy of $\bar{W}$ at $\{z = 0\}$. This is conceptually and mathematically very similar to the standard use of the von-Mise in other works, the only exception being that $\bar{W} < 0$ instead of $\bar{W} > 0$ (for $z > 0$), but this does not pose any difficulties. 

On the other hand, the Upper von-Mise inversion formula is quite different. One needs to input $\psi(s, z_0)$, where $z_ 0 > 1$. This function of $s$ cannot be determined in a stand-alone way by the Upper von-Mise analysis, nor the Lower von-Mise analysis (as $z_0 > 1$, so $\bar{W}$ has switched sign here). Therefore, this is really an input that is determined by a combination: Lower von-Mise \& Crocco $\rightarrow \psi(s, z_0)$. This is a quintessential example of \textit{information flow upwards due to incompressibility} that persists throughout many components of the analysis. 

Mathematically, these notions are reflected in the analysis of Section \ref{Section:Outer} where a relatively asymmetric treatment of the Lower and Upper regions is required. On the one hand, for Lower von-Mise, the bounds are stand-alone, but contend with degeneracy of $\bar{W}|_{z = 0} = 0$. On the other hand, the Upper von-Mise needs to separate out the one dimensional variable $\psi(s, z_0)$ and cannot be stand-alone estimates. The Upper von-Mise does not deal with the degeneracy in $\bar{W}$ as it is always strictly positive in the ``upper region", but it does need to quantify weights at $\infty$ of the type $\langle z \rangle^n$.  

\subsection{Closing the Main Linearized Scheme, Sections \ref{LpLqestsec} - \ref{yikes} }

The main objective in Section \ref{LpLqestsec} is to glue together the analyses in the previous sections of $\mathcal{L}_{\mathrm{Airy}}, \mathcal{L}_{\mathrm{Crocco}}, \mathcal{L}_{\mathrm{von-Mise}}$, which are all active in different spatial regimes, in order to produce a global control over solutions to \eqref{forme}, $\mathcal{L}_{\mathrm{Modulation}}$. From the point-of-view of Figure \ref{org:tree}, we go from the three children operators to the parent operator. This involves estimating the right-hand sides that appear in the $\mathcal{L}_{\mathrm{Airy}}$, $\mathcal{L}_{\mathrm{Crocco}}$, and $\mathcal{L}_{\mathrm{von-Mise}}$ analyses. The key lemma that summarizes all of these bounds is Lemma \ref{keylemma}. 

As a matter of perspective, we identify $\| \psi \|_{L^\infty}$ and $\| u \|_{L^\infty}$ as the main ``global" quantities that enter as error terms in the bounds. For example, $\| \psi \|_{L^\infty}$ can control the non-local one-dimensional function $\psi(s, z_0)$ appearing in \eqref{jo2}. For this reason, when we glue together the Airy, Lower von-Mise, Upper von-Mise, and Crocco analyses, we need to ensure that we end up with a norm that is strong enough to control $\| \psi \|_{L^\infty}$ and $\| u \|_{L^\infty}$. This key calculation is seen in the proof of estimate \eqref{nlyougo:2}.

The main outcome of Section \ref{LpLqestsec} is to close the bound on solutions \eqref{forme} on the \textit{global norm}, \eqref{global:n}. However, we want to peel back this information one more time to the grandparent system, \eqref{urk}. Moreover, we want to do so in a way that does not rely on any cutoff functions or changes of functions: we want a clean norm stated in terms of just the solution $u_{R, k}$. This is because the source terms in $\mathcal{R}_{k}$ are in terms of the original variables and the original function $u_{R,k}$. This role is played by the \textit{effective norm}, \eqref{eff:norm}, which is defined and controlled in Section \ref{yikes}. 

An important subtlety in these sections is the propagation of \textit{two-tiers} of estimates, as one sees by examining \eqref{main:G} and \eqref{main:G:g}, and correspondingly \eqref{main:E} and \eqref{main:Ekju}. The quantities on the left-hand sides of \eqref{main:G:g} are controlled by $\| \cdot \|_{\mathrm{Global}}$, which appears on the left-hand side of \eqref{main:G}. However, the right-hand side of \eqref{main:G:g} is weaker than the right-hand side of \eqref{main:G}, due to the extra $\bar{L}^{\frac13}$ factor multiplying $\| \p_z^j \zeta_{\mathrm{Left}} \chi_{O,j}^+ \langle z \rangle^n \|_{L^2_z}$. In turn, this extra factor will be crucial to close the bounds in Section \ref{final:sec}, as explained below in the discussion of Section \ref{hguyjhjhjh}.

\subsection{Restoring the $k$ Dependence, Sections \ref{ksec} - \ref{final:sec}}

We want to apply the effective norm bounds $k$-by-$k$ to the system \eqref{urk}. This is the basis for our main \textit{linear norm}, \eqref{Linear:norm}. One of the noteworthy features of this linear norm is that the weight in $\langle y \rangle$ is actually $k$-dependent: higher derivatives decay at a weaker rate, which we refer to as a ``downward cascade". 

\subsubsection{Estimates of Source Terms, Section \ref{ksec}}

The main point of Section \ref{ksec} is to provide careful estimates on the $k$-dependent source terms in \eqref{urk}. This downward cascade of weights is motivated by the loss of one weight in the linear commutator terms appearing in $\mathcal{R}_k$. Indeed, consider a linear term of the form $v_{FS} \p_y u_R$. Upon taking $L\p_x$, assume this falls on the coefficient, which produces the commutator term $\p_y u_R (L \p_x) v_{FS}$, which subsequently enters as a forcing term for the $k = 1$ bound. It is the case that $\p_x v_{FS} \sim \langle y \rangle$ for large $y$, and therefore in order to control this growth, we need to estimate the $k = 1$ norm with one weaker weight in $y$ compared to $k = 0$. This type of calculation is seen precisely in estimate \eqref{dwc:1}. 

\subsubsection{Estimates on the Data, Section \ref{bdsdata:Sec}} 

The data in this problem is given in terms of \eqref{PhiIota:1} -- \eqref{PhiIota:4}. However, as we are performing higher order in $\p_x$ estimates, we need to ascertain the corresponding data for the $k$'th derivative. To simplify the discussion, we focus only on the Airy formulation, \eqref{offo}. As usual with parabolic problems, we can solve for $\p_s \Omega|_{s = 1} = \frac{\p_z^2 \Omega|_{s = 1} + F|_{s =1 }}{Z}$, for $Z > 0$ and also analogously for $s = 1 + \bar{L}$ for $Z < 0$. Of course, this involves ensuring a compatibility condition at $(s, Z) = (1, 0)$. In the Prandtl setting in particular, it turns out that iterating this procedure requires a very detailed bookkeeping strategy. Indeed, let us consider the model problem 
\begin{align} \label{throughme:1}
Z \p_s \Omega - \p_Z^2 \Omega = \tau_0 \Omega + \tau_1 \p_z \Omega + \tau_{-1} u,
\end{align}
which retains some of the structure from the full Crocco formulation of the equation \eqref{abs:3}: it re-inserts the $u$ dependent term which is put into the source term, $S_I$, in \eqref{abs:3}, but also drops certain terms that are not relevant for the present discussion. The important component to pay attention to is the velocity contribution, $u$. Due to \eqref{PhiIota:2}, up to a very minor change of coordinate, we can think of $\Omega|_{s = 1}(Z) \sim F_{\mathrm{Left}}(\frac{Z}{L^{\frac13}})$, that is, for $L << 1$, the prescribed data is in terms of a \textit{fast variable}, $Z/L^{\frac13}$. 

We now turn our attention to $u$. Indeed, using the expression \eqref{gammauint}, we see that the first term $\gamma_u \frac{u_{FS}'(y)}{u_{FS}'(y_1(\eps))}$ is actually a function of the original, \textit{slow variable}, $y  \sim Z$. In fact, by rewriting the second term in \eqref{gammauint}, 
\begin{align*} 
\int_{y_1(\eps)}^y \frac{F_{\mathrm{Left}}(\frac{1}{L^{\frac13}}( \frac{\bar{y}}{y_1(\eps)} - 1) )}{u_{FS'}(\bar{y})} \ud \bar{y} = &-\int_{y}^{y^\ast}  \frac{F_{\mathrm{Left}}(\frac{1}{L^{\frac13}}( \frac{\bar{y}}{y_1(\eps)} - 1) )}{u_{FS'}(\bar{y})} \ud \bar{y}  \\
&+ \int_{y_1(\eps)}^{y^\ast}  \frac{F_{\mathrm{Left}}(\frac{1}{L^{\frac13}}( \frac{\bar{y}}{y_1(\eps)} - 1))}{u_{FS'}(\bar{y})} \ud \bar{y},
\end{align*}
we see there is a secondary \textit{slow} contribution of the form $u_{FS}'(y)\int_{y_1(\eps)}^{y^\ast}  \frac{F_{\mathrm{Left}}(\frac{1}{L^{\frac13}}( \frac{\bar{y}}{y_1(\eps)} - 1) + 1)}{u_{FS'}(\bar{y})} \ud \bar{y}$. 

Therefore, even though $\Omega|_{s = 1}$ is prescribed in the variable $Z/L^{1/3}$, $u_R$ (and $u$) will have a fast-slow decomposition. This fast-slow decomposition then feeds back into the expression for $\p_s \Omega_{s = 1}$ through \eqref{throughme:1}. This therefore forces us to design a very careful bookkeeping scheme in order to keep track of the structure of the fast-slow data at $s = 1, 1 + \bar{L}$. 

As a secondary complexity, the integral formula for $u_R$, \eqref{gammauint}, contains a number $\gamma_u$ which is not prescribed; it is solution dependent. Therefore, any mechanism to keep track of the $k$'th data must also distinguish between the solution dependent components (which depend on $2k$ numbers $\gamma_{u; k}, \gamma_{\psi; k}$ defined in \eqref{def:gamma:u:k}) and the prescribed components. In turn, this dependence on $\gamma_{u; k}, \gamma_{\psi; k}$ is the reason we introduce a \textit{two-tiered} system of estimates, as explained below in Section \ref{hguyjhjhjh}.

A tertiary complexity is that the Crocco formulation to read off the data of $\p_s^k \Omega_I$ clearly is only valid locally near the interface. To determine the $k \mapsto k+1$ data iteration in the exterior regions, we need to also invoke the formulation \eqref{forme}. 

Due to the various ingredients that enter this section, the analysis at the boundaries $\{x = 1\} \cap \{y > y_1(\eps)\}$ and $\{x = 1 + L\} \cap \{0 < y < y_{1 + L}\}$ ends up becoming some of the more subtle analysis of the entire proof. This is yet another example of how distinguished the Prandtl system is from the model mixed-type problem \eqref{offo} due to nonlocal effects. Our analysis of these boundary contributions culminates in the bounds \eqref{est:data:M} -- \eqref{est:data:below:M}.

\subsubsection{Maximal Regularity Theory, Section \ref{Lpr}} \label{s:Max:reg}

We will need to trade away $\p_x$ derivatives to gain back $\p_y$ derivatives in order to control nonlinear terms. This is quantified by our \textit{Maximal Regularity norm}, which we define in \eqref{MR:norm} and devote Section \ref{Lpr} to controlling. Since we are propagating control over the scaled derivative $L \p_x$, we can only expect to gain $\p_y$ derivatives with weights of $L$ in front, which is quantified by the weights of $L^j$ that appear in \eqref{MR:norm}. Moreover, since higher $\p_x$ derivatives are controlled with weaker weights at $y = \infty (z = \infty)$, we need to account for slower decay when we invoke this norm, as quantified by the powers of $\langle y \rangle^{N_k -j}$ appearing in \eqref{MR:norm}. 

It is in this section that we will need the sixth formulation of the Prandtl system, namely $\mathcal{L}_{\mathrm{Lin}}$. Since we are less concerned with $\p_x$ derivative loss, it is more convenient to work with the equation in its linear form, \eqref{linHGYU:1}, with the semilinear and quasilinear contributions grouped together. Due to the precise numerology of derivatives and weights we need to keep track of in the norm definition, \eqref{MR:norm}, some of the lemmas in this section require some precise numerology in order to execute (see, for example, Lemma \ref{lmrsem}).

\subsubsection{Closing the Main \textit{a-priori} Estimate, Section \ref{final:sec}}  \label{hguyjhjhjh}

The above analyses enable us to close our full \textit{a-priori} estimate on $\| \psi_R \|_{\mathcal{Z}} := \| \psi_R \|_{\mathrm{Linear}} + \| \psi_R \|_{\mathrm{MR}}$, of the form 
\begin{align} \label{omix:1}
\| \psi_R \|_{\mathcal{Z}} \lesssim \eps L^{-1}\| \psi_R \|_{\mathcal{Z}}^2 + \text{Contributions from $\{x = 1\}, \{x = 1 + L\}$}.
\end{align}
Normally, this would be enough to conclude the argument. However, a look at the Contributions from $\{x = 1\}, \{x = 1 + L\}$ terms appearing above, specifically defined in \eqref{boldDat} and estimated in Proposition \ref{prop:data:M}, estimate \eqref{est:data:M} shows a dependence, with $O(1)$ constant on $|\gamma_{u; k}|$ and $|\gamma_{u;\psi}|$. The presence of these factors can be appreciated, even for $k = 0$, for instance by inspecting the formula \eqref{PhiIota:1}: any time $u$ itself shows up as a data element, $\gamma_{u}$ will necessarily accompany it. Therefore, the bound above, \eqref{omix:1} really reads: 
\begin{align}
\| \psi_R \|_{\mathcal{Z}} \lesssim \eps L^{-1}\| \psi_R \|_{\mathcal{Z}} + \sum_{k' = 0}^k (|\gamma_{u; k}| + |\gamma_{\psi; k}|) + \text{Contributions of }(\mathring{F}_{\mathrm{Left}}, \mathring{F}_{\mathrm{Right}}, G_{\mathrm{Left}}, G_{\mathrm{Right}}).
\end{align}
The reader should see \eqref{sbu:ghy:1} -- \eqref{dat:est:yea} for a more precise representation of the bounds. 

Therefore, we need to in addition de-couple the role of $ \sum_{k' = 0}^k (|\gamma_{u; k}| + |\gamma_{\psi; k}|)$ from the estimate above. We do this by realizing that these contributions are determined primarily from the bottom, whereas these contributions appear on the right-hand side primarily from the upper von-Mise data elements. Therefore, we may estimate $\sum_{k' = 0}^k (|\gamma_{u; k}| + |\gamma_{\psi; k}|)$ without invoking $\sum_{k' = 0}^k (|\gamma_{u; k}| + |\gamma_{\psi; k}|)$ itself. More precisely, we close the bounds: 
\begin{align} \n
\| \psi_R \|_{\mathcal{Z}} \lesssim &\eps L^{-1}\| \psi_R \|_{\mathcal{Z}}^2 + \sum_{k' = 0}^k (|\gamma_{u; k}| + |\gamma_{\psi; k}|) \\ \label{seq1}
&+ \text{Contributions of }(\mathring{F}_{\mathrm{Left}}, \mathring{F}_{\mathrm{Right}}, G_{\mathrm{Left}}, G_{\mathrm{Right}}), \\ \label{seq2}
\sum_{k' = 0}^{k_\ast} (|\gamma_{u; k}| + |\gamma_{\psi; k}|) \lesssim &\eps L^{-1}\| \psi_R \|_{\mathcal{Z}}^2 +  \text{Contributions of }(\mathring{F}_{\mathrm{Left}}, \mathring{F}_{\mathrm{Right}}, G_{\mathrm{Left}}, G_{\mathrm{Right}}).
\end{align}
In order to close the second, ``decoupled" bound above, we run a separate layer of bounds throughout the entire analysis on the special quantities $\gamma_{u; k}, \gamma_{\psi; k}$. Here, we are referring specifically to estimates \eqref{main:G:g}, \eqref{main:Ekju}, \eqref{shbujL2}, which show that $\gamma_{u; k}, \gamma_{\psi; k}$ are controlled by a more refined norm on the sides, namely $\mathrm{\bold{Dat}}_{\mathrm{Below}}$ defined in \eqref{boldDat}. In turn, $\eqref{boldDat}$ \textit{does not} see the dependence of $\gamma_{u; k}, \gamma_{\psi; k}$ at $O(1)$, rather with a small factor of $L^{\frac13}$ in front. In turn, this sequence \eqref{seq1} -- \eqref{seq2} are sufficient to close our main quadratic bounds for $\| \psi_R \|_{\mathcal{Z}}$ in terms of the given data. 

We emphasize that the role of this two-tiered system of bounds in order to decouple the dependence of the data on $\gamma_{u; k}, \gamma_{\psi; k}$ is yet another instantiation of the incompressibility effect (which is what creates the dependence of the $\mathrm{Left}$ data on these numbers to begin with).

\subsection{Spectral Analysis \& Discrete Set of Resonances, Section \ref{sec:spec}} \label{spectral:S}

In the final section of the paper, we want to generalize our result from $0 < L << 1$ to $L$ values potentially large, but away from a discrete set of ``resonances". To do so, we use crucially the ``hidden ellipticity" we have discovered in \eqref{frac:pb:1} -- \eqref{frac:pb:3}. The analysis of this section is motivated by an analogy to the (much simpler) elliptic problem, $- \Delta + V_L$ defined on $(0, 1)$, where $V_L$ is a potential with analytic dependence on the parameter $L$. Classical elliptic theory implies that for such operators, the set of $L$ for which these operators are not invertible is discrete. In our context, we want to use this observation to take $L$ large, as long as a discrete set of values are avoided. 

One of the main instances of the small $L$ hypothesis is to ``close" the estimate on the elliptic problem \eqref{frac:pb:1}. Indeed, the right-hand side depends on $F_I$, which itself depends on the solution (see the definitions in \eqref{defFCL}, for example), and therefore on the trace itself, $\gamma_\Omega^{(0)}$. We may thus think of the most dangerous $F_I$ contributions as a potential $\underline{\bold{K}}_L(\gamma_\Omega^{(0)})$. Therefore, a simplified version of how we are ``closing" \eqref{frac:pb:1} (which occurs in actuality in \eqref{nlyougo:1} -- \eqref{nlyougo:4}) is by thinking of \eqref{frac:pb:1} as 
\begin{align*}
&(-\Delta_D)^{\frac16} \gamma_\Omega^{(0)} =  \underline{\bold{K}}_L(\gamma_\Omega^{(0)})+ \text{$(\Xi_{\mathrm{Left}}, \Xi_{\mathrm{Right}})$ contributions},
\end{align*}
where we estimate the potential term $\| \underline{ \bold{K}}_L(\gamma_\Omega^{(0)}) \|_{L^2_s} \lesssim \| \gamma_\Omega^{(0)} \|_{L^2_s}$. Such an equation admits bounds of the type $\| \gamma_\Omega^{(0)} \|_{L^{2}_s} \lesssim L^{\frac13} \| \underline{\bold{K}}_L(\gamma_\Omega^{(0)}) \|_{L^2_s} + \text{Data Contributions}$. Because the potential term is lower order, we can close an analysis for $L << 1$ (essentially from a scaling point of view, one gets a small factor of $L^{\frac13}$ from inverting the Laplacian).

In order to generalize our result for $L$ not small, we essentially want to move the potential term to the left-hand side: 
\begin{align*}
&(-\Delta_D)^{\frac16} \gamma_\Omega^{(0)} - \underline{ \bold{K}}_L(\gamma_\Omega^{(0)}) = \text{$(\Xi_{\mathrm{Left}}, \Xi_{\mathrm{Right}})$ contributions}, \qquad s \in (1, 1 + \bar{L}).
\end{align*}
By making an analogy with the simpler elliptic problem $-\Delta \gamma_\Omega^{(0)} + V(\gamma_\Omega^{(0)})$ set on $(1,1 + \bar{L})$, where $V(\gamma_\Omega^{(0)})$ is bounded $L^2 \rightarrow L^2$, there are discretely many values $\bar{L}$ (and therefore $L$) for which this operator is not invertible. To see this in our context, we essentially rescale the above equation to be set on the normalized domain $(0, 1)$, and let $\Gamma_\Omega^{(0)}(t) = \gamma_\Omega^{(0)}(s)$, where $t = \frac{s-1}{\bar{L}}$. This produces 
\begin{align*}
&(-\Delta_D)^{\frac16} \Gamma_\Omega^{(0)} - \bold{K}_L[\Gamma_\Omega^{(0)}] = \text{$(\Xi_{\mathrm{Left}}, \Xi_{\mathrm{Right}})$ contributions}, \qquad t \in (0, 1)
\end{align*}
for an appropriately defined potential $\bold{K}_L[\Gamma_\Omega^{(0)}]$. The invertibility of the operator on the left-hand side then boils down to ruling out a kernel to $(-\Delta_D)^{\frac16} + \bold{K}_L$. We rely upon two facts: 
\begin{itemize}
\item[(1)] There is no kernel when $0 < L << 1$ (as shown by our small $L$ analysis, Sections \ref{Section:Airy} -- \ref{final:sec}).

\item[(2)] The operator $(-\Delta_D)^{\frac16} + \bold{K}_L$ depends analytically on the parameter $L$ (in a sense to be made precise in Section \ref{sec:spec}).
\end{itemize}
These two facts imply that the set of $L$ where kernel elements could appear are discrete (essentially, it is because a nonzero analytic function can only have discretely many zeros). The physical interpretation of these discrete resonances, $\{L_k\}$, is an interesting question.

\section{Six Formulations of the Stationary Prandtl System} \label{SectionCH}

We start the proof of Theorems \ref{main:thm} and \ref{thm:res} by performing several changes of variables and transformations. Overall, the conclusion of this section will be the derivation of six canonical systems to study the stationary Prandtl equations. These will be summarized below in Section \ref{sec:four:canonical}. 

\subsection{The $(x, y)$ variables and $\mathcal{L}_{\text{Prandtl}}, \mathcal{L}_{\mathrm{Lin}}$}

We are considering the problem given by \eqref{eq:PR:0}, which reads
\begin{subequations}
\begin{align}
&u_P \p_x u_P + v_P \p_y u_P - \p_y^2 u_{P} = -\p_x p_E(x), \qquad (x, y) \in (1, 1+L) \times (0, \infty), \\
&u_P|_{y = 0} = 0, \qquad u_P|_{y \rightarrow \infty}  = u_E(x). 
\end{align}
\end{subequations}
Recall from \eqref{psir}, that we consider perturbations of the form
\begin{align} \label{mu:pert}
u_P = u_{FS} + \eps u_R, \qquad v_P = v_{FS} + \eps v_R, \qquad \Lambda(x) = \Lambda_{G}(x) + \eps \Xi(x), 
\end{align}
where $\Lambda(x)$ defines the perturbed free boundary (and $\eps \Xi(x)$ the perturbation):
\begin{align}
u_P(x, \Lambda(x)) = u_{FS}(x, \Lambda(x)) + \eps u_R(x, \Lambda(x)) = 0
\end{align}
Expanding around $(u_{FS}, v_{FS})$, which is an exact solution of \eqref{eq:PR:0}, we obtain the equation for the perturbation $(u_R, v_R)$ via
\begin{align} \label{exp:eq}
(u_{FS} + \eps u_R) \p_x u_R + u_R \p_x u_{FS} + (v_{FS} + \eps v_R) \p_y u_R + v_R \p_y u_{FS} - \p_y^2 u_R = 0. 
\end{align}

\vspace{2 mm}

\noindent \underline{$\p_\tau = L \p_x$ Vector Field:} We will need to differentiate tangentially the system above $k_{\ast}$ times. To do this, we introduce the rescaling
\begin{align} \label{dtau}
\tau := \frac{x-1}{L} \Rightarrow \p_{\tau} = L \p_x
\end{align}
We now differentiate this equation by applying $\p_{\tau}^k$. We introduce the notation
\begin{align} \label{Rknot}
u_{R,k}(x, y) := \p_{\tau}^k u_R(x, y), \qquad v_{R,k}(x, y) := \p_{\tau}^{k} v_R(x, y). 
\end{align}
Upon grouping terms according to order of derivative, we obtain the following system:
\begin{align}
\left. \begin{aligned} \label{sideLH}
&u_P \p_x u_{R, k} + u_{Py} v_{R,k} + \bold{C}_1  \p_y u_{R,k} + \bold{C}_2 u_{R,k}  +\bold{C}_3 \psi_{R,k} - \p_y^2 u_{R, k}= \overline{ \mathcal{R}}_{k}, \\
&u_R|_{y = 0} = 0, \qquad u_R|_{y \rightarrow \infty = 0}.
\end{aligned}\right|
\end{align}
The coefficients are defined as follows:
\begin{subequations}
\begin{align} \label{bold:C1:pcts}
\bold{C}_1 := &v_{FS} + \eps 1_{k \ge 1} v_R\\  \label{bold:C2:pcts}
\bold{C}_2 := &\p_x u_{FS} + 1_{k \ge 1} k \p_x u_{P} + \eps 1_{k \ge 2} \p_x u_R \\  \label{bold:C3:pcts}
\bold{C}_3 := & 1_{k = 1} \p_{xy} u_{FS} + 1 _{k \ge 2} k  \p_{xy} u_P, 
\end{align}
\end{subequations}
and the source terms $\overline{\mathcal{R}}_{k}$ are given by two components,
\begin{align} \label{defRKRK:hat}
\overline{\mathcal{R}}_{k} = &\overline{\mathcal{R}}_{k, \mathrm{Lin}} + \overline{\mathcal{R}}_{k, \mathrm{Semi}}.
\end{align}
The linear source terms are defined as
\begin{align} \n
\overline{\mathcal{R}}_{k, \mathrm{Lin}} :=& -1_{k \ge 2} \sum_{k' = 0}^{k-2} \binom{k}{k'} u_{FS,k-k'} \p_x u_{R, k'}   - 1_{k \ge 1} \sum_{k' = 0}^{k-1} u_{R,k'} \p_x u_{FS, k-k'} \\ \label{defRKRK:temp}
& - 1_{k \ge 1} \sum_{k' = 0}^{k-1} \binom{k}{k'} v_{FS, k-k'} \p_y u_{R,k'} - 1_{k \ge 2} \sum_{k' = 0}^{k-2} \binom{k}{k'} v_{R,k'} \p_y u_{FS,k-k'} =: \sum_{i = 1}^4 R_{k, \mathrm{Lin}}^{(i)},
\end{align}
whereas the semilinear terms as follows:
\begin{align}
\overline{\mathcal{R}}_{k, \mathrm{Semi}} :=&-1_{k \ge 3} \eps \sum_{k' = 1}^{k-2} \binom{k}{k'} u_{R,k-k'} \p_x u_{R, k'}- 1_{k \ge 3} \eps \sum_{k' = 1}^{k-2} \binom{k}{k'} v_{R,k'} \p_y u_{R,k-k'} = \sum_{i = 1}^2 \mathcal{R}_{k, \mathrm{Semi}}^{(i)}.
\end{align}

\vspace{2 mm}

\noindent \underline{Quaslinearization:} It will turn out to be convenient to further manipulate the form of our system before we are ready for performing analysis. In particular, we want to ``quasilinearize" the transport term $u_P \p_x u_{R; k} + u_{Py} v_{R; k}$, meaning that we will only consider the full nonlinear background $u_P$ for higher orders of $k$. For lower orders of $k$, we are fine with moving some $\p_x$ terms to the right-hand side. In turn, working with a linear background simplifies the analysis at lower orders of $k$. 

To make matters precise, define $k_s = k_\ast - 1$. We next rewrite 
\begin{align}
u_P = u_{FS} + \eps u_R = \underbrace{u_{FS} + \eps u_{R; \mathrm{Quasi}}}_{:= U_P} + \eps u_{R, \mathrm{Pert}},
\end{align}
where 
\begin{align} \label{urQuasi}
u_{R, \mathrm{Quasi}}(x, y) := \begin{cases}  - \frac{u_{FS}(x, y) - u_{FS}(x, y + \chi(y)(\Lambda_G(x) - \Lambda(x)))}{\eps}, \qquad 0 \le k \le k_s \\ u_R(x, y), \qquad k_s < k \le k_\ast\end{cases}, 
\end{align}
and 
\begin{align} \label{uRpert}
u_{R, \mathrm{Pert}}(x, y) := &  \begin{cases} \frac{u_{FS}(x, y) - u_{FS}(x, y + \chi(y)(\Lambda_G(x) - \Lambda(x)))}{\eps} +  u_R(x, y), \qquad 0 \le k \le k_s \\ 0 \qquad k_s < k \le k_\ast \end{cases}
\end{align}

\begin{remark}The reason for such a complicated looking expression is as follows. Ideally, for $0 \le k \le k_s$, one wants to omit the term $\frac{u_{FS}(x, y) - u_{FS}(x, y + \chi(y)(\Lambda_G(x) - \Lambda(x)))}{\eps}$ from both $u_{R, \mathrm{Quasi}}$ and $u_{R, \mathrm{Pert}}$. In that case, the background would simply be $U_P = u_{FS}$, which notably does not have the same level set, $\Lambda(x)$, as $u_P$. In the case presented, we have that $U_P = u_{FS}(x, y + \chi(y) (\Lambda_G(x) - \Lambda(x)))$, which satisfies $U_P(x, \Lambda(x)) = 0$. While this is not strictly necessary for the proof, we find it conceptually simpler to preserve the same $\Lambda(x)$ for all $k$. 
\end{remark}

While the decomposition creates a $k$-dependence on the background, $U_P$, we choose to suppress this dependence in the notation. We may now rewrite \eqref{sideLH} as follows
\begin{align} \label{sideLH:2}
&U_P \p_x u_{R, k} + U_{Py} v_{R,k} + \bold{C}_1  \p_y u_{R,k} + \bold{C}_2 u_{R,k}  +\bold{C}_3 \psi_{R,k} - \p_y^2 u_{R, k}= \mathcal{R}_{k}.
\end{align}
The source term $\mathcal{R}_k$ is split into linear and semilinear source terms, 
\begin{align} \label{defRKRK}
\mathcal{R}_{k} = & \mathcal{R}_{k, \text{Lin}} + \mathcal{R}_{k, \text{Semi}},
\end{align}
where we define
\begin{align} \label{RkLin}
\mathcal{R}_{k, \text{Lin}}  = \overline{\mathcal{R}}_{k, \text{Lin}} 
\end{align}
and 
\begin{align} \n
\mathcal{R}_{k, \text{Semi}} := & -1_{k \ge 3} \eps \sum_{k' = 1}^{k-2} \binom{k}{k'} u_{R,k-k'} \p_x u_{R, k'}- 1_{k \ge 3} \eps \sum_{k' = 1}^{k-2} \binom{k}{k'} v_{R,k'} \p_y u_{R,k-k'}  \\ \label{RkSemi}
&  - \eps u_{R, \mathrm{Pert}} \p_x u_{R,k} + \eps \p_y u_{R, \mathrm{Pert}} \p_x \psi_{R,k}=: \sum_{i = 1}^4 \mathcal{R}^{(i)}_{k, \text{Semi}}.
\end{align}

\vspace{2 mm}

\noindent \underline{Relaxation of Background Profile \& Study of the Linearized Equations:} It turns out to be convenient to work in a slightly more general framework, which we now explain. Consider \eqref{sideLH:2}; viewing $\bold{C}_i$ as \textit{given} background coefficients we could almost equally easily develop the analysis we perform in this paper for the \textit{linearized} problem, which is more general. In particular, this allows us to easily apply the analysis developed in this paper to the iteration scheme to prove existence in \cite{IM22b}. 

To motivate our formulation, we \textit{relax} the requirement on the background profile of the transport term, $U_P \p_x u_{R, k} + U_{Py} v_{R, k}$, as follows. Recall that $U_P = u_{FS} + \eps u_{R; \mathrm{Quasi}}$. Instead of insisting on backgrounds of this form, where the $\eps$ perturbation is necessarily tied to the solution, we define more generally backgrounds of the form 
\begin{align} \label{genbd:1}
\bar{U} := u_{FS} + \eps u_B,
\end{align}
where $u_B$ is an abstract background profile. Since $u_B$ is supposed to be a stand-in for \eqref{urQuasi}, to be precise, we need to also develop a stand-in for $u_R$, and the corresponding $\Lambda$. Therefore, we will assume $u_B$ is of the form 
\begin{align} \label{uB}
u_{B}(x, y) := \begin{cases}  - \frac{u_{FS}(x, y) - u_{FS}(x, y + \chi(y)(\Lambda_G(x) - \Lambda(x)))}{\eps}, \qquad 0 \le k \le k_s \\ u_{B,R}(x, y), \qquad k_s < k \le k_\ast\end{cases}, 
\end{align}
where $\Lambda$ is defined through
\begin{align} \label{defLambdagn}
u_{FS}(x, \Lambda(x)) + \eps u_{B, R}(x, \Lambda(x)) = 0. 
\end{align}
Notice this necessarily implies that $\bar{U}(x, \Lambda(x)) = 0$. Therefore, one can think of the generalized form: $u_{R} \mapsto u_{B, R}$ and $u_{R; \mathrm{Quasi}} \mapsto u_B$. In order to prove existence in our paper \cite{IM22b}, since we will rely on an iterative process, this level of generality allows us to set $u_{B,R} = u_R^{(n-1)}$ (the previous iterate).

We define $\Xi$ to be the perturbation of the free-boundary 
\begin{align} 
\bar{U}(x, \Lambda(x)) = 0, \qquad \Lambda(x) = \Lambda_G(x) + \eps \Xi(x),
\end{align}
where still $u_{FS}(x, \Lambda_G(x)) = 0$. We then consider the linearized system:
\begin{align} \label{sideLH:3}
&\bar{U} \p_x u_{R, k} + \bar{U}_y v_{R,k} + \bold{C}_1  \p_y u_{R,k} + \bold{C}_2 u_{R,k}  +\bold{C}_3 \psi_{R,k} - \p_y^2 u_{R, k}= \mathcal{R}_{k},
\end{align}
which will eventually motivate the definition of the linearized operator $\mathcal{L}_{\mathrm{Prandtl}}$ below, \eqref{abs:1}.

\vspace{2 mm}

\noindent \underline{Full Linearization:} It will turn out that in order to do our maximal regularity analysis in Section \ref{Lpr}, that it is more convenient to fully linearize the system. Essentially, this is because we care about generating $\p_y$ derivatives from $\p_x$ derivatives, and so we can effectively move all the $\p_x$ terms to the right-hand side. More precisely, in Section \ref{Lpr}, we will use the following formulation from \eqref{sideLH}: 
\begin{align} \label{linHGYU:1}
u_{FS} \p_x u_{R,k} + \p_y u_{FS} v_{R,k} + v_{FS} \p_y u_{R,k} + \p_x u_{FS} u_{R,k} + \widetilde{\bold{C}}_3  \psi_{R,k} - \p_y^2 u_{R,k} = \mathcal{R}_{k} + \mathcal{M}_{k, \mathrm{Quasi}},
\end{align}
where the coefficient $\widetilde{\bold{C}}_3$ contains only the linear part of $\bold{C}_3$: 
\begin{align}
\widetilde{\bold{C}}_3 := 1_{k \ge 1} k \p_{xy} u_{FS},
\end{align}
and where the quasilinear terms are put into $\mathcal{M}_{k, \mathrm{Quasi}}$: 
\begin{align} \n
\mathcal{M}_{k, \text{Quasi}} := & - \eps u_R \p_x u_{R,k} - \eps u_{Ry} v_{R, k} - \eps 1_{k \ge 1} v_R \p_y u_{R,k} - \eps (1_{k \ge 1} k + 1_{k \ge 2}) \p_x u_R u_{R,k} \\ \label{def:MkQuasi}
& - \eps 1_{k \ge 2} k \p_{xy} u_R \psi_{R,k} := \sum_{i = 1}^5 \mathcal{M}^{(i)}_{k, \text{Quasi}}.
\end{align} 
This formulation given in \eqref{linHGYU:1} will motivate our formulation of $\mathcal{L}_{\mathrm{Lin}}$, \eqref{linform:1}.

\subsection{The modulated $(s, z)$ coordinates and $\mathcal{L}_{\text{Modulation}}, \mathcal{L}_{\text{von-Mise}}$}

We now change to self-similar variables, $(s, z)$, via
\begin{align} \label{z:variable}
\frac{\ud s}{\ud x} = \frac{1}{\Lambda(x)^2}, \qquad z := \frac{y}{\Lambda(x)} = \frac{y}{\Lambda_G(x) + \eps \Xi(x)}.
\end{align}
We now transform our functions to 
\begin{subequations}
\begin{align} \label{change:function}
&u_k(s, z) = u_{R, k}(x, y), \qquad v_k(s, z) := - \int_0^z \p_s u_k, \qquad \psi_k(s, z) := \int_0^z u_k, \\ 
&\bar{u}(s, z) :=  u_{FS}(x, y), \qquad \bar{v}(s, z) := - \int_0^z \p_s \bar{u}, \qquad \bar{\psi}(s, z) := \int_0^z \bar{u}, \\
&\lambda(s) := \Lambda(x), \qquad \lambda_G(s) := \Lambda_G(x), \qquad \mu(s) :=  \Xi(x).
\end{align}
\end{subequations}
We also define the notation: 
\begin{align}
u_P(x, y) = & u_{FS}(x, y) + \eps u_R(x, y) = \bar{u}(s, z) + \eps u(s, z) = \bar{w}(s, z), \\ \label{dubbar:1}
\bar{U}(x, y) =& u_{FS}(x, y) + \eps u_B(x, y) = \bar{u}(s, z) + \eps w_B(s, z) = \bar{W}(s, z)
\end{align}
to stand for the total background profiles. Due to \eqref{uB}, these backgrounds depend on $k$, although we elect to suppress the $k$ dependence.  

Associated to \eqref{z:variable}, we record the following identity for a generic function $f = f(s, z)$,
\begin{align} \label{chain:rule:1}
\p_x f(s, z) = \frac{1}{\lambda(s)^2} \p_sf- \frac{\lambda'(s)}{\lambda(s)^3} z \p_z f
\end{align}
We now write the system for $(u, v)$ in the $(s, z)$ coordinates, which after repeated application of the chain rule \eqref{chain:rule:1} to \eqref{sideLH:2} results in 
\begin{align} \label{res:in:1}
& \bar{W} \p_s u_k - \p_z \bar{W} \p_s \psi_k +  \bold{c}_1 \p_z u_k + \bold{c}_2 u_k   + \bold{c}_3 \psi_k - \p_z^2 u_k = \bold{R}_k, \qquad (s, z) \in (1, 1 + \bar{L}) \times (0, \infty),
\end{align}
where 
\begin{align} \label{res:in:2}
\bold{R}_k = \lambda^2 \mathcal{R}_k
\end{align}
and $\mathcal{R}_k$ has been defined in \eqref{defRKRK}, and with coefficients 
\begin{align} \label{defn:bold:c1}
\bold{c}_1 := & \bold{C}_1(x, y) - \frac{\lambda'}{\lambda}z \bar{W},\qquad \bold{c}_2 := \bold{C}_2(x, y) + \frac{\lambda'}{\lambda} z \bar{W}_{z}, \qquad \bold{c}_3 := \bold{C}_3(x, y) - \frac{\lambda'}{\lambda}\bar{W}_{z}.
\end{align}
Above, we have put $\bar{L} := \int_{1}^{1+L} \frac{1}{\Lambda(x)^2} \ud x$. We note that $|\bar{L}| \le L \sup_{x \in (1, 1+L)} \frac{1}{\Lambda(x)^2} \lesssim L$. We will of course need to ensure the condition that $\bar{U}(x, \Lambda(x)) = 0$, which reads in the $(s, z)$ coordinates
\begin{align} \label{mu:pert:1}
\bar{U}(s, \frac{\lambda_G(s) + \eps \mu(s)}{\lambda(s)}) + \eps w_B(s, 1) = 0. 
\end{align}
The closed, coupled, system for $(u, v, \mu(s))$ then reads: 
\begin{align}  \label{eq:for:v:1}
\left. \begin{aligned}
&\bar{W} \p_s u_k + v_k \p_z \bar{W} +  \bold{c}_1 \p_z u_k + \bold{c}_2 u_k   + \bold{c}_3 \psi_k - \p_z^2 u_k = \lambda^2 \mathcal{R}_k,  \\  
&\bar{u}(x, \frac{\lambda_G(s) + \eps \mu(s)}{\lambda(s)}) + \eps w_B(s, 1) = 0.
\end{aligned} \right| (s, z) \in (1, 1+\bar{L}) \times (0, \infty)
\end{align}
The above formulation will motivate our the definition of $\mathcal{L}_{\mathrm{Modulation}}$ below, \eqref{merc:1lkj}.




We now introduce the so-called von-Mise unknowns. Notice that the first two terms on the left-hand side of \eqref{eq:for:v:1} have derivatives in $s$. To write a transport equation on a ``good-unknown" which is transported in $s$, we introduce: 
\begin{align} \label{goodunk:1}
U_k := \bar{W} u_k -\bar{W}_z \psi_k.
\end{align}
Our system then reads:
\begin{align}  \label{eqeq1}
\left.\begin{aligned}
&\p_s U_k +  \bold{b}_1 \p_z u_k + \bold{b}_2 u_k   + \bold{b}_3 \psi_k - \p_z^2 u_k =  \lambda^2 \mathcal{R}_k,  \\ 
&U_k|_{s = 1} = U_{\mathrm{Left}; k}(z), \qquad z > 1, \\ 
&U_k|_{s = 1 + \bar{L}} = U_{\mathrm{Right}; k}(z), \qquad 0 < z <  1, \\ 
&u|_{z = 0} = u|_{z \rightarrow \infty} = 0. 
\end{aligned} \right| (s, z) \in (1, 1+\bar{L}) \times (0, \infty)
\end{align}
Above, we have defined the new coefficients as follows 
\begin{align} \label{tru:1}
\bold{b}_1 := & \bold{c}_1, \qquad \bold{b}_2 :=  \bold{c}_2 - \p_s \bar{ W}, \qquad \bold{b}_3 :=  \bold{c}_3 + \p_{sz} \bar{W}.
\end{align}
This motivates our system $\mathcal{L}_{\mathrm{von-Mise}}$ defined below in \eqref{abs:4:bljk}. As we will be performing energy estimates on \eqref{eqeq1}, we retain the data on the left and on the right as part of the formulation of the abstract operator. In turn, $[U_{\mathrm{Left}; k}(z), U_{\mathrm{Right}; k}(z)]$ will be computed in terms of the given data in Section \ref{bdsdata:Sec}.


\subsection{The Crocco $(s, Z)$ coordinates and $\mathcal{L}_{\text{Airy}}, \mathcal{L}_{\text{Crocco}}$}

\subsubsection{Straightening the Background: the $Y$ variable}

Much of our analysis will take place near the interface $\{z = 1\}$. Due to the structure of $\bar{w}$ near $\{z = 1\}$, we may perform further changes of variables which are valid only in a thin vertical strip around $\{z = 1\}$, which simplify the structure of our system. Indeed, for the remaining changes of coordinates, we will assume that $\delta_1$ is chosen such that $u_{FS}' \ge \frac{1}{2}  u_{FS}'(x, \Lambda_G(x) ) > 0$ (this quantity is $x$-independent) when $|y - \Lambda_G(x)| < \delta_1$. Define
\begin{align} \label{change:1}
Y := \bar{W}(s, z), \qquad \phi_k(s, Y) := \psi_k(s, z) 
\end{align}
We now compute the equation \eqref{eq:for:v:1} on the function $\phi_k$, which results in 
\begin{align}   \n
Y \p_{sY} \phi_k - \p_s \phi_k - \bar{W}_z^2 \p_Y^3 \phi_k = & (\bar{W} \bar{W}_s - 3 \bar{W}_{zz} + \bold{c}_1 \bar{W}_z ) \p_Y^2 \phi_k \\ \n
&+ \Big( \frac{\bar{W} \bar{W}_{sz} - \bar{W}_s \bar{W}_z - \bar{W}_{zzz} + \bold{c}_2  \bar{W}_z + \bold{c}_1 \p_z^2 \bar{W}}{\bar{W}_z}\Big) \p_Y \phi_k \\ \label{vel:st:1}
&+ \frac{\bold{c}_3}{\bar{W}_z} \phi_k  + \frac{ \mathcal{R}_k}{\bar{W}_z}.
\end{align}
Due to the fact that the map $(s, z) \mapsto (s, Y)$ defined by \eqref{change:1} is only invertible when $\bar{W}_z > 0$, we will only be performing this change of variables (and hence considering the formulation \eqref{vel:st:1}) in a thin strip around $\{Y = 0\}$. We note that while this change of variable has the effect of moving the interface from $\{z = 1\}$ to $\{Y = 0\}$, one should not think of $Y \approx z - 1$, because this type of translation would not eliminate the nonlocal term, the second term in \eqref{eq:for:v:1}, whereas our definition of $Y$ will result in a subtle cancellation of this nonlocal term. We now describe these details. 

We introduce the vorticity in the straightened coordinates, $(s,Y)$:
\begin{align} \label{change:2}
V_k(s, Y) := \p_Y^2 \phi_k(s, Y). 
\end{align}
Differentiating equation \eqref{vel:st:1}, we obtain the following system on $V$:
\begin{align} \label{real:sys:omY}
Y \p_s V_k - |\bar{W}_z|^2\p_Y^2 V_k = &   \underline{\tau}_1 \p_Y V_k + \underline{\tau}_0 V_k + \underline{\tau}_{-1}  \p_Y \phi_k + \underline{\tau}_{-2} \phi_k +  \frac{\p_z}{\bar{W}_z} \Big( \frac{\mathcal{R}_k}{\bar{W}_z} \Big).
\end{align}
Above, the coefficients are indexed according to $Y$ regularity in $V$. They are defined are follows: 
\begin{subequations} \label{utau:def}
\begin{align}
\underline{\tau}_1(s, Y) := &\bar{W} \bar{W}_s + \bold{c}_1 \bar{W}_z - \bar{W}_{zz}, \\ 
\underline{\tau}_0(s, Y) := & \frac{\p_z}{\bar{W}_z}( \bar{W} \bar{W}_s + \bold{c}_1 \bar{W}_z - \bar{W}_{zz} ) + \frac{1}{\bar{W}_z} ( \bar{W} \bar{W}_{sz} - \bar{W}_s \bar{W}_z + \bold{c}_2 \bar{W}_z + \bold{c}_1 \bar{W}_{zz} - \bar{W}_{zzz} )   \\
\underline{\tau}_{-1}(s, Y) := & \frac{\p_z}{\bar{W}_z}\Big( \frac{ \bar{W} \bar{W}_{sz} - \bar{W}_s \bar{W}_z - \bar{W}_{zzz} + \bold{c}_2 \bar{W}_z + \bold{c}_1 \bar{W}_{zz} }{\bar{W}_z} \Big) + \frac{\bold{c}_3}{\bar{W}_z}, \\
\underline{\tau}_{-2}(s, Y) := & \frac{\p_z}{\bar{W}_z} \Big( \frac{\bold{c}_3}{\bar{W}_z} \Big).
\end{align}
\end{subequations}
%
We also record the following identity, which shows the connection between our $(s, Y)$ coordinate system and good unknown, $V(s, Y)$, and the Masmoudi-Wong good-unknown, \cite{MW}. Indeed, computing using \eqref{change:1} and \eqref{change:2}, we get the identities
\begin{align}
& u_k(s, z) = \bar{W}_z(s, z) \p_Y \phi_k(s, Y), \\
&\p_z u_k(s, z) = \bar{W}_{zz} \p_Y \phi_k + \bar{W}_z^2 V_k(s, Z) = \frac{\bar{W}_{zz}}{\bar{W}_z} u_k + \bar{W}_z^2 V_k, 
\end{align}
which upon rearranging, results in 
\begin{align} \label{good:unk:1:V}
V_k(s, Y) = \frac{1}{\bar{W}_z^2} (\p_z u_k - \frac{\bar{W}_{zz}}{\bar{W}_z} u_k),
\end{align}
which is essentially equivalent to the Masmoudi-Wong good unknown, up to the pre-factor of $\frac{1}{\bar{W}_z^2}$. We find our new perspective on this good unknown noteworthy. 

\subsubsection{Eikonal Equation for the $Z$ Variable} \label{subsection:local}

Our aim is to now perform a further change of variables, $(s, Y) \mapsto (s, Z)$, which converts the prefactor of $\bar{W}_z^2$ in front of the diffusion in \eqref{real:sys:omY} into a constant prefactor of $1$ (as opposed to a nonconstant function of $s$) when restricted to $Y = 0$ (coinciding with $Z = 0$). We introduce
\begin{align} \label{philz1}
V_k(s, Y) = \Omega_k(s,Z) = \Omega_k(s, p(s, Y)), \qquad Z = p(s, Y). 
\end{align}
Note that this is a change of variable and function we use only \textit{at the vorticity level}, unlike the introduction of the $Y$ variable in \eqref{change:1}, which occurs at the level of the stream function. We also note that, putting \eqref{philz1} together with \eqref{good:unk:1:V}, we get
\begin{align}\label{good:unk:1}
\Omega_k(s, Z) = V_k(s, Y) = \frac{1}{\bar{W}_z^2} (\p_z u_k - \frac{\bar{W}_{zz}}{\bar{W}_z}  u_k).
\end{align}
We will also introduce a notation for the stream function in $(s, Z)$ coordinates:
\begin{align}
\Phi_k(s, Z) = \phi_k(s, Y)
\end{align}
Notice while $\p_Y^2 \phi_k = V_k$, a more complicated elliptic equation relates $\Phi_k$ and $\Omega_k$:
\begin{align}
\Omega_k(s, Z) = |p_Y|^2 \p_Z^2 \Phi_k + p_{YY} \p_Z \Phi_k.
\end{align}
We choose the function $p(s, Y)$ explicitly via
\begin{align} \label{Eikonal:1}
p(s, Y) = \frac{Y}{\bar{W}_z(s, 1)^{\frac23}}. 
\end{align}
We record the following identities from an application of the chain rule,
\begin{subequations}
\begin{align}
Y \p_s V_k =  &Y \p_s \Omega_k + Y p_s \p_Y \Omega_k \\
\p_Y V_k = & p_Y \p_Z \Omega_k, \\
\p_Y^2 V_k= & p_Y^2 \p_Z^2 \Omega_k + p_{YY} \p_Z \Omega_k.
\end{align}
\end{subequations}
From here, we obtain upon invoking \eqref{Eikonal:1} 
\begin{align} \n
Y \p_s V_k - |\bar{W}'|^2 \p_Y^2 V_k = & Z \bar{W}_z(s,1)^{\frac23} \p_s \Omega_k - \frac{\bar{W}_z^2}{\bar{W}_z(s,1)^{\frac43}}  \p_Z^2 \Omega_k + (p_Y^2 |\bar{W}_z|^2 Z p_s  )\p_Z \Omega_k.
\end{align}
We therefore observe that the left-hand side of \eqref{real:sys:omega} is conjugated to a pure Airy operator in the $(s, Z)$ variable, up to lower order terms. Indeed, equation \eqref{real:sys:omY} in the $(s, Y)$ coordinates transforms into 
\begin{subequations} \label{real:sys:omega}
\begin{align} 
&Z \p_s \Omega_k - \p_Z^2 \Omega_k = \tau_2 \p_Z^2 \Omega_k +   \tau_1 \p_Z \Omega_k + \tau_0 \Omega_k + \tau_{-1} \p_Z \Phi_k + \tau_{-2} \Phi_k + \bold{F}_k. 
\end{align}
\end{subequations}
Above, we have defined the following quantities: 
\begin{subequations} \label{tau:def}
\begin{align}
\tau_2(s, Z) := &\frac{\bar{W}_z^2}{\bar{W}_z(s, 1)^2} -1 \\
\tau_1(s, Z) := & - \frac23 \frac{\bar{W}_z^2 \bar{W}_{sz}(s, 1)}{\bar{W}_z(s, 1)^{\frac{11}{3}}}Z + \frac{\underline{\tau}_1}{\bar{W}_z(s,1)^{\frac43}}, \\ 
\tau_0(s, Z) := & \frac{ \underline{\tau}_0}{\bar{W}_z(s, 1)^{\frac23}} \\ 
\tau_{-1}(s, Z) := & \frac{ \underline{\tau}_{-1}}{\bar{W}_z(s, 1)^{\frac43}} \\
\tau_{-2}(s, Z) := &   \frac{ \underline{\tau}_{-2}}{\bar{W}_z(s, 1)^{\frac23}}
\end{align}
\end{subequations}
Finally, we define 
\begin{align} \label{boldF}
\bold{F}_k(s, Z) := & \frac{\p_z}{\bar{W}_z} \Big( \frac{\mathcal{R}_k}{\bar{W}_z} \Big)(s, z).
\end{align}

\subsubsection{Localization}

We choose the parameter $\delta = \delta(u_{FS}) > 0$ so $(s, Z) \in (1, 1 + \bar{L}) \times (-\delta, \delta)$ implies that $(x, y) \in (1, 1 + L) \times (\lambda_G(x) - \delta_1, \lambda_G(x) + \delta_1)$ (in particular, $u_{FS}'$ is bounded below uniformly in this set). To localize $Z$, we cut-off the vorticity $\Omega_k$ via
\begin{align} \label{defn:Chi:I}
\Omega_{I, k} (s, Z) := \Omega_k(s, Z) \chi(\frac{Z}{\delta}) = \Omega_k(s, Z) \chi_I(Z)
\end{align}
where we define $\chi_I(Z) = \chi(\frac{Z}{\delta})$, and where the cut-off $\chi$ is defined in \eqref{chidef}. Consolidating the problem formulation so far, we have from  \eqref{real:sys:omega}
\begin{subequations}
\begin{align} \label{loc:sys:1}
&Z \p_s \Omega_{I, k} - \p_Z^2 \Omega_{I, k} =   F_{I; k}, \qquad (s, Z) \in (1, 1 + \bar{L}) \times \mathbb{R} \\
&\Omega_{k,I}(s, \pm \infty) = 0,  \qquad s \in (1, 1+\bar{L})
\end{align}
\end{subequations}
where the forcing, $F_{I; k}$, appearing above is defined by 
\begin{align}\n
F_{I; k} := & \chi\Big(\frac{Z}{\delta}\Big) (\tau_2 \p_Z^2 \Omega_k +   \tau_1 \p_Z \Omega_k + \tau_0 \Omega_k + \tau_{-1} \p_Z \Phi_k + \tau_{-2} \Phi_k)  \\  \label{loc:sys:I1}
&  - \frac{\chi''(\frac{Z}{\delta})}{\delta^2} \Omega_k - \frac{2}{\delta} \chi'(\frac{Z}{\delta})\p_Z \Omega_k  + \bold{F}_k \chi\Big(\frac{Z}{\delta}\Big).
\end{align}
We assume data on the sides is as follows:
\begin{align}
\Omega_{I, k}(1, Z) =   \Omega_{I, \mathrm{Left}; k}(Z), \qquad \Omega_{I, k}(1 + \overline{L}, Z) =  \Omega_{I, \mathrm{Right}; k}(Z),
\end{align}
In turn, $ \Omega_{I, \mathrm{Left}; k},  \Omega_{I, \mathrm{Right}; k}$ will be computed in terms of the given data in Section \ref{bdsdata:Sec}.

It will be beneficial to commute the cut-off into the $\p_Z$ derivatives on the right-hand side above, which results in the following form of $F_{I; k}$:
\begin{align} \n
F_{I; k} = & \tau_2 \p_Z^2 \Omega_{I, k}+  \tau_1 \p_Z \Omega_{I, k} + \tau_0 \Omega_{I, k}+\chi(\frac{Z}{\delta}) ( \tau_{-1} \p_Z \Phi_k + \tau_{-2} \Phi_k)   \\ \label{form:FI}
&-( (1 + \tau_2) \frac{\chi''(\frac{Z}{\delta})}{\delta^2}  + \frac{\tau_1}{\delta} \chi'(\frac{Z}{\delta}) ) \Omega_k - \frac{2}{\delta} (1 + \tau_2) \chi'(\frac{Z}{\delta}) \p_Z \Omega_k+ \bold{F}_k \chi(\frac{Z}{\delta}), 
\end{align}
Given \eqref{loc:sys:1} and \eqref{form:FI}, we will use two different formulations within this coordinate system. First, for the in Section \ref{Section:Airy},  we need to view the primary operator as the $Z \p_s - \p_Z^2$ operator, and therefore we keep the equation as written \eqref{loc:sys:1}. On the other hand, for the analysis in Section \ref{Section:Interior}, we want to take advantage of integrations by parts through energy estimates, and therefore it is advantageous to move $\tau_2 \p_Z^2 \Omega_I+  \tau_1 \p_Z \Omega_{I, k} + \tau_0 \Omega_{I, k}$ to the left-hand side, resulting in 
\begin{align} \n
&Z \p_s \Omega_{I, k} - (1 + \tau_2) \p_Z^2 \Omega_{I, k} -  \tau_1 \p_Z \Omega_{I, k} - \tau_0 \Omega_{I, k} \\ \n
=& \chi(\frac{Z}{\delta}) ( \tau_{-1} \p_Z \Phi_k + \tau_{-2} \Phi_k)  -( (1 + \tau_2) \frac{\chi''(\frac{Z}{\delta})}{\delta^2}  + \frac{\tau_1}{\delta} \chi'(\frac{Z}{\delta}) ) \Omega_k \\  \label{bedlowo}
&  - \frac{2}{\delta} (1 + \tau_2) \chi'(\frac{Z}{\delta})\p_Z \Omega_k+ \bold{F}_k \chi(\frac{Z}{\delta})
\end{align}

\subsection{Summary of Six Canonical Systems} \label{sec:four:canonical}

We are now prepared to consolidate the various transformations we have made above. In particular, in the course of our analysis, we will have need for six formulations of the Prandtl system. Our argument will provide abstract \textit{a-priori} bounds on the systems presented in this section, and then devote separate sections to the estimation of the corresponding error terms for each $k$.Therefore, subindices of $k$ do not appear in what follows, and the unknowns here are abstract placeholders for $u_{R; k}$, etc...

Motivated by the system \eqref{sideLH:3}, we consider the following abstract problem:
\begin{align}
\left. \begin{aligned} \label{abs:1}
&\bar{U} \p_x u_{R} + \bar{U}_y v_{R} + \bold{C}_1  \p_y u_{R} + \bold{C}_2 u_{R}  +\bold{C}_3 \psi_{R} - \p_y^2 u_R = \mathcal{R}, \\
&u_R|_{y = 0} = 0, \qquad u_R|_{y \rightarrow \infty = 0}, \\
&\psi_R = \mathcal{L}_{\text{Prandtl}}[\mathcal{R}].
\end{aligned}\right| (x, y) \in (1, 1+L) \times \mathbb{R}_+
\end{align}
Motivated by the system \eqref{linHGYU:1}, we consider the following abstract problem: 
\begin{align}
\left. \begin{aligned} \label{linform:1}
&u_{FS} \p_x u_{R} + \p_y u_{FS} v_{R} + v_{FS} \p_y u_{R} + \p_x u_{FS} u_{R} + \widetilde{\bold{C}}_3  \psi_{R} - \p_y^2 u_{R} = \mathcal{R} + \mathcal{M}, \\
&u_R|_{y = 0} = 0, \qquad u_R|_{y \rightarrow \infty = 0}, \\
&\psi_R = \mathcal{L}_{\mathrm{Lin}}[\mathcal{R}, \mathcal{M}], 
\end{aligned} \right| (x, y) \in (1, 1 + L) \times \mathbb{R}_+
\end{align}
Next, motivated by \eqref{res:in:1}, we have two formulations in the $(s, z)$ coordinate system. 
\begin{align}  \label{merc:1lkj}
\left. \begin{aligned}
&\bar{W} \p_s u + v \p_z \bar{W}  - \p_z^2 u + \bold{c}_1 \p_z u + \bold{c}_2 u + \bold{c}_3 \psi   = \bold{R}, \\
&u|_{z = 0} = u|_{z \rightarrow \infty} = 0, \\
&\psi = \mathcal{L}_{\text{Modulation}}[\bold{R}], \\
\end{aligned}\right| (s, z) \in  (1, 1 + \bar{L}) \times \mathbb{R}_+
\end{align}
and the related von-Mise formulation: 
\begin{align}  \label{abs:4:bljk}
\left. \begin{aligned}
&\p_s U + \bold{b}_1\p_z u + \bold{b}_2 u + \bold{b}_3 \psi - \p_z^2 u = \bold{R}, \\  
&U|_{s = 1} = \zeta_{\text{Left}}(z), \\  
&U|_{s = 1 + \bar{L}} = \zeta_{\text{Right}}(z), \\ 
&u|_{z = 0} = u|_{z \rightarrow \infty} = 0, \\
&U := \bar{W} u - \bar{W}_z \psi, \\
&U = \mathcal{L}_{\text{von-Mise}}[\bold{R}, \zeta_{\text{Left}}, \zeta_{\text{Right}}].
\end{aligned}\right| (s, z) \in  (1, 1 + \bar{L}) \times \mathbb{R}_+
\end{align}
Above $\bold{R} = \lambda^2 \mathcal{R}$, according to \eqref{res:in:1}. Motivated by the system \eqref{loc:sys:1}, we will consider the following abstract problem:
\begin{align} \label{abs:2:b:sec3}
\left. \begin{aligned}
&Z \p_s \Omega_I - \p_Z^2 \Omega_I = F_I \\
&\Omega_I(s, \pm \infty) = 0, \\
&\Omega_I|_{s = 1} = \Xi_{\text{Left}}(\frac{Z}{\bar{L}^{\frac13}}), \qquad Z > 0, \\
&\Omega_I|_{s = 1 + \bar{L}} = \Xi_{\text{Right}}(\frac{Z}{\bar{L}^{\frac13}}), \qquad Z < 0, \\
&\Omega_I = \mathcal{L}_{\text{Airy}}[F_I, \Xi_{\text{Left}}, \Xi_{\text{Right}}]  
\end{aligned} \right| (s, Z) \in (1, 1 + \bar{L}) \times \mathbb{R}, 
\end{align}
Motivated by \eqref{bedlowo}, we consider the following abstract problem for energy estimates: 
\begin{align} \label{abs:3}
\left. \begin{aligned}
&Z \p_s \Omega_I - \tau_0 \Omega_I - \tau_1 \p_Z \Omega_I - (1 + \tau_2) \p_Z^2 \Omega_I = S_I \\
&\Omega_I(s, \pm \infty) = 0, \\
&\Omega_I|_{s = 1} = \Xi_{\text{Left}}(\frac{Z}{\bar{L}^{\frac13}}), \qquad Z > 0, \\
&\Omega_I|_{s = 1 + \bar{L}} = \Xi_{\text{Right}}(\frac{Z}{\bar{L}^{\frac13}}), \qquad Z < 0, \\
&\Omega_I = \mathcal{L}_{\text{Crocco}}[S_I, \Xi_{\text{Left}}, \Xi_{\text{Right}}]  
\end{aligned} \right| (s, Z) \in (1, 1 + \bar{L}) \times \mathbb{R}, 
\end{align}
The source terms $F_I, S_I$ appearing in \eqref{abs:2:b:sec3} and \eqref{abs:3} are, of course, designed to mimic \eqref{form:FI} and \eqref{bedlowo}, and are defined in a $k$-independent way in \eqref{defFCL} -- \eqref{defSCL}.

A few remarks are in order. First, the operator $\mathcal{L}_{\mathrm{Lin}}$ is only used to perform Maximal Regularity estimates, Section \ref{Lpr}, and is somehow not \textit{as} prominent in the analysis as the other five operators. Regarding the other five, we think of the organizational structure as shown in Figure \ref{org:tree}. That is, $\mathcal{L}_{\mathrm{Prandtl}}$ is thought of as a ``grandparent" operator and $\mathcal{L}_{\mathrm{Modulation}}$ as a ``parent operator". These operators govern the linearized system. To study them, we introduce three different sub-operators, $\mathcal{L}_{\mathrm{von-Mise}}$, $\mathcal{L}_{\mathrm{Airy}}$, $\mathcal{L}_{\mathrm{Crocco}}$ on which we actually perform estimates. The data that appears in our eventual bounds are those from \eqref{abs:4:bljk}, \eqref{abs:2:b:sec3}, and \eqref{abs:3}. 

The data elements $[\zeta_{\mathrm{Left}}, \zeta_{\mathrm{Right}}]$ are placeholders for $[U_{\mathrm{Left}; k}, U_{\mathrm{Right}; k}]$, and the data elements $[\Xi_{\mathrm{Left}}, \Xi_{\mathrm{Right}}]$ are placeholders for $[\Omega_{I; \mathrm{Left}; k}, \Omega_{I; \mathrm{Right}; k}]$. Finally, we remark that we think of $[\Xi_{\mathrm{Left}}, \Xi_{\mathrm{Right}}]$ as functions of the rescaled variable $Z/\bar{L}^{\frac13}$, because this will turn out to be the dominant form of the data due to the corresponding rescaling given in \eqref{PhiIota:2} -- \eqref{PhiIota:3}. Indeed, it is in Section \ref{bdsdata:Sec} that we restore the $k$ dependent definitions of the data elements and subsequently estimate them.


%
%

\section{Analysis of $\mathcal{L}_{\text{Airy}}[F_I, \Xi_{\text{Left}}, \Xi_{\text{Right}}]$} \label{Section:Airy}

The main system of study in this section will be \textit{a-priori} estimates on the abstract system, \eqref{abs:2:b:sec3}. We define the following traces:
\begin{align} \label{defngammaomegaj}
\gamma_{\Omega}^{(j)}(s) := \p_Z^j \Omega_I|_{Z = 0}(s), \qquad 1 \le s \le 1 + \bar{L}.
\end{align}

\begin{proposition} \label{prop:Airy} Let $F_I \in L^2_{sZ}$. Assume the data satisfies $\sum_{i = 0}^2\| \langle \rho \rangle (\p_\rho^i \Xi_{\mathrm{Left}},\p_\rho^i \Xi_{\mathrm{Right}}) \|_{L^2_\rho} < \infty$. The unique strong solution $\Omega_I$ to the system \eqref{abs:2:b:sec3} satisfies the following estimate:
\begin{align} \label{tr:es:1}
\| \gamma_{\Omega}^{(1)} \|_{L^{2}_s}   \lesssim & \bar{L}^{\frac16} \| F_I \|_{L^2_{sZ}} + \bar{L}^{\frac16} \sum_{i = 0}^2\| \langle \rho \rangle (\p_\rho^i \Xi_{\mathrm{Left}},\p_\rho^i \Xi_{\mathrm{Right}}) \|_{L^2_\rho}.
\end{align}
\end{proposition}
\begin{remark} This estimate can be compared to that of Pagani, \cite{Pagani2}, namely Theorem 4.2, in which the trace $\gamma_{\Omega}^{(1)}$ is shown to be in $\dot{H}^{\frac16}$, which is consistent with the scaling factor of $\bar{L}^{\frac16}$ appearing in \eqref{tr:es:1}. One can actually recover estimate \eqref{tr:es:1} by introducing the change of variables $t = \frac{s -1 }{\bar{L}}, \rho = \frac{Z}{\bar{L}^{\frac13}}$, and applying Theorem 4.2 in \cite{Pagani2}. We include our proof for self-containedness, and since our proof is different on the technical level than Pagani's, who uses the Mellin transform whereas ours will proceed via a Fourier transform in $s$. Moreover, the proof itself involves the derivation of the fractional elliptic equation \eqref{def:Op:A13} in Lemma \ref{lemma:ODE:Analysis}, which is not explicitly written in Pagani's work, and will be useful for our analysis in the case for $L$ not necessarily small, Section \ref{sec:spec}. 
\end{remark}

\subsection{Tangential Extensions}

We will eventually perform a Dirichlet-to-Neumann matching procedure in order to identify the traces $\gamma_{\Omega}^{(0)}(s), \gamma_{\Omega}^{(1)}(s)$. In order to compute the Dirichlet-to-Neumann map for $Z > 0$ and $Z < 0$, we need to: given a Dirichlet condition $\gamma_{\Omega}^{(0)}(s)$, compute the corresponding Neumann trace. In order to do so, we extend tangentially. Importantly, the logic is that we are \textit{given} $\gamma_{\Omega}^{(0)}$, then we extend in order to compute the corresponding Neumann trace. We now turn to the details. 

\vspace{2 mm}

\noindent \textit{Extensions for $Z > 0$:} Assume $\gamma_{\Omega}^{(0)}(s)$ is a given function. We first define its extension as follows 
\begin{align} \label{ext:gamma:1}
\gamma_{\Omega, \mathrm{Ext}}^{(0)}(s) := \begin{cases} \Xi_{\mathrm{Left}}(0) \chi(\frac{s-1}{\bar{L}}), \qquad s \le 1 \\ \gamma_{\Omega}^{(0)}(s), \qquad 1 < s < 1 + \bar{L} \\ \Xi_{\mathrm{Right}}(0) \chi(\frac{s- (1 + \bar{L})}{\bar{L}}) \qquad s \ge 1 + \bar{L}  \end{cases}
\end{align}
where $\chi$ is defined in \eqref{chidef}.

We now extend $F_I$ into $F_{\rightarrow}(s, Z) := 1_{1 \le s \le 1 + \bar{L}}F_I(s, Z)$. We then extend $\Omega_I$ to $\Omega_{\rightarrow}(s, Z)$ via the solution to 
\begin{align}
&(Z \p_s - \p_Z^2) \Omega_{\rightarrow} = F_{\rightarrow}, \qquad 1 < s < \infty \\
&\Omega_{\rightarrow}|_{s = 1} = \Xi_{\mathrm{Left}}(\frac{Z}{\bar{L}^{\frac13}}), \qquad Z > 0, \\
&\Omega_{\rightarrow}|_{Z = 0}= \gamma_{\Omega}^{(0)}(s) 1_{1 < s < 1 + \bar{L}} + \Xi_{\mathrm{Right}}(0) \chi(\frac{s- (1 + \bar{L})}{\bar{L}})  1_{s > 1 + \bar{L}}
\end{align}
Finally, we define the extension of $\Omega$ from $\Omega_{\rightarrow}$ to values of $s < 1$, which we denote by $\Omega_{\mathrm{Ext}}$ by 
\begin{align}
\Omega_{\mathrm{Ext}, +}(s, Z) := \begin{cases} \chi(\frac{s-1}{\bar{L}})  \Xi_{\mathrm{Left}}(\frac{Z}{\bar{L}^{\frac13}}) \qquad s < 1 \\ \Omega_{\rightarrow}(s, Z) \qquad s > 1 \end{cases}
\end{align}
We then define 
\begin{align}
F_{\mathrm{Ext}, +}(s, Z) := \begin{cases}  \frac{Z}{L} \chi'(\frac{s- 1}{\bar{L}}) \Xi_{\mathrm{Left}}(\frac{Z}{\bar{L}^{\frac13}}) - L^{-\frac23}\chi(\frac{s- 1}{\bar{L}}) \Xi''_{\mathrm{Left}}(\frac{Z}{\bar{L}^{\frac13}}), \qquad s < 1  \\  F_{\rightarrow}(s, Z), \qquad s > 1 \end{cases}
\end{align}
It is then the case that 
\begin{align}
&Z \p_s \Omega_{\mathrm{Ext}, +} - \p_Z^2 \Omega_{\mathrm{Ext}, +} = F_{\mathrm{Ext}, +}, \qquad (s, Z) \in \mathbb{R} \times (0, \infty), \\
&\Omega_{\mathrm{Ext}, +}|_{Z = 0} = \gamma^{(0)}_{\Omega, \mathrm{Ext}}(s), \qquad s \in \mathbb{R}. 
\end{align}
For technical reasons, we find it convenient to lift the boundary values $\Xi_{\mathrm{Left}}(0), \Xi_{\mathrm{Right}}(0)$ from \eqref{ext:gamma:1}. We do so by introducing 
\begin{align} \label{kew:plus}
Q_+(s, Z) := \Omega_{\mathrm{Ext}, +} -  \Xi_{\mathrm{Left}}(0) \chi(\frac{s-1}{\bar{L}}) \chi(\frac{Z}{\bar{L}^{\frac13}}) -  \Xi_{\mathrm{Right}}(0) \chi(\frac{s-(1 + \bar{L})}{\bar{L}}) \chi(\frac{Z}{\bar{L}^{\frac13}}).
\end{align}
Then, the following system is satisfied 
\begin{align}
&Z \p_s Q_{ +} - \p_Z^2 Q_{+} = F_{Q,+}, \qquad (s, Z) \in \mathbb{R} \times (0, \infty), \\
&Q_+|_{Z = 0} = \gamma^{(0)}_{Q, \mathrm{Ext}}(s), \qquad s \in \mathbb{R}. 
\end{align}
Above, the new lifted source term is defined as follows: 
\begin{align} \n
F_{Q,+}(s, Z) := &F_{\mathrm{Ext}, +}(s, Z) + \Xi_{\mathrm{Left}}(0)( - \frac{Z}{\bar{L}} \chi'(\frac{s - 1}{\bar{L}}) \chi(\frac{Z}{\bar{L}^{\frac13}}) + \frac{1}{\bar{L}^{\frac23}} \chi(\frac{s-1}{\bar{L}}) \chi''(\frac{Z}{\bar{L}^{\frac23}}) ) \\
& +  \Xi_{\mathrm{Right}}(0)( - \frac{Z}{\bar{L}} \chi'(\frac{s - (1 + \bar{L})}{\bar{L}}) \chi(\frac{Z}{\bar{L}^{\frac13}}) + \frac{1}{\bar{L}^{\frac23}} \chi(\frac{s - (1 + \bar{L})}{\bar{L}}) \chi''(\frac{Z}{\bar{L}^{\frac23}}) )
\end{align}
The boundary term $\gamma_{Q, \mathrm{Ext}}^{(0)}(s)$ is obtained by evaluating \eqref{kew:plus} at $\{Z = 0\}$ and using $\Omega_{\mathrm{Ext}, +}(s, 0) = \gamma_{\Omega, \mathrm{Ext}}^{(0)}(s)$, which is defined in \eqref{ext:gamma:1}, we obtain 
\begin{align}
\gamma^{(0)}_{Q, \mathrm{Ext}}(s) := & \begin{cases} 0, \qquad s \in (1, 1 + \bar{L})^C \\ \gamma_{\Omega}^{(0)}(s) -  \Xi_{\mathrm{Left}}(0) \chi(\frac{s-1}{\bar{L}})  -  \Xi_{\mathrm{Right}}(0) \chi(\frac{s - (1 + \bar{L})}{\bar{L}}), \qquad s \in (1, 1 + \bar{L}). \end{cases}
\end{align}
For future use, we will define 
\begin{align} \label{gamma:Q:def:he}
\gamma^{(0)}_Q(s) := \gamma_{\Omega}^{(0)}(s) -  \Xi_{\mathrm{Left}}(0) \chi(\frac{s-1}{\bar{L}}) -  \Xi_{\mathrm{Right}}(0) \chi(\frac{s - (1 + \bar{L})}{\bar{L}}), \qquad s \in (1, 1 + \bar{L}),
\end{align}
after which $\gamma^{(0)}_{Q, \mathrm{Ext}}$ becomes extension of $\gamma_Q^{(0)}(s)$ by zero to $(1, 1 + \bar{L})$.  At this point, it is convenient to separate out the components of $F_{Q, +}$ which depend on the data from the dependence on the original source term, $F_I$. Therefore, we define 
\begin{align} \n
F_{\mathrm{Data}, +} := & F_{\mathrm{Data}, +}[\Xi_{\mathrm{Left}}, \Xi_{\mathrm{Right}}] \\ \n
= & 1_{s < 1}( \frac{Z}{\bar{L}} \chi'(\frac{s-1}{\bar{L}})\Xi_{\mathrm{Left}}(\frac{Z}{\bar{L}^{\frac13}}) - \bar{L}^{-\frac23} \chi(\frac{s-1}{\bar{L}})\Xi''_{\mathrm{Left}}(\frac{Z}{\bar{L}^{\frac13}})) \\ \n
&\Xi_{\mathrm{Left}}(0)( - \frac{Z}{\bar{L}} \chi'(\frac{s - 1}{\bar{L}}) \chi(\frac{Z}{\bar{L}^{\frac13}}) + \frac{1}{\bar{L}^{\frac23}} \chi(\frac{s-1}{\bar{L}}) \chi''(\frac{Z}{\bar{L}^{\frac23}}) ) \\ \label{F:plus}
& +  \Xi_{\mathrm{Right}}(0)( - \frac{Z}{\bar{L}} \chi'(\frac{s - (1 + \bar{L})}{\bar{L}}) \chi(\frac{Z}{\bar{L}^{\frac13}}) + \frac{1}{\bar{L}^{\frac23}} \chi(\frac{s - (1 + \bar{L})}{\bar{L}}) \chi''(\frac{Z}{\bar{L}^{\frac23}}) ).
\end{align}
After this, we have 
\begin{align}
F_{Q, +} = 1_{1 \le s \le 1 + \bar{L}}F_I +  F_{\mathrm{Data}, +}[\Xi_{\mathrm{Left}}, \Xi_{\mathrm{Right}}].
\end{align}

\vspace{2 mm}

\noindent \textit{Extensions for $Z < 0$:} A completely symmetric construction is performed for $Z < 0$; we denote those extensions by $ \Omega_{\mathrm{Ext}, -}$, $F_{\mathrm{Ext}, -}$, and $Q_{-}$, and the corresponding operator on the data by $F_{\mathrm{Data}, -}$. 

\subsection{Derivation of Fractional Elliptic Problem}
We begin by defining Dirichlet-to-Neumann operators for the inhomogeneous problem 
\begin{subequations}
\begin{align}\label{dada1}
&Z \p_s Q_{\iota} - \p_Z^2 Q_{\iota} = F_{Q, \iota}, \qquad (s, Z) \in \mathbb{R} \times \mathbb{R}_{\iota}, \qquad \iota \in \{ +, -\} \\ \label{dada2}
&Q_{\iota}|_{Z = 0} = \gamma^{(0)}_{Q, \mathrm{Ext}}(s), \qquad Q_{\iota}|_{Z = \iota \infty} = 0, \qquad s \in \mathbb{R}.
\end{align}
\end{subequations}
We want to consider the \textit{restriction} of these Neumann trace of the extended problem:  
\begin{align}
\gamma_{Q, \iota}^{(1)}[\gamma^{(0)}_{Q}](s) := \mathcal{R}_{(1, 1 + \bar{L})} \p_Z Q_{\iota}(s, 0), \qquad s \in (1, 1 + \bar{L}), \qquad \iota \in \{+, -\}.
\end{align} 
where $\mathcal{R}_{(1, 1 + \bar{L})}$ is the restriction operator on the interval $(1, 1 + \bar{L})$.

\begin{lemma} \label{lemma:ODE:Analysis} Let $F_I \in L^2_{sZ}$. Assume the data satisfies $\sum_{i = 0}^2\| \langle \rho \rangle (\p_\rho^i \Xi_{\mathrm{Left}},\p_\rho^i \Xi_{\mathrm{Right}}) \|_{L^2_\rho} < \infty$. There exists a unique solution, $\gamma^{(1)}_{\mathcal{Q}}(s)$, to the Dirichlet to Neumann equation: 
\begin{align}
\gamma_{Q, +}^{(1)}[\gamma^{(0)}_{Q}] = \gamma_{Q, -}^{(1)}[\gamma^{(0)}_{Q}]. 
\end{align}
Moreover, this fixed point satisfies the following fractional elliptic equation  
\begin{align} \label{a:Green:1}
(-\Delta_D)^{\frac16} \gamma_{Q}^{(0)} =  \sum_{\iota \in \pm} \mathcal{T}_{- \frac 1 3, \iota} [F_I 1_{1 \le s \le 1 + \bar{L}}] + \sum_{\iota \in \pm} \mathcal{T}_{- \frac 1 3, \iota} [F_{\mathrm{Data}, \iota}[\Xi_{\mathrm{Left}}, \Xi_{\mathrm{Right}}]], \qquad s \in (1, 1 + \bar{L})
\end{align}
where the linear operators above are defined as follows:
\begin{align} \label{def:Op:A13}
(-\Delta_D)^{\frac16}\gamma_{Q}^{(0)} := & \mathcal{F}^{-1} \Big[|\xi|^{\frac13}  \mathcal{F}[ \gamma^{(0)}_{Q, \mathrm{Ext}} ] \Big],  \\ \label{defn:T:op}
 \mathcal{T}_{- \frac 1 3, \iota} [\cdot] := &(-1)^{\iota} \mathcal{R}_{(1, 1+\bar{L})} \mathcal{F}^{-1} \Big[ \mathcal{G}_{\iota}[\cdot](\xi) \Big] . 
\end{align}
and the integral operators $\mathcal{G}_{\pm}[\cdot](\xi)$ are defined in \eqref{G:top:defn}, \eqref{G:bot:defn}. 
\end{lemma}

\begin{proof} The proof proceeds by finding a unique $\gamma^{(0)}_{Q}(s)$ so that the top Dirichlet-to-Neumann operator coincides with the bottom Dirichlet-to-Neumann operator. As such, we will separately compute the top and bottom Dirichlet-Neumann operators. 

\vspace{2 mm}

\noindent \textit{The ``top" problem, $Z > 0$:} Going to Fourier in $s$, we obtain from \eqref{dada1} -- \eqref{dada2}: 
\begin{subequations}
\begin{align} \label{jj1}
&i \xi Z \widehat{Q}_{+} -\p_Z^2 \widehat{Q}_{+} = \widehat{F}_{Q, +}, \qquad (\xi, Z) \in \mathbb{R} \times \mathbb{R}_+, \\ \label{jjb}
&\widehat{Q}_{+}(\xi, 0) = \widehat{\gamma}^{(0)}_{Q, \mathrm{Ext}}(\xi), \qquad \widehat{Q}_{+}(\xi,  \infty) = 0
\end{align}
\end{subequations}
We solve this ODE via
\begin{align} \label{repL}
\widehat{Q}_{+}(\xi, Z) = c_1(\xi) ai( (i \xi)^{\frac 1 3}Z) + c_2(\xi) bi((i \xi)^{\frac 1 3} Z) + \widehat{U}_{P,+}(\xi, Z),
\end{align}
where $ai(\cdot), bi(\cdot)$ are the standard Airy functions. We refer the reader to Appendix \ref{app:Airy} for definitions and properties of these functions, and the reference \cite{Olver} for further properties.  

The particular solution is defined via 
\begin{align} \n
 \widehat{U}_{P,+}(\xi, Z) := &   \frac{1}{(i \xi)^{\frac 1 3} \mathcal{W}[ai, bi]} ai( (i \xi)^{\frac 1 3}Z) \int_0^Z bi( (i \xi)^{\frac 1 3} Z') \widehat{F_{Q, +}}(\xi, Z') \ud Z' \\ \label{repL1}
 &+ \frac{ 1 }{(i \xi)^{\frac 1 3}\mathcal{W}[ai, bi]} bi( (i \xi)^{\frac 1 3}Z) \int_Z^{\infty} ai( (i \xi)^{\frac 1 3} Z') \widehat{F_{Q, +}}(\xi, Z') \ud Z'.
\end{align}
Above, $\mathcal{W}[ai, bi]$ is the Wronskian between the two fundamental solutions, $ai(\cdot), bi(\cdot)$, and is therefore a constant by classical ODE theory. We note that the due to $\widehat{F}_{Q,+} \in L^2_{\xi Z}$ and the exponential decay of $ai((i\xi)^{\frac13}Z')$ as $Z' \rightarrow \infty$ according to Lemma \ref{lemma:fund:solutions}, the second integral above converges (the first integral above is clearly finite for each fixed $Z < \infty$).  

In order to choose decaying solutions as $Z \rightarrow \infty$, we must set $c_2(\xi) = 0$ in \eqref{repL}, according to \eqref{airy:asy:1}. This, in turn, is due to the fact that for each fixed $\xi \in \mathbb{R}$, for $Z \ge 0$, the quantity $(i\xi)^{\frac13}Z \in \Sigma_0$.  

We now have an equation for the remaining coefficient $c_1(\xi)$:
\begin{align} \label{sys:ci:1}
&c_1(\xi) ai(0) =\widehat{\gamma}^{(0)}_{Q, \mathrm{Ext}}(\xi) - \widehat{U}_{P,+}(\xi, 0),
\end{align}
which we immediately solve to obtain 
\begin{align} \label{df:c1}
c_1(\xi) = \frac{\widehat{\gamma}^{(0)}_{Q, \mathrm{Ext}}(\xi) - \widehat{U}_{P,+}(\xi, 0)}{ai(0)}.
\end{align}

We thus obtain the normal derivative as: 
\begin{align}\label{pztop}
\p_Z \widehat{Q}_{+}(\xi, 0) = &  (i \xi)^{\frac 13} c_1(\xi) ai'(0) + \p_Z \widehat{U}_{P,+}(\xi, 0)  =  (i \xi)^{\frac 1 3} \frac{ai'(0)}{ai(0)} \widehat{\gamma}^{(0)}_{Q, \mathrm{Ext}}(\xi) + \mathcal{G}_{+}[F_{Q, +}](\xi),
\end{align}
where we define the integral operator 
\begin{align} \label{G:top:defn}
\mathcal{G}_{+}[F_{Q, +}](\xi) := & \frac{1}{\mathcal{W}[ai, bi]} \Big( bi'(0) - bi(0) \frac{ai'(0)}{ai(0)} \Big) \int_0^{\infty} ai ( (i \xi)^{\frac 1 3} Z' ) \widehat{F}_{Q, +}(\xi, Z') \ud Z'.
\end{align}
We introduce the quantity
\begin{align} \label{def:M:bot:delta:s}
M_{+} := \frac{ai'(0)}{ai(0)}.
\end{align}
We now take the formula \eqref{pztop} and restrict to $s \in (1, 1+ \bar{L})$ on the spatial side to obtain our top Dirichlet to Neumann operator: 
\begin{align} \label{DN:top}
\gamma^{(1)}_{Q, +}[\gamma^{(0)}_{Q}] = & \mathcal{R}_{(1, 1 + \bar{L})} \mathcal{F}^{-1} \Big[ (i \xi)^{\frac 1 3} M_{+}(\xi) \widehat{\gamma}^{(0)}_{Q, \mathrm{Ext}}(\xi) + \mathcal{G}_{+}[F_{Q, +}](\xi) \Big].
\end{align}

We now repeat the same steps for the bottom problem. 

\vspace{2 mm} 

\noindent \textit{The ``bottom" problem, $Z < 0$:} We arrive at the problem set on $s \in \mathbb{R}$:
\begin{subequations}
\begin{align}
&Z \p_s Q_{ -} - \p_Z^2 Q_{ -} = F_{Q, -}, \qquad (s, Z) \in \mathbb{R} \times (-\infty, 0) \\
&Q_{-}|_{Z = 0} = \gamma^{(0)}_{Q, \mathrm{Ext}}(s), \qquad Q_{ -}|_{Z = - \infty} = 0, \qquad s \in \mathbb{R}.
\end{align}
\end{subequations}
Going again to Fourier in the $s$ variable, we have for the bottom problem: 
\begin{subequations}
\begin{align} \label{jj11}
&i \xi Z \widehat{Q}_{ -}  - \p_Z^2 \widehat{Q}_{-}  = \widehat{F}_{Q, -}, \qquad (\xi, Z) \in \mathbb{R} \times (-\infty,  0), \\ \label{jjb1}
&\widehat{Q}_{-}(\xi, 0) = \widehat{\gamma}^{(0)}_{Q, \mathrm{Ext}}(\xi), \qquad \widehat{Q}_{-}(\xi,  -\infty) = 0
\end{align}
\end{subequations}
Recall now from Lemma \ref{lemma:fund:solutions} that the fundamental solutions in the region $\Sigma_1$ are $\{ ai(w), ai(e^{- \frac{2}{3} \pi i}w) \}$ and on $\Sigma_{-1}$ are $\{ ai(w), ai(e^{\frac{2}{3} \pi i}w) \}$. We therefore define the functions 
\begin{align}
B_{-1}(w) := ai(e^{\frac{2}{3} \pi i}w), \qquad B_{+1}(w) := ai(e^{-\frac{2}{3} \pi i}w)
\end{align}
We will treat the case $\xi > 0$ in detail, and the case of $\xi < 0$ will follow with similar arguments. In this case $(i \xi)^{\frac 1 3} Z$ for $Z < 0$ lies in the region $\Sigma_{-1}$. Hence, we now write a solution
\begin{align} \label{relPP}
\widehat{\mathcal{Q}}_{-}(\xi, Z) = d_1(\xi) ai( (i \xi)^{\frac 13}Z) + d_2(\xi) B_{-1}( (i \xi)^{\frac 1 3}Z) + \widehat{U}_{P,-}(\xi, Z),
\end{align}
where the particular solution above is defined to be 
\begin{align} \n
 \widehat{U}_{P,-}(\xi, Z) := &- \frac{1}{(i \xi)^{\frac 1 3} \mathcal{W}[ai, B_{-1}] } B_{-1}( ( i \xi )^{\frac 1 3} Z ) \int_Z^0 ai( (i \xi)^{\frac 1 3} Z' ) \widehat{F}_{Q, -}(\xi, Z') \ud Z' \\
 & - \frac{1}{(i \xi)^{\frac 1 3} \mathcal{W}[ai, B_{-1}]} ai( (i \xi)^{\frac 1 3} Z ) \int_{-\infty}^Z B_{-1}( (i\xi)^{\frac 1 3} Z' ) \widehat{F}_{Q,-}(\xi, Z') \ud Z'. 
\end{align}
In this case, since for $\xi > 0$ and $Z < 0$, we have $(i\xi)^{\frac13}Z \in \Sigma_{-1}$, we have according to \eqref{airy:asy:3} that $ai((\xi)^{\frac13}Z)$ grows as $Z \rightarrow - \infty$ and $B_{-1}((i\xi)^{\frac13}Z)$ decays as $Z \rightarrow - \infty$. Therefore, we can set $d_1(\xi) = 0$, and obtain the simpler representation:  
\begin{align} \label{relPPP}
\widehat{\mathcal{Q}}_{-}(\xi, Z) =  d_2(\xi) B_{-1}( (i \xi)^{\frac 1 3}Z) + \widehat{U}_{P,-}(\xi, Z),
\end{align}
Evaluation at $Z = 0$ gives the following equation for the Fourier multiplier $d_2(\xi)$, 
\begin{align} \label{eq:d1:a}
& d_2(\xi) B_{-1}(0) = \widehat{\gamma}^{(0)}_{Q, \mathrm{Ext}}(\xi) - \widehat{U}_{P,-}(\xi, 0),  
\end{align}
which we immediately solve to obtain 
\begin{align}
d_2(\xi) =  \frac{\widehat{\gamma}^{(0)}_{Q, \mathrm{Ext}}(\xi) - \widehat{U}_{P,-}(\xi, 0)}{B_{-1}(0)}.
\end{align}

We thus obtain the following expression for the normal derivative:
\begin{align} \label{dYuB}
\p_Z \widehat{\mathcal{Q}}_{-}(\xi, 0) = &  (i \xi)^{\frac 1 3} d_2(\xi) B_{-1}'(0) + \widehat{U}_{P,-}'(\xi, 0) =  (i \xi)^{\frac 1 3} M_{-}(\xi) \widehat{\gamma}^{(0)}_{Q, \mathrm{Ext}}(\xi) + \mathcal{G}_{-}[F_{\mathcal{Q}, -}](\xi), 
\end{align}
where we have defined the quantity  
\begin{align} \label{def:M:bot:delta}
M_{-}(\xi) := & \frac{B_{-1}'(0)}{B_{-1}(0)}, \qquad \xi > 0. 
\end{align}
and the integral operator (for $\xi > 0$)
\begin{align} \label{G:bot:defn}
\mathcal{G}_{-}[F_{Q, -}](\xi) :=& - \frac{1}{\mathcal{W}[ai, B_{-1}]} \Big( ai'(0) - ai(0) \frac{B_{-1}'(0)}{B_{-1}(0)} \Big) \int_{-\infty}^0 B_{-1} \Big( (i \xi)^{\frac 1 3} Z' \Big) \widehat{F}_{Q, -}(\xi, Z') \ud Z'. 
\end{align}
We will see that the expression \eqref{def:M:bot:delta} will change for $\xi < 0$, and therefore we think of $M_{-}(\xi)$ as a Fourier multiplier as opposed to a simple constant. 

We now address the case when $\xi < 0$. In this case, the main change will be that the symbol $M_{-}(\xi)$ will be given by $M_{-}(\xi) := \frac{B_{+1}'(0)}{B_{+1}(0)}$, as opposed to \eqref{def:M:bot:delta}. Therefore, to unify these two cases, we define
\begin{align} \label{na}
M_{-}(\xi) :=  \frac{B_{-sgn(\xi)}'(0)}{B_{-sgn(\xi)}(0)}, \qquad \xi \in \mathbb{R}
\end{align}

 We now take the formula \eqref{dYuB} and restrict to $s \in (1, 1 + \bar{L})$ on the spatial side to obtain our bottom Dirichlet to Neumann operator: 
\begin{align}\label{DN:bot}
\gamma_{\mathcal{Q}, -}^{(1)}[\gamma^{(0)}_{Q}] = & \mathcal{R}_{(1, 1+ \bar{L})} \mathcal{F}^{-1} \Big[ (i \xi)^{\frac 1 3} M_{-}(\xi) \widehat{\gamma}^{(0)}_{Q, \mathrm{Ext}}(\xi) + \mathcal{G}_{-}[F_{Q, -}](\xi) \Big]. 
\end{align}

\vspace{2 mm}

\noindent \textit{Matching:} We now want to solve the fixed point equation 
\begin{align}
\gamma_{Q, -}^{(1)}[\gamma^{(0)}_{Q}](s) =\gamma_{Q, +}^{(1)}[\gamma^{(0)}_{Q}](s) , \qquad s \in (1, 1+\bar{L}),
\end{align}
which, upon equating \eqref{DN:top} with \eqref{DN:bot}, we get
\begin{align} \label{we:get}
\mathcal{R}_{(1, 1 + \bar{L})} \mathcal{F}^{-1} \Big[(i \xi)^{\frac 1 3} \mathcal{M}(\xi)  \widehat{\gamma}^{(0)}_{Q, \mathrm{Ext}}(\xi) \Big] = &   \mathcal{R}_{(1,1 + \bar{L})} \mathcal{F}^{-1} \Big[ \mathcal{G}_{-}[F_{Q, -}](\xi) - \mathcal{G}_{+}[F_{Q, +}](\xi) \Big],
\end{align} 
where we have defined the Fourier multiplier 
\begin{align} \label{defn:M:major}
\mathcal{M}(\xi) := M_{+}(\xi) - M_{-}(\xi). 
\end{align}
This Fourier multiplier will be studied below in Lemma \ref{prop:Fourier:mult}, and the following structural identity will be shown: $\mathcal{M}(\xi) = \frac{1}{2\pi ai(0)^2} e^{- i \frac{\pi}{6} sgn(\xi)}$. Temporarily accepting this structural identity, we proceed to analyze the left-hand side of \eqref{we:get} as follows: 
\begin{align*}
\mathcal{R}_{(1, 1 + \bar{L})} \mathcal{F}^{-1} \Big[(i \xi)^{\frac 1 3} \mathcal{M}(\xi)  \widehat{\gamma}^{(0)}_{Q, \mathrm{Ext}}(\xi) \Big]  := & \mathcal{R}_{(1, 1 + \bar{L})} \mathcal{F}^{-1} \Big[(i \xi)^{\frac 1 3}  \frac{1}{2\pi ai(0)^2} e^{- i \frac{\pi}{6} sgn(\xi)}   \widehat{\gamma}^{(0)}_{Q, \mathrm{Ext}} \Big]  \\
= &  \frac{1}{2\pi ai(0)^2} \mathcal{R}_{(1,1 +\bar{L})} \mathcal{F}^{-1} \Big[ |\xi|^{\frac 13}  \widehat{\gamma}^{(0)}_{Q, \mathrm{Ext}} \Big] \\
= &  \frac{1}{2\pi ai(0)^2} \mathcal{R}_{(1,1 +\bar{L})} (-\Delta_D)^{\frac 16}  \gamma^{(0)}_{Q}. 
\end{align*}
Above, we have used the elementary identity 
\begin{align} \label{Wrons:cancel:0}
(i \xi)^{\frac 1 3} e^{- i \frac{\pi}{6} sgn(\xi)} = |\xi|^{\frac 1 3} e^{i \frac{\pi}{6} sgn(\xi)} e^{- i \frac{\pi}{6} sgn(\xi)}  = |\xi|^{\frac 1 3}. 
\end{align}
This concludes the proof of the lemma. 
\end{proof}

We now need to establish the crucial structural identity used above. 
\begin{lemma} \label{prop:Fourier:mult}  We have the following structural identity for the symbol defined in \eqref{defn:M:major}:
\begin{align} \label{structural:1}
\mathcal{M}(\xi) = \frac{1}{2\pi ai(0)^2} e^{- i \frac{\pi}{6} sgn(\xi)}.
\end{align}

\end{lemma}
\begin{proof} We have the expression 
\begin{align} \n
\mathcal{M}(\xi) = & \Big( \frac{ai'(0)}{ai(0)} -\frac{B_{-sgn(\xi)}'(0)}{B_{-sgn(\xi)}(0)}  \Big) \\ \n
= & \frac{1}{ai(0) B_{-sgn(\xi)}(0)} \Big( B_{-sgn(\xi)}ai'(0) - ai(0) B_{-sgn(\xi)}'(0) \Big)\\ \n
= & \frac{1}{ai(0)^2} \mathcal{W}[ai, B_{-sgn(\xi)}]  \\ 
= & \frac{1}{ai(0)^2} \frac{1}{2\pi} e^{- i sgn(\xi) \frac{\pi}{6} }.
\end{align}
Above, we have used that $B_{-sgn(\xi)}(0) = ai(0)$ by definition, which enables us to pull out the $\xi$ independent factor of $ \frac{1}{ai(0) B_{-sgn(\xi)}(0)}$, and subsequently appeal to the explicit numerical computation for the associated Wronskians on P. 416 of \cite{Olver}.

\end{proof}

\subsection{Proof of Proposition \ref{prop:Airy}}

We now analyze the fractional Poisson problem \eqref{a:Green:1} to give a proof of Proposition \ref{prop:Airy}. 

\begin{proof}[Proof of Proposition \ref{prop:Airy}] The proof will proceed in several steps, which we explicitly delineate below. 

\vspace{2 mm}

\noindent \textit{Step 1: Bound on $\gamma^{(0)}_{Q}$: } First, we let $G_Q$ equal the right-hand side of \eqref{a:Green:1}, so that 
\begin{align} \label{a:Green:fr}
&(-\Delta_D)^{\frac16} \gamma_{Q}^{(0)} =  G_Q \text{ for }s \in (1, 1 + \bar{L}), \qquad \gamma_{Q, \mathrm{Ext}}^{(0)}|_{(1, 1+ \bar{L})^C} = 0. 
\end{align}
We now introduce rescalings which normalize our functions to $(0, 1)$:
\begin{align}
t := \frac{s-1}{\bar{L}}, \qquad \Gamma(t) := \gamma^{(0)}_{Q}(\bar{L}t), \qquad  H_{Q}(t) := G_{Q}(s). 
\end{align}
We define again the extension $\Gamma_{\mathrm{Ext}}(t) := \Gamma(t) 1_{t \in (0, 1)}$. In the normalized $t$ variable, \eqref{a:Green:fr} converts to 
\begin{align} \label{poisson:4}
&(-\Delta_D)^{\frac16} \Gamma =  \bar{L}^{\frac13} H_{Q}, \text{ for } t \in (0, 1), \qquad \Gamma_{\mathrm{Ext}}|_{(0, 1)^C} = 0. 
\end{align}
From standard estimates on the fractional Laplacian operator set on $(0, 1)$ (see for instance Theorem 1.1 in \cite{Gulsa} and the references therein, for instance \cite{Grubb}), we have 
\begin{align} \label{where:1}
\| \mathcal{F}^{-1} |\xi|^{\frac13} \widehat{\Gamma}_{\mathrm{Ext}}(\xi) \|_{L^2_t(\mathbb{R})} \lesssim \bar{L}^{\frac13} \| H_{Q} \|_{L^2_t(0, 1)}. 
\end{align}
We now rewrite the operator on the left-hand side via
\begin{align} \n
\mathcal{F}^{-1}\{ |\xi|^{\frac13} \widehat{\Gamma}_{\mathrm{Ext}}(\xi) \}(t) = & \int e^{i \xi t} \frac{|\xi|^{\frac13}}{\bar{L}} \widehat{\gamma}^{(0)}_{Q, \mathrm{Ext}}(\frac{\xi}{\bar{L}}) \ud \xi = \bar{L}^{- \frac23}  \int e^{i \frac{\xi}{\bar{L}} \bar{L}t} \Big(\frac{|\xi|}{\bar{L}} \Big)^{\frac13} \widehat{\gamma}^{(0)}_{Q, \mathrm{Ext}}(\frac{\xi}{\bar{L}}) \ud \xi \\
= & \bar{L}^{\frac13} \int e^{i \eta \bar{L}t} |\eta|^{\frac13} \widehat{\gamma}^{(0)}_{Q, \mathrm{Ext}}(\eta) \ud \eta = \bar{L}^{\frac13} \mathcal{F}^{-1} \{ |\xi|^{\frac13} \widehat{\gamma}^{(0)}_{Q, \mathrm{Ext}}(\xi) \}(\bar{L}t). 
\end{align}
Inserting back into \eqref{where:1}, we get
\begin{align}
\bar{L}^{\frac13} \| \mathcal{F}^{-1} \{ |\xi|^{\frac13} \widehat{\gamma}^{(0)}_{Q, \mathrm{Ext}}(\xi) \}(L\cdot)\|_{L^2_t(\mathbb{R})} \lesssim \bar{L}^{\frac13} \| H_{Q} \|_{L^2_t(0, 1)} = \bar{L}^{\frac13} \| G_{Q}(\bar{L}\cdot) \|_{L^2_t(0, 1)}
\end{align}
The factor of $\bar{L}^{\frac13}$ cancels from both sides. We now change variables $t \mapsto s$, which contributes a Jacobian of $\bar{L}^{\frac12}$ on both sides, which also cancels out. This results in the bound 
\begin{align} \n
&\| \mathcal{F}^{-1}  |\xi|^{\frac 1 3} \widehat{\gamma}^{(0)}_{Q, \mathrm{Ext}} \|_{L^2_s(\mathbb{R})} \\ \n
 \lesssim & \| (-\Delta_D)^{\frac 16} \gamma^{(0)}_{Q} \|_{L^{2}_s(1,1 +\bar{ L})} = \| G_{Q} \|_{L^2(1,1 + \bar{L})} \\ \label{lower:A23}
\lesssim & \sum_{\iota \in \pm} \| \mathcal{T}_{- \frac 1 3, \iota} [F_I 1_{1 \le s \le 1 + \bar{L}}]  \|_{L^2_s(1, 1+ \bar{L})}+ \sum_{\iota \in \pm} \| \mathcal{T}_{- \frac 1 3, \iota} [F_{\mathrm{Data}, \iota}[\Xi_{\mathrm{Left}}, \Xi_{\mathrm{Right}}]] \|_{L^2_s(1,1+ \bar{L})}.
\end{align} 
While this furnishes a bound on $\gamma_{Q}^{(0)}$ (and hence can also be used to produce an estimate on $\gamma_{\Omega}^{(0)}$), we do not record these bounds at this stage. Our main motivation is to obtain a bound on $\gamma_{\Omega}^{(1)}$. 

\vspace{2 mm}

\noindent \textit{Step 2: Bound on $\gamma_{Q}^{(1)}$: } We notice now through the expression \eqref{DN:top} (we could equally well use \eqref{DN:bot}), upon invoking \eqref{lower:A23} that we have 
\begin{align} \label{asiun:1}
\| \gamma_{Q}^{(1)} \|_{L^2_s(1, 1 + \bar{L})} \lesssim & \sum_{\iota \in \pm} \| \mathcal{T}_{- \frac 1 3, \iota} [F_I 1_{1 \le s \le 1 + \bar{L}}]  \|_{L^2_s(1, 1+ \bar{L})}+ \sum_{\iota \in \pm} \| \mathcal{T}_{- \frac 1 3, \iota} [F_{\mathrm{Data}, \iota}[\Xi_{\mathrm{Left}}, \Xi_{\mathrm{Right}}]] \|_{L^2_s(1,1+ \bar{L})}.
\end{align} 

\vspace{2 mm}

\noindent \textit{Step 3: Bound on $\gamma_{\Omega}^{(1)}$: } By differentiating \eqref{kew:plus} in $Z$, and evaluating at $Z = 0$, we have 
\begin{align}
\gamma^{(1)}_{Q}(s) = \gamma^{(1)}_{\Omega}(s) - \bar{L}^{-\frac13} \Xi_{\mathrm{Left}}(0) \chi(\frac{s-1}{\bar{L}})  - \bar{L}^{-\frac13}  \Xi_{\mathrm{Right}}(0) \chi(\frac{s - (1 + \bar{L})}{\bar{L}}).
\end{align}
From here, we easily obtain 
\begin{align} \n
\| \gamma^{(1)}_{\Omega} \|_{L^2_s(1, 1 + \bar{L})} \lesssim & \| \gamma^{(1)}_{Q} \|_{L^2_s(1, 1 + \bar{L})} + \bar{L}^{\frac16}(  \| \Xi_{\mathrm{Left}}, \Xi_{\mathrm{Right}} \|_{L^\infty} ) \\ \n
\lesssim & \sum_{\iota \in \pm} \| \mathcal{T}_{- \frac 1 3, \iota} [F_I 1_{1 \le s \le 1 + \bar{L}}]  \|_{L^2_s(1, 1+ \bar{L})}+ \sum_{\iota \in \pm} \| \mathcal{T}_{- \frac 1 3, \iota} [F_{\mathrm{Data}, \iota}[\Xi_{\mathrm{Left}}, \Xi_{\mathrm{Right}}]] \|_{L^2_s(1,1+ \bar{L})} \\ \label{guy:1}
& + \bar{L}^{\frac16}(  \| \Xi_{\mathrm{Left}}, \Xi_{\mathrm{Right}} \|_{L^\infty} ).
\end{align}

\vspace{2 mm}

\noindent \textit{Step 4: Evaluation of Source Term $F_I$: } We turn to the evaluation of the source term contribution from \eqref{guy:1}, namely $\sum_{\iota \in \pm} \| \mathcal{T}_{- \frac 1 3, \iota} [F_I 1_{1 \le s \le 1 + \bar{L}}]  \|_{L^2_s(1, 1+ \bar{L})}$. By symmetry, it suffices to consider $\iota = +$. Recall the definition \eqref{defn:T:op}, from which we have 
\begin{align} \n
\| \mathcal{T}_{- \frac 1 3, +} [F_I 1_{1 \le s \le 1 + \bar{L}}]  \|_{L^2_s(1, 1+ \bar{L})} = & \| \mathcal{R}_{(1, 1+\bar{L})} \mathcal{F}^{-1} \Big[ \mathcal{G}_{+}[ F_I 1_{1 \le s \le 1 + \bar{L}}](\xi) \Big] \|_{L^2_s(1, 1 + \bar{L})} \\ \n
\lesssim &  \|  \mathcal{F}^{-1} \Big[ \mathcal{G}_{+}[ F_I 1_{1 \le s \le 1 + \bar{L}}](\xi) \Big] \|_{L^2_s(\mathbb{R})} \\ \label{asinasin:2}
= & \| \mathcal{G}_{+}[ F_I 1_{1 \le s \le 1 + \bar{L}}](\xi)\|_{L^2_\xi(\mathbb{R})}.  
\end{align}
Above, we have used the trivial bound that $\mathcal{R}_{(1, 1 + \bar{L})}$ is a bounded operator, followed by Plancharel. We now recall the definition \eqref{G:top:defn} to bound
\begin{align} \n
&\| \mathcal{G}_{+}[ F_I 1_{1 \le s \le 1 + \bar{L}}](\xi)\|_{L^2_\xi(\mathbb{R})} \lesssim \|  \int_0^{\infty} ai \Big( (i \xi)^{\frac 1 3} Z' \Big) \mathcal{F}[F_I 1_{1 \le s \le 1 + \bar{L}}](\xi, Z') \ud Z' \|_{L^2_\xi(\mathbb{R})} \\ \n
\lesssim & \|  (\int_0^{\infty} (ai ( (i \xi)^{\frac 1 3} Z'))^2 (i\xi)^{\frac13} \ud Z' )^{\frac12} (\int_0^\infty (i\xi)^{-\frac13}  \mathcal{F}[F_I 1_{1 \le s \le 1 + \bar{L}}](\xi, Z')^2 \ud Z')^{\frac12} \|_{L^2_\xi(\mathbb{R})} \\ \label{iheap:1}
\lesssim & \| |\p_s|^{-\frac16} F_I 1_{1 \le s \le 1 + \bar{L}}  \|_{L^2_{sZ}} \lesssim  \| F_I 1_{1 \le s \le 1 + \bar{L}}  \|_{L^2_Z(L^{\frac32}_s)} \lesssim  \bar{L}^{\frac16} \| F_I \|_{L^2_{sZ}}. 
\end{align}
In the last line, we have used the one dimensional fractional Sobolev inequality (for a generic function $f(s)$), $\| |\p_s|^{-\frac16} f \|_{L^2(\mathbb{R})} \lesssim \| f \|_{L^{\frac32}(\mathbb{R})}$, coupled with the bound $\| f 1_{1 \le s \le 1 + \bar{L}} \|_{L^{\frac32}} \le \bar{L}^{\frac16} \| f 1_{1 \le s \le 1 + \bar{L}}\|_{L^2}$. 

\vspace{2 mm}

\noindent \textit{Step 5: Evaluation of Data Term $F_{\mathrm{Data}, \pm}$: } We now turn to the evaluation of the data contributions from \eqref{guy:1}, namely $\sum_{\iota \in \pm} \| \mathcal{T}_{- \frac 1 3, \iota} [F_{\mathrm{Data}, \iota}[\Xi_{\mathrm{Left}}, \Xi_{\mathrm{Right}}]] \|_{L^2_s(1,1+ \bar{L})}$. Again by symmetry, we treat the case $\iota = +$. By repeating \eqref{asinasin:2}, we need to estimate 
\begin{align} \n
&\| \mathcal{T}_{- \frac 1 3, +} [F_{\mathrm{Data}, +}[\Xi_{\mathrm{Left}}, \Xi_{\mathrm{Right}}]] \|_{L^2_s(1,1+ \bar{L})} \lesssim  \| \mathcal{G}_{+}[ F_{\mathrm{Data}, +}[\Xi_{\mathrm{Left}}, \Xi_{\mathrm{Right}}] ](\xi)\|_{L^2_\xi(\mathbb{R})} \\ \n
\lesssim & \| \mathcal{G}_{+}[ 1_{s < 1} \frac{Z}{\bar{L}} \chi'(\frac{s - 1}{\bar{L}})\Xi_{\mathrm{Left}}(\frac{Z}{\bar{L}^{\frac13}})  ](\xi)\|_{L^2_\xi(\mathbb{R})} + \| \mathcal{G}_{+}[ 1_{s < 1} \bar{L}^{-\frac23} \chi(\frac{s-1}{\bar{L}})\Xi''_{\mathrm{Left}}(\frac{Z}{\bar{L}^{\frac13}})  ](\xi)\|_{L^2_\xi(\mathbb{R})} \\ \n
& + \| \mathcal{G}_{+}[\Xi_{\mathrm{Left}}(0)( - \frac{Z}{\bar{L}} \chi'(\frac{s-1}{\bar{L}}) \chi(\frac{Z}{\bar{L}^{\frac13}}) + \frac{1}{\bar{L}^{\frac23}} \chi(\frac{s-1}{\bar{L}}) \chi''(\frac{Z}{\bar{L}^{\frac23}}) )  ](\xi)\|_{L^2_\xi(\mathbb{R})} \\ \n
& + \| \mathcal{G}_{+}[ \Xi_{\mathrm{Right}}(0)( - \frac{Z}{\bar{L}} \chi'(\frac{s - (1 + \bar{L})}{\bar{L}}) \chi(\frac{Z}{\bar{L}^{\frac13}}) + \frac{1}{\bar{L}^{\frac23}} \chi(\frac{s - (1 + \bar{L})}{\bar{L}}) \chi''(\frac{Z}{\bar{L}^{\frac23}}) ) ](\xi)\|_{L^2_\xi(\mathbb{R})} \\ \label{spurt:1}
=: &\sum_{i = 1}^4 D_i. 
\end{align}
where above we have recalled the definition \eqref{F:plus}. By symmetry $D_3$ and $D_4$ are treated the same. We have upon essentially repeating the bound on the operator $G_+$ from \eqref{iheap:1}
\begin{align} \label{spurt:2} 
D_1 \lesssim &L^{-\frac23} \bar{L}^{\frac16} \| \chi'(\frac{s-1}{\bar{L}}) \frac{Z}{\bar{L}^{\frac13}} \Xi_{\mathrm{Left}}(\frac{Z}{\bar{L}^{\frac13}}) \|_{L^2_{sZ}} \lesssim \| \frac{Z}{\bar{L}^{\frac13}} \Xi_{\mathrm{Left}}(\frac{Z}{\bar{L}^{\frac13}}) \|_{L^2_Z} \lesssim \bar{L}^{\frac16} \| \rho \Xi_{\mathrm{Left}} \|_{L^2_\rho}, \\ \label{spurt:3}
D_2 \lesssim & L^{-\frac23} \bar{L}^{\frac16} \| \chi(\frac{s-1}{\bar{L}})  \Xi''_{\mathrm{Left}}(\frac{Z}{\bar{L}^{\frac13}}) \|_{L^2_{sZ}} \lesssim \bar{L}^{\frac16} \| \Xi''_{\mathrm{Left}} \|_{L^2_\rho}, \\ \label{spurt:4}
D_3 \lesssim & \bar{L}^{\frac16} \| \Xi_{\mathrm{Left}} \|_{L^\infty}.
\end{align}
Finally, to conclude the proof, we simply use that $\| \Xi_{\mathrm{Left}} \|_{L^\infty} \lesssim \|  \Xi_{\mathrm{Left}} \|_{H^1_\rho}$. The proposition is proven. 
\end{proof}

\section{Analysis of $\mathcal{L}_{\text{Crocco}}[S_I, \Xi_{\text{Left}}, \Xi_{\text{Right}}]$} \label{Section:Interior}

In this section, we are interested in the abstract problem, \eqref{abs:3}. Again, as we are performing \textit{a-priori} estimates, we take as input the trace $\gamma^{(1)}_{\Omega}$ (which of course cannot be prescribed, they are determined through the Airy equation). However the point-of-view in this section is to have estimates on the bulk behavior of $\Omega$ \textit{in terms of the} trace. Moreover, since our eventual application of these bounds will be to $\Omega_I$ defined in \eqref{defn:Chi:I}, we may assume $\Omega_I$ has compact support in $Z$ (hence, we will freely use bounds such as $|Z| \lesssim |Z|^{1/2}$). 

Define the following norm: 
\begin{align} \n
\| \Omega_I \|_{\text{Crocco}}^2 := &\sup_{1 \le s \le \bar{L}} \Big( \int_{\mathbb{R}} \Omega_I^2 \ud Z + \int_{\mathbb{R}} Z^2 |\p_Z \Omega_I|^2 \ud Z + \int_{\mathbb{R}} Z^4 \p_Z^2 \Omega_{I}^2 \ud Z  \Big) \\ \label{def:Crocco:norm}
& + \int_1^{1 + \bar{L}} \int_{\mathbb{R}} |Z| |\p_Z^2 \Omega_{I}|^2 \ud Z \ud s + \int_1^{1 + \bar{L}} \int_{\mathbb{R}} |Z|^3 |\p_Z^3 \Omega_{I}|^2 \ud Z \ud s.
\end{align}
The main proposition of this section concerns an \textit{a-priori} bound on the $\| \cdot \|_{\text{Crocco}}$ norm introduced above, and is formulated as follows: 
\begin{proposition} \label{prop:Crocco}Assume $\Omega_I$ satisfies the equation \eqref{abs:3}, and suppose  
\begin{align} \label{fte:sup}
\text{supp}(\Omega_I) \subset \{ (s, Z) \in (1, 1+ \bar{L}) \times (-10,10) \}. 
\end{align}
Assume the coefficients $\tau_0, \tau_1, \tau_2$ satisfy the following estimates: 
\begin{align} \label{tau:b:1}
\sum_{i =0}^2 \sum_{j = 0}^2 \| \p_Z^j \tau_i \|_{L^\infty(-2\delta < Z < 2\delta)} \lesssim & 1, \\ \label{tau:b:2}
|\tau_2| \ge & \frac{1}{4}, \qquad -2\delta < Z < 2\delta,
\end{align}
where $\delta$ is defined above \eqref{defn:Chi:I}. Assume the source term satisfies $\sum_{i = 0}^1 \| Z^i \p_Z^i S_I \|_{L^2_{sZ}}^2 < \infty$. Then the following estimates are valid for $\bar{L} << 1$: 
\begin{align} \label{pro:cro:1}
\| \Omega_I \|_{\mathrm{Crocco}}^2 \lesssim &  \| \gamma_{\Omega}^{(1)} \|_{L^2_s}^2 + \sum_{i = 0}^1 \| Z^i \p_Z^i S_I \|_{L^2_{sZ}}^2  +  \bar{L}^{\frac13} \sum_{i = 1}^2 \sum_{\iota \in \{\mathrm{Left, Right}\}} \|\rho^i \p_\rho^i \Xi_{\iota}\|_{L^2_\rho}^2.
\end{align}
\end{proposition}

\begin{remark} The coefficients $\tau_0, \tau_1, \tau_2$ always accompany a factor of $\Omega_I$ itself, and therefore only need to be defined (and controlled) on the support of $\Omega_I$. Hence, we state \eqref{tau:b:1} -- \eqref{tau:b:2} on the interval $-2\delta < Z < 2\delta$. 
\end{remark}


\subsection{Preliminary Interpolation Bounds}

We give here some interpolation bounds which are used above to close the estimate on $\Omega_I$. 
\begin{lemma} The following interpolation estimates are valid: 
\begin{align}  \label{ode:2}
\| \p_Z \Omega_I \|_{L^2_{sZ}}^2 \lesssim & \bar{L}^{\frac13} \Big( \sup_{1 \le s \le 1 + \bar{L}} \int_{\mathbb{R}_+} Z^2 |\p_Z \Omega_I|^2 \ud Z + \int_{1}^{1 + \bar{L}}\int_{\mathbb{R}_+} Z |\p_Z^2 \Omega_{I}|^2 \ud Z \ud s \Big), \\ \label{ode:3}
\|Z\p_Z^2\Omega_{I} \|_{L^2_{sZ}}^2 \lesssim & \bar{L}^{\frac13} \Big( \sup_{1 \le s \le 1 + \bar{L}} \int_{\mathbb{R}_+} Z^4 |\p_Z^2 \Omega_{I}|^2 \ud Z + \int_{1}^{1 + \bar{L}}\int_{\mathbb{R}_+} Z^3 |\p_Z^3 \Omega_{I}|^2 \ud Z \ud s \Big).
\end{align}
\end{lemma}
\begin{proof} We fix a function $f(s, Z)$ that decays as $Z \rightarrow \infty$. We also fix an integer $n$ (to be taken either $n = 0$ or $n = 1$ for our applications). To simplify calculations involving cutoff functions, we take the convention for this proof that, given a small parameter $\delta_L > 0$, the function $\chi(Z \le \delta_{L}) := \chi_{\delta_{L}}(Z)$ and $\chi(Z \ge \delta_{L}) := 1 - \chi_{\delta_{L}}(Z)$, where $\chi_{\delta_{L}}$ is defined beneath \eqref{chidef}.

Fix a number $\delta_L > 0$ to be chosen later. We have 
\begin{align*}
\| Z^{n} f \|_{L^2_{sZ}}^2 \le & \| Z^{n} f \chi(Z \le \delta_L)\|_{L^2_{sZ}}^2 + \| Z^{n} f \chi(Z \ge \delta_L)\|_{L^2_{sZ}}^2 \\
\le &  \| Z^{n} f \chi(Z \le \delta_L)\|_{L^2_{sZ}}^2 + \frac{C}{\delta^2} \| Z^{n+1} f \chi(Z \ge \delta_L)\|_{L^2_{sZ}}^2 \\
\le &  \| Z^{n} f \chi(Z \le \delta_L)\|_{L^2_{sZ}}^2 + \frac{C \bar{L}}{\delta_L^2} \| Z^{n+1} f\|_{L^\infty_{s}L^2_Z}^2.
\end{align*} 
For the near-field term, we have 
\begin{align} \n
\| Z^{n} f \chi(Z \le \delta_L)\|_{L^2_{sZ}}^2 = & \int_1^{1 + \bar{L}} \int_{\mathbb{R}_+} \frac{\p_Z \{ Z^{2n+1}\}}{2n+1} f^2 \chi(Z \le \delta_L)^2 \ud Z \ud s \\ \n
= & - \int_1^{1 + \bar{L}} \int_{\mathbb{R}_+} \frac{  Z^{2n+1}}{2n+1} 2 f f_Z \chi(Z \le \delta_L)^2 \ud Z \ud s \\ \label{sam:3}
& - \frac{1}{\delta_L} \int_1^{1 + \bar{L}} \int_{\mathbb{R}_+} \frac{  Z^{2n+1}}{2n+1}2 f^2 \chi(Z \le \delta_L) \chi'(Z \le \delta_L) \ud Z \ud s = A_{1} + A_2.
\end{align}
The term $A_2$ is analogous to the far-field term, and is estimated above by $A_2 \le \frac{C \bar{L}}{\delta_L^2} \| Z^{n+1} f \|_{L^\infty_{s}L^2_Z}^2$. For the term $A_1$, we use Cauchy-Schwartz and then Young's inequality for products to estimate via 
\begin{align*}
|A_1| \le & C \| Z^n f \chi(Z \le \delta_L) \|_{L^2_{sZ}} \| Z^{n+1} f_Z \chi(Z \le \delta_L) \|_{L^2_{sZ}}\\
 \le& \frac12 \| Z^n f \chi(Z \le \delta_L) \|_{L^2_{sZ}}^2 + C \| Z^{n+1} f_Z \chi(Z \le \delta_L) \|_{L^2_{sZ}}^2. 
\end{align*}
We again use the $1/2$ factor to absorb into the left-hand side of \eqref{sam:3}. This produces 
\begin{align*}
\| Z^{n} f \|_{L^2_{sZ}}^2 \le &\frac{C \bar{L}}{\delta_L^2} \| Z^{n+1} f\|_{L^\infty_{s}L^2_Z}^2 + C \| Z^{n+1} f_Z \chi(Z \le \delta_L) \|_{L^2_{sZ}}^2 \\
\lesssim & \frac{\bar{L}}{\delta_L^2} \| Z^{n+1} f \|_{L^\infty_{s}L^2_Z}^2 +  \delta_L \| Z^{n+\frac12} f_Z  \|_{L^2_{sZ}}^2.
\end{align*}
We now choose $\delta_L = \bar{L}^{\frac13}$ and either $f = \p_Z \Omega_I, n = 0$ for \eqref{ode:2} or $f = \p_Z^2 \Omega_{I}, n = 1$ for \eqref{ode:3}. The claim is proven. 
\end{proof}

We now consolidate the interpolation bounds above in a concise manner to be used repeatedly to bound linear contributions appearing on the right-hand sides of the forthcoming energy estimates.
\begin{corollary} The following estimates are valid: 
\begin{align} \label{west:1}
\| \Omega_I \|_{L^2_{sZ}}^2 \lesssim & \bar{L} \| \Omega_I \|_{\mathrm{Crocco}}^2, \\ \label{west:2}
\|\p_Z \Omega_I \|_{L^2_{sZ}}^2 + \| Z\p_Z^2 \Omega_{I} \|_{L^2_{sZ}}^2 \lesssim & \bar{L}^{\frac13} \| \Omega_I \|_{\mathrm{Crocco}}^2.
\end{align}
\end{corollary}

\subsection{Crocco Energy Estimates}

The forthcoming lemmas will imply the main proposition, Proposition \ref{prop:Crocco}.

\begin{lemma} \label{klem1}Under the assumptions of Proposition \ref{prop:Crocco}, the following estimates are valid on $Z > 0$:
\begin{align} \n
&\sup_{1 \le s \le 1 + \bar{L}} \int_{\mathbb{R}_+} (|\Omega_I|^2 +Z^2 |\p_Z \Omega_I|^2) \ud Z + \int_{1}^{1 + \bar{L}}\int_{\mathbb{R}_+} Z |\p_Z^2 \Omega_{I}|^2 \ud Z \ud s \\ \label{croc:en:2}
& \qquad \qquad \lesssim \bar{L}^{\frac13} \| \Omega_I \|_{\mathrm{Crocco}}^2 + \| \gamma_{\Omega}^{(1)} \|_{L^2_s}^2 + \| S_I \|_{L^2_{sZ}}^2 +  \bar{L}^{\frac13} \sum_{i = 0}^1 \|\rho^i \p_\rho^i \Xi_{\mathrm{Left}} \|_{L^2_\rho}^2,
\end{align}
and a completely analogous set of bounds are valid on $Z < 0$ (with $Z$ weights above being replaced by $-Z$).
\end{lemma}
\begin{proof} We take the inner-product of \eqref{abs:3} with the quantity $-\p_Z^2 \{Z \Omega_I\}$. This produces the identity 
\begin{align*}
 \int_{\mathbb{R}_+} S \p_Z^2 \{ Z \Omega_I \} \ud Z = &- \int_{\mathbb{R}_+} Z \p_s \Omega_I \p_Z^2 \{Z \Omega_I\} \ud Z + \int_{\mathbb{R}_+} (1 + \tau_2) Z |\p_Z^2 \Omega_{I}|^2 \ud Z \\
 &+  \int_{\mathbb{R}_+} (1 + \tau_2)  \p_Z^2 \Omega_{I} \p_Z \Omega_{I} \ud Z  + \int_{\mathbb{R}_+} \tau_0 \Omega \p_Z^2 \{ Z \Omega_I \} \ud Z \\
 &+ \int_{\mathbb{R}_+} \tau_1 \p_Z \Omega_I \p_Z^2 \{ Z \Omega_I \} \ud Z \\
 = &\frac{\p_s}{2}\int_{\mathbb{R}_+} |\p_Z \{Z \Omega_I\}|^2 \ud Z + \int_{\mathbb{R}_+} (1 + \tau_2) Z |\p_Z^2 \Omega_{I}|^2 \ud Z \\
 &- \frac12  \int_{\mathbb{R}_+} \p_Z \tau_2  |\p_Z \Omega_{I}|^2 \ud Z  + \int_{\mathbb{R}_+} \tau_0 \Omega \p_Z^2 \{ Z \Omega_I \} \ud Z \\
 &+ \int_{\mathbb{R}_+} \tau_1 \p_Z \Omega_I \p_Z^2 \{ Z \Omega_I \} \ud Z - \frac12 (1 + \tau_2) |\gamma_{\Omega}^{(1)}(s)|^2,
\end{align*}
where we recall $\gamma_{\Omega}^{(1)}(s) = \p_Z \Omega_I(s, 0)$ by \eqref{defngammaomegaj}. Integration over $s \in (1, s_\ast)$, taking the supremum over $s_\ast$, and using that $\tau_2 \ge -\frac14$ yields the inequality 
\begin{align} \n
&\sup_{1 \le s \le 1 + \bar{L}} \int_{\mathbb{R}_+} |\p_Z \{Z \Omega_I\}|^2 \ud Z + \int_{1}^{1 + \bar{L}} \int_{\mathbb{R}_+}  Z |\p_Z^2 \Omega_{I}|^2 \ud Z \ud s \\ \n
\lesssim & \int_{\mathbb{R}_+} |\p_Z \{Z \Xi_{\text{Left}}(\frac{Z}{\bar{L}^{\frac13}})\}|^2 \ud Z + \frac12  \int_{1}^{1 + \bar{L}} \int_{\mathbb{R}_+} | \p_Z \tau_2 | |\p_Z \Omega_{I}|^2 \ud Z \ud s \\ \n
&+| \int_{1}^{1 + \bar{L}} \int_{\mathbb{R}_+} \tau_0 \Omega_I \p_Z^2 \{ Z \Omega_I\} \ud Z \ud s| +| \int_{1}^{1 + \bar{L}} \int_{\mathbb{R}_+} \tau_1 \p_Z \Omega_I \p_Z^2 \{ Z \Omega_I \} \ud Z \ud s| \\ \n
&+ \frac12 \int_{1}^{1 + \bar{L}} |1 + \tau_2| |\gamma_{\Omega}^{(1)}(s)|^2 \ud s + | \int_1^{1 + \bar{L}} \int_{\mathbb{R}_+} S_I \p_Z^2 \{ Z \Omega_I \} \ud Z \ud s|\\ \label{bo:1}
= & \sum_{i = 1}^{6} \text{Err}_i.
\end{align}
As a matter of convention, we include the universal constant appearing on the right-hand side of the above inequality as part of the definitions of $\text{Err}_i$. 

First of all, we note that the quantity appearing on the left-hand side can be bounded below using the classical Hardy inequality, namely $\| \frac{f}{Z} \|_{L^2_Z} \lesssim \| \p_Z f \|_{L^2_Z}$ for functions $f|_{Z = 0}= 0$. In this case, $f = Z \Omega_I$ clearly vanishes at $Z = 0$, and therefore, 
\begin{align}
&\| \Omega_I \|_{L^2_Z} = \| \frac{Z\Omega_I}{Z} \|_{L^2_Z} \lesssim \| \p_Z \{ Z \Omega_I \} \|_{L^2_Z}, \\
&\| Z \p_Z \Omega_I \|_{L^2_Z} =  \| \p_Z \{Z \Omega_I\} - \Omega_I \|_{L^2_Z} \le \| \p_Z \{Z \Omega_I\} \|_{L^2_Z} + \| \Omega_I \|_{L^2_Z} \lesssim \| \p_Z \{Z \Omega_I\} \|_{L^2_Z}.
\end{align}
 Next, it is clear that $|\text{Err}_5| \lesssim \| 1 + \tau_2 \|_{L^\infty} \| \gamma_{\Omega}^{(1)} \|_{L^2_s}^2$. The bulk terms can be estimated as follows:
\begin{align*}
\sum_{i = 2}^4 |\text{Err}_i| \le & C_\lambda (1 + \| \p_Z \tau_2 \|_{L^\infty}^2 + \| \tau_0 \|_{L^\infty}^2 + \| \tau_1 \|_{L^\infty}^2) (\| \Omega_I \|_{L^2_{sZ}}^2 + \| \p_Z \Omega_I \|_{L^2_{sZ}}^2 ) + \lambda \| Z^{\frac12} \p_Z^2 \Omega_I \|_{L^2_{sZ}}^2 \\
\le & C_\lambda \bar{L}^{\frac13} \| \Omega_I \|_{\text{Crocco}}^2 + \lambda \| Z^{\frac12} \p_Z^2 \Omega_I \|_{L^2_{sZ}}^2,
\end{align*}
for any $0 < \lambda << 1$, where we have used the Young's inequality for products as well as \eqref{west:2}.  For the error term $\text{Err}_6$, we estimate as follows:
\begin{align*}
|\text{Err}_6| \lesssim \| S_I \|_{L^2_{sZ}} ( \| Z \p_Z^2 \Omega_{I} \|_{L^2_{sZ}}  + \|\p_Z \Omega_I \|_{L^2_{sZ}}) \le  C_\lambda \| S_I \|_{L^2_{sZ}}^2 + \lambda \| Z^{\frac12} \p_Z^2 \Omega_{I} \|_{L^2_{sZ}}^2 + C \bar{L}^{\frac13} \| \Omega_I \|_{\text{Crocco}}^2. 
\end{align*}
We therefore choose $\lambda = \frac{1}{200}$ so that $2\lambda \| Z^{\frac12} \p_Z^2 \Omega_I \|_{L^2_{sZ}}^2 \le \frac{1}{100} \| Z^{\frac12} \p_Z^2 \Omega_I \|_{L^2_{sZ}}$ and these contributions can therefore be absorbed to the left-hand side of \eqref{bo:1}. 

Finally, we estimate the initial data, 
\begin{align} \n
|\text{Err}_1| \lesssim & \bar{L}^{\frac13} (\| \Xi_{\text{Left}} \|_{L^2_\rho}^2 + \| \rho \p_{\rho} \Xi_{\text{Left}} \|_{L^2_\rho}^2).
\end{align}
The lemma is proven. 
\end{proof}

\begin{lemma} \label{klem2}Under the assumptions of Proposition \ref{prop:Crocco}, the following estimates are valid on $Z > 0$:
\begin{align} \n
&\sup_{1 \le s \le 1 + \bar{L}} \int_{\mathbb{R}_+} Z^4 |\p_Z^2 \Omega_{I}|^2 \ud Z + \int_{1}^{1 + \bar{L}}\int_{\mathbb{R}_+} Z^3 |\p_Z^3 \Omega_{I}|^2 \ud Z \ud s \\ \n
  \lesssim & \bar{L}^{\frac13} \| \Omega_I \|_{\mathrm{Crocco}}^2 +(\| Z \p_Z \Omega_I \|_{L^\infty_s L^2_Z}^2 + \| Z^{\frac12} \p_Z^2 \Omega_{I} \|_{L^2_{sZ}}^2) +  \| \gamma_{\Omega}^{(1)} \|_{L^2_s}^2 \\ \label{croc:en:3}
&  + \sum_{i = 0}^1 \| Z^i \p_Z^i S_I \|_{L^2_{sZ}}^2  +  \bar{L}^{\frac13} \sum_{i = 1}^2 \|\rho^i \p_\rho^i \Xi_{\mathrm{Left}} \|_{L^2_\rho}^2,
\end{align}
and a completely analogous set of bounds are valid on $Z < 0$ (with $Z$ weights above being replaced by $-Z$).
\end{lemma}
\begin{proof}We first of all differentiate twice \eqref{abs:3} in $Z$ to obtain 
\begin{align} \n
\p_s (Z \Omega_I)_{ZZ} - (1 + \tau_2) \p_Z^4 \Omega_I = & \p_Z^2 S_I + (\tau_1 + 2 \p_Z \tau_2) \p_Z^3 \Omega_I + (\tau_0 + 2 \p_Z \tau_1 + \p_Z^2 \tau_2) \p_Z^2 \Omega_I \\
& + (2 \p_Z \tau_0 + \p_Z^2 \tau_1) \p_Z \Omega_I + \p_Z^2 \tau_0 \Omega_I. 
\end{align}
We subsequently multiply by $Z^2 \p_Z^2\{Z \Omega_I\}$ to both sides of the above equation and integrate over $Z \in \mathbb{R}_+$. This produces the identity 
\begin{align} \n
&\frac{\p_s}{2} \int_{\mathbb{R}_+} Z^2 |\p_Z^2 \{ Z \Omega_I\}|^2 \ud Z - \underbrace{ \int_{\mathbb{R}_+} (1 + \tau_2) \p_Z^4 \Omega_I Z^2 \p_Z^2\{Z \Omega_I\} \ud Z}_{:= -D_1} \\ \n
= & \int_{\mathbb{R}_+} \p_Z^2 S Z^2 \p_Z^2\{Z \Omega_I\} \ud Z + \int_{\mathbb{R}_+}  (\tau_1 + 2 \p_Z \tau_2) \p_Z^3 \Omega_I Z^2 \p_Z^2\{Z \Omega_I\} \ud Z \\ \n
& + \int_{\mathbb{R}_+} (\tau_0 + 2 \p_Z \tau_1 + \p_Z^2 \tau_2) \p_Z^2 \Omega_I Z^2 \p_Z^2\{Z \Omega\} \ud Z + \int_{\mathbb{R}_+}(2 \p_Z \tau_0 + \p_Z^2 \tau_1) \p_Z \Omega_I Z^2 \p_Z^2\{Z \Omega_I\} \ud Z \\ \label{yioo:0}
& + \int_{\mathbb{R}_+}  \p_Z^2 \tau_0 \Omega Z^2 \p_Z^2\{Z \Omega_I\} \ud Z.
\end{align}
To extract the positive diffusive term from $D_1$, we need to perform a long integration by parts which we isolate here: 
\begin{align} \n
D_1 = & - \int_{\mathbb{R}_+} (1 + \tau_2) Z^3 \p_Z^4 \Omega_I \p_Z^2 \Omega_I \ud Z - \int_{\mathbb{R}_+} 2 (1 + \tau_2) Z^2 \p_Z^4 \Omega_I \p_Z \Omega_I \ud Z \\ \n
= & \int_{\mathbb{R}_+} (1 + \tau_2) Z^3 |\p_Z^3 \Omega_I|^2 \ud Z + \int_{\mathbb{R}_+} \p_Z \tau_2 Z^3 \p_Z^3 \Omega_I \p_Z^2 \Omega_I \ud Z + \int_{\mathbb{R}_+}3(1+ \tau_2) Z^2 \p_Z^3 \Omega_I \p_Z^2 \Omega_I \ud Z \\ \n
& + \int_{\mathbb{R}_+} 2\p_Z \tau_2 Z^2 \p_Z^3 \Omega_I \p_Z \Omega_I \ud Z  + \int_{\mathbb{R}_+}2(1+\tau_2) Z^2 \p_Z^3 \Omega_I \p_Z^2 \Omega_I \ud Z+ \int_{\mathbb{R}_+}4(1+\tau_2) Z \p_Z^3 \Omega_I \p_Z \Omega_I \ud Z \\ \n
= & \int_{\mathbb{R}_+} (1 + \tau_2) Z^3 |\p_Z^3 \Omega_I|^2 \ud Z + \int_{\mathbb{R}_+} \p_Z \tau_2 Z^3 \p_Z^3 \Omega_I \p_Z^2 \Omega_I \ud Z + \int_{\mathbb{R}_+}3(1+ \tau_2) Z^2 \p_Z^3 \Omega_I \p_Z^2 \Omega_I \ud Z \\ \n
& + \int_{\mathbb{R}_+} 2\p_Z \tau_2 Z^2 \p_Z^3 \Omega_I \p_Z \Omega_I \ud Z + \int_{\mathbb{R}_+}2(1+\tau_2) Z^2 \p_Z^3 \Omega_I \p_Z^2 \Omega_I \ud Z - \int_{\mathbb{R}_+}4(1+\tau_2) Z |\p_Z^2 \Omega_I|^2 \ud Z \\ \n
& - \int_{\mathbb{R}_+}4\p_Z \tau_2 Z \p_Z^2 \Omega_I \p_Z \Omega_I \ud Z- \int_{\mathbb{R}_+}4(1+ \tau_2)  \p_Z^2 \Omega_I \p_Z \Omega_I \ud Z \\ \n
= & \int_{\mathbb{R}_+} (1 + \tau_2) Z^3 |\p_Z^3 \Omega_I|^2 \ud Z + \int_{\mathbb{R}_+} \p_Z \tau_2 Z^3 \p_Z^3 \Omega_I \p_Z^2 \Omega_I \ud Z + \int_{\mathbb{R}_+}3(1+ \tau_2) Z^2 \p_Z^3 \Omega_I \p_Z^2 \Omega_I \ud Z \\ \n
& + \int_{\mathbb{R}_+} 2\p_Z \tau_2 Z^2 \p_Z^3 \Omega_I \p_Z \Omega \ud Z + \int_{\mathbb{R}_+}2(1+\tau_2) Z^2 \p_Z^3 \Omega_I \p_Z^2 \Omega_I \ud Z - \int_{\mathbb{R}_+}4(1+\tau_2) Z |\p_Z^2 \Omega_I|^2 \ud Z \\ \label{yioo:1}
& - \int_{\mathbb{R}_+}4\p_Z \tau_2 Z \p_Z^2 \Omega_I \p_Z \Omega_I \ud Z+  \int_{\mathbb{R}_+}2 \p_Z \tau_2 | \p_Z \Omega_I|^2 \ud Z + 2(1 + \tau_2) |\gamma_\Omega^{(1)}(s)|^2.
\end{align}
We next integrate by parts the source term from \eqref{yioo:0} to obtain 
\begin{align} \label{yioo:3}
 \int_{\mathbb{R}_+} \p_Z^2 S_I Z^2 \p_Z^2\{Z \Omega_I\} \ud Z = -  \int_{\mathbb{R}_+} \p_Z S_I Z^2 \p_Z^3\{Z \Omega_I\} \ud Z - 2 \int_{\mathbb{R}_+} \p_Z S_I Z \p_Z^2\{Z \Omega_I\} \ud Z.
\end{align}
We now pair \eqref{yioo:1} and \eqref{yioo:3} with \eqref{yioo:0}, integrate over $s \in (1, s_\ast)$ and take supremum over $s_\ast$ to obtain 
\begin{align} \n
&\sup_{1 \le s \le 1 + \bar{L}} \int_{\mathbb{R}_+} Z^2 |\p_Z^2 \{ Z \Omega_I\}|^2 \ud Z + \int_{1}^{1 + \bar{L}} \int_{\mathbb{R}_+}  Z^3 |\p_Z^3 \Omega_I|^2 \ud Z \\ \n
\lesssim & \int_{\mathbb{R}_+} Z^2 |\p_Z^2 \{ Z \Xi_{\text{Left}}(\frac{Z}{\bar{L}^{\frac13}})\}|^2 \ud Z + | \int_1^{1 + \bar{L}} \int_{\mathbb{R}_+} \p_Z \tau_2 Z^3 \p_Z^3 \Omega_I \p_Z^2 \Omega_I \ud Z \ud s| \\ \n
&+ |\int_1^{1 + \bar{L}} \int_{\mathbb{R}_+}(1+ \tau_2) Z^2 \p_Z^3 \Omega_I \p_Z^2 \Omega_I \ud Z \ud s|  + |\int_1^{1 + \bar{L}} \int_{\mathbb{R}_+} 2\p_Z \tau_2 Z^2 \p_Z^3 \Omega_I \p_Z \Omega_I \ud Z \ud s| \\ \n
&+ | \int_1^{1 + \bar{L}} \int_{\mathbb{R}_+}4(1+\tau_2) Z |\p_Z^2 \Omega_I|^2 \ud Z \ud s|  +| \int_1^{1 + \bar{L}}  \int_{\mathbb{R}_+}4\p_Z \tau_2 Z \p_Z^2 \Omega_I \p_Z \Omega_I \ud Z \ud s| \\ \n
&+ |\int_1^{1 + \bar{L}} \int_{\mathbb{R}_+}2 \p_Z \tau_2 | \p_Z \Omega_I|^2 \ud Z \ud s| + \| 1 + \tau_2\|_{L^\infty} \int_1^{1 + \bar{L}} |\gamma_\Omega^{(1)}(s)|^2 \ud s| \\ \n
& + |\int_1^{1 + \bar{L}} \int_{\mathbb{R}_+} \p_Z S_I Z^2 \p_Z^3\{Z \Omega_I\} \ud Z \ud s| + |\int_1^{1 + \bar{L}} \int_{\mathbb{R}_+} \p_Z S_I Z \p_Z^2\{Z \Omega_I\} \ud Z \ud s| \\ \n
&+  | \int_1^{1 + \bar{L}} \int_{\mathbb{R}_+}  (\tau_1 + 2 \p_Z \tau_2) \p_Z^3 \Omega_I Z^2 \p_Z^2\{Z \Omega_I\} \ud Z \ud s|  \\ \n
&+| \int_1^{1 + \bar{L}} \int_{\mathbb{R}_+} (\tau_0 + 2 \p_Z \tau_1 + \p_Z^2 \tau_2) \p_Z^2 \Omega_I Z^2 \p_Z^2\{Z \Omega_I\} \ud Z \ud s| \\ \n
& +|\int_1^{1 + \bar{L}} \int_{\mathbb{R}_+}(2 \p_Z \tau_0 + \p_Z^2 \tau_1) \p_Z \Omega_I Z^2 \p_Z^2\{Z \Omega_I\} \ud Z \ud s| +|\int_1^{1 + \bar{L}} \int_{\mathbb{R}_+}  \p_Z^2 \tau_0 \Omega_I Z^2 \p_Z^2\{Z \Omega_I\} \ud Z \ud s| \\ \label{yioo:4}
=: & \sum_{i = 1}^{14} \text{Err}_i.
\end{align}
Again, as a matter of convention the $\text{Err}_i$ are defined to include the implicit constant appearing above. We now estimate each of the terms appearing on the right-hand side above. First, we have the contribution of the data: 
\begin{align*}
|\text{Err}_1| \lesssim \bar{L}^{\frac13} (\| \rho \p_\rho \Xi_{\text{Left}}\|_{L^2_\rho}^2 +\| \rho^2 \p_\rho^2 \Xi_{\text{Left}}\|_{L^2_\rho}^2 ).
\end{align*}
The next three terms contain factors of $\p_Z^3 \Omega_I$, so we estimate them as follows: 
\begin{align} \n
\sum_{i = 2}^4 |\text{Err}_i| \lesssim& (\| \tau_2 \|_{L^\infty} + \| \p_Z \tau_2 \|_{L^\infty}) \| Z^{\frac32}\p_Z^3 \Omega_{I} \|_{L^2_{sZ}} (\| Z^{\frac12} \p_Z^2 \Omega_{I} \|_{L^2_{sZ}} + \|\p_Z \Omega_I \|_{L^2_{sZ}}) \\ \n
\le & \lambda \| Z^{\frac32} \p_Z^3 \Omega_{I} \|_{L^2_{sZ}}^2 + C_\lambda \| Z^{\frac12}\p_Z^2 \Omega_{I} \|_{L^2_{sZ}}^2 + C_\lambda \|\p_Z \Omega_I \|_{L^2_{sZ}}^2 \\  \label{nugget:1}
\le & \lambda \| Z^{\frac32} \p_Z^3 \Omega_{I} \|_{L^2_{sZ}}^2 + C_\lambda \| Z^{\frac12} \p_Z^2 \Omega_{I} \|_{L^2_{sZ}}^2 + C_\lambda L^{\frac13} \| \Omega_I \|_{\text{Crocco}}^2,
\end{align}
where we have used Young's inequality for products ($0 < \lambda << 1$ will be chosen below relative to only universal constants) as well as the estimate \eqref{west:2}. 

The next three terms are lower order and can be estimated as follows: 
\begin{align*}
\sum_{i = 5}^7 |\text{Err}_i| \lesssim & (\| \tau_2 \|_{L^\infty} + \| \p_Z \tau_2 \|_{L^\infty})(\| Z^{\frac12} \p_Z^2 \Omega_{I} \|_{L^2_{sZ}}^2 + \|\p_Z \Omega_I \|_{L^2_{sZ}}^2 ) \lesssim  \| Z^{\frac12} \p_Z^2 \Omega_{I} \|_{L^2_{sZ}}^2 + \bar{L}^{\frac13} \| \Omega_I \|_{\text{Crocco}}^2,
\end{align*}
again by invoking \eqref{west:2}. 

The next term contains the crucial boundary contribution: 
\begin{align*}
|\text{Err}_8| \lesssim \| 1 + \tau_2 \|_{L^\infty} \| \gamma_\Omega^{(1)} \|_{L^2_s}^2. 
\end{align*}
The next two terms contain the contribution from the source term, $S$, and are estimated as follows: 
\begin{align} \n
\sum_{i = 9}^{10}|\text{Err}_i| \lesssim & \| Z \p_Z S_I \|_{L^2_{sZ}} ( \| Z^2 \p_Z^3 \Omega_{I} \|_{L^2_{sZ}} + \| Z \p_Z^2 \Omega_{I} \|_{L^2_{sZ}} + \|\p_Z \Omega_I \|_{L^2_{sZ}} ) \\ \n
\le & \lambda  \| Z^{\frac32} \p_Z^3 \Omega_{I} \|_{L^2_{sZ}}^2 + C_\lambda \| Z \p_Z S_I \|_{L^2_{sZ}}^2 + C\| Z \p_Z^2 \Omega_{I} \|_{L^2_{sZ}}^2 +C \| \p_Z \Omega_{I} \|_{L^2_{sZ}}^2 \\ \label{nugget:2}
\le & \lambda  \| Z^{\frac32} \p_Z^3 \Omega_{I} \|_{L^2_{sZ}}^2 + C_\lambda \| Z \p_Z S_I \|_{L^2_{sZ}}^2 + C\| Z \p_Z^2 \Omega_{I} \|_{L^2_{sZ}}^2 +C \bar{L}^{\frac13} \| \Omega_I \|_{\text{Crocco}}^2, 
\end{align}
where we have applied the Young's inequality for products as well as \eqref{west:2}.

The next term has a factor of $\p_Z^3 \Omega_I$, and therefore we estimate it as follows: 
\begin{align} \n
|\text{Err}_{11}| \lesssim & (\| \tau_1 \|_{L^\infty} + \| \p_Z \tau_2 \|_{L^\infty}) \| Z^{\frac32} \p_Z^3 \Omega_I \|_{L^2_{sZ}}( \| Z\p_Z^2 \Omega_{I} \|_{L^2_{sZ}} + \| \p_Z \Omega_I \|_{L^2_{sZ}} ) \\ \label{nugget:3}
\le & \lambda  \| Z^{\frac32} \p_Z^3 \Omega_{I} \|_{L^2_{sZ}}^2 +  C_\lambda \| Z^{\frac12} \p_Z^2 \Omega_{I} \|_{L^2_{sZ}}^2 +C_\lambda \bar{L}^{\frac13} \| \Omega_I \|_{\text{Crocco}}^2, 
\end{align}
The remaining terms are all lower order, and we estimate them as follows: 
\begin{align*}
\sum_{i = 12}^{14} |\text{Err}_{i}| \lesssim & (\sum_{i = 0}^2 (\| \p_Z^j \tau_0 \|_{L^\infty} + \| \p_Z^j \tau_1 \|_{L^\infty} + \| \p_Z^j \tau_2 \|_{L^\infty})  ) (\| \Omega_I \|_{L^2_{sZ}}^2 + \|\p_Z \Omega_I \|_{L^2_{sZ}}^2 + \| Z \p_Z^2 \Omega_{I} \|_{L^2_{sZ}}^2) \\
\lesssim & \bar{L}^{\frac13} \| \Omega_I \|_{\text{Crocco}}^2,
\end{align*}
upon invoking \eqref{west:2}. 

We now choose $\lambda = \frac{1}{30}$ so that the $\lambda \| Z^{\frac32} \p_Z^3 \Omega_{I} \|_{L^2_{sZ}}^2$ contributions from \eqref{nugget:1}, \eqref{nugget:2}, \eqref{nugget:3} can all be absorbed to the left-hand side of \eqref{yioo:4} (any concrete choice $3\lambda < 1$ would suffice).

To conclude the proof, we consider the first term on the left-hand side of \eqref{yioo:4}, and invoke the reverse triangle inequality as follows: 
\begin{align*}
\| Z \p_Z^2 \{ Z \Omega_I \} \|_{L^2_{Z}}^2 \gtrsim \| Z^2\p_Z^2  \Omega_I \|_{L^2_Z}^2 - \| Z\p_Z \Omega_I \|_{L^2_Z}^2, 
\end{align*}
after which we observe the $\| Z\p_Z \Omega_I \|_{L^2_Z}^2$ has already been controlled in \eqref{croc:en:2}: 
\begin{align*}
\sup_{1 \le s \le 1 + \bar{L}} \int_{\mathbb{R}_+} Z^4 |\p_Z^2 \Omega_{I}|^2 \ud Z \lesssim \sup_{1 \le s \le 1 + \bar{L}} \int_{\mathbb{R}_+} Z^2 |\p_Z^2 \{ Z \Omega_I \}|^2 \ud Z + \sup_{1 \le s \le 1 + \bar{L}} \int_{\mathbb{R}_+} Z^2 |\p_Z \Omega_I|^2 \ud Z.
\end{align*}
The proof of the lemma is complete. 
\end{proof}

\subsection{Proof of Proposition \ref{prop:Crocco}}

We now bring together the above lemmas in order to prove the main proposition of this section. 
\begin{proof}[Proof of Proposition \ref{prop:Crocco}] We collect the estimates \eqref{croc:en:2} and \eqref{croc:en:3} (and their counterparts on $\mathbb{R}_-$). For the term $\| Z \p_Z \Omega_I \|_{L^\infty_s L^2_Z}^2 + \| Z^{\frac12} \p_Z^2 \Omega_{I} \|_{L^2_{sZ}}^2$ appearing on the right-hand side of \eqref{croc:en:3}, we notice that these terms appear on the left-hand side of \eqref{croc:en:2}. Therefore, upon substituting \eqref{croc:en:2} for these terms, we obtain the new set of bounds: 
\begin{align} \n
&\sup_{1 \le s \le 1 + \bar{L}} \int_{\mathbb{R}} (|\Omega_I|^2 +Z^2 |\p_Z \Omega_I|^2) \ud Z + \int_{1}^{1 + \bar{L}}\int_{\mathbb{R}} |Z| |\p_Z^2 \Omega_{I}|^2 \ud Z \ud s \\ \label{croc:en:2:again}
& \qquad \qquad \lesssim L^{\frac13} \| \Omega_I \|_{\text{Crocco}}^2 + \| \gamma_{\Omega}^{(1)} \|_{L^2_s}^2 + \| S_I \|_{L^2_{sZ}}^2 +  \bar{L}^{\frac13} \sum_{i = 0}^1 \|\rho^i \p_\rho^i \Xi_{\text{Left}} \|_{L^2_\rho}^2, \\ \n
&\sup_{1 \le s \le 1 + \bar{L}} \int_{\mathbb{R}} Z^4 |\p_Z^2 \Omega_{I}|^2 \ud Z + \int_{1}^{1 + \bar{L}}\int_{\mathbb{R}} |Z|^3 |\p_Z^3 \Omega_{I}|^2 \ud Z \ud s \\ \label{croc:en:1:again}
 & \qquad \qquad \lesssim  \bar{L}^{\frac13} \| \Omega_I \|_{\text{Crocco}}^2 +  \| \gamma_{\Omega}^{(1)} \|_{L^2_s}^2  + \sum_{i = 0}^1 \| Z^i \p_Z^i S_I \|_{L^2_{sZ}}^2  +  \bar{L}^{\frac13} \sum_{i = 0}^2 \sum_{\iota \in \text{Left, Right}} \|\rho^i \p_\rho^i \Xi_{\iota} \|_{L^2_\rho}^2.
\end{align}
We add both inequalities to obtain 
\begin{align} \n
&\| \Omega_I \|_{\text{Crocco}}^2 \lesssim  \bar{L}^{\frac13} \| \Omega_I \|_{\text{Crocco}}^2 +  \| \gamma_{\Omega}^{(1)} \|_{L^2_s}^2  + \sum_{i = 0}^1 \| Z^i \p_Z^i S_I \|_{L^2_{sZ}}^2  +  \bar{L}^{\frac13} \sum_{i = 0}^2 \sum_{\iota \in \text{Left, Right}} \|\rho^i \p_\rho^i \Xi_{\iota} \|_{L^2_\rho}^2.
\end{align}
For $\bar{L} << 1$, the $\bar{L}^{\frac13} \| \Omega_I \|_{\text{Crocco}}^2$ term can be absorbed to the left-hand side to close the estimate \eqref{pro:cro:1}. The proof of the proposition is complete. 
\end{proof}

\section{Analysis of $\mathcal{L}_{\text{von-Mise}}[\bold{R}, \zeta_{\text{Left}}, \zeta_{\text{Right}}]$} \label{Section:Outer}

In this section, we consider the von-Mise formulation of the problem, namely \eqref{abs:4:bljk}. We define a sequence of ``outer" cutoff functions as follows: Define a set of parameters 
\begin{align} \label{para:1}
1 < z_0^+ < \bar{z}_0^+ < z_1^+ < \bar{z}_1^+ < z_2^+ < \bar{z}_2^+ < z_3^+ < \bar{z}_3^+ < z_{\text{Out}}^+ < \bar{z}_{\text{Out}}^+, \\ \label{para:2}
1 >z_0^- > \bar{z}_0^- > z_1^- > \bar{z}_1^- > z_2^- > \bar{z}_2^- > z_3^- > \bar{z}_3^- > z_{\text{Out}}^- > \bar{z}_{\text{Out}}^-.
\end{align}
Furthermore, we choose $\bar{z}_{\text{Out}}^{\pm}$ so that the set $\{\bar{z}_{\text{Out}}^-,\bar{z}_{\text{Out}}^+\} \subset \{\chi_I(Z) = 1\}$. We fix smooth cutoff functions such that 
\begin{align}
\chi_{O,j}(z) = \begin{cases} 0 \qquad z_j^- \le z \le z_j^+ \\ 1 \qquad z < \bar{z}_j^- \\ 1 \qquad z > \bar{z}_j^+ \end{cases}, \qquad 0 \le j \le 3. 
\end{align}
We similarly define 
\begin{align}
\chi_{O}(z) = \begin{cases} 0 \qquad z_{\text{Out}}^- \le z \le z_{\text{Out}}^+ \\ 1 \qquad z < \bar{z}_{\text{Out}}^- \\ 1 \qquad z > \bar{z}_{\text{Out}}^+ \end{cases}. 
\end{align}
Finally, we define 
\begin{align} \label{folgers:1}
&\chi_{O,j}^{+} := \chi_{O,j} \bold{1}_{z \ge 1}, \qquad \chi_{O,j}^{-} := \chi_{O,j} \bold{1}_{z \le 1}, \\ \label{sightglass:1}
&\chi_{O}^{+} := \chi_{O} \bold{1}_{z \ge 1}, \qquad \chi_{O}^{-} := \chi_{O} \bold{1}_{z \le 1},
\end{align}
where $\bold{1}_{z \ge 1}$ is an indicator function, equal to $1$ if $z \ge 1$ and $0$ if $z < 1$, and $\bold{1}_{z \le 1}$ is defined analogously. Our choice of parameters, \eqref{para:1} -- \eqref{para:2}, ensures that the cutoffs are (1) ``nested" from $j$ to $j+1$, and also that (2) they overlap favorably with $\chi_I$. More precisely:  
\begin{align} \label{ygh1}
\text{supp}\{ \chi_{O, j+1} \} \subset & \{ \chi_{O, j} = 1 \}, \qquad 0 \le j \le 2 \\ \label{ygh2}
\text{supp}\{ \chi_O \} \subset & \{ \chi_{O, 3} = 1 \}, \\
\{(s, z) : \chi_I(Z) \neq 1 \} \subset & \{\chi_{O} = 1 \},  \\ \label{parahyun}
\{(s, z) : \chi_{O, j} \neq 1 \}  \subset& \{(s, z) : \chi_{O} \neq 1 \} \subset  \{\chi_I = 1 \}, \qquad 0 \le j \le 3.
\end{align}
We define here a sequence of Energy-Dissipation functionals. We write them separately for the domain $0 < z < 1$, using a minus sign, and the domain $1 < z < \infty$ indexed by a plus sign. For the bottom problem, we have one index, $\p_z$ regularity, whereas for the top problem, we have two indices: $\p_z$ regularity and weights in $z$. First, 
\begin{align} \label{yahL1}
\mathcal{E}_{O,-}^{(j)} := &\int_0^1 |\p_z^j U|^2 (\chi_{O,j}^-)^2 \ud z, \qquad \mathcal{D}_{O,-}^{(j)} := \int_0^1 |\bar{W}| |\p_z \p_z^j u|^2 (\chi_{O,j}^-)^2 \ud z,
\end{align}
and next, 
\begin{align} \label{canyou}
\mathcal{E}_{O,+}^{(j, n)} := & \int_1^{\infty} |\p_z^j U|^2 (\chi_{O,j}^+)^2 \langle z \rangle^{2n} \ud z, \qquad \mathcal{D}_{O,+}^{(j, n)} := \int_1^{\infty} \bar{W} |\p_z \p_z^j u|^2 (\chi_{O,j}^+)^2\langle z \rangle^{2n} \ud z.
\end{align}
We consolidate these via 
\begin{align} \label{Ejn}
\mathcal{E}_{O}^{(j, n)} := \mathcal{E}_{O,-}^{(j)} +  \mathcal{E}_{O,+}^{(j, n)}, \qquad \mathcal{D}_{O}^{(j, n)} := \mathcal{D}_{O,-}^{(j)} +  \mathcal{D}_{O,+}^{(j, n)}.
\end{align}
It is useful to have notation for norms in addition to $s$-dependent functionals. Therefore, we define 
\begin{align}\label{yahL2}
\| \psi \|_{VM^{j}_-}^2 :=& \sup_{1 \le s \le 1 + \bar{L}}  \mathcal{E}_{O,-}^{(j)}(s) + \int_1^{1 + \bar{L}}  \mathcal{D}_{O,-}^{(j)}(s) \ud s, \\
\| \psi \|_{VM^{j, n}_+}^2 :=& \sup_{1 \le s \le 1 + \bar{L}}  \mathcal{E}_{O,+}^{(j, n)}(s) + \int_1^{1 + \bar{L}}  \mathcal{D}_{O,+}^{(j, n)}(s) \ud s, \\
\| \psi \|_{VM^{j, n}}^2 :=&\| \psi \|_{VM^{j}_-}^2 + \| \psi \|_{VM^{j, n}_+}^2. 
\end{align}
Finally, we define our consolidated norms: 
\begin{align}  \n
\| \psi \|_{\text{von-Mise}_{-}}^2 :=& \sum_{j = 0}^3 \| \psi \|_{VM_-^{j}}^2, \qquad \| \psi \|_{\text{von-Mise}_{+,n}}^2 := \sum_{j = 0}^3 \| \psi \|_{VM_+^{j, n}}^2, \\ \label{vonMisenorm}
\| \psi \|_{\text{von-Mise}_n}^2 :=& \sum_{j = 0}^3 \| \psi \|_{VM^{j, n}}^2 = \| \psi \|_{\text{von-Mise}_{-}}^2 + \| \psi \|_{\text{von-Mise}_{+,n}}^2.
\end{align}
Our main proposition in this section will be the following estimate on the $\| \cdot \|_{\text{von-Mise}_n}$ norm to solutions of \eqref{abs:4:bljk}.
\begin{proposition} \label{pro:vonMise} Fix any $n \in \mathbb{N}$. Assume the background, $\bar{W}$, satisfies the following estimates: 
\begin{align} \label{sushiL1}
&\| \bar{W} \|_{L^\infty} + \sum_{j = 1}^5 \| \p_z^j \bar{W} \langle z \rangle^{n+1} \|_{L^\infty} \lesssim 1, \\ \label{sushiL2}
&|\bar{W}| \ge c_0 z, \qquad 0 < z << 1, 
\end{align}
Assume the source term, $\bold{R}$, satisfies $\bold{R}|_{z = 0} = \p_z \bold{R}|_{z = 0} = 0$. Assume the coefficients $\bold{b}_i$, $i = 1, 2, 3$, satisfy the following: 
\begin{align} \label{eng:1}
\bold{b}_1(s, 0) = & 0, \\ \label{eng:2}
\| \frac{\bold{b}_1}{\langle z \rangle} \|_{L^\infty} +  \sum_{j = 1}^3 \| \p_z^j \bold{b}_1 \|_{L^\infty} \lesssim & 1, \\ \label{eng:3}
\sum_{j = 0}^3 \| \p_z^j \bold{b}_2 \|_{L^\infty} \lesssim & 1, \\ \label{eng:4}
\sum_{j = 0}^3 \| \p_z^j \bold{b}_3 \langle z \rangle^{n+1} \|_{L^\infty} \lesssim & 1
\end{align}
Then the following bounds are valid for $0 < \bar{L} << 1$: 
\begin{align} \n
\| \psi \|_{\mathrm{von-Mise}_n}^2 \lesssim & \sum_{j = 0}^3 ( \| \p_z^j \zeta_{\mathrm{Left}} \chi_{O,j}^+ \langle z \rangle^n \|_{L^2_z}^2 + \| \p_z^j \zeta_{\mathrm{Right}} \chi_{O,j}^- \|_{L^2_z}^2 )\\ \label{mainvonM}
&+ \sum_{j = 0}^3 \| \p_z^j \bold{R} \chi_{O,j}  \langle z \rangle^n \|_{L^2_{sz}}^2 + \bar{L}(\| \psi \chi_I \|_{L^\infty}^2 + \| u \chi_I \|_{L^\infty}^2), \\ \n
\| \psi \|_{\mathrm{von-Mise}_-}^2 \lesssim & \sum_{j = 0}^3 (  \| \p_z^j \zeta_{\mathrm{Right}} \chi_{O,j}^- \|_{L^2_z}^2  )+ \sum_{j = 0}^3 \| \p_z^j \bold{R} \chi_{O,j}  \langle z \rangle^n \|_{L^2_{sz}}^2 \\  \label{lowvonM}
&+ \bar{L}(\| \psi \chi_I \|_{L^\infty}^2 + \| u \chi_I \|_{L^\infty}^2).
\end{align}
\end{proposition}

\begin{remark} We state separately the lower estimate, \eqref{lowvonM}, as we will need this bound in the future to control the quantities $\gamma_\psi, \gamma_u$, defined in \eqref{dgu}, for instance in \eqref{main:G:g}. The key point of this estimate will be the lack of dependence on $\zeta_{\mathrm{Left}}$.
\end{remark}

\subsection{Preliminary Hardy-type Inequalities}

Here we develop some Hardy type inequalities adapted to decreasing cutoff functions.
\begin{lemma} Let the function $f(\cdot), g(\cdot): \mathbb{R}_+ \rightarrow \mathbb{R}$ satisfy $f(0) = 0$ and $g(0) = g'(0) = 0$. Let $\chi_-$ be one of the smooth decreasing cutoff functions $\chi_{O,j}^-$, defined in \eqref{folgers:1}. Then 
\begin{align} \label{Hardy:order:1}
\| \p_z^j \Big( \frac{f}{z} \Big) \chi_-(z) \|_{L^2_z}  \lesssim & \| \p_z^{j+1} f \chi_-(z) \|_{L^2_z} + \| \p_z^{j} f \|_{L^2_z}, \\  \label{Hardy:order:2}
\| \p_z^j \Big( \frac{g}{z^2} \Big) \chi_-(z) \|_{L^2_z} \lesssim & \| \p_z^{j+2} g \chi_-(z) \|_{L^2_z} + \| \p_z^{j+1} g \|_{L^2_z}.
\end{align}
\end{lemma}
\begin{proof} The $j = 0$ case of \eqref{Hardy:order:1} follows from the standard Hardy inequality. Thus, fix $j \ge 1$. Define $F := \frac{f}{z}$, which is well defined due to the assumption $f(0) = 0$. The left-hand side of \eqref{Hardy:order:1} thus reads 
\begin{align*}
\| \p_z^j F \chi_-(z) \|_{L^2_z}^2 = & \int \p_z \{ z \} |\p_z^j F(z)|^2 \chi_-(z)^2 \ud z \\
= & - 2 \int z \p_z^j F(z) \p_z^{j+1}F(z)  \chi_-(z)^2 \ud z - 2 \int z |\p_z^j F(z)|^2 \chi_-'(z) \chi_-(z) \ud z  \\
= & - 2 \int \p_z^j F(z) \p_z^{j+1} \{ z F \} \chi_-(z)^2 \ud z + 2 (j+1) \int |\p_z^j F(z)|^2 \chi_-(z)^2 \ud z \\
&  - 2 \int z |\p_z^j F(z)|^2 \chi_-'(z) \chi_-(z) \ud z. 
\end{align*}
Moving the middle term to the left-hand side, we obtain the identity 
\begin{align*}
(2j + 1) \int |\p_z^j F(z)|^2 \chi_-(z)|^2 = 2 \int \p_z^j F(z) \p_z^{j+1} \{ z F \} \chi_-(z)^2 \ud z + 2 \int z |\p_z^j F(z)|^2 \chi_-'(z) \chi_-(z) \ud z.
\end{align*}
The latter term above is clearly bounded by $\| \p_z^{j} f \|_{L^2_z}^2$ by reinserting the definition $F = \frac{f}{z}$ and using that $z$ is bounded below on the support of $\chi_-'(z)$. For the former term we apply Cauchy-Schwartz and Young's inequality for products, which gives \eqref{Hardy:order:1}. Estimate \eqref{Hardy:order:2} follows by applying \eqref{Hardy:order:1} twice.  
\end{proof}

\subsection{Estimates on von-Mise Inversion Map}

The following bound informally equates information on $U$ to information on $\psi$. The formal derivative count gives us that $U \sim \psi$ near $z = 0$ (where $\bar{W}$ vanishes) and $U \sim u$ away from $z = 0$. However, there is an anomaly for $\psi$ and $u = \p_z \psi$ itself due to the ``loss of derivative" created by division by $\frac{1}{\bar{W}^2}$. This loss eventually becomes ``tame" for $j \ge 2$. These features are all reflected in the bounds below. 
\begin{lemma} For any $0 \le j \le 3$, the following bounds are valid:
\begin{align}  \label{outer:est:l2:0}
\| \bar{W} u^{(j)}  \chi^{-}_{O, j} \|_{L^2_z}^2 + \sum_{j' = 0}^j \| \p_z^{j'} \psi  \chi^{-}_{O, j} \|_{L^2_z}^2 \lesssim & \sum_{j' = 0}^j  \mathcal{E}_{O,-}^{(j')}(s) + \sum_{j' = 0}^2 \mathcal{E}_{O,-}^{(j')}(s).
\end{align}
\end{lemma}
\begin{proof} The key necessary step is to invert the map $U \mapsto \psi$. The inverse to the formula \eqref{goodunk:1} is 
\begin{align} \label{mellowdy:1}
\psi = \bar{W} \int_0^z \frac{U}{\bar{W}^2} \ud z',
\end{align}
which we differentiate to obtain the identities 
\begin{align} \label{inv:u}
\p_z \psi  = &\p_z \bar{W} \int_{0}^z \frac{U}{\bar{W}^2} \ud \tilde{z} + \frac{U}{\bar{W}}, \\ \label{traveler:1}
\p_z^2 \psi  = & \p_z^2 \bar{W} \int_0^z \frac{U}{\bar{W}^2} \ud \tilde{z} + \p_z \bar{W} \frac{U}{\bar{W}^2}+ \p_z \Big( \frac{U}{\bar{W}} \Big), \\  \label{traveler:2}
\p_z^3 \psi = & \p_z^3 \bar{W} \int_0^z \frac{U}{\bar{W}^2} \ud \tilde{z} + 2 \p_z^2 \bar{W} \frac{U}{\bar{W}^2} + \p_z \Big( \frac{U}{\bar{W}^2} \Big) + \p_z^2 \Big( \frac{U}{\bar{W}} \Big).
\end{align}

\vspace{2 mm}

\noindent \textit{Case $j = 0$:} We first estimate $\psi$ itself using \eqref{mellowdy:1}. Due to the bounded support of $\chi_{O,0}^-$, it suffices to establish an $L^\infty_z$ bound on $\psi$. We decompose the integration as follows 
\begin{align} \label{ess:1}
|\psi \chi_{O,0}^-| \lesssim & \chi_{O,0}^- |\int_0^z \frac{U}{\bar{W}^2} \chi_{O,2}^- \ud z'| + \chi_{O,0}^- |\int_0^z \frac{U}{\bar{W}^2} (1- \chi_2) \ud z'| \\  \label{ess:2}
\lesssim & \| \frac{U}{\bar{W}^2} \chi_{O,2}^- \|_{L^2_z} + \widetilde{\chi}_{O,0}^- |\int_0^z \frac{U}{\bar{W}^2}  \chi_{O,0}^- (1- \chi_2) \ud z'| \\  \label{ess:3}
\lesssim & \| \frac{U}{z^2}\chi_{O,2}^- \|_{L^2_z} +\| U \chi_{O,0}^- \|_{L^2_z}\\  \label{ess:4}
\lesssim & \| \p_z^2 U \chi_{O,2}^- \|_{L^2_z} + \| \p_z U \chi_{O,1}^- \|_{L^2_z} +\| U \chi_{O,0}^- \|_{L^2_z} \lesssim  \sum_{j' = 0}^2 \mathcal{E}_{O,-}^{(j')}(s).
\end{align}
Above, to go from \eqref{ess:1} to \eqref{ess:2}, we let $ \widetilde{\chi}_{O,0}^-$ be a slightly fattened version of the cutoff $\chi_{O,0}^-$ so that $\chi_{O,-}\widetilde{\chi}_{O,0}^- = \chi_{O,-}$ and simultaneously the support of $\widetilde{\chi}_{O,0}^-$ avoids $z = 1$. We subsequently observed that $\chi_{O,0}^-(z) \le \chi_{O,0}^-(z')$ when $z' \le z$ in order to bring in this cutoff function into the integration. To go from \eqref{ess:2} to \eqref{ess:3}, we have observed that $\bar{W} \gtrsim z$ in the support of $\chi_{O,2}^-$, and $\bar{W} \gtrsim 1$ in the support of $(1 - \chi_{O,2}^-)$. To go from \eqref{ess:3} to \eqref{ess:4}, we have used the standard Hardy inequality coupled with the property \eqref{ygh1}.

To conclude the estimate for the case $j = 0$, we notice upon using the identity $U = \bar{W} u - \bar{W}_z \psi$ that 
\begin{align*}
\| \bar{W} u \chi_{O,0}^- \|_{L^2_z} \lesssim \| U \chi_{O,0}^- \|_{L^2_z} + \| \psi \chi_{O,0}^- \|_{L^2_z} \lesssim & \sum_{j = 0}^2 \| \p_z^j U \chi_{O,j}^- \|_{L^2_z}  \lesssim \sum_{j' = 0}^2 \mathcal{E}_{O,-}^{(j')}(s),
\end{align*}
where we have invoked \eqref{ess:4} for the second-to-last inequality above. 

\vspace{2 mm}

\noindent \textit{Case $j = 1$:} Recalling the case $j = 0$, we have already established, for any $0 \le j \le 3$ 
\begin{align}
\chi_{O,j}^- |\int_0^z \frac{U}{\bar{W}^2} \ud z' | \lesssim \Big( \sum_{j' = 0}^2 \mathcal{E}_{O,-}^{(j')}(s) \Big)^{\frac12}.
\end{align}
We therefore estimate using the Hardy inequality coupled with the property \eqref{ygh1}
\begin{align*}
\| \frac{U}{\bar{W}} \chi_{O,1}^- \|_{L^2_z} \lesssim \| \frac{U}{z} \chi_{O,1}^- \|_{L^2_z} \lesssim \| \p_z U \chi_{O,1}^- \|_{L^2_z} + \| U \chi_{O,0}^- \|_{L^2_z}.
\end{align*}
From these estimates, we conclude the estimate on $\p_z \psi \chi_{O,1}^-$. To obtain the estimate on $\bar{W} u_z \chi_{O,1}$, we notice the identity 
\begin{align} \label{td:1}
\bar{W} u_z = \p_z U + \bar{W}_{zz}\psi, 
\end{align}
after which the result follows upon applying the previously established estimates on $\psi$. 

\vspace{2 mm}

\noindent \textit{Case $j = 2$:} Using \eqref{traveler:1}, we have 
\begin{align*}
\| \p_z^2 \psi \chi_{O,2}^- \|_{L^2_z}^2 \lesssim & \| \chi_{O,2}^- \int_0^z \frac{U}{\bar{W}^2} \ud z' \|_{L^2_z}^2 + \| \frac{U}{\bar{W}^2} \chi_{O,2}^- \|_{L^2_z}^2 + \| \p_z \{ \frac{U}{\bar{W}} \} \chi_{O,2}^- \|_{L^2_z}^2 \\
\lesssim &  \sum_{j' = 0}^2 \mathcal{E}_{O,-}^{(j')}(s) + \| \frac{U}{z^2} \chi_{O,2}^- \|_{L^2_z}^2 + \| \p_z \{ \frac{z}{\bar{W}} \frac{U}{z} \} \chi_{O,2}^- \|_{L^2_z}^2 \\
\lesssim &  \sum_{j' = 0}^2 \mathcal{E}_{O,-}^{(j')}(s) +\| \p_z \{ \frac{U}{z} \} \chi_{O,2}^- \|_{L^2_z}^2 \\
\lesssim &  \sum_{j' = 0}^2 \mathcal{E}_{O,-}^{(j')}(s).
\end{align*}
Above, we have used \eqref{Hardy:order:1}. Differentiating \eqref{td:1}, we obtain 
\begin{align}
\bar{W} \p_z^2 u = \p_z^2 U - \bar{W}_z u_z + \bar{W}_{zz}u + \bar{W}_{zzz}u,
\end{align}
after which we obtain 
\begin{align*}
\| \bar{W} \p_z^2 u \chi_{O,2}^- \|_{L^2}^2 \lesssim \| \p_z^2 U \chi_{O,2}^- \|_{L^2_z}^2 + \sum_{j' = 0}^2 \| \p_z^{j'} \psi \chi_{O,2}^- \|_{L^2_z}^2 \lesssim \sum_{j' = 0}^2 \mathcal{E}_{O,-}^{(j')}(s).
\end{align*}

\vspace{2 mm}

\noindent \textit{Case $j = 3$:} This follows in much the same way as $j = 2$, upon using \eqref{Hardy:order:1} and \eqref{Hardy:order:2}. The lemma is proven. 
\end{proof}

We now obtain an analogue of the previous lemma for the region above the free boundary, $\{z = 0\}$. These estimates, however, are \textit{not} symmetric. One can see this by noticing the inversion formula \eqref{mellowdy:1} which uses $\psi|_{z = 0} = 0$. Heuristically, the formula \eqref{mellowdy:1} is one instance of ``information getting sent upwards" through the incompressibility condition. Therefore, the forthcoming estimates will be distinguished through the appearance of $\| \psi \chi_I \|_{L^\infty}$ and $\| u \chi_I \|_{L^\infty}$. 

\begin{lemma} For $1 \le j \le 4$, for any $n$, the following bounds are valid:
\begin{align}  \label{wo:est:l2:0}
\| \p_z^j \psi \langle z \rangle^{n} \chi^{+}_{O, j-1} \|_{L^2_z}^2 \lesssim & \| \psi \chi_I \|_{L^\infty}^2 + \| u \chi_I \|_{L^\infty}^2  + \sum_{j' = 0}^{j-1} \mathcal{E}_{O,+}^{(j',n)} ,  \\ \label{wo:est:0}
\| \psi \chi_{O,0} \|_{L^\infty_z}^2 \lesssim & \| \psi \chi_I \|_{L^\infty}^2 +  \mathcal{E}_{O,+}^{(0, 1)}.
\end{align}
\end{lemma}
\begin{proof} Here the main objective is to obtain the weighted estimates in $z$. For any $z_0 > 1$, we have the inversion map: 
\begin{align} \label{psi:inv:up}
\psi(s, z) = & \frac{ \psi(s, z_0)}{\bar{W}(s, z_0)}\bar{W}(s, z) + \bar{W}(s, z) \int_{z_0}^z \frac{U(s, z')}{\bar{W}(s, z')^2} \ud z', \qquad z \ge z_0. 
\end{align}
This gives us the result \eqref{wo:est:0} as follows. First, it clearly suffices to estimate $\psi (1 - \chi_I)$. Let $z \in \text{supp}(1 - \chi_I)$. In this case we may choose $1 < z_0 < z$ satisfying $z_0 \in \{ \chi_I = 1 \} \cap \{ \chi_{O,0} = 1 \}$. We apply \eqref{psi:inv:up} with this choice of $z_0$. Notice that such a choice of $z_0$ guarantees that $(z_0, z) \in \{\chi_{O,0} = 1\}$. Then we obtain 
\begin{align*}
|\psi(s, z)| \lesssim \| \psi \chi_I \|_{L^\infty} + \| U \chi_{O,0} \langle z \rangle \|_{L^2_z} \lesssim \| \psi \chi_I \|_{L^\infty}+  (\mathcal{E}_{O,+}^{(0, 1)})^{\frac12},
\end{align*}
as required by \eqref{wo:est:0}. 

We now differentiate \eqref{psi:inv:up} to generate the identities for $z > z_0$:
\begin{align} \label{inv:u:diff}
\p_z \psi  = &\frac{ \psi(s, z_0)}{\bar{W}(s, z_0)}\p_z \bar{W}(s, z) + \p_z \bar{W} \int_{z_0}^z \frac{U}{\bar{W}^2} \ud \tilde{z} + \frac{U}{\bar{W}}, \\ \label{mvp:1}
\p_z^2 \psi  = &\frac{ \psi(s, z_0)}{\bar{W}(s, z_0)}\p_z^2 \bar{W}(s, z) + \p_z^2 \bar{W} \int_{z_0}^z \frac{U}{\bar{W}^2} \ud \tilde{z} + \p_z \bar{W} \frac{U}{\bar{W}^2}+ \p_z \Big( \frac{U}{\bar{W}} \Big), \\ \label{mvp:2}
\p_z^3 \psi = &\frac{ \psi(s, z_0)}{\bar{W}(s, z_0)}\p_z^3 \bar{W}(s, z) +  \p_z^3 \bar{W} \int_{z_0}^z \frac{U}{\bar{W}^2} \ud \tilde{z} + 2 \p_z^2 \bar{W} \frac{U}{\bar{W}^2} + \p_z \Big( \frac{U}{\bar{W}^2} \Big) + \p_z^2 \Big( \frac{U}{\bar{W}} \Big), \\ \n
\p_z^4 \psi = & \frac{ \psi(s, z_0)}{\bar{W}(s, z_0)}\p_z^4 \bar{W}(s, z) + \p_z^4 \bar{W} \int_{z_0}^z \frac{U}{\bar{W}^2} \ud \tilde{z} +  3 \p_z^3 \bar{W} \frac{U}{\bar{W}^2} + 2 \p_z^2 \bar{W} \p_z ( \frac{U}{\bar{W}^2})   \\ \label{mvp:3}
& + \p_z^2 \Big( \frac{U}{\bar{W}^2} \Big) + \p_z^3 \Big( \frac{U}{\bar{W}} \Big).
\end{align}
We choose $1 < z_0$ in such a way that $z_0 \in \{ \chi_I = 1\}$ and $\{z > z_0\} \supset \text{supp}(\chi^+_{O,0})$. Notice that $\bar{W} \gtrsim z_0 - 1 > 0$ in this regime. 

We will in particular need to estimate the integrals which appear in \eqref{inv:u:diff} -- \eqref{mvp:3}. To do so, we proceed as follows. Fix another point $z_1 > z_0$ satisfying the following criteria: $z_1 \in \{\chi_I = 1\} \cap \{\chi_{O,0}^+ = 1\}$. In particular, this insures $(z_0, z_1) \subset \{\chi_I = 1\}$ and $(z_1, \infty) \subset \{\chi_{O,0}^+ = 1\}$. Then we have for any $z > z_0$, 
\begin{align} \n
|\int_{z_0}^z \frac{U}{\bar{W}^2} \ud z'| \le & |\int_{z_0}^{\infty} \frac{U}{\bar{W}^2} \ud z'| \le |\int_{z_0}^{z_1} \frac{U}{\bar{W}^2} \ud z'| + |\int_{z_1}^{\infty} \frac{U}{\bar{W}^2} \ud z'| \\ \n
\le & |\int_{z_0}^{z_1} \frac{U}{\bar{W}^2} \chi_I \ud z'| + |\int_{z_1}^{\infty} \frac{U}{\bar{W}^2} \chi_{O,0}^+  \ud z'| \\ \n
\lesssim & \| U \chi_I \|_{L^\infty} + \| U \langle z \rangle \chi_{O,0}^+ \|_{L^2_z} \\
\lesssim & \| \psi \chi_I \|_{L^\infty} + \| u \chi_I \|_{L^\infty} + \| U \langle z \rangle \chi_{O,0}^+ \|_{L^2_z}.
\end{align}
We now have from \eqref{inv:u:diff} -- \eqref{mvp:3} for $1 \le j \le 4$, 
\begin{align*}
\| \p_z^j \psi \chi_{O,j-1}^+ \langle z \rangle^n \|_{L^2_z} \lesssim & \sum_{j' = 1}^4 \| \p_z^j \bar{W} \langle z \rangle^{n+1} \|_{L^\infty_z}  \Big( \| \psi \chi_I \|_{L^\infty} +\| u \chi_I \|_{L^\infty} + \| U \langle z \rangle \chi_{O,0}^+ \|_{L^2_z} \\
& + \sum_{j' = 0}^{j-1} \| \p_z^{j'} U \langle z \rangle^n \chi_{O,j-1}^+ \|_{L^2_z} \Big) \\
\lesssim & \sum_{j' = 1}^4 \| \p_z^j \bar{W} \langle z \rangle^{n+1} \|_{L^\infty_z}  \Big( \| \psi \chi_I \|_{L^\infty} + \| u \chi_I \|_{L^\infty}+ \| U \langle z \rangle \chi_{O,0}^+ \|_{L^2_z}\\
&+ \sum_{j' = 0}^{j-1} \| \p_z^{j'} U \langle z \rangle^n \chi_{O,j'}^+ \|_{L^2_z} \Big).
 \end{align*}
The lemma is proven. 
\end{proof}

It will be useful for us to collect the following interpolation estimates, which in particular allow a pairing with $u^{(j)}$ and $\chi_{O,j}^-$ upon invoking a piece of the diffusion. This is only relevant for the bottom portion due to the vanishing of $\bar{W}$, as the top portion is already obtained in \eqref{wo:est:l2:0}. We note that these are the von-Mise analogues of the bounds \eqref{ode:2} -- \eqref{ode:3} obtained in the Crocco variables. 
\begin{lemma} For any $0 \le j \le 3$, the following bounds are valid:
\begin{align} \label{ah:1}
\| u^{(j)} \chi_{O,j}^- \|_{L^2_{sz}}^2 \lesssim & \bar{L}^{\frac13}( \| \sqrt{\bar{W}} u_z^{(j)} \chi_{O,j}^- \|_{L^2_{sz}}^2 +\|\bar{W} u^{(j)} \chi_{O,j}^- \|_{L^\infty_{s}L^2_z}^2).
\end{align}
\end{lemma}
\begin{proof} We fix a cutoff function at scale $\delta_L$ for a small parameter $\delta_L > 0$ to be chosen. We then have 
\begin{align*}
\| u^{(j)} \chi_{O,j}^-  \|_{L^2_{sz}}^2 \lesssim &\| u^{(j)} \chi_{O,j}^- \chi_{\delta_L} \|_{L^2_{sz}}^2 + \| u^{(j)} \chi_{O,j}^- (1 - \chi_{\delta_L}) \|_{L^2_{sz}}^2 \\
\lesssim & \| \bar{W} u_z^{(j)} \chi_{O,j}^- \chi_{\delta_L} \|_{L^2_{sz}}^2 + \frac{1}{\delta_L^2} \|\bar{W} u^{(j)} \chi_{O,j}^- \|_{L^2_{sz}}^2 \\
\lesssim & \delta_L \| \sqrt{\bar{W}} u_z^{(j)} \chi_{O,j}^- \chi_{\delta_L} \|_{L^2_{sz}}^2 + \frac{\bar{L}}{\delta_L^2}\|\bar{W} u^{(j)} \chi_{O,j}^- \|_{L^\infty_{s}L^2_z}^2.
\end{align*}
We then choose $\delta_L = \bar{L}^{\frac13}$. The lemma is proven. 
\end{proof}
As a corollary to this, upon pairing with \eqref{outer:est:l2:0}, we can estimate $\| u^{(j)} \chi_{O,j}^- \|_{L^2_{sz}}^2$ for any $0 \le j \le 3$ in terms of the $\| \cdot \|_{\mathrm{von-Mise}_-}$ norm. In fact, we summarize the following bounds (which are the analogue of \eqref{west:1} -- \eqref{west:2} in the Crocco setting) that will be used to estimate linear error terms appearing in our von-Mise energy estimates.  
\begin{corollary}For any $0 \le j \le 3$, the following bounds are valid:
\begin{align} \label{ujell2} 
\| u^{(j)} \chi_{O,j}^- \|_{L^2_{sz}}^2 \lesssim & \bar{L}^{\frac13} \| \psi \|_{\mathrm{von-Mise}_-}^2, \\ \label{md:2}
\| u^{(j)} \chi_{O,j}^+ \langle z \rangle^n \|_{L^2_{sz}}^2 \lesssim & \bar{L} (\| \psi \chi_I \|_{L^\infty}^2 +\| u \chi_I \|_{L^\infty}^2 ) + \bar{L} \| \psi \|_{\mathrm{von-Mise}_{+, n}}^2, \\ \label{md:3}
\| \psi \chi_{O,0}^- \|_{L^2_{s}L^\infty_z}^2 \lesssim & \bar{L}\| \psi \|_{\mathrm{von-Mise}_-}^2, \\ \label{md:4}
\| \psi \chi_{O,0}^+ \|_{L^2_{s}L^\infty_z}^2 \lesssim & \bar{L} \| \psi \chi_I \|_{L^\infty}^2 + \bar{L} \| \psi \|_{\mathrm{von-Mise}_{+, 1}}^2.
\end{align}
\end{corollary}
\begin{proof} To prove \eqref{ujell2}, we invoke \eqref{ah:1} and subsequently \eqref{outer:est:l2:0}. To prove \eqref{md:2} (resp. \eqref{md:4}), we simply integrate both sides of \eqref{wo:est:l2:0} (resp. \eqref{wo:est:0}) over $s \in (1, 1 + \bar{L})$. To prove \eqref{md:3}, we integrate both sides of \eqref{ess:4}. 

\end{proof}

\subsection{Estimates on von-Mise Forward Map}

The lemmas above essentially provide estimates on the inversion map: converting information about $U$ into information about $\psi$ and $u$, which are useful as the energy functionals \eqref{yahL1}, \eqref{canyou} are expressed in terms of $U$. On the other hand, the dissipation functions are expressed in terms of $u$, and so it turns out to be convenient to provide the following estimates which also go ``forwards": estimate $U$ in terms of $u$. These will primarily be used to control commutators coming from the diffusion terms. 
\begin{lemma} Under the assumptions on $\bar{W}$, \eqref{sushiL1} -- \eqref{sushiL2}, for any $0 \le j \le 3$,
\begin{align} \label{deri:1}
\| \p_z U^{(j)} \chi_{O,j}^- \|_{L^2_{sz}}^2 \lesssim &\| \psi \|_{VM^{j}_-}^2 +\bar{L}^{\frac13} \| \psi \|_{\mathrm{von-Mise}_-}^2,  \\ \label{deri:2}
\| \p_z U^{(j)} \chi_{O,j}^+ \langle z \rangle^n \|_{L^2_{sz}}^2 \lesssim &\| \psi \|_{VM^{j,n}_+}^2 + \bar{L}\| \psi \|_{\mathrm{von-Mise}_{+,n}}^2 +  \bar{L} (\| \psi \chi_I \|_{L^\infty}^2 +\| u \chi_I \|_{L^\infty}^2 ).
\end{align}
\end{lemma}
\begin{proof} We differentiate the map $U = \bar{W} u - \bar{W}_z \psi$ to obtain 
\begin{align} \label{derived:1}
\p_z^{j+1} U = \bar{W} \p_z u^{(j)} + \sum_{j' = 0}^{j} \binom{j+1}{j'} \bar{W}^{(j+1 - j')} u^{(j')} - \sum_{j' = 0}^{j+1} \binom{j+1}{j'} \bar{W}^{(j+2-j')} \psi^{(j')},
\end{align}
after which we obtain 
\begin{align*}
\| \p_z U^{(j)} \chi_{O,j}^- \|_{L^2_{sz}}^2 \lesssim & \| \bar{W} \p_z u^{(j)} \chi_{O,j}^- \|_{L^2_{sz}}^2 +\| \psi \chi_{O,j}^- \|_{L^2_{sz}}^2 + \sum_{j' = 0}^{j} \| u^{(j')} \chi_{O,j}^- \|_{L^2_{sz}}^2 \\
\lesssim & \|\sqrt{ |\bar{W}|} \p_z u^{(j)} \chi_{O,j}^- \|_{L^2_{sz}}^2 +\| \psi \chi_{O,0}^- \|_{L^2_{sz}}^2 + \sum_{j' = 0}^{j} \| u^{(j')} \chi_{O,j'}^- \|_{L^2_{sz}}^2 \\
\lesssim & \|\sqrt{ |\bar{W}|} \p_z u^{(j)} \chi_{O,j}^- \|_{L^2_{sz}}^2 +\bar{L}^{\frac13} \| \psi \|_{\text{von-Mise}_-}^2 \\
\lesssim & \| \psi \|_{VM^{j}_-}^2 + \bar{L}^{\frac13} \| \psi \|_{\text{von-Mise}_-}^2 
\end{align*}
where in the second-to-last line we have invoked \eqref{ujell2} and \eqref{md:3}, and in the final line we have invoked the definitions \eqref{yahL1} and \eqref{yahL2}.

For the bound \eqref{deri:2}, we again use \eqref{derived:1} to estimate as follows: 
\begin{align*}
\| \p_z U^{(j)} \chi_{O,j}^+ \langle z \rangle^n \|_{L^2_{sz}}^2 \lesssim & \| \bar{W} \p_z u^{(j)} \chi_{O,j}^+ \langle z \rangle^n \|_{L^2_{sz}}^2 +\| \psi \chi_{O,j}^- \|_{L^2_sL^\infty_{z}}^2 + \sum_{j' = 0}^{j} \| u^{(j')} \chi_{O,j}^+ \langle z \rangle^n\|_{L^2_{sz}}^2 \\
\lesssim & \|\sqrt{ \bar{W}} \p_z u^{(j)} \chi_{O,j}^+ \langle z \rangle^n \|_{L^2_{sz}}^2 +\| \psi \chi_{O,0}^+ \|_{L^2_{s}L^\infty_z}^2 + \sum_{j' = 0}^{j} \| u^{(j')} \chi_{O,j'}^+  \langle z \rangle^n \|_{L^2_{sz}}^2 \\
\lesssim & \|\sqrt{ \bar{W}} \p_z u^{(j)} \chi_{O,j}^+  \langle z \rangle^n \|_{L^2_{sz}}^2 +\bar{L} \| \psi \|_{\text{von-Mise}_{+, n}}^2 + \bar{L} (\| \psi \chi_I \|_{L^\infty}^2 +\| u \chi_I \|_{L^\infty}^2 ) \\
\lesssim & \| \psi \|_{VM^{j,n}_{+}}^2 + \bar{L} \| \psi \|_{\text{von-Mise}_{+,n}}^2 + \bar{L} (\| \psi \chi_I \|_{L^\infty}^2 +\| u \chi_I \|_{L^\infty}^2 ),
\end{align*}
where we have invoked \eqref{md:2}, \eqref{md:4}. The lemma is proven. 
\end{proof}

\subsection{von-Mise Energy Estimates}

Since we will be differentiating the system three times, it is helpful to present the form of the equation after $\p_z^j$ has been applied, specifically with regards to commutators arising from the left-hand side. We introduce the local notation $g^{(j)} := \p_z^j g$ for an abstract function $g(s, z)$. We have 
\begin{align}
\p_s U^{(j)} + \bold{b}_1 \p_z u^{(j)} + \bold{b}_2 u^{(j)} + \bold{b}_3 \psi^{(j)} - \p_z^2 u^{(j)} = - \sum_{i = 1}^3 LC_{i; j} + \p_z^j \bold{R}, 
\end{align}
where ``LC" stands for "Linear Commutator" and these terms are as follows:
\begin{align}
LC_{1; j} :=  &\sum_{j' = 0}^{j-1} \binom{j}{j'} \bold{b}_1^{(j-j')} u^{(j'+1)}, \\
LC_{2; j} := & \sum_{j' = 0}^{j-1} \binom{j}{j'} \bold{b}_2^{(j-j')} u^{(j')}, \\
LC_{3; j} := & \sum_{j' = 0}^{j-1} \binom{j}{j'} \bold{b}_3^{(j-j')} \psi^{(j')}.
\end{align}

We first record the following observation about the boundary data at $\{z = 0\}$, which is important for us to perform our high order energy method. 
\begin{lemma}[Boundary conditions at $z = 0$] \label{LB1} Assume $\bold{R}|_{z = 0} = \p_z \bold{R}|_{z = 0} = 0$. The following boundary conditions hold:
\begin{align} \label{SPE:1}
&u|_{z = 0} = \p_z^2 u|_{z = 0} = \p_z^3 u|_{z = 0} = 0, \qquad \p_z^j U|_{z = 0} = 0, \qquad j = 0, 1, 3.
\end{align}
\end{lemma}
\begin{proof} The proof follows by repeatedly using equation \eqref{eq:for:v:1}, and $U$ vanish to second order at $z = 0$. We thus just need to consider $\p_z^3 U, \p_z^4 U$. For this, we record the following identity via the product rule:
\begin{align}
\p_z^3 U = & \bar{w} \p_z^3 u + 2 \bar{w}_z \p_z^2 u - 2  \p_z^3 \bar{w} u - \p_z^4 \bar{w} \psi.
\end{align} 
Evaluation at $z = 0$ and using that $\psi|_{z = 0} = u|_{z= 0} = \p_z^2 u|_{z = 0} = \p_z^3 u|_{z = 0} = 0$, as well as $\p_z^3 \bar{u}|_{z = 0} = 0$, we obtain the result.
\end{proof}

We perform the bottom and top estimates separately, as there are different features that present in each estimate. For the bottom estimate, we need to take care for the boundary conditions at $\{z = 0\}$ and degeneracies due to the vanishing of $\bar{W}$ at $\{z  = 0\}$ (which in particular creates the apparent loss of derivative in \eqref{outer:est:l2:0}). On the other hand, for the top estimate, we need to propagate weights in $z$ as well as contend with the ``nonlocal" effect (the $\psi \chi_I, u \chi_I$ terms appearing on the right-hand side of \eqref{wo:est:l2:0} -- \eqref{wo:est:0}). 
\begin{lemma}[Bottom Estimate] \label{klem3} Under the assumptions of Proposition \ref{pro:vonMise}, for any $0 \le j \le 3$, 
\begin{align}\n
\| \psi \|_{VM^{j}_-}^2 \lesssim & \| \p_z^j \zeta_{\mathrm{Right}} \chi_{O,j}^- \|_{L^2_z}^2 + \| \p_z^j \bold{R} \chi_{O,j}^- \|_{L^2_{sz}}^2 + \bar{L}(\| \psi \chi_I \|_{L^\infty}^2 + \| u \chi_I \|_{L^\infty}^2) \\  \label{vm:YES:1}
& + \bar{L}^{\frac13} \| \psi \|_{\mathrm{von-Mise}_-}^2 + \bold{1}_{1 \le j \le 3} \| \psi \|_{VM^{j-1}_-}^2.
\end{align}
\end{lemma}
\begin{proof} We apply the multiplier $U^{(j)} ( \chi_{O,j}^{-})^2$. Because there are many terms in this estimate, it is useful to keep in mind the overall strategy of this estimate, which we now discuss.

\vspace{2 mm} 

\noindent \textit{Step 0: Strategy of the Proof} We will obtain an identity of the form: 
\begin{align} \label{form:1}
\frac{\p_s}{2} \mathcal{E}^{(j)}_{O,-}(s) - \mathcal{D}^{(j)}_{O,-}(s) = \text{Err}(s). 
\end{align}
Integrating the above equality backwards from $s = 1 + \bar{L}$ to $s = s_\ast$ produces 
\begin{align} \n
 \mathcal{E}^{(j)}_{O,-}(s_\ast) + \int_{s_\ast}^{1 + \bar{L}}  \mathcal{D}^{(j)}_{O,-}(s) \ud s = & \mathcal{E}^{(j)}_{O,-}(1 + \bar{L}) -  \int_{s_\ast}^{1 + \bar{L}} \text{Err}(s) \ud s \\
 \le &  \mathcal{E}^{(j)}_{O,-}(1 + \bar{L}) + \int_1^{1 + \bar{L}} |\text{Err}(s)| \ud s. 
\end{align}
Taking the supremum in $s_\ast$, we obtain 
\begin{align} \label{ocj:1}
\| \psi \|_{VM^j_-}^2 \le \mathcal{E}^{(j)}_{O,-}(1 + \bar{L}) +  \int_1^{1 + \bar{L}} |\text{Err}(s)| \ud s \le \| \p_z^j \zeta_{\text{Right}} \chi_{O,j}^- \|_{L^2_z}^2 + \int_1^{1 + \bar{L}} |\text{Err}(s)| \ud s.
\end{align}
The proof of inequality \eqref{vm:YES:1} will therefore follow upon establishing the following bound:
\begin{align} \n
\int_1^{1 + \bar{L}} |\text{Err}(s)| \ud s \le & \frac{1}{2} \| \psi \|_{VM^j_-}^2  +C\Big( \| \p_z^j \bold{R} \chi_{O,j}^- \|_{L^2_{sz}}^2 + \bar{L}(\| \psi \chi_I \|_{L^\infty}^2 + \| u \chi_I \|_{L^\infty}^2) \\  \label{vm:YES:2}
& + \bar{L}^{\frac13} \| \psi \|_{\text{von-Mise}_-}^2 + \bold{1}_{1 \le j \le 3} \| \psi \|_{VM^{j-1}_-}^2\Big),
\end{align}
where the $ \frac{1}{2} \| \psi \|_{VM^j_-}^2$ will be absorbed to the left-hand side of \eqref{ocj:1}. 

We turn now to the specific form of $\text{Err}(s)$, and the corresponding estimates. 

\vspace{2 mm} 

\noindent \textit{Step 1: Energy Identity} Doing so produces the following identity 
\begin{align} \n
&\frac{\p_s}{2} \int_0^1 |U^{(j)}|^2 (\chi_{O,j}^{-})^2 \ud z + \int_0^1 \p_z u^{(j)} \p_z U^{(j)} (\chi_{O,j}^{-})^2 \ud z \\ \n
=& - 2 \int_{0}^1 \p_z u^{(j)} U^{(j)} \chi_{O,j}^{-} \p_z \chi_{O,j}^{-} \ud z - \int_0^1 \bold{b}_1 \p_z u^{(j)} U^{(j)} (\chi_{O,j}^{-})^2 \ud z  - \int_0^1 \bold{b}_2 u^{(j)} U^{(j)} (\chi_{O,j}^{-})^2 \ud z \\ \n
& -  \int_0^1 \bold{b}_3 \psi^{(j)} U^{(j)} (\chi_{O,j}^{-})^2 \ud z -  \sum_{i = 1}^3 \int_0^1 LC_{i;j} U^{(j)} (\chi_{O,j}^{-})^2 \ud z + \int_0^1 \p_z^j \bold{R} U^{(j)} (\chi_{O,j}^{-})^2 \ud z \\ \label{vmEE:1}
=:& \sum_{i = 1}^8 \mathrm{Err}_{i; j},
\end{align}
where we note the boundary contribution at $\{z = 0\}$ vanishes due to \eqref{SPE:1}.

First of all, in order to rewrite this expression in the form \eqref{form:1}, we need to expand the diffusion term on the left-hand side. This process of expanding, applying the product rule, and integrating by parts produces many (linear) commutator terms which we need to estimate carefully. Indeed, we have   
\begin{align} \n
&\int_0^1 \p_z u^{(j)} \p_z U^{(j)} (\chi_{O,j}^{-})^2 \ud z=  \int_0^1 \p_z u^{(j)} \p_z \p_z^j \{ \bar{W} u - \bar{W}_z \psi  \} (\chi_{O,j}^{-})^2 \ud z \\ \n
= & \int_0^1 \bar{W} |\p_z u^{(j)}|^2 (\chi_{O,j}^{-})^2 \ud z + \int_0^1 j \bar{W}_z \p_z u^{(j)} u^{(j)} (\chi_{O,j}^{-})^2 \ud z \\ \n
&+ \sum_{j' = 0}^{j-1} \binom{j+1}{j'} \int_0^1 \bar{W}^{(1 + j-j')} \p_z u^{(j)}  u^{(j')} (\chi_{O,j}^{-})^2 \ud z \\
& - \sum_{j' = 0}^{j} \binom{j+1}{j'} \int_0^1 \bar{W}_z^{(1 + j-j')} \p_z u^{(j)}  \psi^{(j')} (\chi_{O,j}^{-})^2 \ud z \\ \n
= & \int_0^1 \bar{W} |\p_z u^{(j)}|^2 (\chi_{O,j}^{-})^2 \ud z - \frac12 \int_0^1 j \bar{W}_z  |u^{(j)}|^2 (\chi_{O,j}^{-})^2 \ud z - \frac{j}{2} \bar{W}_z |u^{(j)}|^2|_{z = 0} \\ \n
& - \int_0^1 j \bar{W}_z  |u^{(j)}|^2 \chi_{O,j}^{-} \p_z\chi_{O,j}^- \ud z - \sum_{j' = 0}^{j-1} \binom{j+1}{j'} \int_0^1 \bar{W}^{(1 + j-j')} u^{(j)} \p_z u^{(j')} (\chi_{O,j}^{-})^2 \ud z \\ \n
& - \sum_{j' = 0}^{j-1} \binom{j+1}{j'} \int_0^1 \bar{W}^{(2 + j-j')} u^{(j)} u^{(j')} (\chi_{O,j}^{-})^2 \ud z -2 \sum_{j' = 0}^{j-1} \binom{j+1}{j'} \int_0^1 \bar{W}^{(1 + j-j')} u^{(j)} u^{(j')} \chi_{O,j}^{-} \p_z \chi_{O,j}^{-} \ud z \\ \n
&-  \sum_{j' = 0}^{j-1} \binom{j+1}{j'} \bar{W}^{(1 + j-j')}u^{(j)}  u^{(j')} \Big|_{z = 0} + \sum_{j' = 0}^{j} \binom{j+1}{j'} \int_0^1 \bar{W}_z^{(1 + j-j')}  u^{(j)}  u^{(j')} (\chi_{O,j}^{-})^2 \ud z \\ \n
&+  \sum_{j' = 0}^{j} \binom{j+1}{j'} \int_0^1 \bar{W}_{zz}^{(1 + j-j')}  u^{(j)}  \psi^{(j')} (\chi_{O,j}^{-})^2 \ud z  +2  \sum_{j' = 0}^{j} \binom{j+1}{j'} \int_0^1 \bar{W}_{z}^{(1 + j-j')}  u^{(j)}  \psi^{(j')} \chi_{O,j}^{-} \p_z \chi_{O,j}^{-} \ud z \\ \n
&+ \sum_{j' = 0}^{j} \binom{j+1}{j'}  \bar{W}_z^{(1 + j-j')}  u^{(j)}  \psi^{(j')} \Big|_{z = 0} \\ \label{yes:1}
=:&  \int_0^1 \bar{W} |\p_z u^{(j)}|^2 (\chi_{O,j}^{-})^2 \ud z - \sum_{i = 1}^{11} J_{i},
\end{align}
where the $J_i$ are error terms to be estimated. We now insert the identity \eqref{yes:1} into \eqref{vmEE:1} in order to derive 
\begin{align}
&\frac{\p_s}{2} \int_0^1 |U^{(j)}|^2 (\chi_{O,j}^{-})^2 \ud z +\int_0^1 \bar{W} |\p_z u^{(j)}|^2 (\chi_{O,j}^{-})^2 \ud z = \sum_{i = 1}^8 \mathrm{Err}_{i; j} +  \sum_{i = 1}^{11} J_{i} := \text{Err}(s). 
\end{align}
We note that the above equality is of the form \eqref{form:1} upon noting that $\bar{W} = - |\bar{W}|$ in the support of $\chi_{O,j}^-$. 
\vspace{2 mm}

\noindent \textit{Step 2: Diffusive Error Terms, $J_i(s)$} First, we have 
\begin{align*}
|J_1| \lesssim &\| \bar{W}_z \|_{L^\infty} \| u^{(j)} \chi_{O,j}^{-} \|_{L^2_z}^2.
\end{align*}
Next, we notice that when $0 \le j \le 3$, the only boundary contribution is at $j = 1$, for which we have by a standard Trace inequality: 
\begin{align*}
|J_2| \lesssim &\| \bar{W}_z \|_{L^\infty}|u_z(s, 0)|^2 \lesssim \| \bar{W}_z \|_{L^\infty} (\| u^{(1)} \chi_{O,1}^{-} \|_{L^2_z}^2 + \| u^{(2)} \chi_{O,2}^{-} \|_{L^2_z}^2)
\end{align*}
The terms $J_4, J_5, J_8, J_9$ are treated together as follows: 
\begin{align*}
|J_4| + |J_5| + |J_8| + |J_9| \lesssim & \| \bar{W}^{(3 + j)} \|_{L^\infty} \| u^{(j)} \chi_{O,j}^{-} \|_{L^2_z} \sum_{j' = 0}^{j+1} \| \psi^{(j')} \chi_{O,j} \|_{L^2_z} \\
\lesssim & \| \bar{W}^{(3 + j)} \|_{L^\infty} \| u^{(j)} \chi_{O,j}^{-} \|_{L^2_z} ( \| \psi \chi_{O,0}^- \|_{L^2_z} +  \sum_{j' = 0}^{j} \| u^{(j')} \chi_{O,j' } \|_{L^2_z}),
\end{align*}
where we have used that $\chi_{O,j} \le \chi_{O,j'}$ if $j' \le j$. For the terms where a derivative hits $\chi_{O,j}$, we treat separately the cases when $j \ge 1$ and when $j = 0$.  First, we have 
\begin{align*}
|J_3| \lesssim & \bold{1}_{j \ge 1} \| \bar{W}_z \|_{L^\infty} \| u^{(j)} \chi_{O,j}^{-} \|_{L^2_z} \| \sqrt{|\bar{W}|} \p_z u^{(j-1)} \chi_{O,j-1} \|_{L^2_z}.
\end{align*}
Next, we have upon using that $\p_z \chi_{O,j}^-$ is supported in the set $\{ \chi_{O,j-1} = 1\}$, 
\begin{align*}
|J_6| \lesssim \bold{1}_{j \ge 1} \| \bar{W}^{(1 + j)} \|_{L^\infty} \| u^{(j)} \chi_{O,j}^{-} \|_{L^2_z} \sum_{j' = 0}^{j-1} \| u^{(j')} \chi_{O,j'}^{-} \|_{L^2_z},
\end{align*}
and finally, we have 
\begin{align*}
|J_{10}| \lesssim & \bold{1}_{j \ge 1} \| \bar{W}^{(2 + j)} \|_{L^\infty} \| u^{(j)} \chi_{O,j}^{-} \|_{L^2_z} (\| \psi \chi_{O,0}^- \|_{L^2_z} + \sum_{j' = 0}^{j-1} \| u^{(j')} \chi_{O,j'}^- \|_{L^2_z}) \\
& + \bold{1}_{j = 0} \| \bar{W}_{zz} \|_{L^\infty} \| u \chi_{O,0}^- \|_{L^2_z} \| \psi \chi_I \|_{L^\infty_z}
\end{align*}
We finally notice that $J_7 = J_{11} = 0$ for $0 \le j \le 3$. 

We now consolidate all of the bounds on $J_i$ as follows: 
\begin{align*}
\sum_{i = 1}^{11} |J_i(s)| \lesssim \| \psi \chi_{O,0}^- \|_{L^\infty_z}^2 + \sum_{j' = 0}^j \| u^{(j')} \chi^-_{O,j'} \|_{L^2_z}^2 + \bold{1}_{j \ge 1} \| \sqrt{|\bar{W}|} \p_z u^{(j-1)} \chi_{O,j-1}^- \|_{L^2_z}^2 + \bold{1}_{j = 0} \| \psi \chi_I \|_{L^\infty_z}^2.  
\end{align*}
Integrating this in $s$, and invoking \eqref{ujell2}, yields 
\begin{align} \label{mpsid:1}
\int_1^{1 + \bar{L}}\sum_{i = 1}^{11} |J_i(s)| \ud s \lesssim \bar{L}^{\frac13} \| \psi \|_{\text{von-Mise}_-}^2 + \bold{1}_{1 \le j \le 3} \| \psi \|_{VM^{j-1}_-} + \bold{1}_{j = 0} \bar{L} \| \psi \chi_I \|_{L^\infty}^2.
\end{align}
\vspace{2 mm}

\noindent \textit{Step 3: Estimation of $\mathrm{Err}_{i; j}(s)$} We begin with $\text{Err}_{1;j}$. This is the primary term that talks to the previous order of derivative, and in the case of $j = 0$, talks to the interior norm. We first treat the case when $j \ge 1$. We have upon using that $|\bar{W}| \gtrsim 1$ in the support of $\p_z \chi_{O,j}^-$, 
\begin{align*}
|\text{Err}_{1;j}|\bold{1}_{j \ge 1}  \lesssim &\bold{1}_{j \ge 1} \| \sqrt{|\bar{W}|} \p_z u^{(j)} \chi_{O,j}^- \|_{L^2_z} \| \p_z U^{(j-1)} \chi_{O,j-1}^- \|_{L^2_z}  \\
\le & \frac{1}{4}\mathcal{D}^{(j)}_{O,-}(s) +  \bold{1}_{j \ge 1} C\| \p_z U^{(j-1)} \chi_{O,j-1}^- \|_{L^2_z}^2,
\end{align*}
where we have used Young's inequality for products. 

For the $j \ge 1$ case, we have invoked the bound \eqref{deri:1}. For the $j = 0$ case, we use that $\chi_{O,0}'$ is supported in the set $\{\chi_I = 1\}$, and integrate by parts as follows: 
\begin{align*}
\text{Err}_{1; 0} = & - 2 \int_0^1 \p_z u U \chi_{O,0}^- \p_z \chi_{O,0}^- \ud z= - 2 \int_0^1 \p_z u \{ \bar{W} u - \bar{W}_z \psi \} \chi_{O,0}^- \p_z \chi_{O,0}^- \ud z \\
= &- \int_0^1 \bar{W}_z u^2 \chi_{O,0}^- \p_z \chi_{O,0}^- \ud z +  \int_0^1 \bar{W} u^2 \p_z \{ \chi_{O,0}^- \p_z \chi_{O,0}^- \} \ud z - \int_0^1 \bar{W}_{zz} u \psi \chi_{O,0}^- \p_z \chi_{O,0}^- \ud z \\
& -  \int_0^1 \bar{W}_{z} u \psi \p_z \{ \chi_{O,0}^- \p_z \chi_{O,0}^- \} \ud z, 
\end{align*}
after which we get $|\text{Err}_{1;0}| \lesssim \| \bar{W}_{zz} \|_{L^\infty}( \| u \chi_I \|_{L^2}^2 + \| \psi \chi_I \|_{L^2}^2)$.

Next, we estimate 
\begin{align*}
|\text{Err}_{2;j}| \lesssim & \Big\| \frac{\bold{b}_1}{\bar{W}} \chi_{O}^- \Big\|_{L^\infty} \| \sqrt{\bar{|W|}} \p_z u^{(j)} \chi_{O,j}^- \|_{L^2} \| U^{(j)} \chi_{O,j}^- \|_{L^2_z} \\
\le & \frac{1}{4}\| \sqrt{\bar{|W|}} \p_z u^{(j)} \chi_{O,j}^- \|_{L^2}^2 + C\| U^{(j)} \chi_{O,j}^- \|_{L^2_z}^2 \\
\le & \frac{1}{4}\mathcal{D}^{(j)}_{O,-}(s) + C\| U^{(j)} \chi_{O,j}^- \|_{L^2_z}^2,
\end{align*}
where we have used again Young's inequality for products, as well as the assumptions \eqref{eng:1} -- \eqref{eng:2}.

Next, we successively have
\begin{align*}
|\text{Err}_{3;j}| \lesssim & \| \bold{b}_2 \|_{L^\infty} \| u^{(j)} \chi_{O,j}^- \|_{L^2} \| U^{(j)} \chi_{O,j}^- \|_{L^2}, \\
|\text{Err}_{4;j}| \lesssim & \| \bold{b}_3 \|_{L^\infty} \| \psi^{(j)} \chi_{O,j}^- \|_{L^2} \| U^{(j)} \chi_{O,j}^- \|_{L^2},
\end{align*}
where we have used the bootstrap assumptions \eqref{eng:3} -- \eqref{eng:4}. Next, we have the linear commutators arising from $LC_1, LC_2, LC_3$. We have (upon using that $\chi_{O,j} \le \chi_{O,j'}$ for $j' \le j$)
\begin{align*}
|\text{Err}_{5; j}| \lesssim & (\sum_{j' = 0}^j \| \p_z^{j'} \bold{b}_1 \|_{L^\infty}) (\sum_{j' = 0}^j \| u^{(j')} \chi_{O,j'}^- \|_{L^2_z}) \| U^{(j)} \chi_{O,j}^- \|_{L^2_z}, \\
|\text{Err}_{6; j}| \lesssim &  (\sum_{j' = 0}^j \| \p_z^{j'} \bold{b}_2 \|_{L^\infty}) (\sum_{j' = 0}^j \| u^{(j')} \chi_{O,j'}^- \|_{L^2_z}) \| U^{(j)} \chi_{O,j}^- \|_{L^2_z}, \\
|\text{Err}_{7; j}| \lesssim &  (\sum_{j' = 0}^j \| \p_z^{j'} \bold{b}_3 \|_{L^\infty}) (\sum_{j' = 0}^j \| \psi^{(j')} \chi_{O,j'}^- \|_{L^2_z}) \| U^{(j)} \chi_{O,j}^- \|_{L^2_z}.
\end{align*}
For the final $\text{Err}_{8,j}$ term, we estimate via Cauchy-Schwartz as follows:
\begin{align*}
|\text{Err}_{8; j}| \lesssim & \| \p_z^j \bold{R} \chi_{O,j}^- \|_{L^2_{z}}^2 + \| U^{(j)} \chi_{O,j}^- \|_{L^2_z}^2. 
\end{align*}

Consolidating the above estimates on $\text{Err}_{i; j}$ together, we have 
\begin{align} \n
\sum_{i = 1}^7 |\text{Err}_{i; j}(s)| \le & \frac{1}{2}\mathcal{D}^{(j)}_{O,-}(s) +C \Big( \sum_{j' = 0}^j \| u^{(j')} \chi_{O,j'}^- \|_{L^2_z}^2 + \sum_{j' = 0}^j \| \psi^{(j')} \chi_{O,j'}^- \|_{L^2_z}^2 + \| U^{(j)} \chi_{O,j}^- \|_{L^2_z}^2 \\
&  + \bold{1}_{j \ge 1} \| \p_z U^{(j-1)} \chi_{O,j-1}^- \|_{L^2_z} + \bold{1}_{j = 0} ( \| u \chi_I \|_{L^2}^2 + \| \psi \chi_I \|_{L^2}^2) + \| \p_z^j \bold{R} \chi_{O,j}^- \|_{L^2_{z}}^2\Big).
\end{align}
We now integrate both sides in $s$ and invoke \eqref{ujell2}, \eqref{outer:est:l2:0}, and \eqref{deri:1} to obtain 
\begin{align} \n
\int_1^{1 + \bar{L}} \sum_{i = 1}^7 |\text{Err}_{i; j}(s)| \ud s \le &  \frac{1}{2} \| \psi \|_{VM^j_-}^2 + \bold{1}_{1 \le j \le 3} C\| \psi \|_{VM^{j-1}_-}^2+ C \bar{L}^{\frac13} \| \psi \|_{\text{von-Mise}_-}^2 \\ \label{mpsid:2}
&+ \bold{1}_{j = 0} C\bar{L} ( \| u \chi_I \|_{L^\infty}^2 + \| \psi \chi_I \|_{L^\infty}^2) +C \| \p_z^j \bold{R} \chi_{O,j}^- \|_{L^2}^2.
\end{align}

\vspace{2 mm}

\noindent \textit{Step 4: Conclusion of the Proof} We combine the estimates \eqref{mpsid:1} and \eqref{mpsid:2} in order to achieve the estimate \eqref{vm:YES:2}. We then insert \eqref{vm:YES:2} into the right-hand side of \eqref{ocj:1}, establishing \eqref{vm:YES:1}. This concludes the proof of the lemma. 
\end{proof}

We now perform the same estimate on the domain $(1, \infty)$. As indicated earlier, here there will be two simplifications: the lack of degeneracy of $\bar{W}$ and not needing to deal with boundary conditions at $\{z = 0\}$. On the other hand, there are two new features: the presence of nonlocal contributions $\| \psi \chi_I \|_{L^\infty} + \| u \chi_I \|_{L^\infty}$ as well as weights of $\langle z \rangle^n$. A further (comparatively minor) difference is that we cannot put $\psi$ itself in $L^2_z$ (due to the unboundedness of the domain), but rather only in $L^\infty_z$. This explains why the lowest order term is often singled out in the forthcoming calculations. 

\begin{lemma}[Top Estimate]Under the assumptions of Proposition \ref{pro:vonMise}, for any $0 \le j \le 3$, for any $n \in \mathbb{N}$, 
\begin{align}\n
\| \psi \|_{VM^{j,n}_+}^2 \lesssim & \| \p_z^j \zeta_{\mathrm{Left}} \chi_{O,j}^+ \langle z \rangle^n \|_{L^2_z}^2 + \| \p_z^j \bold{R} \chi_{O,j}^+  \langle z \rangle^n \|_{L^2_{sz}}^2 + \bar{L}(\| \psi \chi_I \|_{L^\infty}^2 + \| u \chi_I \|_{L^\infty}^2) \\  \label{vm:YES:1:plus}
& + \bar{L} \| \psi \|_{\mathrm{von-Mise}_{+,n}}^2 + \bold{1}_{1 \le j \le 3} \| \psi \|_{VM^{j-1, n}_+}^2.
\end{align}
\end{lemma}
\begin{proof} We apply the multiplier $U^{(j)} ( \chi_{O,j}^{+})^2 \langle z \rangle^{2n}$. We again first provide an overview of the strategy of the proof, even though it is essentially the same as the previous lemma. 

\vspace{2 mm} 

\noindent \textit{Step 0: Strategy of the Proof} We will obtain an identity of the form: 
\begin{align} \label{form:1:top}
\frac{\p_s}{2} \mathcal{E}^{(j, n)}_{O,+}(s) + \mathcal{D}^{(j, n)}_{O,+}(s) = \text{Err}(s). 
\end{align}
Integrating the above equality forwards from $s = 1$ to $s = s_\ast$ produces 
\begin{align} \n
 \mathcal{E}^{(j, n)}_{O,-}(s_\ast) + \int_{1}^{s_\ast}  \mathcal{D}^{(j, n)}_{O,+}(s) \ud s = & \mathcal{E}^{(j,n)}_{O,+}(1) +  \int_{1}^{s} \text{Err}(s) \ud s \\
 \le &  \mathcal{E}^{(j,n)}_{O,+}(1)  + \int_1^{1 + \bar{L}} |\text{Err}(s)| \ud s. 
\end{align}
Taking the supremum in $s_\ast$, we obtain 
\begin{align} \label{ocj:2}
\| \psi \|_{VM^{j,n}_+}^2 \le \mathcal{E}^{(j,n)}_{O,+}(1 + \bar{L}) +  \int_1^{1 + \bar{L}} |\text{Err}(s)| \ud s \le \| \p_z^j \zeta_{\text{Left}} \chi_{O,j}^+ \langle z \rangle^n \|_{L^2_z}^2 + \int_1^{1 + \bar{L}} |\text{Err}(s)| \ud s.
\end{align}
The proof of inequality \eqref{vm:YES:1:plus} will therefore follow upon establishing the following bound:
\begin{align} \n
\int_1^{1 + \bar{L}} |\text{Err}(s)| \ud s \le & \frac{1}{2} \| \psi \|_{VM^{j,n}_+}^2  +C\Big( \| \p_z^j \bold{R} \chi_{O,j}^+ \langle z \rangle^n \|_{L^2_{sz}}^2 + \bar{L}(\| \psi \chi_I \|_{L^\infty}^2 + \| u \chi_I \|_{L^\infty}^2) \\  \label{vm:YES:2:plus}
& + \bar{L} \| \psi \|_{\text{von-Mise}_{+,n}}^2 + \bold{1}_{1 \le j \le 3} \| \psi \|_{VM^{j-1, n}_+}^2\Big),
\end{align}
where the $ \frac{1}{2} \| \psi \|_{VM^{j,n}_+}^2$ will be absorbed to the left-hand side of \eqref{ocj:2}. 

We turn now to determining the specific form of $\text{Err}(s)$, and then performing the corresponding estimates. 

\vspace{2 mm}

\noindent \textit{Step 1: Energy Identity} This produces the following identity 
\begin{align} \n
&\frac{\p_s}{2} \int_1^{\infty} |U^{(j)}|^2 (\chi_{O,j}^{+})^2 \langle z \rangle^{2n} \ud z + \int_1^{\infty} \p_z u^{(j)} \p_z U^{(j)} (\chi_{O,j}^{+})^2 \langle z \rangle^{2n} \ud z \\ \n
=& - 2 \int_{1}^{\infty} \p_z u^{(j)} U^{(j)} \chi_{O,j}^{+} \p_z \chi_{O,j}^{+} \langle z \rangle^{2n} \ud z - \int_1^{\infty} \bold{b}_1 \p_z u^{(j)} U^{(j)} (\chi_{O,j}^{+})^2 \langle z \rangle^{2n} \ud z  \\ \n
&- \int_1^{\infty} \bold{b}_2 u^{(j)} U^{(j)} (\chi_{O,j}^{+})^2 \langle z \rangle^{2n} \ud z  -  \int_1^{\infty} \bold{b}_3 \psi^{(j)} U^{(j)} (\chi_{O,j}^{+})^2\langle z \rangle^{2n} \ud z \\ \n
&-  \sum_{i = 1}^3 \int_1^{\infty} LC_{i;j} U^{(j)} (\chi_{O,j}^{+})^2 \langle z \rangle^{2n} \ud z + \int_1^{\infty} \p_z^j \bold{R} U^{(j)} (\chi_{O,j}^{+})^2 \langle z \rangle^{2n} \ud z \\ \n
& - 2n \int_1^\infty \p_z u^{(j)} U^{(j)} \langle z \rangle^{2n-1} (\chi_{O,j}^+)^2 \ud z \\ \label{vmEE:1:top}
=:& \sum_{i = 1}^9 Err_{i; j}.
\end{align}
Compared to \eqref{vmEE:1}, we notice the presence of an extra commutator term, namely $\text{Err}_{9;j}$. This new error term will  be estimated below in \eqref{new:com}. 

Just as before, we need to expand the diffusion term on the left-hand side, which analogously produces 
\begin{align} \n
&\int_1^{\infty} \p_z u^{(j)} \p_z U^{(j)} (\chi_{O,j}^{+})^2 \langle z \rangle^{2n} \ud z \\ \n
= & \int_1^{\infty} \bar{W} |\p_z u^{(j)}|^2 (\chi_{O,j}^{+})^2 \langle z \rangle^{2n} \ud z + \sum_{j' = 0}^j \binom{j+1}{j'} \int_1^\infty \p_z u^{(j)} \p_z^{j+1 - j'} \bar{W} u^{(j')} (\chi_{O,j}^{+})^2 \langle z \rangle^{2n} \ud z \\ \n
& - \sum_{j' = 1}^{j+1} \binom{j+1}{j'} \int_1^\infty \p_z u^{(j)} \p_z^{j+1-j'} \bar{W}_z \psi^{(j')}(\chi_{O,j}^{+})^2 \langle z \rangle^{2n} \ud z - \int_1^\infty \p_z u^{(j)} \p_z^{j+1} \bar{W}_z \psi (\chi_{O,j}^{+})^2 \langle z \rangle^{2n} \ud z \\
=:& \int_1^{\infty} \bar{W} |\p_z u^{(j)}|^2 (\chi_{O,j}^{+})^2 \langle z \rangle^{2n} \ud z - ( J_1 + J_2 + J_3)
\end{align}

Pairing this identity with \eqref{vmEE:1:top}, we have obtained our full energy identity which reads: 
\begin{align} \label{merced:1}
\frac{\p_s}{2} \int_1^{\infty} |U^{(j)}|^2 (\chi_{O,j}^{+})^2 \langle z \rangle^{2n} \ud z + \int_1^{\infty} \bar{W} |\p_z u^{(j)}|^2 (\chi_{O,j}^{+})^2 \langle z \rangle^{2n} \ud z =  \sum_{i = 1}^9 \mathrm{Err}_{i; j} + \sum_{i = 1}^3 J_i =: \text{Err}(s),
\end{align}
which is of the form \eqref{form:1:top}.

\vspace{2 mm}

\noindent \textit{Step 2: Diffusive Error Terms, $J_i(s)$} There are far fewer terms compared to \eqref{yes:1}: we do not need to integrate by parts due to the nonvanishing of $\bar{W}$ on the ``top" domain. Indeed, we may estimate these error terms as follows: 
\begin{align*}
|J_1| + |J_2| \lesssim & (\sum_{j' = 0}^{j+1} \| \p_z^{j'} \bar{W} \|_{L^\infty}) \| \p_z u^{(j)} \chi_{O,j}^+ \langle z \rangle^n \|_{L^2} \sum_{j' = 0}^j \| u^{(j')} \chi_{O,j'}^+ \langle z \rangle^n \|_{L^2_z}, \\
|J_3| \lesssim & \| \p_z^{j+2} \bar{W} \langle z \rangle^{n} \|_{L^2} \| \p_z u^{(j)} \chi_{O,j}^+ \langle z \rangle^n \|_{L^2} \| \psi \chi_{O,0}^+ \|_{L^\infty_z} 
\end{align*}
Since $\bar{W}$ is bounded below on the support of $\chi_{O,j}^+$, we may apply Young's inequality for products to obtain 
\begin{align} \label{nk:1}
\sum_{i = 1}^3 |J_i(s)| \le \frac{1}{4} \| \sqrt{\bar{W}} \p_z u^{(j)} \chi_{O,j}^+ \langle z \rangle^n \|_{L^2_z}^2 + C(\| \psi \chi_{O,0}^+ \|_{L^\infty_z}^2 +\sum_{j' = 0}^j \| u^{(j')} \chi_{O,j'}^+ \langle z \rangle^n \|_{L^2_z}^2).
\end{align}
\vspace{2 mm}

\noindent \textit{Step 3: Estimation of RHS of \eqref{vmEE:1:top}} We now turn to estimating the terms $\text{Err}_{i;j}$.The heuristic will be as follows: our main motivation for integration by parts is to cancel out linearly growing coefficients, as opposed to loss of derivatives.

For the first term, $\text{Err}_{1;j}$, we note that on the support of $\p_z \chi_{O,j}^+$, the weight $\langle z \rangle^n$ is bounded and therefore plays essentially no role. We subsequently estimate as follows: 
\begin{align*}
|\text{Err}_{1;j}| \lesssim & \bold{1}_{j \ge 1} \| \p_z u^{(j)} \chi_{O,j}^+ \|_{L^2_z} \| \p_z U^{(j-1)} \chi_{O,j-1}^+ \|_{L^2_z} + \bold{1}_{j = 0} \| \p_z u^{(j)} \chi_{O,j}^+ \|_{L^2_z} \| U \chi_I \|_{L^2_z} \\
\lesssim &  \bold{1}_{j \ge 1} \| \p_z u^{(j)} \chi_{O,j}^+ \|_{L^2_z} \| \p_z U^{(j-1)} \chi_{O,j-1}^+ \|_{L^2_z} + \bold{1}_{j = 0} \| \p_z u^{(j)} \chi_{O,j}^+ \|_{L^2_z} (\| \psi \chi_I \|_{L^\infty} + \| u \chi_I \|_{L^\infty}) \\
\le & \frac{1}{100}  \| \sqrt{\bar{W}} \p_z u^{(j)} \chi_{O,j}^+ \langle z \rangle^n \|_{L^2_z}^2 + C\bold{1}_{j \ge 1} \| \p_z U^{(j-1)} \chi_{O,j-1}^+ \|_{L^2_z}^2 + C\bold{1}_{j = 0}(\| \psi \chi_I \|_{L^\infty}^2 + \| u \chi_I \|_{L^\infty}^2),
\end{align*}
upon again noting that $\bar{W}$ is bounded uniformly below and subsequently using Young's inequality for products. 

We now address the error term $\text{Err}_{2;j}$ which requires an integration by parts in $z$ due to the linear growth in $z$ of the coefficient $\bold{b}_1$. We use the identity \eqref{derived:1} with $j$ being replaced by $j-1$
to write 
\begin{align*}
\text{Err}_{2;j} =& \int_1^{\infty} \bold{b}_1 \bar{W} \p_z u^{(j)}  u^{(j)}  (\chi_{O,j}^{+})^2 \langle z \rangle^{2n} \ud z + \sum_{j' = 0}^{j-1} \binom{j}{j'} \int_1^{\infty}  \bar{W}^{(j-j')} \bold{b}_1  \p_z u^{(j)}  u^{(j')}  (\chi_{O,j}^{+})^2 \langle z \rangle^{2n} \ud z \\
& - \sum_{j' = 1}^{j} \binom{j}{j'}  \int_1^{\infty} \bold{b}_1 \bar{W}_z^{(j-j')} \p_z u^{(j)}  \psi^{(j')}   (\chi_{O,j}^{+})^2 \langle z \rangle^{2n} \ud z -  \int_1^{\infty} \bold{b}_1\bar{W}_z^{(j)} \psi  \p_z u^{(j)}    (\chi_{O,j}^{+})^2 \langle z \rangle^{2n} \ud z \\
= & - \frac12 \int_1^{\infty} \p_z \bold{b}_1 \bar{W} |u^{(j)}|^2  (\chi_{O,j}^{+})^2 \langle z \rangle^{2n} \ud z - \frac12 \int_1^{\infty} \bold{b}_1 \bar{W}_z |u^{(j)}|^2  (\chi_{O,j}^{+})^2 \langle z \rangle^{2n} \ud z\\
&   -n \int_1^{\infty} \bold{b}_1 \bar{W} |u^{(j)}|^2  (\chi_{O,j}^{+})^2  \langle z \rangle^{2n-1} \ud z - \int_1^{\infty} \bold{b}_1 \bar{W} |u^{(j)}|^2  \chi_{O,j}^{+} \p_z \chi_{O,j}^{+} \langle z \rangle^{2n} \ud z\\
&+ \sum_{j' = 0}^{j-1} \binom{j}{j'} \int_1^{\infty}  \bar{W}^{(j-j')} \bold{b}_1  \p_z u^{(j)}  u^{(j')}  (\chi_{O,j}^{+})^2 \langle z \rangle^{2n} \ud z \\
& - \sum_{j' = 1}^{j} \binom{j}{j'}  \int_1^{\infty} \bold{b}_1 \bar{W}_z^{(j-j')} \p_z u^{(j)}  \psi^{(j')}   (\chi_{O,j}^{+})^2 \langle z \rangle^{2n} \ud z -  \int_1^{\infty} \bold{b}_1\bar{W}_z^{(j)} \psi  \p_z u^{(j)}    (\chi_{O,j}^{+})^2 \langle z \rangle^{2n} \ud z \\
=:& \sum_{i = 1}^7 \text{Err}_{2, i ;j}.
\end{align*}
where above we have integrated by parts the pure transport term $\int_1^{\infty} \bold{b}_1 \bar{W} \p_z u^{(j)}  u^{(j)}  (\chi_{O,j}^{+})^2 \langle z \rangle^{2n} \ud z$ to cancel the linear growth exhibited by $\bold{b}_1$. We now estimate each of these terms successively. 
\begin{align*}
\sum_{i = 1}^3 |\text{Err}_{2, i ;j}| \lesssim & \Big( \| \p_z \bold{b}_1 \|_{L^\infty} + \| \frac{\bold{b}_1}{\langle z \rangle} \|_{L^\infty} \Big)\Big( \| \bar{W} \|_{L^\infty} + \| \bar{W}_z \langle z \rangle\|_{L^\infty} \Big) \| u^{(j)} \chi_{O,j}^+ \langle z \rangle^n \|_{L^2_z}^2,
\end{align*}
where we have invoked \eqref{eng:1} -- \eqref{eng:2}. Next, we have upon noting the support of $\p_z \chi_{O,j}^+$ is for bounded values of $z$ (so the weights $\langle z \rangle^n$ are in particular uniformly bounded), 
\begin{align*}
 |\text{Err}_{2, 4 ;j}| \lesssim & \bold{1}_{j \ge 1} \| \frac{\bold{b}_1}{\langle z \rangle} \|_{L^\infty} \| u^{(j)} \chi_{O,j}^+ \langle z \rangle^n \|_{L^2_z} \| \p_z u^{(j-1)} \chi_{O,j-1}^+ \langle z \rangle^n \|_{L^2_z}  + \bold{1}_{j = 0} \| \frac{\bold{b}_1}{\langle z \rangle} \|_{L^\infty} \| u \chi_I \|_{L^\infty_z}^2 \\
 \lesssim & \| u^{(j)} \chi_{O,j}^+ \langle z \rangle^n \|_{L^2_z}^2 +  \bold{1}_{j \ge 1} \mathcal{D}^{(j-1, n)}_{O,+}(s) + \bold{1}_{j =0}  \| u \chi_I \|_{L^\infty_z}^2.
\end{align*}
Next, upon again using the boundedness below of $\bar{W}$ on the support of $\chi_{O,j}^+$ and Young's inequality for products, 
\begin{align*}
\sum_{i = 5}^6 |\text{Err}_{2, i ;j}| \lesssim &  \| \frac{\bold{b}_1}{\langle z \rangle} \|_{L^\infty} (\sum_{j' = 1}^j \| \bar{W}^{(j')} \langle z \rangle \|_{L^\infty}) \| \p_z u^{(j)} \chi_{O,j}^+ \langle z \rangle^n \|_{L^2_z} (\sum_{j' = 0}^{j-1} \| u^{(j')} \chi_{O,j'} \langle z \rangle^n \|_{L^2_z}) \\
\le & \frac{1}{100} \| \sqrt{\bar{W}} \p_z u^{(j)} \chi_{O,j}^+ \langle z \rangle^n \|_{L^2_z}^2 +C \sum_{j' = 0}^{j-1} \| u^{(j')} \chi_{O,j'} \langle z \rangle^n \|_{L^2_z}^2
\end{align*}
Lastly from this set of terms, 
\begin{align*}
|\text{Err}_{2, 7 ;j}| \lesssim &\| \frac{\bold{b}_1}{\langle z \rangle} \|_{L^\infty} \| \bar{W}^{(j)}_z \langle z \rangle^{n+1} \|_{L^2} \| \psi \chi_{O,0}^+ \|_{L^\infty_z} \| \p_z u^{(j)} \chi_{O,j}^+ \langle z \rangle^n \|_{L^2_z} \\
\le & \frac{1}{100} \| \sqrt{\bar{W}} \p_z u^{(j)} \chi_{O,j}^+ \langle z \rangle^n \|_{L^2_z}^2 +C \| \psi \chi_{O,0}^+ \|_{L^\infty_z}^2
\end{align*}
The treatment of $\text{Err}_{2;j}$ is now concluded. Next, we have 
\begin{align*}
|\text{Err}_{3 ;j}| \lesssim & \| \bold{b}_2 \|_{L^\infty} \| u^{(j)} \chi_{O,j}^+ \langle z \rangle^n \|_{L^2_z} \| U^{(j)} \chi_{O,j}^+ \langle z \rangle^n \|_{L^2_z}, \\
|\text{Err}_{4 ;j}| \lesssim & \bold{1}_{j \ge 1} \| \bold{b}_3 \|_{L^\infty} \| u^{(j-1)} \chi_{O,j-1}^+ \langle z \rangle^n \|_{L^2_z} \| U^{(j)} \chi_{O,j}^+ \langle z \rangle^n \|_{L^2_z} \\
& + \bold{1}_{j = 0} \| \bold{b}_3 \langle z \rangle^{n+1} \|_{L^\infty} \| \psi \chi_{O,0}^+ \|_{L^\infty} \| U \chi_{O,0}^+ \langle z \rangle^n \|_{L^2_z} 
\end{align*}
Next, we have 
\begin{align*}
|\text{Err}_{5;j}| \lesssim & (\sum_{j' = 1}^j \| \p_z^{j'} \bold{b}_1 \|_{L^\infty}) (\sum_{j' = 1}^j \| u^{(j')} \chi_{O,j'}^+ \langle z \rangle^n \|_{L^2_z}) \| U^{(j)} \chi_{O,j}^+ \langle z \rangle^n \|_{L^2}, \\
|\text{Err}_{6;j}| \lesssim &(\sum_{j' = 1}^j \| \p_z^{j'} \bold{b}_2 \|_{L^\infty}) (\sum_{j' = 0}^{j-1} \| u^{(j')} \chi_{O,j'}^+ \langle z \rangle^n \|_{L^2_z}) \| U^{(j)} \chi_{O,j}^+ \langle z \rangle^n \|_{L^2}, \\
|\text{Err}_{7;j}| \lesssim &(\sum_{j' = 1}^{j-1} \| \p_z^{j'} \bold{b}_3 \|_{L^\infty})(\sum_{j' = 0}^{j-2} \| u^{(j')} \chi_{O,j'}^+ \langle z \rangle^n \|_{L^2_z}) + \| \bold{b}_3^{(j)} \langle z \rangle^{n+1} \|_{L^\infty} \| \psi \chi_{O,0}^+ \|_{L^\infty} \| U^{(j)} \chi_{O,j}^+ \langle z \rangle^n \|_{L^2},
\end{align*}
which are acceptable contributions due to the assumptions \eqref{eng:3} -- \eqref{eng:4}. Next, we have our source term, which we estimate simply by Cauchy-Schwartz and Young's inequality: 
\begin{align}
|\text{Err}_{8;j}| \lesssim & \| \p_z^j \bold{R} \chi_{O,j}^+ \langle z \rangle^n \|_{L^2_{z}}^2 + \| U^{(j)} \chi_{O,j}^+ \langle z \rangle^n \|_{L^2_{z}}^2.
\end{align}
We now move to our final term, 
\begin{align} \label{new:com}
|\text{Err}_{9;j}| \lesssim &\| \p_z u^{(j)} \chi_{O,j}^+ \langle z \rangle^n \|_{L^2_z} \| U^{(j)} \chi_{O,j}^+ \langle z \rangle^n \|_{L^2_z} \\ \n
\le &  \frac{1}{100} \| \sqrt{\bar{W}} \p_z u^{(j)} \chi_{O,j}^+ \langle z \rangle^n \|_{L^2_z}^2 +C\| U^{(j)} \chi_{O,j}^+ \langle z \rangle^n \|_{L^2_z}^2. 
\end{align}

We now bring together all the bounds of the $\text{Err}_{i,j}$ terms in the following manner: 
\begin{align} \n
\sum_{i = 1}^9 |\text{Err}_{i,j}(s)| \le & \frac{1}{4}  \| \sqrt{\bar{W}} \p_z u^{(j)} \chi_{O,j}^+ \langle z \rangle^n \|_{L^2_z}^2 \\ \n
&+ C\bold{1}_{j \ge 1} (\| \p_z U^{(j-1)} \chi_{O,j-1}^+ \|_{L^2_z}^2 + \mathcal{D}^{(j-1,n)}_{O,+}(s))  \\ \n
&+ C\bold{1}_{j = 0}(\| \psi \chi_I \|_{L^\infty}^2 + \| u \chi_I \|_{L^\infty}^2) \\ \n
&+ C(\| U^{(j)} \chi_{O,j}^+ \langle z \rangle^n \|_{L^2_z}^2 +  \| \psi \chi_{O,0}^+ \|_{L^\infty}^2 + \sum_{j' = 1}^j \| u^{(j')} \chi_{O,j'}^+ \langle z \rangle^n \|_{L^2_z}^2) \\ \label{nk:2}
& + C \| \p_z^j \bold{R} \chi_{O,j}^+ \langle z \rangle^n \|_{L^2_{z}}^2.
\end{align}
Above, we have grouped terms with respect to the eventual right-hand side of \eqref{vm:YES:2:plus}.

\vspace{2 mm}

\noindent \textit{Step 4: Conclusion of Proof} We bring together bounds \eqref{nk:1} and \eqref{nk:2} (along with the definition of $\text{Err}(s)$ from \eqref{merced:1}) to obtain the consolidated estimate 
\begin{align} \n
|\text{Err}(s)| \le &  \frac{1}{2}  \| \sqrt{\bar{W}} \p_z u^{(j)} \chi_{O,j}^+ \langle z \rangle^n \|_{L^2_z}^2 \\ \n
&+ C\bold{1}_{j \ge 1} (\| \p_z U^{(j-1)} \chi_{O,j-1}^+ \|_{L^2_z}^2 + \mathcal{D}^{(j-1,n)}_{O,+}(s))  \\ \n
&+ C\bold{1}_{j = 0}(\| \psi \chi_I \|_{L^\infty}^2 + \| u \chi_I \|_{L^\infty}^2) \\ \n
&+ C(\| U^{(j)} \chi_{O,j}^+ \langle z \rangle^n \|_{L^2_z}^2 +  \| \psi \chi_{O,0}^+ \|_{L^\infty}^2 + \sum_{j' = 1}^j \| u^{(j')} \chi_{O,j'}^+ \langle z \rangle^n \|_{L^2_z}^2) \\ \label{nk:3}
& + C \| \p_z^j \bold{R} \chi_{O,j}^+ \langle z \rangle^n \|_{L^2_{z}}^2
\end{align}
We integrate both sides in $s$ and invoke the bounds \eqref{deri:2}, \eqref{md:2}, and \eqref{md:4}, as well as the definition of the energy, \eqref{canyou}, in order to establish exactly \eqref{vm:YES:2:plus}. We then insert \eqref{vm:YES:2:plus} into the right-hand side of \eqref{ocj:2}. Upon absorbing the $\frac12 \| \psi \|_{VM^{j,n}_+}^2$ to the left-hand side, this exactly establishes \eqref{vm:YES:1:plus}. This concludes the proof of the lemma.  
\end{proof}

\subsection{Proof of Proposition \ref{pro:vonMise}}

We now bring together the previously established lemmas in order to prove the main proposition of this section. 
\begin{proof}[Proof of Proposition \ref{pro:vonMise}] To make notations convenient, we set  $\mathrm{RHS}_{VM}$ to stand for the right-hand side of the desired estimate, \eqref{mainvonM}. We recall the bounds \eqref{vm:YES:1} and \eqref{vm:YES:1:plus} in the case $j = 0$, which imply: 
\begin{align} 
\| \psi \|_{VM^{0}_-}^2 \lesssim & \mathrm{RHS}_{VM}+ \bar{L}^{\frac13} \| \psi \|_{\text{von-Mise}_-}^2, \\
\| \psi \|_{VM^{0,n}_+}^2 \lesssim & \mathrm{RHS}_{VM} + \bar{L} \| \psi \|_{\text{von-Mise}_{+,n}}^2,
\end{align}
and the general $j = 1, 2, 3$ case reads 
\begin{align} 
\| \psi \|_{VM^{j}_-}^2 \lesssim &\mathrm{RHS}_{VM}+ \bar{L}^{\frac13} \| \psi \|_{\text{von-Mise}_-}^2 +\bold{1}_{1 \le j \le 3}\| \psi \|_{VM^{j-1}_-}^2 , \\
\| \psi \|_{VM^{j,n}_+}^2 \lesssim &\mathrm{RHS}_{VM} + \bar{L} \| \psi \|_{\text{von-Mise}_{+,n}}^2 + \bold{1}_{1 \le j \le 3}\| \psi \|_{VM^{j-1,n}_+}^2,
\end{align}
which clearly implies 
\begin{align} 
\| \psi \|_{VM^{j}_-}^2 \lesssim &\mathrm{RHS}_{VM}+ \bar{L}^{\frac13} \| \psi \|_{\text{von-Mise}_-}^2, \\
\| \psi \|_{VM^{j,n}_+}^2 \lesssim & \mathrm{RHS}_{VM} + \bar{L} \| \psi \|_{\text{von-Mise}_{+,n}}^2.
\end{align}
Upon summing the above inequalities together and subsequently summing over $\sum_{j = 0}^3$, recalling the definition \eqref{vonMisenorm}, we have 
\begin{align} 
\| \psi \|_{\text{von-Mise}_n}^2 \lesssim &\mathrm{RHS}_{VM} + \bar{L}^{\frac13} \| \psi \|_{\text{von-Mise}_n}^2,
\end{align}
which closes the estimate $\| \psi \|_{\text{von-Mise}_n}^2 \lesssim \mathrm{RHS}_{VM}$, as desired in \eqref{mainvonM}. This proves \eqref{mainvonM}. The proof of \eqref{lowvonM} is nearly identical. The proof of the proposition is complete.  
\end{proof}

\section{Analysis of $\mathcal{L}_{\text{Modulation}}[\bold{R}]$} \label{LpLqestsec}

Our main objective in this section is to study the $\mathcal{L}_{\text{Modulation}}[\bold{R}]$ operator, \eqref{merc:1lkj}. The strategy will be to introduce the various transformations we have employed in Section \ref{SectionCH} in order to arrive at the abstract systems $\mathcal{L}_{\mathrm{Airy}}$, $\mathcal{L}_{\mathrm{Crocco}}$, $\mathcal{L}_{\mathrm{von-Mise}}$, with the appropriate source terms. Referring to Figure \ref{org:tree}, we will perform the following steps in this section: 
\begin{figure}[h] 
\hspace{50 mm} \includegraphics[scale=0.3]{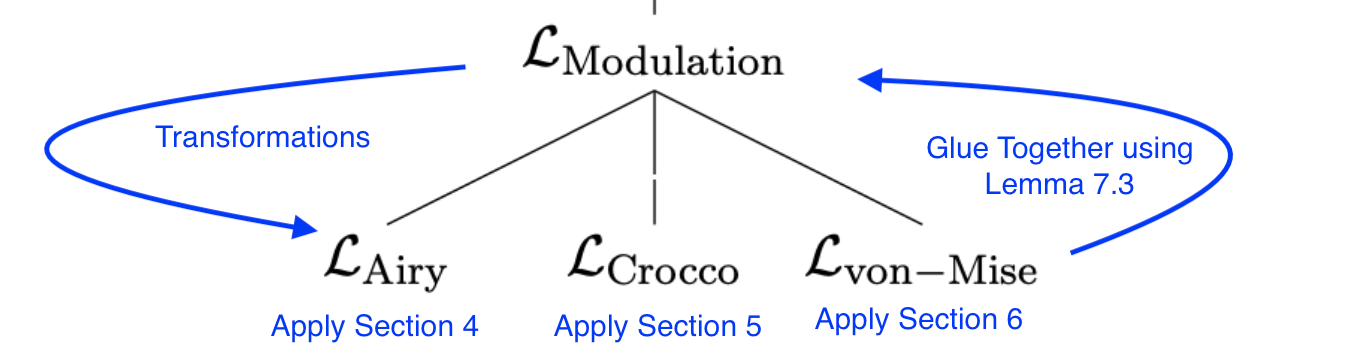}
\caption{Closing the Global Norm} \label{org:tree:mod}
\end{figure}

In particular, we have the following relations: 
\begin{align} \label{weds:1}
Z =& \frac{Y}{\bar{W}_z(s, 1)^{\frac23}}, \qquad Y = \bar{W}(s, z), \\ \label{weds:2}
\Phi(s, Z) =& \phi(s, Y) = \psi(s, z), \qquad \Omega(s, Z) = V(s, Y) = \frac{1}{\bar{W}_z^2}(\p_z u - \frac{\bar{W}_{zz}}{\bar{W}_z}u). 
\end{align}
We subsequently define $\Omega_I = \chi_I(Z) \Omega$, where the parameter $\delta$ appearing in the definition of $\chi_I$ will be chosen below (immediately preceding line \eqref{la:1}). For future use, we define 
\begin{align} \label{dgu}
\gamma_u := u(1, 1), \qquad \gamma_\psi := \psi(1, 1). 
\end{align}
With these transformations, we have the following systems on which we apply our $\mathcal{L}_{\mathrm{Airy}}$ and $\mathcal{L}_{\mathrm{Crocco}}$ analyses:  
\begin{align} \label{abs:67}
Z \p_s \Omega_I - \p_Z^2 \Omega_I =  &F_I, \\
Z \p_s \Omega_I - \tau_0 \Omega_I - \tau_1 \p_Z \Omega_I - (1 + \tau_2) \p_Z^2 \Omega_I =& S_I,
\end{align}
where we have defined  
\begin{align} \n
F_I := & F_{\text{Coupled}} + F_{\text{Low}}, \\ \n
F_{\text{Coupled}}:=&\tau_2 \p_Z^2 \Omega_I + \tau_1 \p_Z \Omega_I + \tau_0 \Omega_I + \chi_I (\tau_{-1} \p_Z \Phi + \tau_{-2} \Phi) \\ \label{defFCL}
&-( \frac{\chi''(\frac{Z}{\delta})}{\delta^2}  + \frac{\tau_1}{\delta} \chi'(\frac{Z}{\delta}) ) \Omega - \frac{2}{\delta} \chi'(\frac{Z}{\delta}) \Omega_Z, \\ \label{defFLOW}
F_{\text{Low}} := & \chi_I(Z)\frac{\p_z}{\bar{W}_z} \Big( \frac{\bold{R}}{\bar{W}_z} \Big)(s, z),
\end{align}
and
\begin{align}
S_I :=& S_{\text{Coupled}} + F_{\text{Low}}, \\ \label{defSCL}
S_{\text{Coupled}} := & \chi_I(Z) ( \tau_{-1} \p_Z \Phi + \tau_{-2} \Phi)   -( \frac{\chi''(\frac{Z}{\delta})}{\delta^2}  + \frac{\tau_1}{\delta} \chi'(\frac{Z}{\delta}) ) \Omega - \frac{2}{\delta} \chi'(\frac{Z}{\delta}) \Omega_Z.
\end{align}
We use the phrase ``Coupled" because $F_{\mathrm{Coupled}}$ and $S_{\mathrm{Coupled}}$ contain linear, solution-dependent terms originating from the left-hand side of $\mathcal{L}_{\text{Modulation}}[\bold{R}]$, \eqref{merc:1lkj}, and must therefore be coupled and closed as part of the estimate. On the other hand, we use ``Low" to refer to terms arising from the source term, $\bold{R}$, which is considered a given. Eventually when we apply this analysis $k$-by-$k$ in Section \ref{ksec}, these terms will arise from \eqref{defRKRK:temp}, \eqref{RkSemi}, and will be of order at most $k-1$.   

We recall the definitions of $\underline{\tau}_i$, $-2 \le i \le 1$ and $\tau_i$, $-2 \le i \le 2$ from \eqref{utau:def} and \eqref{tau:def}, respectively. In particular, we note that $\tau_i = \tau_i[\bar{W}, \vec{\bold{c}}]$. Similarly, to apply our von-Mise analysis, we recall the coefficients $\bold{b}_i$ defined through \eqref{tru:1}, and hence we regard $\bold{b}_i[\bar{W}, \vec{\bold{c}}]$. 

We now define several norms that we will use to close our linear analysis. The first notation is simply to save space, as the following two $L^\infty$ quantities frequently arise together:  
\begin{align} \label{S:norm}
\| \psi \|_{S} := \| \psi \|_{L^\infty_{sz}} + \| u \|_{L^\infty_{sz}}.
\end{align} 
Next, we define
\begin{align} \label{global:n}
\| \psi \|_{\text{Global}_n} := \| \psi \|_{\text{von-Mise}_n} + \| \psi \|_S + \| \Omega_I \|_{\text{Crocco}} +  \| \gamma_{\Omega}^{(1)} \|_{L^2_s}.
\end{align}
\begin{remark}The reason we call this a ``global" norm is because it glues together: (1) trace information (at the interface) encoded by the Trace norm, (2) local (to the interface) information encoded by the Crocco norms, and (3) far-field (from the interface) behavior encoded by the von-Mise norms. 
\end{remark}
The main proposition of this section quantifies control over $\| \psi \|_{\text{Global}_n}$ in terms of norms on the right-hand side $\bold{R}$ and in terms of norms on the boundary data. 
\begin{proposition}\label{pro:global} Let $n \in \mathbb{N}$. Assume the background, $\bar{W}$, satisfies the following estimates: 
\begin{align} \label{sushiL1L2}
&\sum_{k = 0}^1 \| \p_s^k \bar{W} \|_{L^\infty} + \sum_{k = 0}^1 \sum_{j = 1}^5 \| \p_s^k \p_z^j \bar{W} \langle z \rangle^{n+1} \|_{L^\infty} \lesssim 1, \\ \label{sushiL2}
&|\bar{W}| \ge c_0 z, \qquad 0 < z << 1, \\ \label{bhu:1}
&\inf_{1 < s < 1 + \bar{L}} |\bar{W}_z(s, 1)| > 0.
\end{align}
Assume the source term, $\bold{R}$, satisfies $\bold{R}|_{z = 0} = \p_z \bold{R}|_{z = 0} = 0$. Assume the coefficients $\bold{c}_i$, $i = 1, 2, 3$, satisfy the following: 
\begin{align} \label{eng:1L}
\bold{c}_1(s, 0) = & 0, \\ \label{eng:2L}
\| \frac{\bold{c}_1}{\langle z \rangle} \|_{L^\infty} +  \sum_{j = 1}^3 \| \p_z^j \bold{c}_1 \|_{L^\infty} \lesssim & 1, \\ \label{eng:3L}
\sum_{j = 0}^3 \| \p_z^j \bold{c}_2 \|_{L^\infty} \lesssim & 1, \\ \label{eng:4L}
\sum_{j = 0}^3 \| \p_z^j \bold{c}_3 \langle z \rangle^{n+1} \|_{L^\infty} \lesssim & 1.
\end{align}
Then the following estimate is valid for a solution $\psi$ to \eqref{merc:1lkj}
\begin{align} \n
\| \psi \|_{\mathrm{Global}_n} \lesssim & \sum_{j = 0}^3 \| \p_z^j \bold{R} \chi_{O,j} \langle z \rangle^n \|_{L^2_{sz}} + \sum_{j = 0}^1 \| \p_z^j \bold{R} \chi_I \|_{L^2_{sz}} + \| d(z) \p_z^2 \bold{R} \chi_I \|_{L^2_{sz}} \\ \n
& + \bar{L}^{\frac16} \sum_{i = 0}^2 \sum_{\iota \in \mathrm{Left, Right}}  \| \langle \rho \rangle \p_\rho^i \Xi_{\iota} \|_{L^2_\rho} +  \bar{L}^{\frac16} \sum_{i = 1}^2 \sum_{\iota \in \mathrm{Left, Right}} \|\rho^i \p_\rho^i \Xi_{\iota}\|_{L^2_\rho} \\  \label{main:G}
& + \sum_{j = 0}^3 ( \| \p_z^j \zeta_{\mathrm{Left}} \chi_{O,j}^+ \langle z \rangle^n \|_{L^2_z} + \| \p_z^j \zeta_{\mathrm{Right}} \chi_{O,j}^- \|_{L^2_z} ).
\end{align}
In addition, the following bound is valid on $\gamma_u$ and $\gamma_\psi$, 
\begin{align}\n
|\gamma_u| + |\gamma_\psi| \lesssim &  \sum_{j = 0}^3 \| \p_z^j \bold{R} \chi_{O,j} \langle z \rangle^n \|_{L^2_{sz}} + \sum_{j = 0}^1 \| \p_z^j \bold{R} \chi_I \|_{L^2_{sz}} + \| d(z) \p_z^2 \bold{R} \chi_I \|_{L^2_{sz}} \\ \n
& + \bar{L}^{\frac16} \sum_{i = 0}^2 \sum_{\iota \in \mathrm{Left, Right}}  \| \langle \rho \rangle \p_\rho^i \Xi_{\iota} \|_{L^2_\rho} +  \bar{L}^{\frac16} \sum_{i = 1}^2 \sum_{\iota \in \mathrm{Left, Right}} \|\rho^i \p_\rho^i \Xi_{\iota}\|_{L^2_\rho} \\  \label{main:G:g}
& + \sum_{j = 0}^3   ( \bar{L}^{\frac13} \| \p_z^j \zeta_{\mathrm{Left}} \chi_{O,j}^+ \langle z \rangle^n \|_{L^2_z} + \| \p_z^j \zeta_{\mathrm{Right}} \chi_{O,j}^- \|_{L^2_z} ).
\end{align}
\end{proposition}

\begin{remark} While $\gamma_u$ and $\gamma_\psi$ are controlled by $\| \psi \|_{\mathrm{Global}_n}$, the key point in stating \eqref{main:G:g} separately is to eliminate the $O(1)$ dependence of the right-hand side of \eqref{main:G:g} on $\zeta_{\mathrm{Left}}$ for these specific quantities, and replace it with a small $\bar{L}^{\frac13}$ factor. This distinction will be important to close the bounds in Section \ref{final:sec}, for instance \eqref{two:points:2}.
\end{remark}

\subsection{Closing Main Linear Loop}

\begin{lemma} \label{keylemma}Under the hypotheses of Proposition \ref{pro:global}, the following sequence of estimates is valid for $\bar{L} << 1$: 
\begin{align} \n
\| \psi \|_{\mathrm{von-Mise}_n}^2 \lesssim & \bar{L} \| \psi \|_S^2 + \sum_{j = 0}^3 \| \p_z^j \bold{R}\chi_{O,j} \langle z \rangle^n \|_{L^2_{sz}}^2 \\  \label{nlyougo:1}
&+ \sum_{j = 0}^3 ( \| \p_z^j \zeta_{\mathrm{Left}} \chi_{O,j}^+ \langle z \rangle^n \|_{L^2_z}^2 + \| \p_z^j \zeta_{\mathrm{Right}} \chi_{O,j}^- \|_{L^2_z}^2 ),\\ \label{nlyougo:1bbb}
\| \psi \|_{\mathrm{von-Mise}_-}^2 \lesssim & \bar{L} \| \psi \|_S^2 + \sum_{j = 0}^3 \| \p_z^j \bold{R}\chi_{O,j} \langle z \rangle^n \|_{L^2_{sz}}^2 + \sum_{j = 0}^3 ( \| \p_z^j \zeta_{\mathrm{Right}} \chi_{O,j}^- \|_{L^2_z}^2 ),\\ \label{nlyougo:2}
\| \psi \|_{S}^2 \lesssim & \| \psi \|_{\mathrm{von-Mise}_n}^2 + \| \Omega_I \|_{\mathrm{Crocco}}^2, \\ \n
\| \Omega_I \|_{\mathrm{Crocco}}^2 \lesssim & \| \gamma_\Omega^{(1)} \|_{L^2_s}^2 + \bar{L}( \| \psi \|_S^2 +  \| \psi \|_{\mathrm{von-Mise}_n}^2 +\| \Omega_I \|_{\mathrm{Crocco}}^2 ) \\ \label{nlyougo:3}
& + \sum_{i = 0}^1 \|Z^i \p_Z^i F_{\mathrm{Low}} \|_{L^2_{sZ}}^2 + \bar{L}^{\frac13} \sum_{i = 1}^2 \sum_{\iota \in \mathrm{Left, Right}} \|\rho^i \p_\rho^i \Xi_{\iota}\|_{L^2_\rho}^2 ,\\ \n
\| \gamma_\Omega^{(1)} \|_{L^2_s} \lesssim &\bar{L}^{\frac13} (\| \psi \|_S +  \| \Omega_I \|_{\mathrm{Crocco}} +\| \psi \|_{\mathrm{von-Mise}_n}) \\ \label{nlyougo:4}
&+ \| F_{\mathrm{Low}} \|_{L^2_{sZ}} + \bar{L}^{\frac16} \sum_{i = 0}^2\| \langle \rho \rangle (\p_\rho^i \Xi_{\mathrm{Left}},\p_\rho^i \Xi_{\mathrm{Right}}) \|_{L^2_\rho}.
\end{align}
\end{lemma}
\begin{proof}[Proof of \eqref{nlyougo:1} and \eqref{nlyougo:1bbb}] First, we note that the hypotheses on $\bar{W}$, $\bold{R}$, and $\bold{c}_i$ imply the hypotheses of Proposition \ref{pro:vonMise} through the definitions \eqref{tru:1}. Therefore, we may apply Proposition \ref{pro:vonMise}, and the claims are an immediate consequence of \eqref{mainvonM} and \eqref{lowvonM}.
\end{proof}
\begin{proof}[Proof of \eqref{nlyougo:2}] We give the estimate on the stronger quantity $\| u \langle z \rangle \|_{L^\infty}$ (which will then imply the bound on both $u$ and $\psi$). Fix a value, $z_{1}$, with the following properties: $0 < z_1 < 1$ and $z_1 \in \{\chi_{O}^- = 1\} \cap \{\chi_I = 1\}$. Fix another value $z_2$ with the following properties: $1 < z_2 < \infty$ and $z_2 \in \{\chi_I = 1\} \cap \{ \chi_{O}^+ = 1\}$. 

\vspace{2 mm}

\noindent \textit{``Lower von-Mise" Region: } If $0 \le z \le z_1$, then we have
\begin{align*}
|u(s, z)| = |\int_0^z u_z(s, z') \ud z'| = |\int_0^z u_z(s, z') \chi_{O}^- \ud z'| \lesssim \| u_z \chi_{O}^- \|_{L^2_z},
\end{align*}
and therefore $\sup_{1 \le s \le 1 + \bar{L}}|u(s, z)| \lesssim \| \psi \|_{\text{von-Mise}_-}$ by \eqref{outer:est:l2:0}. 
\vspace{2 mm}

\noindent \textit{``Intermediate Crocco" Region: } The assumption of \eqref{bhu:1} paired with the second derivative upper bound in \eqref{sushiL1L2} allows us to define $\delta_1 > 0$ such that $\inf_{1 < s < 1 + \bar{L}} \inf_{1- \delta_1 < z < 1 + \delta_1} \bar{W}_z(s, z) > 0$. By choosing $\delta > 0$ such that $\sup_{1 < s < 1 + \bar{L}} \sup_{1- \delta_1 < z < 1 + \delta_1}|\bar{W}|(s, z) \le 10\delta$, we may ensure that $\{ \bar{W} \le \delta \} \subset (1, 1+ \bar{L}) \times \{1-\delta_1 < z < 1 + \delta_1 \}$. Let $(s, Z_1)$ be in the $Z$ coordinate that corresponds to $(s, z_1$). We now estimate $\p_Z \Phi(s, Z_1)$ by 
\begin{align} \label{la:1}
|\p_Z \Phi(s, Z_1)| \lesssim |u(s, z_1)| + |\psi(s, z_1)| \lesssim \| \psi \|_{\text{von-Mise}_-}.
\end{align}
Let $(s, Z_2)$ be in the $Z$ coordinate that corresponds to $(s, z_1$). For any $Z_1 < Z < Z_2$, we have 
\begin{align*}
|\p_Z \Phi(s, Z)| \le & |\p_Z \Phi(s, Z_1)| + |\int_{Z_1}^{Z_2} \p_Z^2 \phi| \le |\p_Z \Phi(s, Z_1)| + |\int_{Z_1}^{Z_2} \p_Z^2 \phi \chi_I| \\
\le &   |\p_Z \Phi(s, Z_1)| + |\int_{Z_1}^{Z_2} \Omega_I| \\
\lesssim &  |\p_Z \Phi(s, Z_1)| + \| \Omega_I \|_{L^2_Z} \\
\lesssim &  \| \psi \|_{\text{von-Mise}_-} + \| \Omega_I \|_{\text{Crocco}},
\end{align*}
where we have appealed to estimate \eqref{la:1}. Of course this implies also that 
\begin{align} \label{gov:2}
\sup_{z_1 \le z \le z_2} |\psi(s, z)| \lesssim \sup_{Z_1 \le Z \le Z_2} |\Phi(s, Z)| \lesssim &  \| \psi \|_{\text{von-Mise}_-} + \| \Omega_I \|_{\text{Crocco}}.
\end{align}
From here, we obtain 
\begin{align} \label{gov:1}
\sup_{z_1 \le z \le z_2} |u(s, z)| \lesssim \sup_{z_1 \le z \le z_2} |\psi(s, z)| + \sup_{Z_1 \le z \le Z_2} |\p_Z \Omega(s, Z)|  \lesssim &  \| \psi \|_{\text{von-Mise}_-} + \| \Omega_I \|_{\text{Crocco}}.
\end{align}

\vspace{2 mm}

\noindent \textit{``Upper von-Mise" Region:} Here, we apply \eqref{inv:u:diff} with $z_0$ in that expression replaced by the $z_2$ in the present lemma. For clarify, we present the resulting expression for $z > z_2$: 
\begin{align}
u(s, z)  = &\frac{ \psi(s, z_2)}{\bar{W}(s, z_2)}\p_z \bar{W}(s, z) + \p_z \bar{W} \int_{z_2}^z \frac{U}{\bar{W}^2} \ud \tilde{z} + \frac{U}{\bar{W}}
\end{align}
We notice that $(z_2, z) \subset \{\chi_{O}^+ = 1\}$ by design of $z_2$. We also notice that $|\psi(s, z_2)| \lesssim \| \psi \|_{\text{von-Mise}_-} + \| \Omega_I \|_{\text{Crocco}}$ upon using \eqref{gov:2}. Therefore, we obtain 
\begin{align} \n
\langle z \rangle |u(s, z)| \lesssim & \| \p_z \bar{W} \langle z \rangle \|_{L^\infty} |\psi(s, z_2)| + \| U \chi_{O}^+ \|_{L^2_z} + \| U \chi_{O}^+ \langle z \rangle \|_{L^\infty_z} \\  \n
\lesssim & \| \psi \|_{\text{von-Mise}_n} + \| \Omega_I \|_{\text{Crocco}} +\| U \chi_{O}^+ \langle z \rangle \|_{L^2_z} + \| \p_z U \chi_{O}^+ \langle z \rangle \|_{L^2_z} + \| U \p_z \chi_{O}^+ \langle z \rangle \|_{L^2_z} \\  \n
\lesssim & \| \psi \|_{\text{von-Mise}_n} + \| \Omega_I \|_{\text{Crocco}} +\| U \chi_{O}^+ \langle z \rangle \|_{L^2_z} + \| \p_z U \chi_{O}^+ \langle z \rangle \|_{L^2_z} + \| U  \chi_{O, 0}^+ \langle z \rangle \|_{L^2_z} \\  \n
\lesssim & \| \psi \|_{\text{von-Mise}_n} + \| \Omega_I \|_{\text{Crocco}},
\end{align}
where we have applied the standard Sobolev embedding as well as the definition of the von-Mise energy functionals, \eqref{canyou}. 
\end{proof}
\begin{proof}[Proof of \eqref{nlyougo:3}] First of all, we note that the assumptions on $\bar{W}, \bold{c}_i$ imply through the relations \eqref{utau:def} and \eqref{tau:def} that the bootstrap hypotheses \eqref{tau:b:1} and \eqref{tau:b:2} hold. Hence, we may apply Proposition \ref{prop:Crocco}. We recall estimate \eqref{pro:cro:1} in which we choose the source term $S_I = F_{\text{Low}} + S_{\text{Coupled}}$. We estimate first $\| S_{\text{Coupled}} \|_{L^2_{sZ}}$ and then $\| Z \p_Z S_{\text{Coupled}} \|_{L^2_{sZ}}$. First, we have 
\begin{align*}
\| S_{\text{Coupled}} \|_{L^2_{sZ}} \lesssim & \bar{L}^{\frac12} (\| \tau_{-1} \p_Z \Phi \chi_I \|_{L^\infty_{sZ}} + \| \tau_{-2}  \Phi \chi_I \|_{L^\infty_{sZ}}) + ( \| \chi_I'' \Omega \|_{L^2_{sZ}} + \| \chi_I' \tau_{1} \p_Z \Omega \|_{L^2_{sZ}}  ).
\end{align*}
 First, using the change of variables \eqref{Eikonal:1} and \eqref{change:1}, we obtain 
\begin{align*}
 &(\| \tau_{-1} \p_Z \Phi \chi_I \|_{L^\infty_{sZ}} + \| \tau_{-2}  \Phi \chi_I \|_{L^\infty_{sZ}}) \\
  \lesssim & (\| \tau_{-1} \|_{L^\infty} + \| \tau_{-2} \|_{L^\infty})(\| \frac{1}{\bar{W}_z} \chi_I \|_{L^\infty} + \| \bar{W}_z \chi_I \|_{L^\infty}) (\| u \|_{L^\infty} + \| \psi \|_{L^\infty}) \\
  \lesssim & \| \psi \|_S,
\end{align*}
where we note that $\| \tau_{-1} \|_{L^\infty} \lesssim 1$ and $\| \tau_{-2} \|_{L^\infty} \lesssim 1$ upon invoking \eqref{tau:def}, \eqref{utau:def}, and the hypotheses on $\bar{W}, \bold{c}_i$. Next, we have upon using the relation between $u$ and $\Omega$, \eqref{weds:2}
\begin{align*}
&( \| \chi_I'' \Omega \|_{L^2_{sZ}} + \| \chi_I' \tau_{1} \p_Z \Omega \|_{L^2_{sZ}}  )\\
\lesssim & \| \tau_1 \|_{L^\infty} \frac{1}{\| \bar{W}_z \chi_{\mathrm{Fat}}  \|_{L^\infty}}( \sum_{i = 1}^3 \| \p_z^i \bar{W} \|_{L^\infty}) (\sum_{i = 0}^2 \| \p_z^i u \chi_O \|_{L^2_{sz}}) \lesssim \sqrt{\bar{L}} \| \psi \|_{\text{von-Mise}}
\end{align*}
where $\chi_{\mathrm{Fat}}$ is a slightly fattened cutoff of $\chi_I$, equal to $1$ on the support of $\chi_I$.  

Next, for the $Z \p_Z S_{\text{Coupled}}$ term, the estimates are largely the same. The only exception is $\chi_I \p_Z \{ \p_Z \Phi\} = \chi_I \Omega = \Omega_I$, which we estimate as follows 
\begin{align*}
\| \Omega \|_{L^2_{sZ}} \lesssim \sqrt{\bar{L}} \| \Omega \|_{L^\infty_s L^2_Z} \lesssim \sqrt{\bar{L}} \| \Omega_I \|_{\text{Crocco}},
\end{align*}
where we have invoked the definition of the Crocco norm, \eqref{def:Crocco:norm}.  

\end{proof}
\begin{proof}[Proof of \eqref{nlyougo:4}] We recall estimate \eqref{tr:es:1} which we apply with the source terms $F_{\text{Low}} + F_{\text{Coupled}}$. We estimate the $L^2_{sZ}$ norm of $F_{\text{Coupled}}$: 
\begin{align*}
\| F_{\text{Coupled}} \|_{L^2_{sZ}} \le & \| \tau_2 \p_Z^2 \Omega_I \|_{L^2_{sZ}} +  ( \|\tau_1 \p_Z \Omega_I \|_{L^2_{sZ}} +\|\tau_0  \Omega_I \|_{L^2_{sZ}}   )  + ( \|\chi_I \tau_{-1} \p_Z \Phi \|_{L^2_{sZ}} +\|\tau_{-2}  \Phi \|_{L^2_{sZ}}   )  \\
& + ( \| \chi_I'' \Omega \|_{L^2_{sZ}} + \| \chi_I' \tau_{1} \p_Z \Omega \|_{L^2_{sZ}}  ).
\end{align*}
Each group of the above terms will be treated separately. First, we have 
\begin{align*}
\| \tau_2 \p_Z^2 \Omega_{k,I}  \|_{L^2_{sZ}} \lesssim & \| \frac{\tau_2}{Z} \|_{L^\infty}  \| Z \p_Z^2 \Omega_{k,I}   \|_{L^{2}_{sZ}} \lesssim  \bar{L}^{\frac16} \| \Omega_I \|_{\text{Crocco}}, 
\end{align*}
where we have appealed to \eqref{ode:3}. 

Next, we have by definition 
\begin{align*}
 \|\tau_1 \p_Z \Omega_I \|_{L^2_{sZ}} +\|\tau_0  \Omega_I \|_{L^2_{sZ}}  \lesssim (\| \tau_0 \|_{L^\infty} + \| \tau_1 \|_{L^\infty}) (\| \Omega_I \|_{L^2_{sZ}} + \| \p_Z \Omega_I \|_{L^2_{sZ}}) \lesssim \bar{L}^{\frac16} \| \Omega_I \|_{\text{Crocco}},
\end{align*}
where we have appealed to \eqref{ode:2}. 

Next, we have 
\begin{align*}
&\|\chi_I \tau_{-1} \p_Z \Phi \|_{L^2_{sZ}} +\|\tau_{-2}  \Phi \|_{L^2_{sZ}}   \\
\lesssim & \bar{L}^{\frac12} (\| \tau_{-1} \p_Z \Phi \chi_I \|_{L^\infty_{sZ}} + \| \tau_{-2}  \Phi \chi_I \|_{L^\infty_{sZ}}) \\
\lesssim & \bar{L}^{\frac12}(\| \tau_{-1} \|_{L^\infty} + \| \tau_{-2} \|_{L^\infty})(\| \frac{1}{\bar{W}_z} \chi_I \|_{L^\infty} + \| \bar{W}_z \chi_I \|_{L^\infty}) (\| u \|_{L^\infty} + \| \psi \|_{L^\infty}) \\
\lesssim &\bar{L}^{\frac12} \| \psi \|_S. 
\end{align*}
Lastly, we have
\begin{align*}
&\| \chi_I'' \Omega \|_{L^2_{sZ}} + \| \chi_I' \tau_{1} \p_Z \Omega \|_{L^2_{sZ}} \lesssim  \| \tau_1 \|_{L^\infty} ( \| \chi_I'' \Omega \|_{L^\infty_{Z} L^2_s} + \| \chi_I'  \p_Z \Omega \|_{L^2_{sZ}}) \\
\lesssim &  \| \tau_1 \|_{L^\infty} \sum_{i = 0}^1 \| \chi_{\text{Fat}} \p_Z^j \Omega \|_{L^2_{sZ}} \\
\lesssim &  \| \tau_1 \|_{L^\infty} \frac{1}{\| \bar{W}_z  \|_{L^\infty}}( \sum_{i = 1}^3 \| \p_z^i \bar{W} \|_{L^\infty}) (\sum_{i = 0}^2 \| \p_z^i u \chi_O \|_{L^2_{sz}}) \\
\lesssim & \sqrt{\bar{L}} \| \psi \|_{\text{von-Mise}},
\end{align*}
where we have invoked the relation between $u$ and $\Omega$, \eqref{weds:2}. The lemma is proven. 
\end{proof}


\subsection{Proof of Proposition \ref{pro:global}}

We are now ready to prove the main proposition of this section, namely Proposition \ref{pro:global}.

\begin{proof}[Proof of Proposition \ref{pro:global}] We will use the smallness of $\bar{L}$ in order to close the set of estimates \eqref{nlyougo:1} -- \eqref{nlyougo:4}. More specifically, we insert \eqref{nlyougo:2} and \eqref{nlyougo:4} into \eqref{nlyougo:1} and \eqref{nlyougo:3}, which gives 
\begin{align} \n
\| \psi \|_{\mathrm{von-Mise}_n}^2 \lesssim & \bar{L} (\| \psi \|_{\mathrm{von-Mise}_n}^2 + \| \Omega_I \|_{\text{Crocco}}^2) + \sum_{j = 0}^3 \| \p_z^j \bold{R}\chi_{O,j} \langle z \rangle^n \|_{L^2_{sz}}^2 \\  \label{nlyougo:1en}
&+ \sum_{j = 0}^3 ( \| \p_z^j \zeta_{\mathrm{Left}} \chi_{O,j}^+ \langle z \rangle^n \|_{L^2_z}^2 + \| \p_z^j \zeta_{\mathrm{Right}} \chi_{O,j}^- \|_{L^2_z}^2 ), \\ \n
\| \Omega_I \|_{\mathrm{Crocco}}^2 \lesssim &  \bar{L}^{\frac23}(   \| \psi \|_{\mathrm{von-Mise}_n}^2 +\| \Omega_I \|_{\mathrm{Crocco}}^2 ) \\ \n
& + \sum_{i = 0}^1 \|Z^i \p_Z^i F_{\mathrm{Low}} \|_{L^2_{sZ}}^2 + \bar{L}^{\frac13} \sum_{i = 1}^2 \sum_{\iota \in \mathrm{Left, Right}} \|\rho^i \p_\rho^i \Xi_{\iota}\|_{L^2_\rho}^2 \\ \label{nlyougo:3en}
& + \| F_{\mathrm{Low}} \|_{L^2_{sZ}}^2 + \bar{L}^{\frac13} \sum_{i = 0}^1 \sum_{\iota \in \mathrm{Left, Right}}  \| \sqrt{\rho} \p_\rho^i \Xi_{\iota} \|_{L^2_\rho}^2.
\end{align}
At this point, we can notice that 
\begin{align} \n
\| F_{\mathrm{Low}} \|_{L^2_{sZ}}^2 + \sum_{i = 0}^1 \|Z^i \p_Z^i F_{\mathrm{Low}} \|_{L^2_{sZ}}^2 \lesssim & \sum_{j = 0}^3 \| \p_z^j \bold{R} \chi_{O,j} \langle z \rangle^n \|_{L^2_{sz}} + \| \bold{R} \chi_I \|_{L^2_{sz}}  \\ \label{reapply:me}
& +  \| \p_z \bold{R} \chi_I \|_{L^2_{sz}} + \| d(z) \p_z^2 \bold{R} \chi_I \|_{L^2_{sz}}
\end{align}
where we have used the definition \eqref{defFLOW}, as well as the change of variables from $(s, z)$ to $(s,Z)$, \eqref{weds:1}. 

Therefore, letting $\text{RHS}_{\ast}$ denote the right-hand side of \eqref{main:G}, \eqref{nlyougo:1en} -- \eqref{nlyougo:3en} yield 
\begin{align} \n
\| \psi \|_{\mathrm{von-Mise}_n}^2 \lesssim & \bar{L} (\| \psi \|_{\mathrm{von-Mise}_n}^2 + \| \Omega_I \|_{\text{Crocco}}^2) + \text{RHS}_\ast^2, \\ \n
\| \Omega_I \|_{\mathrm{Crocco}}^2 \lesssim &  \bar{L}^{\frac23}(   \| \psi \|_{\mathrm{von-Mise}_n}^2 +\| \Omega_I \|_{\mathrm{Crocco}}^2 ) + \text{RHS}_\ast^2.
\end{align}
Clearly, for $\bar{L} << 1$ this closes the bound for $\| \psi \|_{\mathrm{von-Mise}_n}^2 + \| \Omega_I \|_{\text{Crocco}}^2$ and therefore also closes the bound for $\| \psi \|_S^2$ by \eqref{nlyougo:2} and therefore also closes the bound for $\| \gamma_\Omega^{(1)}\|_{L^2_s}$ by applying \eqref{nlyougo:4}. Thus, by definition \eqref{global:n}, the entire $\| \psi \|_{\text{Global}_n}$ is controlled by $\text{RHS}_\ast$. This proves \eqref{main:G}. 

The proof of \eqref{main:G:g} works in a nearly identical manner, upon using \eqref{nlyougo:1bbb} instead of \eqref{nlyougo:1}. We proceed as follows. Define $\mathrm{RHS}_{\ast \ast}$ to be the right-hand side of \eqref{main:G:g}. By \eqref{gov:2} and \eqref{gov:1}, we have $|\gamma_u| + |\gamma_\psi| \lesssim \| \psi \|_{\text{von-Mise}_-} + \| \Omega_I \|_{\text{Crocco}}$. We now invoke \eqref{nlyougo:1bbb}, \eqref{nlyougo:3}, \eqref{nlyougo:4} to obtain 
\begin{align} \label{gho:1}
\| \psi \|_{\mathrm{von-Mise}_-}^2 \lesssim & \bar{L} \times \mathrm{RHS}_\ast^2 + \mathrm{RHS}_{\ast \ast}^2, \\ \n
\| \Omega_I \|_{\mathrm{Crocco}}^2 \lesssim & \| \gamma_\Omega^{(1)} \|_{L^2_s}^2 + \bar{L} \times \mathrm{RHS}_\ast^2  + \sum_{i = 0}^1 \|Z^i \p_Z^i F_{\mathrm{Low}} \|_{L^2_{sZ}}^2 + \bar{L}^{\frac13} \sum_{i = 1}^2 \sum_{\iota \in \mathrm{Left, Right}} \|\rho^i \p_\rho^i \Xi_{\iota}\|_{L^2_\rho}^2 ,\\ \label{gho:2}
\lesssim & \| \gamma_\Omega^{(1)} \|_{L^2_s}^2 + \bar{L} \times \mathrm{RHS}_\ast^2 + \mathrm{RHS}_{\ast\ast}^2, \\ \n
\| \gamma_\Omega^{(1)} \|_{L^2_s} \lesssim &\bar{L}^{\frac13} \times \mathrm{RHS}_\ast + \| F_{\mathrm{Low}} \|_{L^2_{sZ}} + \bar{L}^{\frac16} \sum_{i = 0}^2\| \langle \rho \rangle (\p_\rho^i \Xi_{\mathrm{Left}},\p_\rho^i \Xi_{\mathrm{Right}}) \|_{L^2_\rho} \\ \label{gho:3}
\lesssim & \bar{L}^{\frac13} \times \mathrm{RHS}_\ast + \mathrm{RHS}_{\ast \ast},
\end{align}
where we have invoked \eqref{reapply:me} above as well. Inserting \eqref{gho:3} into \eqref{gho:2}, and using that $\bar{L}^{\frac13} \mathrm{RHS}_\ast \lesssim \mathrm{RHS}_{\ast \ast}$, we obtain \eqref{main:G:g}. The proposition is proven. 
\end{proof}

\section{Analysis of $\mathcal{L}_{\text{Prandtl}}[\mathcal{R}]$} \label{yikes}

In this section, we are considering \eqref{abs:1}. Our main objective in this section is to convert the bounds in Proposition \ref{pro:global} into the physical $(x, y)$ coordinate system from $(s, z)$ coordinates. The reasons are two-fold. First, the global norm, \eqref{global:n}, contains quantities expressed in $(s, z)$ and in $(s, Z)$, and we would like to have a norm expressed in a single coordinate system. Since the system \eqref{abs:1} is expressed in $(x, y)$ coordinates, the most natural choice for such a unified coordinate system is the $(x, y)$ coordinate system. Second, more practically speaking, the source terms appearing in $\mathcal{R}$ (which will eventually be applied to $\mathcal{R}_k$, defined in \eqref{defRKRK:temp} and \eqref{RkSemi}) are left in $(x, y)$ coordinates and therefore will be more convenient to estimate with a norm that is expressed in these coordinates, as opposed to changing variables for each source term. The basic premise of this section will be to transform \eqref{abs:1} into \eqref{merc:1lkj}, apply the estimates from Proposition \eqref{pro:global}, and then switch the coordinate system back. A caricature of our strategy is shown in Figure \ref{org:tree:PR}.
\begin{figure}[h] 
\hspace{50 mm} \includegraphics[scale=0.3]{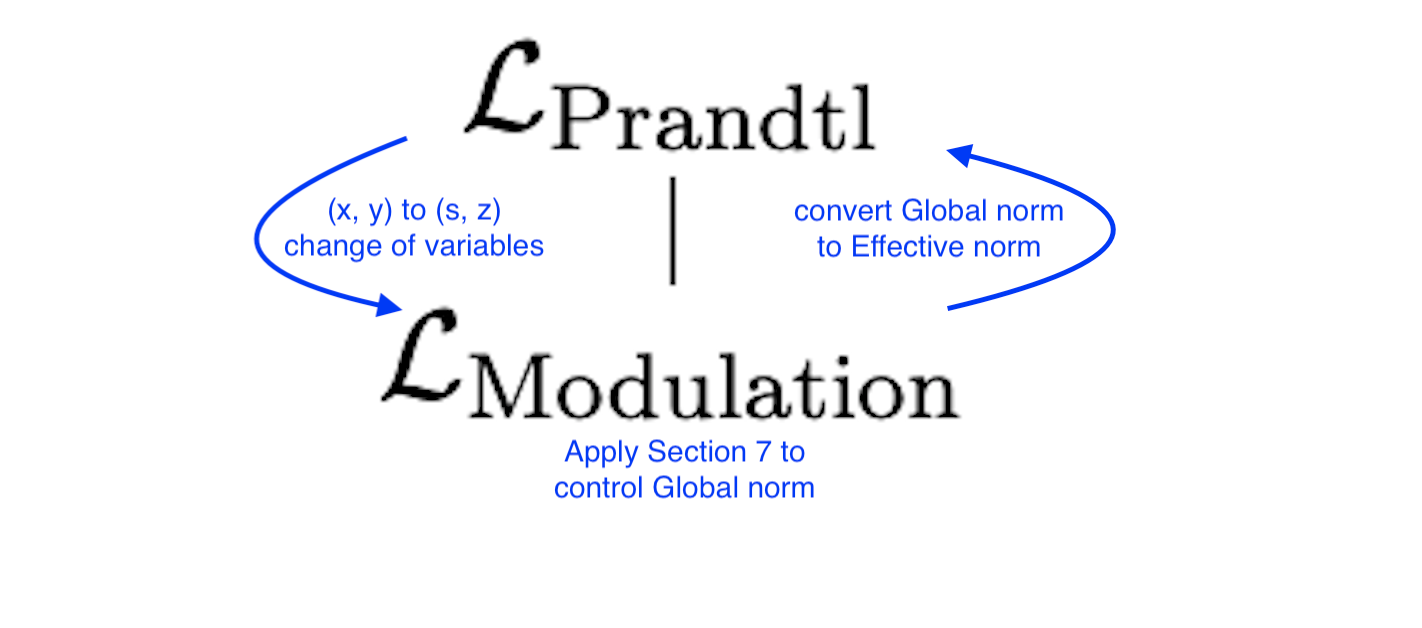}
\caption{Closing the Effective Norm} \label{org:tree:PR}
\end{figure}

To proceed in a self-contained manner, we first recall the form of $\bar{U}$ from \eqref{genbd:1}, and subsequently $\Lambda(x)$ through \eqref{defLambdagn}. Next, we  recall  
\begin{align} \label{z:variable:en}
\frac{\ud s}{\ud x} = \frac{1}{\Lambda(x)^2}, \qquad z := \frac{y}{\Lambda(x)} = \frac{y}{\Lambda_G(x) + \eps \Xi(x)}.
\end{align}
We now transform our unknown functions via 
\begin{align} \label{change:function:as}
&u(s, z) = u_{R}(x, y), \qquad v(s, z) := - \int_0^z \p_s u, \qquad \psi(s, z) := \int_0^z u,
\end{align}
our modulation variables via 
\begin{align}
\lambda(s) := \Lambda(x),  \qquad \lambda_G(s) := \Lambda_G(x), \qquad \mu(s) :=  \Xi(x),
\end{align}
and our backgrounds via 
\begin{align}
&\bar{W}(s, z) = \bar{U}(x, y), \qquad \bar{u}(s, z) = u_{FS}(x, y), \qquad w_B(s, z) = u_B(x, y).
\end{align}
Upon introducing these transformations, the quantities in \eqref{change:function:as} satisfy the equation \eqref{merc:1lkj}. First of all, it is clear from \eqref{defLambdagn} that $\Lambda = \Lambda[u_{B, R}]$, and therefore also $\lambda = \lambda[u_{B, R}]$ (and the same for $\Xi$, $\mu$). Second,  $\bold{c}_i$ are functions of $\bold{C}_i$ and $\lambda(s)$ (and therefore $u_B$ and thus $u_{B,R}$) through \eqref{defn:bold:c1}, so we write $\bold{c}_i = \bold{c}_i[u_{B, R}, \vec{\bold{C}}]$. Finally $\bold{R} = \lambda^2 \mathcal{R}$. 

We note at this point that \eqref{dgu} implies also that 
\begin{align}
\gamma_u = u(1, 1) = u_R(1, y_1(\eps)), \qquad \gamma_\psi = \psi(1, 1) = \psi_R(1, y_1(\eps)). 
\end{align}

One observes that the nature of the global norm, \eqref{global:n} is relatively complicated; it measures several types of quantities in several coordinate systems. We therefore define an ``effective" norm in $(x, y)$ coordinates, that is: a norm that encompasses in a unified manner the essential quantities that have been controlled: 
\begin{align} \n
\| \psi_R \|_{\text{Eff}_n} := & \| \psi_R \|_{L^\infty} + \| u_R \langle y \rangle^n \|_{L^\infty} + \| d(x,y) \p_y^2 u_R \langle y \rangle^n \|_{L^\infty_x L^2_y} + \|  \p_y u_R \langle y \rangle^n \|_{L^\infty_x L^2_y} \\ \label{eff:norm}
& + \|  \p_y^2 u_R \langle y \rangle^n \|_{L^2_{xy}} + \| \sqrt{|d(x,y)|} \p_y^3 u_R \langle y \rangle^n \|_{L^2_{xy}} + \| \sqrt{\bar{W}} \p_y^4 u_R \chi_{O,3} \langle y \rangle^n \|_{L^2_{xy}}.
\end{align}
Above, the function $d(x, y) := d_0( \frac{y}{\Lambda(x)} )$, where $d_0$ is a smooth, monotonically increasing function satisfying $d_0(1) = 0$, and $d_0'(1) = 1$, and $d_0(0) =-1$, $d_0(\infty) = 1$. The purpose of this weight is to mimic $\bar{W}$ in the norms near the interface (it vanishes linearly at $y = \Lambda(x)$, and is, in absolute value, bounded above and below away from $y = \Lambda(x)$) but converges to fixed values away from the interface. It turns out this norm is equivalent to the Global norm, \eqref{global:n}.

We are therefore ready to state the main proposition: 
\begin{proposition}\label{pro:effective}Assume the following bounds on $u_{B,R}$ :
\begin{align} \label{ubboo}
& \sum_{k = 0}^1 \sum_{j = 0}^5 \| \p_x^k \p_y^j u_{B, R} \langle y \rangle^{n+1} \|_{L^\infty} \le  \eps^{-\frac{1}{3}}.
\end{align}
Assume the source term, $\mathcal{R}$, satisfies $\mathcal{R}|_{z = 0} = \p_y \mathcal{R}|_{z = 0} = 0$. Assume the coefficients $\bold{C}_i$, $i = 1, 2, 3$, satisfy the following: 
\begin{align} \label{eng:1Lcap}
\bold{C}_1(s, 0) = & 0, \\ \label{eng:2Lcap}
\| \frac{\bold{C}_1}{\langle z \rangle} \|_{L^\infty} +  \sum_{j = 1}^3 \| \p_z^j \bold{C}_1 \|_{L^\infty} \lesssim & 1, \\ \label{eng:3Lcap}
\sum_{j = 0}^3 \| \p_z^j \bold{C}_2 \|_{L^\infty} \lesssim & 1, \\ \label{eng:4L}
\sum_{j = 0}^3 \| \p_z^j \bold{C}_3 \langle z \rangle^{n+1} \|_{L^\infty} \lesssim & 1.
\end{align}
The following bounds are valid for $0 < L << 1$: 
\begin{align} \n
\| \psi_R \|_{\mathrm{Eff}_n} \lesssim &  \sum_{j = 0}^3 \| \p_y^j \mathcal{R} \chi_{O,j} \langle y \rangle^n \|_{L^2_{xy}} + \sum_{j = 0}^1 \| \p_y^j \mathcal{R} \chi_I \|_{L^2_{xy}} + \| d(x,y) \p_y^2 \mathcal{R} \chi_I \|_{L^2_{xy}} \\ \n
& + L^{\frac16} \sum_{i = 0}^2 \sum_{\iota \in \mathrm{Left, Right}}  \| \langle \rho \rangle \p_\rho^i \Xi_{\iota} \|_{L^2_\rho} +  L^{\frac16} \sum_{i = 1}^2 \sum_{\iota \in \mathrm{Left, Right}} \|\rho^i \p_\rho^i \Xi_{\iota}\|_{L^2_\rho} \\  \label{main:E}
& + \sum_{j = 0}^3 ( \| \p_z^j \zeta_{\mathrm{Left}} \chi_{O,j}^+ \langle z \rangle^n \|_{L^2_z} + \| \p_z^j \zeta_{\mathrm{Right}} \chi_{O,j}^- \|_{L^2_z} ),
\end{align}
and
\begin{align} \n
|\gamma_u| + |\gamma_\psi| \lesssim &  \sum_{j = 0}^3 \| \p_y^j \mathcal{R} \chi_{O,j} \langle y \rangle^n \|_{L^2_{xy}} + \sum_{j = 0}^1 \| \p_y^j \mathcal{R} \chi_I \|_{L^2_{xy}} + \| d(x,y) \p_y^2 \mathcal{R} \chi_I \|_{L^2_{xy}} \\ \n
& + L^{\frac16} \sum_{i = 0}^2 \sum_{\iota \in \mathrm{Left, Right}}  \| \langle \rho \rangle \p_\rho^i \Xi_{\iota} \|_{L^2_\rho} +  L^{\frac16} \sum_{i = 1}^2 \sum_{\iota \in \mathrm{Left, Right}} \|\rho^i \p_\rho^i \Xi_{\iota}\|_{L^2_\rho} \\  \label{main:Ekju}
& + \sum_{j = 0}^3   ( L^{\frac13} \| \p_z^j \zeta_{\mathrm{Left}} \chi_{O,j}^+ \langle z \rangle^n \|_{L^2_z} + \| \p_z^j \zeta_{\mathrm{Right}} \chi_{O,j}^- \|_{L^2_z} ).
\end{align}
\end{proposition}
\begin{remark} Our assumption \eqref{ubboo} shows that $\Lambda$ (hence $\lambda$) is bounded above and below universally. Therefore $s$ and $x$ are comparable in magnitude through the relation \eqref{z:variable:en}. In particular, we have $\bar{L} = \int_1^{1 + L} \frac{1}{\Lambda(x)^2} \ud x$, which implies that 
\begin{align} \label{LbarL}
\bar{L} \lesssim L \lesssim \bar{L}, 
\end{align}
and so at this stage we may translate bounds that require $\bar{L} << 1$ to $L << 1$. 
\end{remark}

\subsection{Effective Lower Bounds}

\begin{lemma} Under the assumptions of Proposition \ref{pro:effective}, the following bounds are valid: 
\begin{align} \label{lyu:1}
\| \psi \|_{L^\infty} + \| u \langle z \rangle^n \|_{L^\infty}  \lesssim &\| \psi \|_{\mathrm{Global}_n}, \\ \label{lyu:2}
 \|  \p_z u \langle z \rangle^n \|_{L^\infty_s L^2_z} + \| d_0(z) \p_z^2 u \langle z \rangle^n \|_{L^\infty_s L^2_z} \lesssim & \| \psi \|_{\mathrm{Global}_n}, \\ \label{lyu:3}
\|  \p_z^2 u \langle z \rangle^n \|_{L^2_{sz}} + \| \sqrt{|d_0(z)|} \p_z^3 u \langle z \rangle^n \|_{L^2_{sz}} \lesssim &  \| \psi \|_{\mathrm{Global}_n}.
\end{align}
\end{lemma}
\begin{proof}[Proof of \eqref{lyu:1}] The proof of \eqref{lyu:1} without the weight $\langle z \rangle^n$ follows immediately from the definition of the $S$-norm, \eqref{S:norm}. We choose a value of $z_{\ast} \in \{\chi_{O} = 1\} \cap \{z > 1\}$, and write 
\begin{align*}
u(s, z)^2 \langle z \rangle^{2n} = & u(s,z_\ast)^2 z_\ast^{2n} + \int_{z_\ast}^z \p_{z'} \{ u^2 \langle z' \rangle^{2n} \} \ud z'  \\
 \lesssim & \| \psi \|_S^2 + \| u \langle z \rangle^n \chi_{O,+} \|_{L^\infty_s L^2_z}^2 +\| \p_z u \langle z \rangle^n \chi_{O,+} \|_{L^\infty_s L^2_z}^2 \\
 \lesssim & \| \psi \|_S^2 + \| \psi \|_{\text{von-Mise}_n}^2,
\end{align*}
where in the final line we used \eqref{wo:est:l2:0}.

By \eqref{outer:est:l2:0} and again \eqref{wo:est:l2:0}, the bounds \eqref{lyu:2} follows if we include $\chi_{O}$: 
\begin{align*}
\| d_0(z) \p_z^2 u \langle z \rangle^n \chi_O \|_{L^\infty_s L^2_z} + \|  \p_z u \langle z \rangle^n\chi_O \|_{L^\infty_s L^2_z} \lesssim & \| \psi \|_{\mathrm{Global}_n}.
\end{align*}
Therefore, it suffices to prove (noting that $\langle z \rangle^n$ is bounded above and below on the support of $\chi_I$)
\begin{align*}
\| d_0(z) \p_z^2 u  \chi_I \|_{L^\infty_s L^2_z} + \|  \p_z u  \chi_I \|_{L^\infty_s L^2_z} \lesssim & \| \psi \|_{\mathrm{Global}_n}
\end{align*}
The proof of \eqref{lyu:1} is complete.
\end{proof}
\begin{proof}[Proof of \eqref{lyu:2}] We first prove the $\p_z u$ estimate. We recall \eqref{weds:2}, which gives us $\p_z u = \bar{W}_z^2 \Omega + \frac{\bar{W}_{zz}}{\bar{W}_z} u$, and therefore 
\begin{align*}
\|  \p_z u  \chi_I \|_{L^\infty_s L^2_z} \le & \| \bar{W}_z^2 \Omega  \chi_I \|_{L^\infty_s L^2_z} + \| \frac{\bar{W}_{zz}}{\bar{W}_z} u \Omega  \chi_I \|_{L^\infty_s L^2_z} \\
\lesssim & (\| \bar{W}_{zz} \|_{L^\infty} + \| \bar{W}_z \|_{L^\infty} + \frac{1}{\bar{W}_z} \chi_I \|_{L^\infty}) ( \| \Omega_I \|_{L^\infty_s L^2_z} + \| u \chi_I \|_{L^\infty_s L^2_z} ) \\
\lesssim & ( \| \Omega_I \|_{L^\infty_s L^2_Z} + \| u \|_{L^\infty_{sz}} )\\
\lesssim & \| \Omega_I \|_{\text{Crocco}} + \| \psi \|_S, 
\end{align*}
where above we have used the boundedness of the change of variable from $L^\infty_s L^2_z$ to $L^\infty_s L^2_Z$. 

We now address the $\p_z^2u$ term in \eqref{lyu:2}. Again, due to \eqref{outer:est:l2:0} and again \eqref{wo:est:l2:0} it suffices to consider the bounds on the support of $\chi_I$. We now differentiate \eqref{weds:2} which produces the identity (upon using the chain rule due to the change of variables $z \rightarrow Y \rightarrow Z$, \eqref{weds:1})
\begin{align} \label{id:cro:2}
\p_z^2 u = \bar{W}_z^2 \frac{\bar{W}_z}{\bar{W}_z(s, 1)^{\frac23}} \Omega_{Z} + 2 \bar{W}_z \bar{W}_{zz} \Omega + \frac{\bar{W}_{zz}}{\bar{W}_z} \p_z u + \p_z \{ \frac{\bar{W}_{zz}}{\bar{W}_z} \} u
\end{align}
We now use equivalence $d_0(z) \sim Z$ to obtain 
\begin{align} \n
\| d_0(z) \p_z^2 u \chi_I \|_{L^2_z} \lesssim & \| \bar{W}_z^2 \frac{\bar{W}_z}{\bar{W}_z(s, 1)^{\frac23}} \|_{L^\infty} \| Z \Omega_Z \chi_I \|_{L^2_z} +  \| \bar{W}_z \bar{W}_{zz} \|_{L^\infty} \| \Omega \chi_I \|_{L^2_z} \\ \label{st:1}
&+ \| \frac{\bar{W}_{zz}}{\bar{W}_z} \|_{L^\infty} \| \p_z u \|_{L^2_z} + \| \p_z \{ \frac{\bar{W}_{zz}}{\bar{W}_z} \} \|_{L^\infty} \| u \chi_I \|_{L^\infty} \\ \label{st:2}
\lesssim & \| Z \p_Z \Omega_I \|_{L^2_z} + \|  \Omega_I \|_{L^2_z} + \| \p_z u \|_{L^2_z} + \| u  \|_{L^\infty} \\ \label{st:3}
\lesssim & \| Z \p_Z \Omega_I \|_{L^2_Z} + \|  \Omega_I \|_{L^2_Z} + \| \p_z u \|_{L^2_z} + \| u  \|_{L^\infty} \\ \label{st:4}
\lesssim & \| \Omega_I \|_{\text{Crocco}} + \| \psi \|_{\text{Global}_n}, 
\end{align}
where we have used the assumed bounds \eqref{ubboo} to estimate the $\bar{W}$ terms, the identity $\p_Z \Omega \chi_I = \p_Z \Omega_I - \Omega \chi_I' = \p_Z \Omega \chi_I = \p_Z \Omega_I - \frac{1}{\bar{W}_z^2}(\p_z u - \frac{\bar{W}_{zz}}{\bar{W}_z}u) \chi_I'$ to go from \eqref{st:1} to \eqref{st:2}, the boundedness of the change of variables from $L^2_z$ to $L^2_Z$ to go from \eqref{st:2} to \eqref{st:3}, and finally the fact that we have already established \eqref{lyu:2} for $\p_z u$ to go from \eqref{st:3} to \eqref{st:4}. The proof of \eqref{lyu:2} is complete.
\end{proof}
\begin{proof}[Proof of \eqref{lyu:3}] To prove the $\p_z^2 u$ inequality of \eqref{lyu:3}, we first notice that since $d(z) \gtrsim 1$ on the support of $1-\chi_I$, we may apply \eqref{lyu:2} to estimate $\|  \p_z^2 u \langle z \rangle^n (1 - \chi_I) \|_{L^2_{sz}}$. Therefore, it suffices to estimate $\|  \p_z^2 u \chi_I \|_{L^2_{sz}}$. For this, we again use the identity \eqref{id:cro:2} and invoke the fact that $\| \p_Z \Omega_I \|_{L^2_{sZ}} \lesssim \| \Omega_I \|_{\text{Crocco}}$, according to \eqref{west:2}. 

To prove the $\p_z^3 u$ estimate, we note that due to \eqref{ujell2} and \eqref{md:3}, it suffices to consider $\| \sqrt{d(z)} \p_z^3 u \chi_I \|_{L^2_{sz}}$. To prove the estimate in this case, we differentiate \eqref{id:cro:2} in $z$. In so doing, the leading order term is $\p_Z^2 \Omega$, which satisfies $\| \sqrt{Z} \p_Z^2 \Omega_I \|_{L^2_{sZ}} \lesssim \| \Omega_I \|_{\text{Crocco}}$ by definition, \eqref{def:Crocco:norm}. The remaining, lower order, terms are estimated using the previously established estimates in \eqref{lyu:1} -- \eqref{lyu:3}. The lemma is proven. 
\end{proof}

\subsection{Proof of Proposition \ref{pro:effective}}

We first establish a lemma that allows us to bound $\Xi(x)$ in terms of the quantity $u_B$. 
\begin{lemma} Given $u_{B, R}$ satisfying \eqref{ubboo}, there exists a perturbation $\Xi(x)$ satisfying the following bounds: 
\begin{align} \label{xi:est}
\| \Xi \|_{L^\infty_x} \lesssim \| u_{B, R} \|_{L^\infty_{xy}}, \qquad \| \Xi' \|_{L^\infty_x} \lesssim \| \p_x u_{B, R} \|_{L^\infty_{xy}} + \| u_{B, R} \|_{L^\infty_{xy}} + \| \p_y u_{B, R} \|_{L^\infty_{xy}}.
\end{align}
\end{lemma}
\begin{proof} Consider the equation \eqref{defLambdagn}, which we rewrite as follows
\begin{align*}
u_{FS}(x, \Lambda_G(x) + \eps \Xi(x)) + \eps u_{B, R}(x, \Lambda_G(x) + \eps \Xi(x)) = 0. 
\end{align*}
We find a solution $\Xi(x)$ to this equation using the following iteration: $\Xi^{(0)}(x) = 0$, 
\begin{align*}
u_{FS}(x, \Lambda_G(x) + \eps \Xi^{(n)}(x)) = - \eps u_{B, R}(x, \Lambda_G(x) + \eps \Xi^{(n-1)}(x)). 
\end{align*}
As $\p_y u_{FS}$ is bounded below in a neighborhood of $x, \Lambda_G(x)$ and $u_{FS}(x, \Lambda_G(x)) = 0$, there exists a smooth inverse $I_{FS}(x, a)$, defined for $|a| < a_{\ast}$ for a universal constant $a_\ast$, such that $u_{FS}(x, \Lambda_G(x) + p) = a \iff p = I_{FS}(x, a)$ and $I_{FS}(x, 0) = 0$. We therefore have 
\begin{align}
\eps \Xi^{(n)}(x) = I_{FS}(x, -\eps u_{B, R}(x, \Lambda_G(x) + \eps \Xi^{(n-1)}(x))).
\end{align}
It is clear that due to the boundedness of $\| u_{B, R} \|_{L^\infty}, \|\p_y u_{B, R} \|_{L^\infty}$ that
\begin{align*}
|\Xi^{(n+1)}(x) - \Xi^{(n)}(x)| \lesssim & \eps \| \p_y u_{B, R} \|_{L^\infty} |\Xi^{(n)}(x) - \Xi^{(n-1)}(x)|, \qquad n \ge 1, \\
|\Xi^{(1)}(x)| \lesssim  &\| u_{B, R} \|_{L^\infty}.
\end{align*}
From here, the standard contraction mapping principle on $\mathbb{R}$ ensures a fixed point ($x$ by $x$), and moreover the fixed point satisfies the bound $\| \Xi \|_{L^\infty_x} \lesssim \| u_{B, R} \|_{L^\infty_{xy}}$.

To obtain the bound on $\Xi'(x)$, we use that we have existence of $\Xi(x)$ to the fixed point equation
\begin{align*}
\eps \Xi(x) = I_{FS}(x, -\eps u_{B, R}(x, \Lambda_G(x) + \eps \Xi(x))),
\end{align*}
which, upon differentiating in $x$ yields 
\begin{align*}
\eps \Xi'(x) = & \p_x I_{FS}(x,  -\eps u_{B, R}(x, \Lambda_G(x) + \eps \Xi(x))) \\
&- \eps \p_a I_{FS}(x, -\eps u_{B, R}(x, \Lambda_G(x) + \eps \Xi(x)))\p_x u_{B, R}(x, \Lambda_G(x) + \eps \Xi(x)) \\
&- \eps \p_a I_{FS}(x, -\eps u_{B, R}(x, \Lambda_G(x) + \eps \Xi(x))) \p_y u_{B, R}(x, \Lambda_G(x) + \eps \Xi(x))\Lambda_G'(x) \\
& -  \eps^2 \p_a I_{FS}(x, -\eps u_{B, R}(x, \Lambda_G(x) + \eps \Xi(x))) \p_y u_{B, R}(x, \Lambda_G(x) + \eps \Xi(x))\Xi'(x) \\
=: & a_1 + a_2 + a_3 + \eps^2 a_4 \Xi'(x).  
\end{align*}
We therefore obtain 
\begin{align*}
 \Xi'(x) = \frac{1}{1-\eps a_4} \frac{1}{\eps} (a_1 + a_2 + a_3). 
\end{align*}
We proceed to estimate 
\begin{align*}
|a_1| \le & |\p_x I_{FS}(x, 0)| + \| \p_{xa} I_{FS} \|_{L^\infty} \eps \| u_{B, R} \|_{L^\infty} \lesssim \eps \| u_{B, R} \|_{L^\infty}, \\
|a_2| \lesssim &\eps \| \p_x u_{B, R} \|_{L^\infty}, \\
|a_3| \lesssim & \eps \|\p_y u_{B, R} \|_{L^\infty}, \\
|a_4| \lesssim & \| \p_y u_{B, R} \|_{L^\infty},
\end{align*}
where for the $a_1$ estimate we use that $I_{FS}(x, 0) = 0$ for all $x$ so that $\p_x I_{FS}(x, 0) = 0$. The proof of the lemma is complete. 
\end{proof}

We are now able to prove Proposition \ref{pro:effective}.
\begin{proof}[Proof of Proposition \ref{pro:effective}] First of all the bounds \eqref{xi:est} show that $\Lambda$ is bounded above and below. We may then translate the information on the left-hand sides of \eqref{lyu:1} -- \eqref{lyu:3} into bounds of the Effective norm.
\begin{align} \label{eff:global:bd}
\| \psi_R \|_{\mathrm{Eff}_n} \lesssim \| \psi \|_{\mathrm{Global}_n}
\end{align}
Indeed, this follows from the chain rule identity $\frac{d}{dy} = \frac{1}{\Lambda(x)} \frac{d}{dz}$, and estimates \eqref{lyu:1} -- \eqref{lyu:3} upon simply noting the equivalence of $L^2_z$ to $L^2_y$ and similarly $L^2_{sz}$ to $L^2_{xy}$.

Second, the assumed bounds on $u_{B, R}$ and $\vec{\bold{C}}$ and the boundary conditions on $\mathcal{R}$ guarantee that we may apply Proposition \ref{pro:global}. Therefore, we may successively apply \eqref{eff:global:bd}, \eqref{main:G}, and \eqref{LbarL} to obtain 
\begin{align} \n
\| \psi_R \|_{\mathrm{Eff}_n} \lesssim  \| \psi \|_{\mathrm{Global}_n} \lesssim &   \sum_{j = 0}^3 \| \p_z^j \bold{R} \chi_{O,j} \langle z \rangle^n \|_{L^2_{sz}} + \sum_{j = 0}^1 \| \p_z^j \bold{R} \chi_I \|_{L^2_{sz}} + \| d(z) \p_z^2 \bold{R} \chi_I \|_{L^2_{sz}} \\ \n
&+ \bar{L}^{\frac16} \sum_{i = 0}^2 \sum_{\iota \in \mathrm{Left, Right}}  \| \langle \rho \rangle \p_\rho^i \Xi_{\iota} \|_{L^2_\rho} +  \bar{L}^{\frac16} \sum_{i = 1}^2 \sum_{\iota \in \mathrm{Left, Right}} \|\rho^i \p_\rho^i \Xi_{\iota}\|_{L^2_\rho} \\  \n
& + \sum_{j = 0}^3 ( \| \p_z^j \zeta_{\mathrm{Left}} \chi_{O,j}^+ \langle z \rangle^n \|_{L^2_z} + \| \p_z^j \zeta_{\mathrm{Right}} \chi_{O,j}^- \|_{L^2_z} ),
\end{align}
which implies \eqref{main:E} upon invoking $\bold{R} = \lambda^2 \mathcal{R}$, using $\frac{d}{dy} = \frac{1}{\Lambda(x)} \frac{d}{dz}$, and changing from $L^2_{sz}$ to $L^2_{xy}$. The proof of \eqref{main:Ekju} is nearly identical. The proposition is proven. 
\end{proof}

\section{Bounds on $k$-dependent Source Terms} \label{ksec}

In this section, our main objective will be to provide bounds on the source terms appearing in \eqref{main:E} and \eqref{main:Ekju}, when the abstract source term $\mathcal{R}$ is replaced by $\mathcal{R}_k$, defined in \eqref{defRKRK}. 

We will moreover need to choose the weight parameter $n$ appearing in \eqref{main:E} and \eqref{main:Ekju} in a $k$-dependent manner. Recall that we define the weight parameter $N_0 = 3k_\ast + 5$. Next, we define the sequence of parameters  
\begin{align} \label{def:NsubK}
N_k := N_0 - k.
\end{align}
Such a choice will be motivated by the eventual form of our linear norm, which we define here \footnote{the verbiage ``Linear" is due to the eventual estimate \eqref{sbU:1}; without nonlinear terms we obtain a bound $\| \psi_R \|_{\mathrm{Linear}} \lesssim \bold{Dat}$.}
\begin{align} \label{Linear:norm}
\| \psi_R \|_{\text{Linear}} := \sum_{k= 0}^{k_\ast} \| \psi_{R,k} \|_{\text{Eff}_{N_k}}
\end{align}
Of note is the \textit{downward weight cascade} created by decreasing $N_k$ at each level of derivative. This is created due to the linear growth of $v_{FS}$, which precisely appears in one term below, namely \eqref{dwc:1}.

We now introduce the notation for the source terms
\begin{align} \label{defSourcek}
\mathrm{\bold{Source}}_k := &\sum_{j = 0}^3 \| \p_y^j \mathcal{R}_k \chi_{O,j} \langle y \rangle^{N_k} \|_{L^2_{xy}} + \sum_{j = 0}^1 \| \p_y^j \mathcal{R}_k \chi_I \|_{L^2_{xy}} + \| d(x,y) \p_y^2 \mathcal{R}_k \chi_I \|_{L^2_{xy}}, 
\end{align}
and
\begin{align} \label{defSourcefull}
\mathrm{\bold{Source}} := & \sum_{k = 0}^{k_\ast} \mathrm{\bold{Source}}_k.
\end{align}

The main proposition of this section will be as follows. 
\begin{proposition}\label{pro:Linear} Fix any $L_{\mathrm{Max}} < \infty$, and let $0 < L < L_{\mathrm{Max}}$. The following bounds are valid: 
\begin{align} \label{sourke}
\mathrm{\bold{Source}}_k \lesssim  L \sum_{k'= 0}^{k-1} \| \psi_{R, k'} \|_{\mathrm{Eff}_{N_{k'}}} + \eps L^{-1} (\sum_{k = 0}^{\lfloor \frac{k_\ast + 1}{2} \rfloor} \sum_{j = 0}^4 \| \p_y^j u_{R,k} \|_{L^\infty} ) \| \psi_R \|_{\mathrm{Linear}} + \eps L^{-1} \| \psi_R \|_{\mathrm{Linear}}^2.
\end{align}
\end{proposition}

\begin{remark}In this section, we collect estimates on the source terms, $\mathrm{\bold{Source}}_k$, which appear on the right-hand sides of \eqref{main:E} and \eqref{main:Ekju}, uniformly over $L \in (0, L_{\mathrm{Max}})$, where $L_{\mathrm{Max}}$ is fixed (as in Theorem \ref{thm:res}). Therefore, the forthcoming bounds will depend poorly on large values of $L_{\mathrm{Max}}$, but we choose to suppress this dependence, and we freely use the inequality $L \lesssim C(L_{\mathrm{Max}}) \lesssim 1$. As all the lemmas in this section are valid on $(0, L_{\mathrm{Max}})$, we do not repeatedly state this for each lemma. 
\end{remark}

\subsection{Bounds on Linear Source Terms}

We begin with estimating the source terms appearing with outer cutoffs. 

\begin{lemma} Fix any $0 \le k \le k_\ast$. The following inequality is valid: 
\begin{align} \label{RkLin:bd:1}
  \sum_{j = 0}^3 \| \p_y^j \mathcal{R}_{k, \mathrm{Lin}} \chi_{O,j} \langle y \rangle^{N_k} \|_{L^2_{xy}}^2 \lesssim L \sum_{k'= 0}^{k-1} \| \psi_{R, k'} \|_{\mathrm{Eff}_{N_{k'}}}.
\end{align}  
\end{lemma}
\begin{proof} We proceed term by term, starting with $\mathcal{R}_{k, \text{Lin}}^{(1)}$ In the forthcoming calculations, we will often use the following property ($u_{FS,k'} = L^{k'} \p_x^{k'} u_{FS}$). We have for any $j$ in the range $0 \le j \le 3$, 
\begin{align} \n
\| \p_y^j \mathcal{R}_{k, \mathrm{Lin}}^{(1)} \langle y \rangle^{N_k}\chi_{O,j} \|_{L^2_{xy}} \lesssim & \bold{1}_{k \ge 2}   \sum_{k' = 0}^{k-2} \sum_{j' = 0}^j L^{k-k' - 1} \| \p_y^{j'} u_{R,k'+1} \langle y \rangle^{N_{k' + 1}} \chi_{O,j} \|_{L^2_{xy}}, \\ \label{down:casc}
\| \p_y^j \mathcal{R}_{k, \mathrm{Lin}}^{(2)} \langle y \rangle^{N_k} \chi_{O,j}\|_{L^2_{xy}} \lesssim & \bold{1}_{k \ge 1} \sum_{k' = 0}^{k-1} \sum_{j' = 0}^j L^{k-k'} \| \p_y^{j'} u_{R,k'} \langle y \rangle^{N_{k' + 1}} \chi_{O,j} \|_{L^2_{xy}},
\end{align}
both of which are easily seen to be bounded by the right-hand side of \eqref{RkLin:bd:1}, upon invoking \eqref{eff:norm}. Above, we have used the fact that $k' + 1 \le k - 1$, and hence $N_{k' + 1} \ge N_{k-1} \ge N_k$, according to \eqref{def:NsubK}. We have also used that $L^{k-k'} \le C(L_{\mathrm{Max}}) L \lesssim L$ and $L^{k-k' - 1} \le C(L_{\mathrm{Max}}) L \lesssim L$. 

The next term is precisely the term that requires us to introduce the ``downward weight cascade", \eqref{def:NsubK}. Indeed, we have
\begin{align} \n
\| \p_y^j \mathcal{R}_{k, \mathrm{Lin}}^{(3)} \langle y \rangle^{N_k}\chi_{O,j} \|_{L^2_{xy}} \lesssim & \bold{1}_{k \ge 1} \sum_{k' = 0}^{k-1} \sum_{j' = 0}^{j-1} \| \p_x^{k-k'} \p_y^{j-j'} v_{FS} \langle y \rangle^{N_k} \|_{L^\infty} \| L^{k-k'} \p_y^{j' + 1}  u_{R,k'} \chi_{O,j} \|_{L^2_{xy}} \\ \label{dwc:1}
& + \bold{1}_{k \ge 1} \sum_{k' = 0}^{k-1}  \| \frac{ \p_x^{k-k'}  v_{FS}}{\sqrt{\bar{W}}} \langle y \rangle^{-1} \|_{L^\infty} \| L^{k-k'} \sqrt{\bar{W}} \p_y^{j + 1}  u_{R,k'} \langle y \rangle^{N_k + 1} \chi_{O,j} \|_{L^2_{xy}},
\end{align}
which is bounded by the right-hand side of \eqref{RkLin:bd:1} upon invoking that $N_{k} + 1 \le N_{k-1} \le N_{k'}$, and again that $L^{k-k'} \le  C(L_{\mathrm{Max}})L \lesssim L$.  

For the next term, we use the rapid decay of $\p_y u_{FS}$ to estimate 
\begin{align*}
\| \p_y^j \mathcal{R}_{k, \mathrm{Lin}}^{(4)} \langle y \rangle^n \chi_{O,j} \|_{L^2_{xy}} \lesssim & \bold{1}_{k \ge 2} \sum_{j' = 1}^j \sum_{k' = 0}^{k-2} \| \p_y^{j-j' + 1} \p_x^{k-k'} u_{FS} \langle y \rangle^{n} \|_{L^\infty} L^{k-k' - 1} \| \p_y^{j' - 1} u_{R,k'+1} \chi_{O,j} \|_{L^2_{xy}} \\
& + \bold{1}_{k \ge 2} \sum_{k' = 0}^{k-2} \| \p_x^{k-k'} \p_y u_{FS} \langle y \rangle^{n + 1} \|_{L^\infty}  L^{k-k' - 1} \| \frac{\psi_{R,k'+1}}{\langle y \rangle} \chi_{O,j} \|_{L^2_{xy}},
\end{align*}
which upon using that $\chi_{O,j} \le \chi_{O,j'-1}$, $L^{k-k' - 1} \le  C(L_{\mathrm{Max}}) L \lesssim L$ (for the range of $k, k'$ in the summations above) and $\| \frac{\psi_{R,k'+1}}{\langle y \rangle} \chi_{O,j} \|_{L^2_{xy}} \lesssim \| \psi_{R,k'+1} \|_{L^\infty_{xy}}$ also immediately proves that this term is controlled by the right-hand side of \eqref{RkLin:bd:1}. The lemma is proven.
\end{proof}

\begin{lemma} Fix any $0 \le k \le k_\ast$. The following inequality is valid: 
\begin{align} \label{pu:2}
\sum_{j = 0}^1 \| \p_y^j \mathcal{R}_{k, \mathrm{Lin}} \chi_I \|_{L^2_{xy}} + \| d(x,y) \p_y^2 \mathcal{R}_{k, \mathrm{Lin}} \chi_I \|_{L^2_{xy}} \lesssim L \sum_{k'= 0}^{k-1} \| \psi_{R, k'} \|_{\mathrm{Eff}_{N_{k'}}}.
\end{align}
\end{lemma}
\begin{proof} We have from \eqref{RkLin} immediately the inequalities, upon invoking the definition of the effective norm, \eqref{eff:norm},
\begin{align*}
\sum_{j = 0}^1 \| \p_y^j \mathcal{R}_{k, \mathrm{Lin}} \chi_I \|_{L^2_{xy}} \lesssim &  C(L_{\mathrm{Max}}) L \sum_{k' = 0}^{k-1} \sum_{j' = 0}^3 \| \p_y^{j'} \psi_{k'} \chi_I \|_{L^2_{xy}} \lesssim L  \sum_{k= 0}^{k-1} \| \psi_{R, k'} \|_{\mathrm{Eff}_{N_{k'}}}, \\
 \| d(x,y) \p_y^2 \mathcal{R}_{k, \mathrm{Lin}} \chi_I \|_{L^2_{xy}} \lesssim &C(L_{\mathrm{Max}}) L\sum_{k' = 0}^{k-1} \sum_{j' = 0}^4 \| d(x, y) \p_y^{j'} \psi_{k'} \chi_I \|_{L^2_{xy}} \lesssim L \sum_{k'= 0}^{k-1} \| \psi_{R, k'} \|_{\mathrm{Eff}_{N_{k'}}},
\end{align*}
where we have generated a factors of $L$ upon again using the property $u_{FS,k'} = L^{k'} \p_x^{k'} u_{FS}$ just as in the previous lemma. The lemma is proven.
\end{proof}

\subsection{Bounds on Semilinear Source Terms}

We are now ready to estimate the semilinear source terms. First, we begin with those semilinear terms containing outer cutoffs.
\begin{lemma} Fix any $0 \le k \le k_\ast$. The following bounds are valid: 
\begin{align} \n
 \sum_{j = 0}^3 \| \p_y^j \mathcal{R}_{k; \mathrm{Semi}} \langle y \rangle^{N_k} \chi_{O,j} \|_{L^2_{xy}} \lesssim & \eps L^{-1} (\sum_{k = 0}^{\lfloor \frac{k_\ast + 1}{2} \rfloor} \sum_{j = 0}^4 \| \p_y^j u_{R,k} \|_{L^\infty} ) \sum_{k'= 0}^{k-1} \| \psi_{R, k'} \|_{\mathrm{Eff}_{N_{k'}}} \\ \label{noise:1}
 &+ \eps L^{-1}(\sum_{k'= 0}^{k-1} \| \psi_{R, k'} \|_{\mathrm{Eff}_{N_{k'}}})^2.
\end{align}
\end{lemma}
\begin{proof} We recall the definition of $\mathcal{R}_{k; \text{Semi}}$ from \eqref{RkSemi}. We establish \eqref{noise:1} for each fixed $j$ in the interval $0 \le j \le 3$. First, 
\begin{align} \n
 \| \p_y^j \mathcal{R}^{(1)}_{k; \text{Semi}} \langle y \rangle^{N_k} \chi_{O,j} \|_{L^2_{xy}}  \lesssim & \eps 1_{k \ge 3} \sum_{k' = 1}^{k-2} \sum_{j' = 0}^j L^{-1} \| \p_y^{j-j'} u_{R,k-k'} \p_y^{j'} u_{R,k' + 1} \langle y \rangle^{N_k} \chi_{O,j} \|_{L^2_{xy}} \\ \n
\lesssim &  \eps 1_{k \ge 3} \sum_{k' = 1}^{\lfloor \frac{k+1}{2} \rfloor - 1} \sum_{j' = 0}^jL^{-1} \| \p_y^{j-j'} u_{R,k-k'} \p_y^{j'} u_{R,k' + 1} \langle y \rangle^{N_k} \chi_{O,j} \|_{L^2_{xy}} \\ \n
& + \eps 1_{k \ge 3} \sum_{k' = \lfloor \frac{k+1}{2} \rfloor}^{k-2} \sum_{j' = 0}^jL^{-1} \| \p_y^{j-j'} u_{R,k-k'} \p_y^{j'} u_{R,k' + 1} \langle y \rangle^{N_k} \chi_{O,j} \|_{L^2_{xy}} \\ \n
\lesssim & \eps L^{-1} ( \sum_{k' = 0}^{\lfloor \frac{k+1}{2} \rfloor} \sum_{j' = 0}^j \| \p_y^{j'} u_{R,k'} \|_{L^\infty}) \sum_{k' = 0}^{k-1} \sum_{j' = 0}^j \| \p_y^{j'} u_{R,k'} \langle y \rangle^{N_{k'}} \chi_{O,j'} \|_{L^2_{xy}} \\ \label{hurricane:1}
\lesssim & \eps L^{-1} (\sum_{k = 0}^{\lfloor \frac{k_\ast + 1}{2} \rfloor} \sum_{j = 0}^4 \| \p_y^j u_{R,k} \|_{L^\infty} ) \sum_{k'= 0}^{k-1} \| \psi_{R, k'} \|_{\text{Eff}_{N_{k'}}},
\end{align}
where we have invoked that $N_{k'} \ge N_k$ for $k' \le k$, as well as the definition of the effective norm, \eqref{eff:norm}.

Next, we split the second semilinear term as follows:
\begin{align*}
\| \p_y^j \mathcal{R}^{(2)}_{k; \text{Semi}} \langle y \rangle^{N_k} \chi_{O,j} \|_{L^2_{xy}} \lesssim A + B,
\end{align*}
where
\begin{align*}
A := &1_{k \ge 3} \eps L^{-1} \sum_{k' = 1}^{k-2} \sum_{j' = 0}^{j-1} \| \p_y^{j-j' - 1} u_{R,k' + 1} \p_y^{j' + 1} u_{R, k-k'}  \langle y \rangle^{N_k} \chi_{O,j} \|_{L^2_{xy}}, \\
B := &1_{k \ge 3} \eps L^{-1} \sum_{k' = 1}^{k-2} \| \psi_{R,k' + 1} \p_y^{j + 1} u_{R,k-k'} \langle y \rangle^{N_k} \chi_{O,j} \|_{L^2_{xy}}.
\end{align*}
Clearly, we have 
\begin{align*}
B \lesssim & \eps L^{-1} (\sum_{k' = 0}^{k-1} \|\frac{ \psi_{R,k'}}{\sqrt{\bar{W}}} \|_{L^\infty}) (\sum_{k' = 0}^{k-1} \|  \sqrt{\bar{W}} \p_y^{j+1} u_{R,k'} \langle y \rangle^{N_{k'}} \chi_{O,j} \|_{L^2_{xy}} ) \\
\lesssim & \eps L^{-1} (\sum_{k' = 0}^{k-1} \| \psi_{R,k'} \|_{L^\infty} + \|u_{R,k'}\|_{L^\infty})  \sum_{k'= 0}^{k-1} \| \psi_{R, k'} \|_{\text{Eff}_{N_{k'}}} \\
\lesssim & \eps L^{-1}( \sum_{k'= 0}^{k-1} \| \psi_{R, k'} \|_{\text{Eff}_{N_{k'}}} )^2. 
\end{align*}
For the $A$ term, noticing that the max $\p_y$ is $\p_y^j$ on both terms, the max $\p_x$ is $k-1$, and that since $(k' + 1) + (k-k') = k+1$ that one of either $k' + 1$ or $k-k'$ has to be $\le \lfloor \frac{k+1}{2} \rfloor$, we have 
\begin{align} \n
A \lesssim & \eps L^{-1}( \sum_{k' = 0}^{\lfloor \frac{k+1}{2} \rfloor} \sum_{j' = 0}^j \| \p_y^{j'} u_{R,k'} \|_{L^\infty} ) (\sum_{k' = 0}^{k-1} \sum_{j' = 0}^j \| \p_y^{j'} u_{R,k'} \langle y \rangle^{N_{k'}} \chi_{O,j'} \|_{L^2_{xy}}  ) \\ \label{rc:2}
\lesssim & \eps L^{-1} (\sum_{k = 0}^{\lfloor \frac{k_\ast + 1}{2} \rfloor} \sum_{j = 0}^4 \| \p_y^j u_{R,k} \|_{L^\infty} ) \sum_{k'= 0}^{k-1} \| \psi_{R, k'} \|_{\text{Eff}_{N_{k'}}}.
\end{align}
Finally, the terms $\mathcal{R}^{(3)}_{k; \text{Semi}}\chi_{O,j}$ and $\mathcal{R}^{(4)}_{k; \text{Semi}} \chi_{O,j}$ are estimated in an almost identical manner to the above. The lemma is proven.
\end{proof}

\begin{lemma}Fix any $0 \le k \le k_\ast$. The following estimate is valid:
\begin{align} \n
 &\sum_{j = 0}^1 \| \p_y^j \mathcal{R}_{k; \mathrm{Semi}} \chi_{I} \|_{L^2_{xy}} +\| d(x,y) \p_y^2 \mathcal{R}_{k; \mathrm{Semi}} \chi_{I} \|_{L^2_{xy}} \\ \label{noise:2}
 \lesssim & \eps L^{-1} (\sum_{k = 0}^{\lfloor \frac{k_\ast + 1}{2} \rfloor} \sum_{j = 0}^4 \| \p_y^j u_{R,k} \|_{L^\infty} ) \sum_{k'= 0}^{k_\ast} \| \psi_{R, k'} \|_{\mathrm{Eff}_{N_{k'}}} + \eps L^{-1}(\sum_{k'= 0}^{k-1} \| \psi_{R, k'} \|_{\mathrm{Eff}_{N_{k'}}})^2.
\end{align}
\end{lemma}
\begin{proof} First, we have for $j = 0, 1, 2$, the following estimate which follows in a nearly identical fashion to \eqref{hurricane:1}
\begin{align} \n
 \| \p_y^j \mathcal{R}^{(1)}_{k; \text{Semi}}\chi_{I} \|_{L^2_{xy}}  \lesssim & \eps 1_{k \ge 3} \sum_{k' = 1}^{k-2} \sum_{j' = 0}^2 L^{-1} \| \p_y^{j-j'} u_{R,k-k'} \p_y^{j'} u_{R,k' + 1} \chi_I \|_{L^2_{xy}} \\ \n
\lesssim &  \eps 1_{k \ge 3} \sum_{k' = 1}^{\lfloor \frac{k+1}{2} \rfloor - 1} \sum_{j' = 0}^2 L^{-1} \| \p_y^{j-j'} u_{R,k-k'} \p_y^{j'} u_{R,k' + 1} \chi_I \|_{L^2_{xy}} \\ \n
& + \eps 1_{k \ge 3} \sum_{k' = \lfloor \frac{k+1}{2} \rfloor}^{k-2} \sum_{j' = 0}^2 L^{-1} \| \p_y^{j-j'} u_{R,k-k'} \p_y^{j'} u_{R,k' + 1} \chi_I \|_{L^2_{xy}} \\ \n
\lesssim & \eps L^{-1} ( \sum_{k' = 0}^{\lfloor \frac{k+1}{2} \rfloor} \sum_{j' = 0}^2 \| \p_y^{j'} u_{R,k'} \|_{L^\infty}) \sum_{k' = 0}^{k-1} \sum_{j' = 0}^2 \| \p_y^{j'} u_{R,k'} \chi_I \|_{L^2_{xy}} \\ \n
\lesssim & \eps L^{-1} (\sum_{k = 0}^{\lfloor \frac{k_\ast + 1}{2} \rfloor} \sum_{j = 0}^4 \| \p_y^j u_{R,k} \|_{L^\infty} ) \sum_{k'= 0}^{k-1} \| \psi_{R, k'} \|_{\text{Eff}_{N_{k'}}},
\end{align}
where we have invoked the definition \eqref{eff:norm}.  

Next, we split the second semilinear term as follows for $j = 0, 1$:
\begin{align*}
\| \p_y^j \mathcal{R}^{(2)}_{k; \text{Semi}} \chi_I \|_{L^2_{xy}} \lesssim A + B, \qquad \| d(x,y) \p_y^2 \mathcal{R}^{(2)}_{k; \text{Semi}} \chi_I \|_{L^2_{xy}} \lesssim C + D,
\end{align*}
where
\begin{align*}
A := &1_{k \ge 3} \eps L^{-1} \sum_{k' = 1}^{k-2} \sum_{j' = 0}^{j-1} \| \p_y^{j-j' - 1} u_{R,k' + 1} \p_y^{j' + 1} u_{R, k-k'}  \chi_I \|_{L^2_{xy}}, \\
B := &1_{k \ge 3} \eps L^{-1} \sum_{k' = 1}^{k-2} \| \psi_{R,k' + 1} \p_y^{j + 1} u_{R,k-k'} \chi_I \|_{L^2_{xy}}, \\
C := &1_{k \ge 3} \eps L^{-1} \sum_{k' = 1}^{k-2} \sum_{j' = 0}^{1} \|d(x,y) \p_y^{j-j' - 1} u_{R,k' + 1} \p_y^{j' + 1} u_{R, k-k'}  \chi_I \|_{L^2_{xy}}, \\
D := &1_{k \ge 3} \eps L^{-1} \sum_{k' = 1}^{k-2} \|d(x,y) \psi_{R,k' + 1} \p_y^{3} u_{R,k-k'} \chi_I \|_{L^2_{xy}}
\end{align*}
For the $A$ and $C$ terms, we have a similar argument to \eqref{rc:2}, which gives
\begin{align} \n
A + C \lesssim & \eps L^{-1}( \sum_{k' = 0}^{\lfloor \frac{k+1}{2} \rfloor} \sum_{j' = 0}^2 \| \p_y^{j'} u_{R,k'} \|_{L^\infty} ) (\sum_{k' = 0}^{k-1} \sum_{j' = 0}^2 \| \p_y^{j'} u_{R,k'} \chi_I \|_{L^2_{xy}}  ) \\ \label{rc:4}
\lesssim & \eps L^{-1} (\sum_{k = 0}^{\lfloor \frac{k_\ast + 1}{2} \rfloor} \sum_{j = 0}^4 \| \p_y^j u_{R,k} \|_{L^\infty} ) \sum_{k'= 0}^{k-1} \| \psi_{R, k'} \|_{\text{Eff}_{N_{k'}}}.
\end{align}
For the $B$ term, we have for $j = 0, 1$: 
\begin{align*}
B \lesssim & \eps L^{-1} (\sum_{k' = 0}^{k-1} \| \psi_{R,k'} \|_{L^\infty}) (\sum_{k' = 0}^{k-1} \| \p_y^{j+1} u_{R,k'} \chi_I \|_{L^2_{xy}} ) \\
\lesssim & \eps L^{-1} (\sum_{k' = 0}^{k-1} \| \psi_{R,k'} \|_{L^\infty} )  \sum_{k'= 0}^{k-1} \| \psi_{R, k'} \|_{\text{Eff}_{N_{k'}}} \lesssim  \eps L^{-1}( \sum_{k'= 0}^{k-1} \| \psi_{R, k'} \|_{\text{Eff}_{N_{k'}}} )^2,
\end{align*}
and essentially identically when $j = 2$:
\begin{align*}
D \lesssim & \eps L^{-1} (\sum_{k' = 0}^{k-1} \| \psi_{R,k'} \|_{L^\infty}) (\sum_{k' = 0}^{k-1} \| d(x, y) \p_y^{3} u_{R,k'} \chi_I \|_{L^2_{xy}} ) \\
\lesssim & \eps L^{-1} (\sum_{k' = 0}^{k-1} \| \psi_{R,k'} \|_{L^\infty} )  \sum_{k'= 0}^{k-1} \| \psi_{R, k'} \|_{\text{Eff}_{N_{k'}}} \lesssim  \eps L^{-1}( \sum_{k'= 0}^{k-1} \| \psi_{R, k'} \|_{\text{Eff}_{N_{k'}}} )^2,
\end{align*}
where we have invoked the definition \eqref{eff:norm}.

We now turn to $\mathcal{R}_{k, \text{Semi}}^{(3)}$ and $\mathcal{R}_{k, \text{Semi}}^{(4)}$, for which we estimate as follows: 
\begin{align*}
&\sum_{j = 0}^2 ( \| \p_y^j \mathcal{R}_{k, \text{Semi}}^{(3)} \chi_I \|_{L^2_{xy}} +  \| \p_y^j \mathcal{R}_{k, \text{Semi}}^{(4)} \chi_I \|_{L^2_{xy}}) \\
\lesssim & \eps L^{-1} \sum_{k \le k_s} (\sum_{j = 0}^3  \| \p_y^j u_{R, \text{Pert}} \|_{L^\infty}) (\| \psi_{R,k+1} \|_{L^\infty} + \sum_{j = 0}^2 \| \p_y^j u_{R,k+1} \|_{L^2_{xy}}) \\
\lesssim & \eps L^{-1} \sum_{k \le k_s} (\sum_{j = 0}^3  \| \p_y^j u_R \|_{L^\infty}) (\| \psi_{R,k+1} \|_{L^\infty} + \sum_{j = 0}^2 \| \p_y^j u_{R,k+1} \|_{L^2_{xy}}) \\
\lesssim & \eps L^{-1}(\sum_{j = 0}^3  \| \p_y^j u_R \|_{L^\infty}) ( \sum_{k \le k_\ast} \| \psi_{R,k} \|_{\text{Eff}_{N_k}}).
\end{align*}
Above, we have used the definition of $u_{R, \text{Pert}}$, \eqref{uRpert}, as well as the definition of the effective norm, \eqref{eff:norm}. The lemma is proven.
\end{proof}

\subsection{Proof of Proposition \ref{pro:Linear}}

\begin{proof}[Proof of Proposition \ref{pro:Linear}]First, we note that by \eqref{RkLin:bd:1} and \eqref{pu:2}, we have 
\begin{align} \n
& \sum_{j = 0}^3 \| \p_y^j \mathcal{R}_{k, \mathrm{Lin}} \chi_{O,j} \langle y \rangle^{N_k} \|_{L^2_{xy}} +   ( \sum_{j = 0}^1 \| \p_y^j\mathcal{R}_{k, \mathrm{Lin}} \chi_I \|_{L^2_{xy}} + \| d(x,y) \p_y^2 \mathcal{R}_{k, \mathrm{Lin}} \chi_I \|_{L^2_{xy}}) \\ \label{mercmerc:2}
  \lesssim & L \sum_{k'= 0}^{k-1} \| \psi_{R, k'} \|_{\text{Eff}_{N_{k'}}},
\end{align}
Similarly, we note that by \eqref{noise:1} and \eqref{noise:2}, we have 
\begin{align} \n
&\sum_{j = 0}^3 \| \p_y^j \mathcal{R}_{k, \mathrm{Semi}} \chi_{O,j} \langle y \rangle^{N_k} \|_{L^2_{xy}} +  ( \sum_{j = 0}^1 \| \p_y^j\mathcal{R}_{k, \mathrm{Semi}} \chi_I \|_{L^2_{xy}} + \| d(x,y) \p_y^2 \mathcal{R}_{k, \mathrm{Semi}} \chi_I \|_{L^2_{xy}}) \\ \label{mercmerc:3}
\lesssim & \eps L^{-1} (\sum_{k = 0}^{\lfloor \frac{k_\ast + 1}{2} \rfloor} \sum_{j = 0}^4 \| \p_y^j u_{R,k} \|_{L^\infty} ) \| \psi_R \|_{\mathrm{Linear}} + \eps L^{-1} \| \psi_R \|_{\mathrm{Linear}}^2,
\end{align}
for each $k$. The proposition is proven. 
\end{proof}

\section{Bounds on Boundary Data} \label{bdsdata:Sec}

In this section, our main objective will be to provide bounds on the data terms appearing in \eqref{main:E} and \eqref{main:Ekju}. We will be applying Proposition \ref{pro:effective} with the specific choices of data instead of the placeholders, $\zeta_{\iota} = U_{\iota; k}$ and $\Xi_{\iota}(\rho) = \Omega_{I, \iota; k}(Z)$. To compactify the notations involving the data elements, we introduce the following two quantities: 
\begin{align} \n
\mathrm{\bold{Dat}}_k := & L^{\frac16} \sum_{i = 0}^2 \sum_{\iota \in \mathrm{Left, Right}}  \| \langle \rho \rangle \p_\rho^i \Omega_{I, \iota; k} \|_{L^2_\rho} + L^{\frac16} \sum_{i = 1}^2 \sum_{\iota \in \mathrm{Left, Right}} \|\rho^i \p_\rho^i \Omega_{I, \iota; k} \|_{L^2_\rho} \\  \label{dat:k:defn}
& + \sum_{j = 0}^3 ( \| \p_z^j U_{\mathrm{Left}; k} \chi_{O,j}^+ \langle z \rangle^{N_k} \|_{L^2_z} + \| \p_z^j U_{\mathrm{Right}; k} \chi_{O,j}^- \|_{L^2_z} ), \\ \n
\mathrm{\bold{Dat}}_{\mathrm{Below},k} := & L^{\frac16} \sum_{i = 0}^2 \sum_{\iota \in \mathrm{Left, Right}}  \| \langle \rho \rangle \p_\rho^i \Omega_{I, \iota; k} \|_{L^2_\rho} +  L^{\frac16} \sum_{i = 1}^2 \sum_{\iota \in \mathrm{Left, Right}} \|\rho^i \p_\rho^i \Omega_{I, \iota; k} \|_{L^2_\rho} \\  \label{dat:k:defn}
& + \sum_{j = 0}^3  ( L^{\frac13}\| \p_z^j U_{\mathrm{Left}; k} \chi_{O,j}^+ \langle z \rangle^{N_k} \|_{L^2_z} + \| \p_z^j U_{\mathrm{Right}; k} \chi_{O,j}^- \|_{L^2_z} ),
\end{align}
which appear as a result of inserting the particular choices $\zeta_{\iota} = U_{\iota; k}$ and $\Xi_{\iota}(\rho) = \Omega_{I, \iota; k}(Z)$ into the data terms of \eqref{main:E}, \eqref{main:Ekju}. Next, we have the total data contributions after taking summation in $k$
\begin{align} \label{boldDat}
\mathrm{\bold{Dat}} := \sum_{k = 0}^{k_\ast} \mathrm{\bold{Dat}}_k, \qquad \mathrm{\bold{Dat}}_{\mathrm{Below}} := \sum_{k = 0}^{k_\ast} \mathrm{\bold{Dat}}_{\mathrm{Below}, k}
\end{align}
Our main goal in this section is to provide estimates on the data terms, $\mathrm{\bold{Dat}}$,  $\mathrm{\bold{Dat}}_{\mathrm{Below}}$, defined in \eqref{boldDat}, in terms of (1) the given data and (2) the solution itself. We will need to define
\begin{align} \label{def:gamma:u:k}
\gamma_{u; k} := u_{R,k}(1, y_1(\eps)), \qquad  \psi_{u; k} := \psi_{R,k}(1, y_1(\eps)).
\end{align}
We recall the form of the given data, \eqref{PhiIota:1} -- \eqref{PhiIota:4}. Our main proposition in this section is the following. 
\begin{proposition} \label{prop:data:M} Fix any $L_{\mathrm{Max}} < \infty$, and let $0 < L < L_{\mathrm{Max}}$. Assume the compatibility conditions \eqref{compat:1} -- \eqref{these:numbers}. Then the following bounds hold: 
\begin{align} \n
\mathrm{\bold{Dat}} \lesssim & L^{\frac16} (\| \mathring{F}_{\mathrm{Left}} \langle \eta \rangle^3 \|_{H^{2 + 3k_\ast}_\eta} + \| \mathring{F}_{\mathrm{Right}} \langle \eta \rangle^3 \|_{H^{2 + 3k_\ast}_\eta}) + \|  G_{\mathrm{Right}} \|_{H^{3 + 3k_\ast}_y} + \|  G_{\mathrm{Left}} \langle y \rangle^{N_0} \|_{H^{3 + 2k_\ast}_y} \\ \label{est:data:M}
& + \sum_{k' = 0}^{k_\ast} (|\gamma_{u;k'}| + |\gamma_{\psi; k'}|), \\ \n
\mathrm{\bold{Dat}}_{\mathrm{Below}} \lesssim & L^{\frac16} (\| \mathring{F}_{\mathrm{Left}} \langle \eta \rangle^3 \|_{H^{2 + 3k_\ast}_\eta} + \| \mathring{F}_{\mathrm{Right}} \langle \eta \rangle^3 \|_{H^{2 + 3k_\ast}_\eta}) + \|  G_{\mathrm{Right}} \|_{H^{3 + 3k_\ast}_y} +L^{\frac13}  \|  G_{\mathrm{Left}} \langle y \rangle^{N_0} \|_{H^{3 + 2k_\ast}_y}  \\ \label{est:data:below:M}
& + L^{\frac16} \sum_{k' = 0}^{k_\ast} (|\gamma_{u;k'}| + |\gamma_{\psi; k'}|).
\end{align}
\end{proposition}
\begin{remark} Sharper estimates are available, but we have chosen to retain the simplest form for the right-hand side. For example, the eventual estimate \eqref{apply:me:k} will show that higher Sobolev norms of $G_{\mathrm{Left}}, G_{\mathrm{Right}}$ will come with higher powers of $L$, which have chosen to discard in the presentation of \eqref{est:data:below:M}, as these are superfluous for our eventual application. 
\end{remark}

\begin{remark} The quantities $\mathrm{\bold{Dat}}$ and $\mathrm{\bold{Dat}}_{\mathrm{Below}}$, though being boundary data, are not fully determined by the \textit{prescribed} data, as can be seen by the presence of $\sum_{k' = 0}^{k_\ast} (|\gamma_{u;k'}| + |\gamma_{\psi; k'}|)$ terms on the right-hand sides of \eqref{est:data:M} and \eqref{est:data:below:M}. This is already seen at the $k = 0$ level through the expression \eqref{PhiIota:1}. At higher orders in $k$, this is due to the fact that the vorticity can be deduced on the boundaries $\{x = 1\} \cap \{y > y_1(\eps) \}$ through the equation, for instance through \eqref{use:me:1}, but then to recover the velocity one needs to integrate vertically, for instance through \eqref{opop:1} -- \eqref{opop:2}.
\end{remark}

\begin{remark}  As in the previous section, the estimates \eqref{est:data:M} -- \eqref{est:data:below:M} are valid for all $L \in (0, L_{\mathrm{Max}})$. We omit the dependence of implicit constants on large values of $L_{\mathrm{Max}}$, and we freely use the inequality $L \lesssim C(L_{\mathrm{Max}}) \lesssim 1$. As all the lemmas in this section are valid on $(0, L_{\mathrm{Max}})$, we do not repeatedly state this for each lemma. 
\end{remark}

We introduce various auxiliary notations to help clarify the relation between various changes of variables. First, we note that on the sides, $\{x = 1\} = \{s = 1\}$ and $\{x = 1 + L\} = \{s = 1 + \bar{L}\}$, the only solution-dependent unknown are the two numbers $\gamma_{\psi; k}, \gamma_{u; k}$; all other information can be reconstructed in terms of the given data and these two numbers. 

We summarize the various variables we have: 
\begin{align*}
z := \frac{y}{\Lambda(x)}, \hspace{3 mm} Y := \bar{u}(s, z) + \eps u(s, z), \hspace{3 mm} Z :=  p(s, Y), \hspace{3 mm} \eta := \frac{y-\Lambda(x)}{\Lambda(x) L^{\frac13}}, \hspace{3 mm} \rho = \frac{Z}{\bar{L}^{\frac13}}.
\end{align*}
We note that, in general these changes of coordinates are $s$ (or $x$)-dependent. Therefore, in this section, they (1) differ based on if we are working on the ``Left" ($s = 1$) or ``Right" ($s = 1 + \bar{L}$), and (2) are themselves determined by the prescribed data in conjunction with the two numbers $\gamma_{\psi}, \gamma_u$. We introduce
\begin{align}
\omega_{NL, k}(x, y) = & \p_y u_{R, k}(x, y) - \frac{\p_y^2 u_{P}}{\p_y u_{P}}u_{R, k}(x, y).
\end{align}
We will take the convention for this section that $A_{\mathrm{Left}}(y) := A(1, y)$, $A_{\mathrm{Right}}(y) := A(1 + L, y)$, for any function $A(x, y)$, and for subscripted functions, such as $A_{j}(x, y)$, then $A_{j; \mathrm{Left}}(y) := A_j(1, y)$, $A_{j; \mathrm{Right}}(y) := A_j(1 + L, y)$. We also sometimes use $'$ to mean $\p_y$ when the quantity is a function of $y$ only.

\subsection{Iterative Structure of Higher Order Data}

Our starting point here is the equation in the original $(x, y)$ variables, which we recall from \eqref{sideLH}, and write in the following manner 
\begin{align} \label{jb43:1}
u_P \p_x u_{R, k} - u_{Py}\p_x \psi_{R,k} = & \p_y^2 u_{R, k} - \bold{C}_1  \p_y u_{R,k} - \bold{C}_2 u_{R,k} - \bold{C}_3 \psi_{R,k} + \overline{\mathcal{R}}_k, \\ \n
u_P \p_x \p_y u_{R,k} - u_{Pyy}\p_x \psi_{R,k}  =& \p_y^3 u_{R, k} - \bold{C}_1  \p_y^2 u_{R,k} - (\p_y \bold{C}_1 + \bold{C}_2) \p_y u_{R,k}-(\p_y \bold{C}_2 + \bold{C}_3) u_{R,k}  \\ \label{jb43:2}
 & - \p_y \bold{C}_3 \psi_{R,k} + \p_y \overline{ \mathcal{R}}_k.
\end{align}
\begin{remark} We use the formulation \eqref{sideLH} here as opposed to quasilinearizing the background $u_P$. Indeed, in order to compute the iterative structure of the data, we need to keep the full background on the left-hand side. 
\end{remark}

\noindent \textit{Identity for $ \omega_{\mathrm{NL}, k+1; \iota}$:} The forthcoming identities will be used on the intervals $[y_1(\eps), y^\ast]$ in the case $\iota = \mathrm{Left}$, and $[y_\ast, y_{1 + L}]$ in the case $\iota = \mathrm{Right}$. We now take the linear combination $\eqref{jb43:2} - \frac{u_{Pyy}}{u_{Py}} \times \eqref{jb43:1}$, which results in the identity
\begin{align} \n
u_P \omega_{\mathrm{NL}, k+1; \iota} = & L \p_y^2 \omega_{\mathrm{NL}, k; \iota} +L ( A_2  \p_y^2 u_{R,k; \iota}  + A_1 \p_y u_{R,k; \iota}  + A_0 u_{R,k; \iota}  + A_{-1} \psi_{R, k; \iota}) \\
&+L (\p_y - \frac{u_{Pyy}}{u_{Py}}) \overline{\mathcal{R}}_{k; \iota},
\end{align}
where we have denoted the coefficients via  
\begin{align}
A_2 := &  - \bold{C}_1, \\
A_1 := &  -2 L \p_y (\frac{u_{Pyy}}{u_{Py}}) - (\p_y \bold{C}_1 - \frac{u_{Pyy}}{u_{Py}} \bold{C}_1) - \bold{C}_2, \\
A_0 := & - L \p_y^2 ( \frac{u_{Pyy}}{u_{Py}}) - (\p_y \bold{C}_2 - \frac{u_{Pyy}}{u_{Py}} \bold{C}_2 ) - \bold{C}_3, \\
A_{-1} := & -(\p_y \bold{C}_3 - \frac{u_{Pyy}}{u_{Py}} \bold{C}_3).
\end{align}
It turns out to be convenient to rewrite the right-hand side in the following manner 
\begin{align} \n
u_P \omega_{\mathrm{NL}, k+1; \iota} = & L \p_y^2 \omega_{\mathrm{NL}, k; \iota} + L B_1 \p_y \omega_{\mathrm{NL}, k; \iota} +L B_0 \omega_{\mathrm{NL}, k; \iota} + L C_0 u_{R, k; \iota} + L A_{-1} \psi_{R, k; \iota} \\
&+L (\p_y - \frac{u_{Pyy}}{u_{Py}}) \overline{ \mathcal{R}}_{k; \iota}, 
\end{align}
where 
\begin{align}
B_1 := &A_2, \\
B_2 := &A_2 \frac{\p_y^2 u_P}{\p_y u_P} + A_1, \\
C_0 := & A_2 (\frac{\p_y^2 u_P}{\p_y u_P})^2 + A_2 \p_y (\frac{\p_y^2 u_P}{\p_y u_P}) + A_1 \frac{\p_y^2 u_P}{\p_y u_P}.
\end{align}
To compactify notation in what follows, we introduce 
\begin{align}
\mathcal{D}_2 \omega_{NL; k} := & \p_y^2 \omega_{\mathrm{NL}; k} + B_1 \p_y \omega_{\mathrm{NL}; k} + B_0 \omega_{\mathrm{NL}; k}, \\ \label{widehatG}
\mathcal{G}_k := &  C_0 u_{R; k} +  A_{-1} \psi_{R; k}, \\ \label{widehatRk}
\overline{r}_k := & (\p_y - \frac{u_{Pyy}}{u_{Py}})\overline{ \mathcal{R}}_k.
\end{align}
after which we obtain 
\begin{align} \label{use:me:1}
u_P \omega_{\mathrm{NL}, k+1; \iota} = & L \mathcal{D}_2 \omega_{\mathrm{NL}, k; \iota} + L\mathcal{G}_{k; \iota} + L\overline{r}_{k; \iota}.
\end{align}

\vspace{2 mm}

\noindent \textit{Identity for $u_{R, k+1}$:} The following will be used on the intervals $[y^\ast, \infty)$ in the case $\iota = \mathrm{Left}$ and $[0, y_\ast]$ in the case $\iota = \mathrm{Right}$. We use \eqref{jb43:1} to write 
\begin{align} \label{expo:1}
u_P^2 \p_y \{ \frac{\psi_{R, k+1; \iota}}{u_P} \} = & L \mathcal{H}_{k; \iota} +L \overline{\mathcal{R}}_{k; \iota}
\end{align}
where we define
\begin{align}
\mathcal{H}_k := \p_y^2 u_{R, k} - \bold{C}_1  \p_y u_{R,k} - \bold{C}_2 u_{R,k} - \bold{C}_3 \psi_{R,k}.
\end{align}
The equation \eqref{expo:1} is an ODE for $\psi_{R, k+1; \iota}$ whose inversion will be done as needed in the forthcoming lemmas to recover $\psi_{R, k+1; \iota}$ and subsequently $u_{R, k+1; \iota}$ upon differentiating. 

\vspace{2 mm}

\noindent \textit{Compatibility Conditions on $Q_{\mathrm{Left}}, Q_{\mathrm{Right}}$:} We can use the iterative identity \eqref{use:me:1} to derive the conditions on $(Q_{\mathrm{Left}}, Q_{\mathrm{Right}})$. These conditions will not be explicit, as they become complicated rapidly. However, we can write them iteratively. The study of $Q_{\mathrm{Left}}, Q_{\mathrm{Right}}$ are symmetric, so for simplicity we write the formulas for $Q_{\mathrm{Left}}$. Indeed, define 
\begin{align} \label{yeayea:1}
Q_{\mathrm{Left}; 0} := &Q_{\mathrm{Left}}, \\ \label{yeayea:2}
Q_{\mathrm{Left}; k} := &\frac{L}{u_P}(\mathcal{D}_2 Q_{\mathrm{Left}; k-1} + \mathcal{G}_{k-1} + \overline{r}_{k-1}).  
\end{align}
Then our compatibility conditions become 
\begin{align} \label{cc:1}
\mathcal{D}_2 Q_{\mathrm{Left}; k}(0) + \mathcal{G}_k(1, y_1(\eps)) + \overline{r}_k(1, y_1(\eps)) = 0. 
\end{align}
Examining the form of $\mathcal{G}_k(1, y_1(\eps)) + \overline{r}_k(1, y_1(\eps))$ from \eqref{widehatG}, \eqref{widehatRk}, it is clear that this condition can be achieved by selecting $Q_{\mathrm{Left}}$ to be a smooth function with appropriately tailored boundary conditions at $0$, satisfying 
\begin{align} \label{QbdLR}
\| Q_{\mathrm{Left}} \|_{C^\infty} + \| Q_{\mathrm{Right}} \|_{C^\infty}  \lesssim L^{\frac23} \sum_{k = 0}^{k_\ast - 1} (|\gamma_{u; k}| + |\gamma_{\psi; k}|). 
\end{align}
%
\begin{lemma} \label{lem:one} Let $0 \le K \le 1$, and $J \in \mathbb{N}$. The following estimates are valid: 
\begin{align}  \label{hm:1}
\| \langle \eta \rangle^{K} (\eta \p_\eta)^{J} \omega_{\mathrm{NL}, k; \mathrm{Left}} \|_{H^m_\eta(y_1(\eps), y^\ast)} \lesssim & \|  (\eta \p_\eta)^{J} F_{\mathrm{Left}} \langle \eta \rangle^K \|_{H^{m + 3k}_\eta} + L^{\frac12-\frac{K}{3}} \sum_{k' = 0}^{k-1}  |\gamma_{u; k'}|, \\  \n
\|  (\eta \p_\eta)^{J} u_{R, k; \mathrm{Left}} \|_{H^m_\eta(y_1(\eps), y^\ast)} \lesssim & \|  (\eta \p_\eta)^{J} F_{\mathrm{Left}}  \|_{H^{m -1 + 3k}_\eta} + L^{-\frac16} |\gamma_{u; k}|\\ \label{hm:2}
& + L^{\frac12} \sum_{k' = 0}^{k-1}  |\gamma_{u; k'}|.
\end{align}
\end{lemma}
\begin{proof}It is convenient to close the bounds for $K = 0, J = 0$ first, with the extension to $K = 1$ and general $J$ being straightforward. We first of all write $u_{R, k; \mathrm{Left}}$ in terms of $\omega_{\mathrm{NL}, k; \mathrm{Left}}$, and subsequently $\psi_{R, k; \mathrm{Left}}$ via the inversion formula:
\begin{align} \label{opop:1}
u_{R, k; \mathrm{Left}}(y) = & \gamma_{u; k} \frac{u_{P; \mathrm{Left}}'(y)}{u_{P; \mathrm{Left}}'(y_1(\eps))} + u'_{P; \mathrm{Left}}(y) \int_{y_1(\eps)}^y \frac{\omega_{NL, k; \mathrm{Left}}}{u'_{P; \mathrm{Left}}}, \\ \label{opop:2}
\psi_{R,k; \mathrm{Left} }(y) = & \gamma_{\psi; k} + \int_{y_1(\eps)}^y u_{R, k; \mathrm{Left}}.
\end{align}
From this expression, we immediately obtain 
\begin{align*}
\|  u_{R, k; \mathrm{Left}} \|_{H^m_\eta(y_1(\eps), y^\ast)} \lesssim& L^{-\frac16} |\gamma_{u; k}| +  \|  \omega_{\mathrm{NL}, k; \mathrm{Left}} \|_{H^{\max\{m-1, 0\}}_\eta(y_1(\eps), y^\ast)}, \\
\| \psi_{R, k; \mathrm{Left}} \|_{L^\infty(y_1(\eps), y^\ast)} \lesssim & ( |\gamma_{\psi; k}| +  |\gamma_{u; k}| ) + L^{\frac16}\|  \omega_{\mathrm{NL}, k; \mathrm{Left}} \|_{L^2_\eta(y_1(\eps), y^\ast)}.
\end{align*}
Next, we use \eqref{use:me:1} to obtain 
\begin{align*}
\omega_{\mathrm{NL}, k+1; \mathrm{Left}} = & L \frac{\mathcal{D}_2 \omega_{\mathrm{NL}, k; \mathrm{Left}} - a_1 }{u_P}+ L \frac{\mathcal{G}_{k; \mathrm{Left}} - a_2}{u_P} + L\frac{\overline{r}_{k; \mathrm{Left}} - a_3}{u_P} := \sum_{i = 1}^3 E_i,
\end{align*}
for $a_1 = \mathcal{D}_2 \omega_{\mathrm{NL}; k}(y_1(\eps))$, $a_2 = \mathcal{G}_k(y_1(\eps))$ and $a_3 = \overline{r}_k(y_1(\eps))$. These sum to zero due to the compatibility condition \eqref{cc:1}. First, we have by using that $u_P \sim y - y_1(\eps)$ on $[y_1(\eps), y^\ast]$,
\begin{align*}
|E_1| \lesssim &L \| \p_y^3 \omega_{\mathrm{NL}, k; \mathrm{Left}} \|_{H^m_\eta(y_1(\eps), y^\ast)} \lesssim \|  \omega_{\mathrm{NL}, k; \mathrm{Left}} \|_{H^{m+3}_\eta(y_1(\eps), y^\ast)} \\
|E_2| \lesssim & L^{\frac23} \|  u_{R, k; \mathrm{Left}} \|_{H^{m+1}_\eta(y_1(\eps), y^\ast)}, \\
|E_3| \lesssim & L^{\frac13} \|  \omega_{\mathrm{NL}, k-1; \mathrm{Left}} \|_{H^{m+2}_\eta(y_1(\eps), y^\ast)} + L^{\frac23}  \|  u_{R,k-1; \mathrm{Left}} \|_{H^{m+1}_\eta(y_1(\eps), y^\ast)} 
\end{align*}
We therefore get the following iteration. For the purpose of saving notation, let 
\begin{align} \n
\alpha_{k, m} := \| \omega_{\mathrm{NL}, k; \mathrm{Left}} \|_{H^m_\eta(y_1(\eps), y^\ast)} , \qquad \beta_{k,m} := \| u_{R, k; \mathrm{Left}} \|_{H^m_\eta(y_1(\eps), y^\ast)},
\end{align}
after which our iteration can be written as
\begin{align}
\beta_{k,m} \lesssim & L^{-\frac16} |\gamma_{u; k}| + \alpha_{k,m-1}, \qquad k \ge 0 \\
\alpha_{k+1,m} \lesssim & \alpha_{k,m+3} + L^{\frac13} 1_{k \ge 1} \alpha_{k-1, m+2} + L^{\frac23} ( \beta_{k, m+1} + 1_{k \ge 1} \beta_{k-1, m+1}), \qquad k \ge 0 \\
\alpha_{0,m} \lesssim & \| F_{\mathrm{Left}} \|_{H^m_\eta}.
\end{align}
From this iteration, the desired bounds follow immediately. The lemma is proven. 
\end{proof}

\begin{lemma} \label{lem:two} Fix any $J \ge 1$. There exists a decomposition $u_{R, k; \mathrm{Left}} =  u_{R, k; \mathrm{Left}}^{\mathrm{(Fast)}} + u_{R, k; \mathrm{Left}}^{\mathrm{(Slow)}}$, where 
\begin{align} \label{js:1}
\| u_{R, k; \mathrm{Left}}^{\mathrm{(Fast)}} \langle \eta \rangle^J \|_{H^m_\eta(\eta^\ast, \infty)} \lesssim & L^{\frac{k+2}{3}} \| F_{\mathrm{Left}} \langle \eta \rangle^{2 + J} \|_{H^{\max(m +2k-1,0)}_\eta}, \\ \n
\| u_{R, k; \mathrm{Left}}^{\mathrm{(Slow)}} \langle y \rangle^J \|_{H^m_y(y^\ast, \infty)} \lesssim & |\gamma_{\psi; k}| + |\gamma_{u; k}| + L^{\frac23} \sum_{k' = 0}^{k-1} (|\gamma_{\psi; k'}| + |\gamma_{u; k'}| ) \\ \label{js:2}
& + L^{\frac16} \| F_{\mathrm{Left}} \langle \eta \rangle^3 \|_{H^{3k}_\eta} + L^k \| G_{\mathrm{Left}} \langle y \rangle^J \|_{H^{m + 2k}_y}.
\end{align}
\end{lemma}
\begin{proof} We first invert the expression \eqref{expo:1} on the domain $(y^\ast, \infty)$ to obtain the expressions
\begin{align} \n
\psi_{R, k+1; \mathrm{Left}}(y) = &\psi_{R, k+1; \mathrm{ Left}}(y^\ast) \frac{u_{P; \mathrm{Left}}(y)}{u_{P; \mathrm{Left}}(y^\ast)} +L u_{P; \mathrm{Left}} \int_{y^\ast}^y \frac{\mathcal{H}_{k; \mathrm{Left}}}{u_{P; \mathrm{Left}}^2} \ud y' \\ \label{mo:by:1}
&+ L u_{P; \mathrm{Left}} \int_{y^\ast}^y \frac{\overline{\mathcal{R}}_{k; \mathrm{Left}}}{u_{P; \mathrm{Left}}^2} \ud y'.
\end{align}
From this expression, we are led to decompose 
\begin{align*}
\psi_{R, k+1; \mathrm{Left}}(y) = \psi^{(\mathrm{Slow})}_{R, k+1; \mathrm{Left}}(y) +  \psi^{(\mathrm{Fast})}_{R, k+1; \mathrm{Left}}(\eta) , 
\end{align*}
where
\begin{align*}
\psi^{(\mathrm{Slow})}_{R, k+1; \mathrm{Left}}(y)  :=  & \psi_{R, k+1; \mathrm{ Left}}(y^\ast) \frac{u_{P; \mathrm{Left}}(y)}{u_{P; \mathrm{Left}}(y^\ast)} + L u_{P; \mathrm{Left}} \int_{y^\ast}^y \frac{\mathcal{H}_{k; \mathrm{Left}}^{(\mathrm{Slow})}}{u_{P; \mathrm{Left}}^2}  \ud y' \\
&+ L u_{P; \mathrm{Left}} \int_{y^\ast}^y \frac{\overline{\mathcal{R}}_{k; \mathrm{Left}}^{(\mathrm{Slow})}}{u_{P; \mathrm{Left}}^2}  \ud y' 
 + L^{\frac43} u_{P; \mathrm{Left}} \int_{\eta^\ast}^\infty \frac{\mathcal{H}_{k; \mathrm{Left}}^{(\mathrm{Fast})}}{u_{P; \mathrm{Left}}^2}  \ud \eta' \\
& + L^{\frac43} u_{P; \mathrm{Left}} \int_{\eta^\ast}^\infty \frac{\overline{\mathcal{R}}_{k; \mathrm{Left}}^{(\mathrm{Fast})}}{u_{P; \mathrm{Left}}^2}  \ud \eta', \\
 \psi^{(\mathrm{Fast})}_{R, k+1; \mathrm{Left}}(\eta)   := &- L^{\frac43} u_{P; \mathrm{Left}} \int_{\eta}^\infty \frac{\mathcal{H}_{k; \mathrm{Left}}^{(\mathrm{Fast})}}{u_{P; \mathrm{Left}}^2}  \ud \eta' - L^{\frac43} u_{P; \mathrm{Left}} \int_{\eta}^\infty \frac{\overline{\mathcal{R}}_{k; \mathrm{Left}}^{(\mathrm{Fast})}}{u_{P; \mathrm{Left}}^2}  \ud \eta'.
\end{align*}
We differentiate both these identities in $y$ to obtain 
\begin{align*}
u^{(\mathrm{Slow})}_{R, k+1; \mathrm{Left}}(y) :=  & \psi_{R,\mathrm{ Left}, k+1}(y^\ast) \frac{\p_y u_{P; \mathrm{Left}}(y)}{u_{P; \mathrm{Left}}(y^\ast)} +  L \p_y u_{P; \mathrm{Left}} \int_{y^\ast}^y \frac{\mathcal{H}_{k; \mathrm{Left}}^{(\mathrm{Slow})}}{u_{P; \mathrm{Left}}^2}  \ud y' + L \frac{\mathcal{H}_{k; \mathrm{Left}}^{(\mathrm{Slow})}}{u_{P; \mathrm{Left}}} \\
& + L \p_y u_{P; \mathrm{Left}} \int_{y^\ast}^y \frac{\overline{\mathcal{R}}_{k; \mathrm{Left}}^{(\mathrm{Slow})}}{u_{P; \mathrm{Left}}^2}  \ud y' + L  \frac{\overline{\mathcal{R}}_{k; \mathrm{Left}}^{(\mathrm{Slow})}}{u_{P; \mathrm{Left}}}  + L^{\frac43} \p_y u_{P; \mathrm{Left}} \int_{\eta^\ast}^\infty \frac{\mathcal{H}_{k; \mathrm{Left}}^{(\mathrm{Fast})}}{u_{P; \mathrm{Left}}^2}  \ud \eta' \\
& + L^{\frac43} \p_y u_{P; \mathrm{Left}} \int_{\eta^\ast}^\infty \frac{\overline{\mathcal{R}}_{k; \mathrm{Left}}^{(\mathrm{Fast})}}{u_{P; \mathrm{Left}}^2}  \ud \eta', \\
u^{(\mathrm{Fast})}_{R, k+1; \mathrm{Left}}(\eta) := & - L^{\frac43} \p_y u_{P; \mathrm{Left}} \int_{\eta}^\infty \frac{\mathcal{H}_{k; \mathrm{Left}}^{(\mathrm{Fast})}}{u_{P; \mathrm{Left}}^2}  \ud \eta' - L \frac{\mathcal{H}_{k; \mathrm{Left}}^{(\mathrm{Fast})}}{u_{P; \mathrm{Left}}} -  L^{\frac43} \p_y u_{P; \mathrm{Left}} \int_{\eta}^\infty \frac{\overline{\mathcal{R}}_{k; \mathrm{Left}}^{(\mathrm{Fast})}}{u_{P; \mathrm{Left}}^2}  \ud \eta'  \\
&- L \frac{\overline{\mathcal{R}}_{k; \mathrm{Left}}^{(\mathrm{Fast})}}{u_{P; \mathrm{Left}}}.
\end{align*}
From these identities, we obtain the following estimates for $J \ge 1$, 
\begin{align} \n
\| u^{(\mathrm{Slow})}_{R, k+1;  \mathrm{Left}} \langle y \rangle^{J} \|_{H^m_{y}(y^\ast, \infty)} \lesssim &   | \psi_{R,\mathrm{ Left}, k+1}(y^\ast) | + L (  \| \mathcal{H}^{(\mathrm{Slow})}_{k; \mathrm{Left}} \langle y \rangle^J \|_{H^m_y(y^\ast, \infty)} + \| \overline{\mathcal{R}}^{(\mathrm{Slow})}_{k; \mathrm{Left}} \langle y \rangle^J \|_{H^m_y(y^\ast, \infty)}) \\  \label{hgyumnb:1}
& +L^{\frac43} \| \mathcal{H}_{k; \mathrm{Left}}^{(\mathrm{Fast})} \langle \eta \rangle \|_{L^2_\eta(\eta^\ast, \infty)} +  L^{\frac43} \| \overline{\mathcal{R}}_{k; \mathrm{Left}}^{(\mathrm{Fast})} \langle \eta \rangle \|_{L^2_\eta(\eta^\ast, \infty)}, \\ \label{hgyumnb:2}
\| u^{(\mathrm{Fast})}_{R, k+1;  \mathrm{Left}} \langle \eta \rangle^{J} \|_{H^m_{\eta}(\eta^\ast, \infty)} \lesssim & L \| \mathcal{H}^{(\mathrm{Fast})}_{k;  \mathrm{Left}} \langle \eta \rangle^{J} \|_{H^m_{\eta}(\eta^\ast, \infty)} +  L \| \overline{\mathcal{R}}^{(\mathrm{Fast})}_{k;  \mathrm{Left}} \langle \eta \rangle^{J} \|_{H^m_{\eta}(\eta^\ast, \infty)}.
\end{align}
We now have the following expressions 
\begin{align} 
\mathcal{H}_{k, \mathrm{Left}}^{(\mathrm{Slow})} := & \p_y^2 u_{R, k; \mathrm{Left}}^{(\mathrm{Slow})} - \bold{C}_1  \p_y u_{R, k; \mathrm{Left}}^{(\mathrm{Slow})} - \bold{C}_2 u^{(\mathrm{Slow})}_{R, k; \mathrm{Left}} - \bold{C}_3 \psi_{R,k;  \mathrm{Left}}^{(\mathrm{Slow})}, \\
\mathcal{H}_{k; \mathrm{Left}}^{(\mathrm{Fast})} := & L^{-\frac23}  \p_\eta^2 u^{(\mathrm{Fast})}_{R,k; \mathrm{Left}} - \bold{C}_1 L^{-\frac13}  \p_\eta u^{(\mathrm{Fast})}_{R,k; \mathrm{Left}} - \bold{C}_2 u^{\mathrm{(Fast)}}_{R,k; \mathrm{Left}} - \bold{C}_3 \psi^{(\mathrm{Fast})}_{R,k; \mathrm{Left}}.
\end{align}
We therefore estimate 
\begin{align} \label{tor:a:1}
\| \mathcal{H}_k^{(\mathrm{Slow})} \langle y \rangle^J \|_{H^m_y(y^\ast, \infty)} \lesssim & \| u^{(\mathrm{Slow})}_{R,k; \mathrm{Left}}\langle y \rangle^{J} \|_{H^{m+2}_{y}(y^\ast, \infty)} + \| \psi^{(\mathrm{Slow})}_{R,k; \mathrm{Left}} \|_{L^{\infty}_{y}(y^\ast, \infty)}, \\ \label{tor:a:2}
\| \mathcal{H}_k^{(\mathrm{Fast})} \langle \eta \rangle^J  \|_{H^m_\eta(\eta^\ast, \infty)} \lesssim & L^{-\frac23} \| u^{(\mathrm{Fast})}_{R,k;  \mathrm{Left}} \langle \eta \rangle^J \|_{H^{m+2}_{\eta}(\eta^\ast, \infty)} + \| \psi^{(\mathrm{Fast})}_{R, k; \mathrm{Left}} \|_{L^\infty_{\eta}(\eta^\ast, \infty)}.
\end{align}
In a similar manner, we have
\begin{align} \label{tor:a:1jk}
\| \overline{\mathcal{R}}_k^{(\mathrm{Slow})} \langle y \rangle^J \|_{H^m_y(y^\ast, \infty)} \lesssim & \| u^{(\mathrm{Slow})}_{R,k-1; \mathrm{Left}}\langle y \rangle^{J} \|_{H^{m+1}_{y}(y^\ast, \infty)} + \| \psi^{(\mathrm{Slow})}_{R,k-1; \mathrm{Left}} \|_{L^{\infty}_{y}(y^\ast, \infty)}, \\ \label{tor:a:2jk}
\| \overline{\mathcal{R}}_k^{(\mathrm{Fast})} \langle \eta \rangle^J  \|_{H^m_\eta(\eta^\ast, \infty)} \lesssim & L^{-\frac13} \| u^{(\mathrm{Fast})}_{R,k-1;  \mathrm{Left}} \langle \eta \rangle^J \|_{H^{m+1}_{\eta}(\eta^\ast, \infty)} + \| \psi^{(\mathrm{Fast})}_{R, k-1; \mathrm{Left}} \|_{L^\infty_{\eta}(\eta^\ast, \infty)}.
\end{align}

To make expressions simpler, we define the following notation for this proof: 
\begin{align*}
\alpha_{k,m} := &\| u^{(\mathrm{Slow})}_{R,k; \mathrm{Left}} \langle y \rangle^J \|_{H^m_{y}(y^\ast, \infty)}, && \beta_{k,m} := \| u^{(\mathrm{Fast})}_{R,k;  \mathrm{Left}}   \langle \eta \rangle^J\|_{H^m_{\eta}(\eta^\ast, \infty)}, \\
\mu_{k} := & \| \psi^{(\mathrm{Slow})}_{R,k; \mathrm{Left}} \|_{L^\infty_{y}(y^\ast, \infty)}, && \gamma_{k} := \| \psi^{(\mathrm{Fast})}_{R,k; \mathrm{Left}} \|_{L^\infty_{\eta}(\eta^\ast, \infty)}.
\end{align*}
Combining the inequalities \eqref{hgyumnb:1} -- \eqref{hgyumnb:2} with \eqref{tor:a:1} -- \eqref{tor:a:2} and \eqref{tor:a:1jk} -- \eqref{tor:a:2jk}, we obtain the following iterative bounds:
\begin{align}\n
\alpha_{k+1,m} \lesssim & | \psi_{R,\mathrm{ Left}, k+1}(y^\ast) |  + L \alpha_{k,m+2} + L \mu_k + L^{\frac23} \beta_{k,2} + L^{\frac43} \gamma_k \\  \label{ind:est:1:lk}
& + 1_{k \ge 1}(L \alpha_{k-1,m+1} + L \mu_{k-1} + L \beta_{k-1,3} + L^{\frac43}\gamma_{k-1}), \\ \label{ind:est:2:lk}
\beta_{k+1,m} \lesssim & L^{\frac13}\beta_{k, m+2} + L \gamma_k + 1_{k \ge 1}( L^{\frac23} \beta_{k-1, m+1} + L \gamma_{k-1}), \\ \label{ind:est:3:lk}
\mu_{k+1} \lesssim & | \psi_{R,\mathrm{ Left}, k+1}(y^\ast) | + \alpha_{k+1,1}, \\
\gamma_{k+1} \lesssim & L^{\frac13} \beta_{k+1,1}.  
\end{align}
The latter two inequalities on the stream function arise from a straightforward Sobolev embedding in one dimension. We may insert the latter two inequalities into the first two and close the iteration fro $\alpha_{k,m}, \beta_{k,m}$ (first for $\beta_{k,m}$, as it is independent of $\alpha_{k,m}$, and then for $\alpha_{k,m}$) to obtain: 
\begin{align} \label{park:1}
\alpha_{k,m} \lesssim & \sum_{k' = 0}^k L^{k-k'} |\psi_{R, k'; \mathrm{Left}}(y^\ast)| + L^k \alpha_{0, m+2k} + L^{\frac{k+2}{3}} \beta_{0,2 + 2k}, \\ \label{park:2}
\beta_{k,m} \lesssim & L^{\frac{k}{3}} \beta_{0, m + 2k}.
\end{align}
It thus remains to provide bounds on $|\psi_{R, k'; \mathrm{Left}}(y^\ast)|$ as well as the initial states $\alpha_{0,m}, \beta_{0,m}$. 

First, for the initial states, we recall the expression \eqref{PhiIota:1}. From here, we decompose 
\begin{align*}
u_{R, 0; \mathrm{Left}}^{(\mathrm{Slow})}(y) := & \chi_{\mathrm{Cont, +}}(y)(\gamma_{u} \frac{u_{FS}'(y)}{u_{FS}'(y_1(\eps))} + L^{\frac13} u_{FS}'(y) \int_{y_1(\eps)}^{\infty} \frac{F_{\mathrm{Left}}(\frac{1}{L^{\frac13}}(\frac{y}{y_1(\eps)} - 1) + 1)}{u_{FS}'} \ud \eta )  \\ 
& + \chi_{\mathrm{Top}}(y) G_{\mathrm{Left}}(y), \\
u_{R, 0; \mathrm{Left}}^{(\mathrm{Fast})}(\eta) := & -\chi_{\mathrm{Cont, +}}(y) L^{\frac13} u_{FS}'(y) \int_{\eta}^{\infty} \frac{F_{\mathrm{Left}}(\frac{1}{L^{\frac13}}(\frac{y}{y_1(\eps)} - 1) + 1)}{u_{FS}'} \ud \eta.
\end{align*}
We now provide estimates on these initial quantities. We have 
\begin{align} \label{in:bd:1}
\alpha_{0,m} \lesssim & |\gamma_u| + L^{\frac23} \| F_{\mathrm{Left}} \langle \eta \rangle \|_{L^2_\eta} + \| G_{\mathrm{Left}} \langle y \rangle^J \|_{H^{m}_y}, \\ \label{in:bd:2}
\beta_{0,m} \lesssim & L^{\frac23} \| F_{\mathrm{Left}} \langle \eta \rangle^{2+ J} \|_{H^{\max(m-1,0)}_\eta}.
\end{align}
Next, we estimate 
\begin{align} \n
|\psi_{R, \mathrm{Left}, k}(y^\ast)| \le & |\gamma_{\psi; k}| + |\int_{y_1(\eps)}^{y^\ast} u_{R, k; \mathrm{Left}} \ud y' | \lesssim  |\gamma_{\psi; k}| + \| u_{R, k; \mathrm{Left}} \|_{L^\infty_y(y_1(\eps), y^\ast)} \\ \n
\lesssim &  |\gamma_{\psi; k}|  + |\gamma_{u; k}| + \| \omega_{NL, k; \mathrm{Left}} \|_{L^2_y(y_1(\eps), y^\ast)} \\ \n
\lesssim &  |\gamma_{\psi; k}|  + |\gamma_{u; k}| + L^{\frac16} \| \omega_{NL, k; \mathrm{Left}} \|_{L^2_\eta(y_1(\eps), \eta^\ast)} \\ \label{in:bd:4}
\lesssim &  |\gamma_{\psi; k}|  + |\gamma_{u; k}| + L^{\frac16} (\| F_{\mathrm{Left}} \|_{H^{3k}_\eta} +L^{\frac12} \sum_{k' = 0}^{k-1} |\gamma_{u; k'}| ),
\end{align}
where we have invoked \eqref{hm:1}.

Pairing the bounds \eqref{park:1} -- \eqref{park:2} with \eqref{in:bd:1} -- \eqref{in:bd:2} and \eqref{in:bd:4} gives the desired result. We note that in so doing, we have chosen to not present the sharpest estimate (powers of $L$) at the gain of simplicity. 
\end{proof}

\begin{lemma} \label{lem:three} Fix any $J \ge 1$. There exists a decomposition $u_{R, k; \mathrm{Right}} =  u_{R, k; \mathrm{Right}}^{\mathrm{(Fast)}} + u_{R, k; \mathrm{Right}}^{\mathrm{(Slow)}}$, where
\begin{align} \label{js:3}
\| u_{R, k; \mathrm{Right}}^{\mathrm{(Fast)}} \langle \eta \rangle^J \|_{H^m_\eta(0, y_\ast)} \lesssim & L^{\frac{k+2}{3}} \| F_{\mathrm{Right}} \langle \eta \rangle^{2 + J} \|_{H^{\max(m +2k-1,0)}_\eta}, \\  \label{js:4}
\| u_{R, k; \mathrm{Right}}^{\mathrm{(Slow)}} \|_{H^m_y(0, y_\ast)} \lesssim & L^{\frac16} \| F_{\mathrm{Right}} \langle \eta \rangle^3 \|_{H^{3k}_\eta} + L^k \| G_{\mathrm{Right}} \|_{H^{m + 3k}_y}.
\end{align}
\end{lemma}
\begin{proof} This proof follows in a very similar manner to Lemma \ref{lem:two}, and we omit repeating the details.

\end{proof}

\begin{lemma} \label{lem:four} Let $0 \le K \le 1$, and $J \in \mathbb{N}$. The following estimates are valid: 
\begin{align}  \label{hm:1:R}
\| \langle \eta \rangle^{K} (\eta \p_\eta)^{J} \omega_{\mathrm{NL}, k; \mathrm{Right}} \|_{H^m_\eta(0, y_\ast)} \lesssim & \|  (\eta \p_\eta)^{J} F_{\mathrm{Right}} \langle \eta \rangle^{\max\{K,3\}} \|_{H^{m + 3k}_\eta} + L^{\frac12-\frac{K}{3}} \| G_{\mathrm{Right}} \|_{H^{3k-2}_y}, \\  \label{hm:2:R}
\|  (\eta \p_\eta)^{J} u_{R, k; \mathrm{Right}} \|_{H^m_\eta(0, y_\ast)} \lesssim & \|  (\eta \p_\eta)^{J} F_{\mathrm{Right}} \langle \eta \rangle^{\max\{K, 3\}} \|_{H^{\max\{m -1 + 3k, 0\}}_\eta} + L^{k-\frac16} \|G_{\mathrm{Right}} \|_{H^{1 + 3k}_y}.
\end{align}
\end{lemma}
\begin{proof} This follows in an identical manner to Lemma \ref{lem:one}, with the quantities $\gamma_{u;k}$ being replaced by $u_{R, k; \mathrm{Right}}(y_\ast)$, which is bounded by 
\begin{align*}
|u_{R, k; \mathrm{Right}}(y_\ast)| =& |\int_0^{y_\ast} \p_y u_{R, k; \mathrm{Right}}(y_\ast)| \lesssim \| u_{R, k; \mathrm{Right}}^{(\mathrm{Slow})} \|_{H^1_y(0, y_\ast)} + \| u_{R, k; \mathrm{Right}}^{(\mathrm{Fast})} \|_{H^1_\eta(0, y_\ast)} \\
\lesssim & L^k \|G_{\mathrm{Right}} \|_{H^{1 + 3k}_y} + L^{\frac16} \| F_{\mathrm{Right}} \langle \eta \rangle^3 \|_{H^{3k}_\eta},
\end{align*}
according to applying \eqref{js:3} and \eqref{js:4} with $m = 1$.

\end{proof}

\subsection{Proof of Proposition \ref{prop:data:M}}

It is useful at this stage to include a computation for the case $k = 0$. Indeed, letting $\mathrm{\bold{Dat}}_k$ be the contribution of the $k'$th element of the sum in \eqref{boldDat}. Then, we have 
\begin{lemma}[$k = 0$ Data Estimate] The following bounds are valid: 
\begin{align} \n
\mathrm{\bold{Dat}}_0 \lesssim & L^{\frac16} (\| F_{\mathrm{Left}} \langle \eta \rangle^3 \|_{H^{2 }_\eta} + \| F_{\mathrm{Right}} \langle \eta \rangle^3 \|_{H^{2}_\eta}) \\ \label{yeah:yea:ye:1}
&+(|\gamma_{u}| + |\gamma_{\psi}|) +  \| G_{\mathrm{Right}} \|_{H^{3}_y} + \|  G_{\mathrm{Left}} \langle y \rangle^{N_0} \|_{H^{3}_y}), \\ \n
\mathrm{\bold{Dat}}_{\mathrm{Below},0} \lesssim & L^{\frac16} (\| F_{\mathrm{Left}} \langle \eta \rangle^3 \|_{H^{2 }_\eta} + \| F_{\mathrm{Right}} \langle \eta \rangle^3 \|_{H^{2}_\eta}) +  \| G_{\mathrm{Right}} \|_{H^{3}_y} \\ \label{yeah:yea:ye:1b}
& + L^{\frac13} (|\gamma_{u}| + |\gamma_{\psi}| +  \|  G_{\mathrm{Left}} \langle y \rangle^{N_0} \|_{H^{3}_y}) . 
\end{align}
\end{lemma}
\begin{proof} We proceed in two steps, namely we will first estimate the quantities involving $\Xi_{\mathrm{Left}}, \Xi_{\mathrm{Right}}$, and then those involving $U_{\mathrm{Left}}, U_{\mathrm{Right}}$.

\vspace{2 mm}

\noindent \underline{Estimates on $\Omega_{I; \mathrm{Left}}$:} The starting point is an expression for $u_{\mathrm{Left}}(z), 1 < z < z_\ast$ in terms of the given functions $F_{\mathrm{Left}}$ and the unknown numbers $(\gamma_\psi, \gamma_u)$. By inversion, we have
\begin{align} \label{kia:1}
u_{\mathrm{Left}}(z)= &  \gamma_u \frac{\bar{u}_z(1, z) }{\bar{u}_z(1,1)} + \bar{u}_z(1, z) \int_1^z \frac{F_{\mathrm{Left}}( \frac{z' - 1}{L^{\frac13}} + 1)}{\bar{u}_z(1, z')} \ud z'.
\end{align}
To emphasize this, we write in operator form, $u_{\mathrm{Left}}  = u_{\mathrm{Left}}[\gamma_u, F_{\mathrm{Left}}]$. This formula will be utilized both to describe the change of variables as well as to describe the change of function. More precisely, we first have 
\begin{align*}
\rho = & \frac{Z}{\bar{L}^{\frac13}} = \frac{p(1, \bar{u}(s, z) + \eps u(1, z))}{\bar{L}^{\frac13}} =   \frac{p(1, \bar{u}(1, z) + \eps u_{\mathrm{Left}}(1, z))}{\bar{L}^{\frac13}} \\
= & \frac{\bar{u}(1, L^{\frac13} \eta + 1) + \eps u_{\mathrm{Left}}(1, L^{\frac13} \eta + 1)}{\bar{L}^{\frac13} \bar{w}_z(1, 1)^{\frac23}} =  \frac{\bar{u}(1, L^{\frac13} \eta + 1) + \eps u_{\mathrm{Left}}(1, L^{\frac13} \eta + 1)}{\bar{L}^{\frac13} ( \bar{u}_z(1, 1) + \eps \omega_\ast(1, 1) + \eps \frac{\bar{u}_{zz}}{\bar{u}_z} \gamma_u )^{\frac23}} \\
=&  \frac{\bar{u}(1, L^{\frac13} \eta + 1) + \eps u_{\mathrm{Left}}(1, L^{\frac13} \eta + 1)}{\bar{L}^{\frac13} ( \bar{u}_z(1, 1) + \eps F_{\mathrm{Left}}(1) + \eps \frac{\bar{u}_{zz}}{\bar{u}_z} \gamma_u )^{\frac23}} = \eta + \eps T_{\mathrm{coord}}[\eta; F_{\mathrm{Left}}, \gamma_u], 
\end{align*} 
where the operator $T_{\mathrm{coord}}[\eta; F_{\mathrm{Left}}, \gamma_u]$ is defined implicitly to be the perturbation of the identity from above. It is also convenient to introduce the inverse operator $S_{\mathrm{coord}}$ via 
\begin{align}
\rho = \eta + \eps T_{\mathrm{coord}}[\eta; F_{\mathrm{Left}}, \gamma_u] \iff \eta = \rho + \eps S_{\mathrm{coord}}[\rho; F_{\mathrm{Left}}, \gamma_u].
\end{align}
We can therefore provide estimates on the Jacobian of the change of coordinates using a straightforward application of Taylor's remainder theorem, 
\begin{align*}
 |\frac{d\rho}{d\eta} - 1| \le \chi_I L^{-\frac13} (L^{\frac23}\eta)(\| \bar{u} \|_{H^2_z} + \eps \|u_{\mathrm{Left}} \|_{H^2_z}) \lesssim \delta (1 + \eps |\gamma_u| + \eps L^{-\frac13} \| F_{\mathrm{Left}} \|_{W^{1,\infty}_\eta}  )
 \end{align*}
We therefore have, in our introduced notation, 
\begin{align*}
\Omega_{I; \mathrm{Left}}(\rho) = & \chi_I(Z)F_{\mathrm{Left}}( \eta).
\end{align*}
We may therefore estimate $\| \Omega_{I; \mathrm{Left}} \|_{L^2_\rho} \lesssim  \| F_{\mathrm{Left}} \|_{L^2_\eta}$.
\vspace{2 mm}

\noindent \underline{Estimates on $\Omega_{I; \mathrm{Right}}$:} These estimates follow in an analogous manner (but are strictly simpler) than those above on $\Omega_{I; \mathrm{Left}}$.

\vspace{2 mm}

\noindent \underline{Estimates on $U_{\mathrm{Left}}$:} We now evaluate $ \sum_{j = 0}^3 ( \| \p_z^j U_{\mathrm{Left}; k} \chi_{O,j}^+ \langle z \rangle^{N_k} \|_{L^2_z}$. The support of $\chi_{O,j}^+$ overlaps the two regions: $y_1(\eps) < y < y^\ast$ as well as $y^\ast < y < \infty$. Out of these, the region $y^\ast < y < \infty$ is the more complicated (as it involves $G_{\mathrm{Left}}$), whereas the contribution from $y_1(\eps) < y < y^\ast$ is identical, just without the $G_{\mathrm{Left}}$. Therefore, we simply treat the former contribution. We note that the change of variable from $z$ to $y$ is nearly the identity at $\{x = 1\}$ (and in fact everywhere). Upon using that $U_{\mathrm{Left}} = \bar{W} u_{\mathrm{Left}} - \bar{W}_z \psi_{\mathrm{Left}}$, it therefore suffices for us to evaluate the quantities $\text{Term}_u := \sum_{j = 0}^3 ( \| \p_y^j u_{R,\mathrm{Left}; k} \chi_{O,j}^+ \langle y \rangle^{N_k} \|_{L^2_y}$ and $\text{Term}_{\psi} := \|  \psi_{R,\mathrm{Left}; k} \chi_{O,j}^+\|_{L^\infty_y}$. From the expression \eqref{PhiIota:1}, we clearly have 
\begin{align*}
|\text{Term}_u| \lesssim |\gamma_u| + \sum_{j =0}^3 \| \p_y^j G_{\mathrm{Left}} \langle y \rangle^{N_0} \|_{L^2_y}.
\end{align*}
Next, by integrating \eqref{kia:1} from $z = 0$ (which corresponds to $y = y_1(\eps)$) as well as \eqref{PhiIota:1}, we have 
\begin{align*}
|\text{Term}_\psi| \lesssim L^{\frac16} \| \langle \eta \rangle^2 F_{\mathrm{Left}} \|_{H^2}+ |\gamma_\psi| + \sum_{j =0}^3 \| \p_y^j G_{\mathrm{Left}} \langle y \rangle^{N_0} \|_{L^2_y}.
\end{align*}

\vspace{2 mm}

\noindent \underline{Estimates on $U_{\mathrm{Right}}$:}  These estimates follow in an analogous manner (but are strictly simpler) than those above on $U_{\mathrm{Left}}$. 

\vspace{2 mm}

Putting these bounds together proves \eqref{yeah:yea:ye:1}. For the proof of \eqref{yeah:yea:ye:1b}, we simply scale the contribution from $U_{\mathrm{Left}}$ by $L^{\frac13}$. This concludes the proof of the lemma. 
\end{proof}

\begin{lemma} The following bounds are valid: 
\begin{align} \n
\mathrm{\bold{Dat}}_k \lesssim & L^{\frac16} (\| F_{\mathrm{Left}} \langle \eta \rangle^3 \|_{H^{2 + 3k}_\eta} + \| F_{\mathrm{Right}} \langle \eta \rangle^3 \|_{H^{2 + 3k}_\eta})  +(|\gamma_{u;k}| + |\gamma_{\psi; k}|)\\ \label{apply:me:k}
& + L^{\frac16}\sum_{k' = 0}^{k-1} (|\gamma_{u; k'}| + |\gamma_{\psi; k'}|)+ L^k( \| G_{\mathrm{Right}} \|_{H^{3+3k}_y} + \|  G_{\mathrm{Left}} \langle y \rangle^{N_0} \|_{H^{3 + 2k}_y}), \\ \n
\mathrm{\bold{Dat}}_{\mathrm{Below}, k} \lesssim &L^{\frac16} (\| F_{\mathrm{Left}} \langle \eta \rangle^3 \|_{H^{2 + 3k}_\eta} + \| F_{\mathrm{Right}} \langle \eta \rangle^3 \|_{H^{2 + 3k}_\eta})+ L^{\frac13}(|\gamma_{u;k}| + |\gamma_{\psi; k}|) \\ \label{apply:me:k:bel}
&+ L^{\frac16}\sum_{k' = 0}^{k-1} (|\gamma_{u; k'}| + |\gamma_{\psi; k'}|) +  L^k( \| G_{\mathrm{Right}} \|_{H^{3+3k}_y} +L^{\frac13} \|  G_{\mathrm{Left}} \langle y \rangle^{N_0} \|_{H^{3 + 2k}_y}).
\end{align}
\end{lemma}
\begin{proof}This follows in an analogous manner to the $k = 0$ case presented above, but with the inductive bounds from Lemmas \ref{lem:one}, \ref{lem:two}, \ref{lem:three}, \ref{lem:four} being used as opposed to the given expressions \eqref{PhiIota:1} -- \eqref{PhiIota:4}.
\end{proof}

We are now ready to prove Proposition \ref{prop:data:M}.
\begin{proof}[Proof of Proposition \ref{prop:data:M}] We simply apply estimates \eqref{apply:me:k} and \eqref{apply:me:k:bel} with \eqref{QbdLR} and sum over $k$ in $0 \le k \le k_\ast$ in order to prove \eqref{est:data:M} and \eqref{est:data:below:M}, respectively.
\end{proof}

\section{Maximum Regularity Bounds} \label{Lpr}

In this section, we continue as in the previous section to study the full Prandtl system. To motivate the analysis of this section, we turn to the second term on the right-hand side of \eqref{sourke}. We observe the $L^\infty$ based term in this bound has the property that they are at low $x$ derivatives ($\lfloor \frac{k_\ast + 1}{2} \rfloor$) but relatively high number of $y$ derivatives ($4$). Therefore, our main goal in this section is to trade away $x$ derivatives to gain $y$ derivatives, which will enable us to control these terms.\footnote{We note that an alternative proof could have avoided this section entirely by propagating higher $y$ derivatives in all the previous sections, say a number $k_{\ast \ast}$ to be chosen, but we opt to differentiate many times in the $x$ direction and recover $y$ derivatives using the analysis of this section. This results in simpler expressions with the product rule throughout the paper.} For this section only (since we can ``lose" $x$ derivatives in these estimates), we fully separate out the linear and quasilinear terms as in \eqref{linHGYU:1}. We define our maximal regularity norm: 
\begin{align} \label{MR:norm}
\| \psi_R \|_{\mathrm{MR}} := & \sum_{k = 0}^{k_\ast - 1} \sum_{j = 1}^{J_k} \| \tilde{L}^j \p_y^{2 + j} u_{R,k} \langle y \rangle^{N_k - j} \|_{L^2_{xy}}, \qquad \tilde{L} := \min\{L, 1\}
\end{align}
where the index $J_k$ appearing above is defined as follows 
\begin{align}
J_k := 2(k_\ast - k).
\end{align}
The feature of the weight index $N_k - j$ appearing is due to the fact that higher $x$ derivatives have weaker decay due to \eqref{def:NsubK}, and so when we trade away $x$ derivatives to gain $y$ derivatives, we also need to account for the weakening in the weight. The feature of the weights in $\tilde{L}^j$ are relevant when $0 < L << 1$. Since control scaled derivatives $L \p_x$ in our linear norm, when $L$ is small this scaled derivative is weaker has the potential to be weak. In order to account for this, we need to also scale the $y$ derivatives appearing in \eqref{MR:norm} by the appropriate factors of $L$. 

The main proposition of this section will be 
\begin{proposition} \label{pro:MR} Let $k_\ast \ge 5$. Fix any $L_{\mathrm{Max}} < \infty$, and let $0 < L < L_{\mathrm{Max}}$. The following estimate is valid: 
\begin{align} \label{main:pro:MR}
\| \psi_R \|_{\mathrm{MR}} \lesssim & \| \psi_R \|_{\mathrm{Linear}} + \eps L^{-1} ( \sum_{k = 0}^{k_\ast - 1} \sum_{j = 0}^{J_k } \|\p_y^j \psi_{R,k} \|_{L^\infty} +  \| \psi_R \|_{\mathrm{Linear}}  )(\| \psi_R \|_{\mathrm{MR}}  + \| \psi_R \|_{\mathrm{Linear}} ). 
\end{align}
\end{proposition}

\begin{remark} As in the previous sections, the estimate \eqref{main:pro:MR} is valid for all $L \in (0, L_{\mathrm{Max}})$. We omit the dependence of implicit constants on large values of $L_{\mathrm{Max}}$, and we freely use the inequality $L \lesssim C(L_{\mathrm{Max}}) \lesssim 1$. As all the lemmas in this section are valid on $(0, L_{\mathrm{Max}})$, we do not repeatedly state this for each lemma. 
\end{remark}

\subsection{Analysis of $\mathcal{L}_{\text{Lin}}[\mathcal{R}, \mathcal{M}]$}

In this subsection, we perform an analysis on the system \eqref{linHGYU:1} as an abstract system with a given source term. More precisely, define 
\begin{align} \label{linHGYU:abs}
u_{FS} \p_x u + \p_y u_{FS} v + v_{FS} \p_y u + \p_x u_{FS} u + \widetilde{\bold{C}}_3  \psi - \p_y^2 u = \mathcal{R} + \mathcal{M},
\end{align}
where $\mathcal{R}, \mathcal{M}$ are given source terms. 

\begin{lemma} Let $j \ge 1, N \ge j$. Let $u$ satisfy \eqref{linHGYU:abs}. Then the following estimates are valid:  
\begin{align} \n
\| \tilde{L}^j \p_y^{2 + j} u \langle y \rangle^{N-j} \|_{L^2_{xy}} \lesssim& \| \psi \|_{\mathrm{Eff}, N} +  \| \tilde{L} \p_x \psi \|_{\mathrm{Eff}, N-1} + \tilde{L} \sum_{j' = 1}^{j-2} \| \tilde{L}^{j' + 1} \p_y^{2 + j'} u_x \langle y \rangle^{N-1 - j'} \|_{L^2} \\ \label{mr:lin}
&+ \sum_{j' = 1}^j  (\| \tilde{L}^{j'} \p_y^{j'} \mathcal{R} \langle y \rangle^{N-j'} \|_{L^2_{xy}} + \| \tilde{L}^{j'} \p_y^{j'} \mathcal{M} \langle y \rangle^{N-j'} \|_{L^2_{xy}}).
\end{align}
\end{lemma}
\begin{proof} It turns out to be more convenient to establish the following bound without the weights of $L$:
\begin{align} \n
\| \p_y^{2 + j} u \langle y \rangle^{N-j} \|_{L^2_{xy}} \lesssim& \| \psi \|_{\mathrm{Eff}, N} +  \|  \p_x \psi \|_{\mathrm{Eff}, N-1} +  \sum_{j' = 1}^{j-2} \|  \p_y^{2 + j'} u_x \langle y \rangle^{N-1 - j'} \|_{L^2} \\ \label{mr:lin:bare}
&+ \sum_{j' = 1}^j  (\| \p_y^{j'} \mathcal{R} \langle y \rangle^{N-j'} \|_{L^2_{xy}} + \| \p_y^{j'} \mathcal{M} \langle y \rangle^{N-j'} \|_{L^2_{xy}}),
\end{align}
after which multiplying both sides by $\tilde{L}^{j}$ (and using that $\tilde{L} \le 1$ by definition) yields \eqref{mr:lin}.

We will induct on $j \ge 1$. Therefore, we first prove the result for $j = 1$ as the base case. 

\vspace{2 mm}

\noindent \textit{Base Case ($j = 1$):} Applying $\p_y$ to \eqref{linHGYU:1} results in the following identity (after rearranging for $\p_y^3 u)$: 
\begin{align} \label{linHGYU:2}
-\p_y^3 u = - \sum_{i = 1}^5 \underline{LT}^{(1)}_i + \p_y \mathcal{R} + \p_y \mathcal{M},
\end{align}
where we define $\underline{LT}_i$ to be the linear transport terms: 
\begin{align*}
\underline{LT}^{(j)}_1 := & \p_y^j \{ u_{FS} \p_x u \}, \\
\underline{LT}^{(j)}_2 := & \p_y^j \{  \p_y u_{FS} v \}, \\
\underline{LT}^{(j)}_3 := &\p_y^j \{  v_{FS} \p_y u \}, \\
\underline{LT}^{(j)}_4 := &\p_y^j \{ \p_x u_{FS} u  \},  \\
\underline{LT}^{(j)}_5 := &\p_y^j \{ \widetilde{\bold{C}}_3 \psi\}. 
\end{align*}
From \eqref{linHGYU:2}, we obtain the estimate 
\begin{align*}
\|  \p_y^{3} u \langle y \rangle^{N - 1} \|_{L^2} \lesssim \sum_{i = 1}^5\| \underline{LT}^{(1)}_i \langle y \rangle^{N-1} \|_{L^2} + \| \p_y \mathcal{R} \langle y \rangle^{N-1} \|_{L^2} + \| \p_y \mathcal{M} \langle y \rangle^{N-1} \|_{L^2}.
\end{align*}
We subsequently estimate 
\begin{align*}
\| \underline{LT}_1^{(1)} \langle y \rangle^{N-1} \|_{L^2} \lesssim & \| u_x \langle y \rangle^{N-1} \|_{L^2} + \| u_{xy} \langle y \rangle^{N-1} \|_{L^2} \lesssim \| \p_x \psi \|_{\text{Eff}, N-1}, \\
\| \underline{LT}_2^{(1)} \langle y \rangle^{N-1} \|_{L^2} \lesssim &\| u_x  \|_{L^2} \lesssim \| \p_x \psi \|_{\text{Eff}, N-1}, \\
\| \underline{LT}_3^{(1)} \langle y \rangle^{N-1} \|_{L^2} \lesssim & \| u_y \langle y \rangle^{N-1} \|_{L^2} + \| u_{yy} \langle y \rangle^{N} \|_{L^2} \lesssim \| \psi \|_{\text{Eff}, N}  \\
\| \underline{LT}_4^{(1)} \langle y \rangle^{N-1} \|_{L^2} \lesssim & \| u \langle y \rangle^{N-1} \|_{L^2} + \| u_y \langle y \rangle^{N-1} \|_{L^2} \lesssim \| \psi \|_{\text{Eff}, N}  \\
\| \underline{LT}_5^{(1)} \langle y \rangle^{N-1} \|_{L^2} \lesssim & \|  \widetilde{\bold{C}}_3 \langle y \rangle^{N-1} \|_{L^\infty}\| u \langle y \rangle^{N-1} \|_{L^2} + \| \p_y \widetilde{\bold{C}}_3 \langle y \rangle^{N} \|_{L^\infty} \| \frac{\psi}{\langle y \rangle} \|_{L^2} \lesssim \| \psi \|_{\text{Eff}, N}. 
\end{align*}
Above, for the term $\underline{LT}_3^{(1)}$ we have used that $v_{FS} \sim y$ for large $y$, hence the loss of one weight. 

\vspace{2 mm}

\noindent \textit{Inductive Step:} We now induct by fixing any $2 \le j$, and assuming that inequality \eqref{mr:lin:bare} is valid up to the index $j - 1$. We apply $\p_y^j$ to \eqref{linHGYU:1}, which produces the identity 
\begin{align}
- \p_y^{2 + j} u =  - \sum_{i = 1}^5 \underline{LT}^{(j)}_i + \p_y^j \mathcal{R} + \p_y^j \mathcal{M},
\end{align}
and subsequently the estimate 
\begin{align*}
\|  \p_y^{2 + j} u \langle y \rangle^{N - j} \|_{L^2} \lesssim \sum_{i = 1}^5\| \underline{LT}^{(j)}_i \langle y \rangle^{N-j} \|_{L^2} + \| \p_y^j \mathcal{R} \langle y \rangle^{N-j} \|_{L^2} + \| \p_y^j \mathcal{M} \langle y \rangle^{N-j} \|_{L^2}
\end{align*}
We turn to estimating the $\underline{LT}_i^{(j)}$ terms on the right-hand side. Here, a delicate bookkeeping of indices is required. We proceed for the first term as follows:
\begin{align*}
\|  \underline{LT}^{(j)}_1 \langle y \rangle^{N-j} \|_{L^2} \lesssim & \sum_{j' = 0}^j \| \p_y^{j'} u_x \langle y \rangle^{N-j} \|_{L^2} \lesssim  \sum_{j' = 0}^2 \| \p_y^{j'} u_x \langle y \rangle^{N-j} \|_{L^2} + \sum_{j' = 3}^j \| \p_y^{j'} u_x \langle y \rangle^{N-j} \|_{L^2} \\
\lesssim &  \sum_{j' = 0}^2 \| \p_y^{j'} u_x \langle y \rangle^{N-j} \|_{L^2} + \sum_{j' = 1}^{j-2} \| \p_y^{2 + j'} u_x \langle y \rangle^{N-j} \|_{L^2} \\
\lesssim &  \sum_{j' = 0}^2 \| \p_y^{j'} u_x \langle y \rangle^{N-1} \|_{L^2} + \sum_{j' = 1}^{j-2} \| \p_y^{2 + j'} u_x \langle y \rangle^{N-1 - j'} \|_{L^2} \\
\lesssim & \| \p_x \psi \|_{\text{Eff}, N-1} +  \sum_{j' = 1}^{j-2} \| \p_y^{2 + j'} u_x \langle y \rangle^{N-1 - j'} \|_{L^2},
\end{align*} 
as required by the right-hand side of \eqref{mr:lin:bare}. The next term is far simpler due to the rapid decay of $\p_y u_{FS}$, and reads as follows: 
\begin{align*}
\|  \underline{LT}^{(j)}_2 \langle y \rangle^{N-j} \|_{L^2} \lesssim & \| \frac{v}{\langle y \rangle} \|_{L^2} + \sum_{j' = 1}^j \| \p_y^{j'-1} u_x  \|_{L^2} \lesssim  \| \p_x \psi \|_{\text{Eff}, N-1} +  \sum_{j' = 1}^{j-2} \| \p_y^{2 + j'} u_x \langle y \rangle^{N-1 - j'} \|_{L^2}.
\end{align*} 
For the third term, due to the linear growth of $v_{FS}$ in $y$, we distinguish this contribution as follows: 
\begin{align*}
\|  \underline{LT}^{(j)}_3 \langle y \rangle^{N-j} \|_{L^2} \lesssim &\| \p_y^{j+1} u \langle y \rangle^{N+1-j} \|_{L^2} + \sum_{j' = 0}^{j-1} \| \p_y^{j' + 1}u \langle y \rangle^{N-j} \|_{L^2} \\
\lesssim & \| \p_y^{(j-1) + 2} u \langle y \rangle^{N-(j-1)} \|_{L^2} + \sum_{j' = 0}^{1} \| \p_y^{j' + 1}u \langle y \rangle^{N-j} \|_{L^2} + \sum_{j' = 2}^{j-1} \| \p_y^{j' + 1}u \langle y \rangle^{N-j} \|_{L^2} \\
\lesssim & \| \p_y^{(j-1) + 2} u \langle y \rangle^{N-(j-1)} \|_{L^2} +\| \psi \|_{\text{Eff},N} + \sum_{j' = 1}^{j-2} \| \p_y^{2 + j'}u \langle y \rangle^{N-j} \|_{L^2}  \\
\lesssim & \| \p_y^{(j-1) + 2} u \langle y \rangle^{N-(j-1)} \|_{L^2} +\| \psi \|_{\text{Eff},N} + \sum_{j' = 1}^{j-2} \| \p_y^{2 + j'}u \langle y \rangle^{N-j'} \|_{L^2}. 
\end{align*}
For these first and third terms appearing on the right-hand side, we appeal to the inductive hypothesis. The next term is again far simpler due to the boundedness of $\p_x u_{FS}$, and we have 
\begin{align*}
\|  \underline{LT}^{(j)}_4 \langle y \rangle^{N-j} \|_{L^2} \lesssim & \sum_{j' = 0}^j \| \p_y^{j'} u \langle y \rangle^{N-j} \|_{L^2} \lesssim  \sum_{j' = 0}^2 \| \p_y^{j'} u \langle y \rangle^{N-j} \|_{L^2} +  \sum_{j' = 3}^j \| \p_y^{j'} u \langle y \rangle^{N-j} \|_{L^2} \\
\lesssim &  \| \psi \|_{\text{Eff},N} +  \sum_{j' = 1}^{j-2} \| \p_y^{2+j'} u \langle y \rangle^{N-j'} \|_{L^2},
\end{align*} 
and we again appeal to the inductive hypothesis for the final quantity appearing on the right-hand side. For the final term, we have 
\begin{align*}
\| \underline{LT}^{(j)}_5 \langle y \rangle^{N-j} \|_{L^2} \lesssim & \| \p_y^j \widetilde{\bold{C}}_3 \langle y \rangle^{N-j + 1} \|_{L^\infty} \| \frac{\psi}{\langle y \rangle} \|_{L^2} + \sum_{j' = 1}^j \| \p_y^{j-j'} \widetilde{\bold{C}}_3 \|_{L^\infty} \| \p_y^{j'} \psi \langle y \rangle^{N-j} \|_{L^2} \\
\lesssim & (\| \p_y^j \widetilde{\bold{C}}_3 \langle y \rangle^{N-j + 1} \|_{L^\infty} + \sum_{j' = 0}^{j-1} \| \widetilde{\bold{C}}_3  \|_{L^\infty}   ) \sum_{j' = 0}^{j-1} \| \p_y^{j'} u \langle y \rangle^{N-j} \|_{L^2} \\
\lesssim & (\| \p_y^j \widetilde{\bold{C}}_3 \langle y \rangle^{N-j + 1} \|_{L^\infty} + \sum_{j' = 0}^{j-1} \| \widetilde{\bold{C}}_3  \|_{L^\infty}   ) ( \sum_{j' = 0}^{2} \| \p_y^{j'} u \langle y \rangle^{N-j} \|_{L^2} \\
& +  \sum_{j' = 1}^{j-3} \| \p_y^{2 + j'} u \langle y \rangle^{N-j'} \|_{L^2}) \\
\lesssim &\| \psi \|_{\text{Eff},N} + \sum_{j' = 1}^{j-3} \| \p_y^{2 + j'} u \langle y \rangle^{N-j'} \|_{L^2}.
\end{align*}
As before, we appeal to the inductive hypothesis for the final term appearing on the right-hand side above. This concludes the proof of the lemma. 
\end{proof}

\subsection{Linear Maximum Regularity Source Terms}

At this point, we deviate from the abstract setting above, and consider the specific system \eqref{linHGYU:1}. We therefore provide estimates on the specific source terms appearing \eqref{linHGYU:1}. We begin with a treatment of $\mathcal{R}_{k, \mathrm{Lin}}$, defined in \eqref{RkLin}.

\begin{lemma} For any $1 \le j \le J_k$, 
\begin{align} \label{des:w:L}
\| \tilde{L}^j \p_y^j \mathcal{R}_{k, \mathrm{Lin}} \langle y \rangle^{N_k-j} \|_{L^2_{xy}} \lesssim  \sum_{k' = 0}^{k-1} \| \psi_{R, k'} \|_{\mathrm{Eff}, N_{k'}} +\tilde{L} \sum_{k' = 0}^{k-1} \sum_{j' = 1}^{j-1} \| \tilde{L}^{j'} \p_y^{2 + j'} u_{R, k'} \langle y \rangle^{N_{k'}-j'} \|_{L^2_{xy}}. 
\end{align}
\end{lemma}
\begin{proof} Just as the previous lemma, it is simpler to prove the $\tilde{L}$-unweighted estimate first, and then scale by factors of $\tilde{L}$ as we need for the future. Therefore, we prove: 
\begin{align} \label{des:wOUT:L}
\| \p_y^j \mathcal{R}_{k, \mathrm{Lin}} \langle y \rangle^{N_k-j} \|_{L^2_{xy}} \lesssim  \sum_{k' = 0}^{k-1} \| \psi_{R, k'} \|_{\mathrm{Eff}, N_{k'}} + \sum_{k' = 0}^{k-1} \sum_{j' = 1}^{j-1} \| \p_y^{2 + j'} u_{R, k'} \langle y \rangle^{N_{k'}-j'} \|_{L^2_{xy}}.
\end{align}
Clearly, multiplying both sides of the above inequality by $\tilde{L}^j$ (and throwing away excess powers of $\tilde{L}$) implies the desired bound, \eqref{des:w:L}.

Second, in order to prove \eqref{des:wOUT:L}, we notice it suffices to prove the estimate: 
\begin{align} \n
\| \p_y^j \mathcal{R}_{k, \mathrm{Lin}} \langle y \rangle^{N_k-j} \|_{L^2_{xy}} \lesssim &\sum_{k' = 0}^{k-1} \| \psi_{R,k'} \|_{L^\infty} +  \sum_{k' = 0}^{k-1} \sum_{j' = 0}^{2} \| \p_y^{j'} u_{R, k'} \langle y \rangle^{N_k+1} \|_{L^2_{xy}} \\ \label{ifitdont}
&+  \sum_{k' = 0}^{k-1} \sum_{j' = 3}^{j+1} \| \p_y^{j'} u_{R, k'} \langle y \rangle^{N_k+3-j'} \|_{L^2_{xy}}
\end{align}
by invoking the definition of the effective norm, \eqref{eff:norm}, and the property of $N_k + 1 \le N_{k'}$ for $k' \le k-1$ according to \eqref{def:NsubK}.  

We recall the definition of $\mathcal{R}_k$ from \eqref{defRKRK}. We have 
\begin{align*}
\| \p_y^j \mathcal{R}^{(1)}_{k, \mathrm{Lin}} \langle y \rangle^{N_k-j} \|_{L^2_{xy}} \lesssim & 1_{k \ge 2} \sum_{k' = 0}^{k-2} \sum_{j' = 0}^j L^{k-k' - 1} \| \p_y^{j'} u_{R,k' + 1} \langle y \rangle^{N_k-j} \|_{L^2_{xy}}, \\
\| \p_y^j \mathcal{R}^{(2)}_{k, \mathrm{Lin}} \langle y \rangle^{N_k-j} \|_{L^2_{xy}} \lesssim &1_{k \ge 1} \sum_{k' = 0}^{k-1} \sum_{j' = 0}^j L^{k-k' } \| \p_y^{j'} u_{R,k'} \langle y \rangle^{N_k-j} \|_{L^2_{xy}},
\end{align*}
both of which satisfy the bound on the right-hand side of \eqref{ifitdont}. Next, we have the $v_{FS}$ term which grows linearly in $y$. To account for this, we have 
\begin{align*}
\| \p_y^j \mathcal{R}^{(3)}_{k, \mathrm{Lin}} \langle y \rangle^{N_k-j} \|_{L^2_{xy}} \lesssim & 1_{k \ge 1} \sum_{k' = 0}^{k-1} \sum_{j' = 0}^j \| \frac{\p_y^{j - j'} \p_x^{k-k'} v_{FS}}{\langle y \rangle} \|_{L^\infty} L^{k-k'} \| \p_y^{j' + 1} u_{R, k'} \langle y \rangle^{N_k+1-j} \|_{L^2_{xy}} \\
\lesssim & 1_{k \ge 1} \sum_{k' = 0}^{k-1} \sum_{j' = 1}^{j+1}  L^{k-k'} \| \p_y^{j'} u_{R, k'} \langle y \rangle^{N_{k'}-j} \|_{L^2_{xy}}.
\end{align*}
For the final linear term, we have 
\begin{align*}
\| \p_y^j \mathcal{R}^{(4)}_{k, \mathrm{Lin}} \langle y \rangle^{N_k-j} \|_{L^2_{xy}} \lesssim &1_{k \ge 2} \sum_{k' = 0}^{k-2} L^{k-k' - 1} \| \psi_{R,k' + 1} \|_{L^\infty_{xy}} \| \p_y^{j+1} \p_x^{k-k'} u_{FS} \langle y \rangle^{N_k-j} \|_{L^2_{xy}} \\
& + 1_{k \ge 2} \sum_{k' = 0}^{k-2} \sum_{j' = 1}^j L^{k-k' - 1} \| \p_y^{j' - 1} u_{R,k' + 1} \|_{L^2_{xy}} \| \p_y^{j-j' + 1} \p_x^{k-k} u_{FS}\langle y \rangle^{N_k-j} \|_{L^2_{xy}} \\
\lesssim & 1_{k \ge 2} \sum_{k' = 1}^{k-1} L^{k - k'}  \| \psi_{R,k' } \|_{L^\infty_{xy}} +  1_{k \ge 2} \sum_{k' = 1}^{k-1} \sum_{j' = 0}^{j-1} L^{k-k' } \| \p_y^{j'} u_{R,k'} \|_{L^2_{xy}},
\end{align*}
which is bounded by the right-hand side of \eqref{ifitdont}. The lemma is proven. 
\end{proof}

\subsection{Semilinear Maximum Regularity Estimates}

We provide estimates on $\mathcal{R}_{k; \text{Semi}}$ which will be used for our maximum regularity analysis. 
\begin{lemma} \label{lmrsem}Let $k_\ast \ge 5$. Fix any $0 \le k \le k_\ast -1$. Let $1 \le j \le J_k = 2(k_\ast - k)$. Define $N_k := n_\ast - k$. Then the following nonlinear estimates hold: 
\begin{align} \n
\| L^j \p_y^j \mathcal{R}_{k; \mathrm{Semi}} \langle y \rangle^{N_k - j} \|_{L^2_{xy}} \lesssim &  \eps L^{-1} ( \sum_{k = 0}^{k_\ast - 1} \sum_{j = 0}^{J_k } \|\p_y^j \psi_{R,k} \|_{L^\infty}  )(\sum_{k' = 0}^{k-1} \| \psi_{R, k'} \|_{\mathrm{Eff}, N_{k'}} \\ \label{pluplu:1}
&\qquad \qquad + \sum_{k' = 0}^{k-1} \sum_{j' = 1}^{J_{k'}} \| L^{j'} \p_y^{2 + j'} u_{R, k'} \langle y \rangle^{N_{k'}-j'} \|_{L^2_{xy}} ).
\end{align}
\end{lemma}
\begin{proof} We recall the definition of $\mathcal{R}_{k, \mathrm{Semi}}$ from \eqref{RkSemi}. We consider the estimate of 
\begin{align} \label{ltouc:1}
\| \p_y^j \mathcal{R}^{(1)}_{k; \text{Semi}} \langle y \rangle^{N_k - j} \|_{L^2_{xy}} \lesssim \eps L^{-1} \sum_{k' = 1}^{k-2} \sum_{j' = 0}^j \| \p_y^{j-j'} u_{R,k-k'} \p_y^{j'} u_{R,k' + 1} \langle y \rangle^{N_k - j} \|_{L^2_{xy}}
\end{align}
In order to decide which of the terms above to put in $L^\infty$ and use Sobolev embedding versus putting in $L^2$, we argue using the ``tame principle" as follows. If $j' \le  J_{k' + 1} - 4$, then we will put the second term in the product in $L^\infty$. Therefore, we assume $j' > J_{k' + 1} - 4 = 2(k_\ast - (k' + 1)) - 4 = 2k_\ast - 2k' - 6$. In this case we provide the following bound on indices: 
\begin{align*}
j - j' \le& J_k - j' = 2k_\ast - 2k - j' < 2k_\ast - 2k - (2k_\ast - 2k' - 6) = 6 - 2k + 2k' \\
\le & (2k_\ast - 4) - 2k + 2k' = 2k_\ast - 2(k-k') - 4 = J_{k-k'} - 4, 
\end{align*}
where to go from the first line to the second line we have used $2k_\ast - 4 \ge 6 \iff k_\ast \ge 5$. The conclusion, therefore, is that either $j' \le J_{k' + 1} - 4$ or $j-j' \le J_{k-k'} - 4$. Using this, we may estimate the term above as follows: 
\begin{align} \n
\|L^j \p_y^j \mathcal{R}^{(1)}_{k; \text{Semi}} \langle y \rangle^{N_k - j} \|_{L^2_{xy}} \lesssim \eps L^{-1} & \Big( \sum_{k' = 0}^{k-1} \sum_{j' = 0}^{J_{k'} - 4} \| \p_y^{j'} u_{R,k'} \|_{L^\infty_{xy}} \Big) \Big( \sum_{k' = 0}^{k-1} \sum_{j' = 0}^{J_{k'}} \| L^{j'} \p_y^{j'} u_{R,k'} \langle y \rangle^{N_{k'} - j'} \|_{L^2_{xy}} \Big) 
\end{align}
For the second semilinear term, we decompose it into two sub-terms as follows: 
\begin{align} \n
\| L^j \p_y^j \mathcal{R}^{(2)}_{k; \text{Semi}} \langle y \rangle^{N_k - j} \|_{L^2_{xy}} \lesssim & \eps L^{-1} \sum_{k' = 1}^{k-2} \sum_{j' = 0}^{j-1} L^j \| \p_y^{j'} u_{R,k'+1} \p_y^{j-j'} u_{R,k-k'} \langle y \rangle^{N_k - j} \|_{L^2_{xy}} \\
& +\eps L^{-1} \sum_{k' = 1}^{k-2}L^j \| \psi_{R,k' + 1} \p_y^{j+1} u_{R,k-k'} \langle y \rangle^{N_k - j} \|_{L^2_{xy}} := A + B. 
\end{align}
A comparison with \eqref{ltouc:1} shows that the estimate of $A$ is identical to that of the previous term. To estimate $B$, we proceed as follows: 
\begin{align*}
B \lesssim & \eps L^{-1} (\sum_{k' = 0}^{k-1} \| \psi_{R,k'} \|_{L^\infty}) (\sum_{k' = 0}^{k-1} \| L^j \p_y^{j+1} u_{R,k'}   \langle y \rangle^{N_k - j} \|_{L^2_{xy}} ) \\
\lesssim & \eps L^{-1} (\sum_{k' = 0}^{k-1} \| \psi_{R,k'} \|_{L^\infty}) (\sum_{k' = 0}^{k-1} \|L^j \p_y^{j+1} u_{R,k'}   \langle y \rangle^{N_{k'} - (j+1)} \|_{L^2_{xy}} ).
\end{align*}
Using the fact that $j \le J_k$, we have $j+1 \le J_{k} + 1 \le J_{k' }$. The lemma is proven. 
\end{proof}

\subsection{Quasilinear Maximum Regularity Estimates}

We provide estimates on $\mathcal{M}_{k; \text{Quasi}}$, defined in \eqref{def:MkQuasi}. 
\begin{lemma} Let $k_\ast \ge 2$. Fix any $0 \le k \le k_\ast-1$. Let $1 \le j \le J_k = 2(k_\ast - k)$. Define $N_k := n_\ast - k$. Then the following nonlinear estimates hold: 
\begin{align} \n
\| L^j \p_y^j \mathcal{M}_{k; \text{Quasi}} \langle y \rangle^{N_k - j} \|_{L^2_{xy}} \lesssim&  \eps L^{-1} ( \sum_{k = 0}^{k_\ast - 1} \sum_{j = 0}^{J_k } \|\p_y^j \psi_{R,k} \|_{L^\infty} + \| \psi \|_{\mathrm{Linear}} )(\sum_{k' = 0}^{k+1} \| \psi_{R, k'} \|_{\mathrm{Eff}, N_{k'}} \\ \label{qmr:bd:1}
&\qquad \qquad + \sum_{k' = 0}^{k+1} \sum_{j' = 1}^{J_{k'}} \| L^{j'} \p_y^{2 + j'} u_{R, k'} \langle y \rangle^{N_{k'}-j'} \|_{L^2_{xy}} ).
\end{align}
\end{lemma}
\begin{proof} We treat separately the cases $k \ge 1$ and $k = 0$. 

\vspace{2 mm}

\noindent \underline{Case $k \ge 1$} We first have 
\begin{align*}
\| L^j \p_y^j \mathcal{M}_{k; \text{Quasi}}^{(1)} \langle y \rangle^{N_k - j} \|_{L^2_{xy}} \lesssim & \sum_{j' = 0}^j \eps L^{-1} \| \p_y^{j-j'} u_R \|_{L^\infty} \| L^{j'} \p_y^{j'} u_{R, k+1} \langle y \rangle^{N_k - j} \|_{L^2_{xy}} \\
\lesssim & \eps L^{-1} (\sum_{j' = 0}^{J_k} \| \p_y^{j'} u_R \|_{L^\infty}) (\sum_{j' = 0}^{j} \| L^{j'}  \p_y^{j'} u_{R, k + 1} \langle y \rangle^{N_k - j} \|_{L^2_{xy}} ) \\
\lesssim & \eps L^{-1} (\sum_{j' = 0}^{J_k} \| \p_y^{j'} u_R \|_{L^\infty}) (\sum_{j' = 0}^{J_k} \| L^{j'}  \p_y^{j'} u_{R, k + 1} \langle y \rangle^{N_{k+1} -( j' - 1)} \|_{L^2_{xy}} ) \\
\lesssim & \eps L^{-1} (\sum_{j' = 0}^{J_0-2} \| \p_y^{j'} u_R \|_{L^\infty})( (\sum_{j' = 1}^{J_{k+1}} \| L^{j'}  \p_y^{2 + j'} u_{R, k + 1} \langle y \rangle^{N_{k+1} -j'} \|_{L^2_{xy}} )  \\
& + (\sum_{j' = 0}^{2} \| \p_y^{j'} u_{R, k + 1} \langle y \rangle^{N_{k+1} } \|_{L^2_{xy}} ) ),
\end{align*}
Clearly, the right-hand side above is bounded by the right-hand side of \eqref{qmr:bd:1}. As in the previous lemmas, by decomposing 
\begin{align*}
 \p_y^j \mathcal{M}_{k; \text{Quasi}}^{(2)} = \p_y^j \{ \eps u_{Ry} v_{R, k} \} = \eps L^{-1} \sum_{j' = 0}^{j-1} \binom{j}{j' + 1} \p_y^{j-j'} u_R \p_y^{j'} u_{R, k+1} + \eps L^{-1} \p_y^{j+1} u_R \psi_{R, k+1}, 
\end{align*}
and noticing that the first sum on the right-hand side is contained in the treatment of $\mathcal{M}_{k; \text{Quasi}}^{(1)}$, it suffices to treat the second term on the right-hand side. We do so as follows:
\begin{align*}
L^j \| \eps L^{-1} \p_y^{j+1} u_R \psi_{R, k+1} \langle y \rangle^{N_k - j} \|_{L^2_{xy}} \lesssim & \eps L^{-1} \| L^j \p_y^{j+1} u_R \langle y \rangle^{N_k - j + 1} \|_{L^2_{xy}} \| \psi_{R, k + 1} \|_{L^\infty} \\
\lesssim & \| \psi_{R, k +1} \|_{L^\infty} ( \sum_{j' = 0}^2 \| \p_y^{j'} u_R \langle y \rangle^{N_0} \|_{L^2_{xy}} + \sum_{j' = 1}^{J_0} \| L^{j'}\p_y^{2+ j'} u_R \langle y \rangle^{N_0 - j'} \|_{L^2_{xy}})\\
\lesssim & \| \psi \|_{\mathrm{Linear}} ( \sum_{j' = 0}^2 \| \p_y^{j'} u_R \langle y \rangle^{N_0} \|_{L^2_{xy}} + \sum_{j' = 1}^{J_0} \|  L^{j'} \p_y^{2+ j'} u_R \langle y \rangle^{N_0 - j'} \|_{L^2_{xy}}),
\end{align*}
which is again bounded by the right-hand side of \eqref{qmr:bd:1}. 

Next, we have 
\begin{align*}
\| L^j \p_y^j \mathcal{M}_{k; \text{Quasi}}^{(3)} \langle y \rangle^{N_k - j} \|_{L^2_{xy}} \lesssim & \eps L^{-1} \| \psi_{R,1} \|_{L^\infty} \|L^j \p_y^{j+1} u_{R,k} \langle y \rangle^{N_k - j} \|_{L^2_{xy}} \\
&+ \eps L^{-1} \sum_{j' = 1}^{j} \| \p_y^{j-j'} u_{R,1} \|_{L^\infty} \| L^j \p_y^{j'} u_{R,k} \langle y \rangle^{N_k - j} \|_{L^2_{xy}} \\
\lesssim & \eps L^{-1} (\| \psi_{R,1} \|_{L^\infty} + \sum_{j' = 0}^{J_1 - 1} \| \p_y^{j'} u_{R, 1} \|_{L^\infty} ) ( (\sum_{j' = 1}^{J_{k}} \| L^{j'} \p_y^{2 + j'} u_{R, k } \langle y \rangle^{N_{k} -j'} \|_{L^2_{xy}} )  \\
& + (\sum_{j' = 0}^{2} \| \p_y^{j'} u_{R, k} \langle y \rangle^{N_{k} } \|_{L^2_{xy}} ) ),
\end{align*}
which is bounded by the right-hand side of \eqref{qmr:bd:1} upon using that $k_\ast \ge 2$. 
\vspace{2 mm}

\noindent \underline{Subcase $k = 1$} In this case, we estimate 
\begin{align*}
\| L^j \p_y^j \mathcal{M}_{k; \text{Quasi}}^{(4)} \langle y \rangle^{N_k - j} \|_{L^2_{xy}} \lesssim & \eps L^{-1} \sum_{j' = 0}^j L^j \| \p_y^{j-j'} u_{R,1} \p_y^{j'} u_{R,1} \langle y \rangle^{N_1 - j} \|_{L^2_{xy}} \\
\lesssim & \eps L^{-1} (\sum_{j' = 0}^{J_1 - 1} \| \p_y^{j'} u_{R,1} \|_{L^\infty})(\sum_{j' = 0}^{2} \| \p_y^{j'} u_{R,1} \langle y \rangle^{N_1 } \|_{L^2_{xy}} \\
& + \sum_{j' = 1}^{J_1} \| L^{j'} \p_y^{2 + j'} u_{R,1} \langle y \rangle^{N_1 - j'} \|_{L^2_{xy}} ),
\end{align*}
which is bounded by the right-hand side of \eqref{qmr:bd:1} upon using that $k_\ast \ge 2$.

\vspace{2 mm}

\noindent \underline{Subcase $k \ge 2$} In this subcase, we have 
\begin{align*}
\| L^j \p_y^j \mathcal{M}_{k; \text{Quasi}}^{(4)} \langle y \rangle^{N_k - j} \|_{L^2_{xy}} \lesssim & ( \sum_{j' = 0}^{J_k} \| \p_y^{j'} u_{R,1} \|_{L^\infty}) (\sum_{j' = 0}^{J_k} \| L^{j'} \p_y^{j'} u_{R,k} \langle y \rangle^{N_k - j} \|_{L^2_{xy}}) \\
\lesssim & ( \sum_{j' = 0}^{J_1 - 2} \| \p_y^{j'} u_{R,1} \|_{L^\infty}) (\sum_{j' = 0}^{J_k} \| L^{j' } \p_y^{j'} u_{R,k} \langle y \rangle^{N_k - j} \|_{L^2_{xy}}),
\end{align*}
which is bounded by the right-hand side of \eqref{qmr:bd:1} upon using that $k_\ast \ge 2$.

For the term $\mathcal{M}_{k, \text{Quasi}}^{(5)}$, we decompose via 
\begin{align*}
\| L^j \p_y^j \mathcal{M}_{k, \text{Quasi}}^{(5)}\langle y \rangle^{N_k - j} \|_{L^2_{xy}}  \lesssim & \eps L^{-1} \sum_{j' = 1}^j  L^j \| \p_y^{j-j' + 1} u_{R,1} \p_y^{j' - 1} u_{R,k} \langle y \rangle^{N_k - j} \|_{L^2_{xy}} \\
&+ \eps L^{-1}  \| \psi_{R,k} \|_{L^\infty} \| L^j \p_y^{j+1} u_{R,1} \langle y \rangle^{N_k - j} \|_{L^2_{xy}} \\
\lesssim &  ( \sum_{j' = 0}^{J_1 - 2} \| \p_y^{j'} u_{R,1} \|_{L^\infty}) (\sum_{j' = 0}^{J_k} \| L^{j'} \p_y^{j'} u_{R,k} \langle y \rangle^{N_k - j} \|_{L^2_{xy}}) \\
& + \eps L^{-1}  \| \psi_{R,k} \|_{L^\infty} \sum_{j = 0}^{J_1} \| L^j \p_y^{j+1} u_{R,1} \langle y \rangle^{N_1 - j} \|_{L^2_{xy}}.
\end{align*},
which is bounded by the right-hand side of \eqref{qmr:bd:1} upon using that $k_\ast \ge 2$.
\vspace{2 mm}

\noindent \underline{Case $k = 0$} We turn now to the case when $k = 0$. Here, we notice that the form of $\mathcal{M}_{0, \text{Quasi}}$ simplifies to  
\begin{align}
\mathcal{M}_{0, \text{Quasi}} = \eps u_R \p_x u_R + \eps v_R \p_y u_R.
\end{align}
We proceed to estimate for $1 \le j \le J_0$,  
\begin{align*}
&\| L^j \p_y^j \mathcal{M}_{0; \text{Quasi}} \langle y \rangle^{N_0 - j} \|_{L^2_{xy}} \\
 \lesssim & \eps L^{-1} \sum_{j' = 0}^{j - 1} \| \p_y^{j'} u_R \|_{L^\infty} \| L^{j-j'} \p_y^{j - j'} u_{R,1} \langle y \rangle^{N_0 - j} \|_{L^2_{xy}} + \eps L^{-1} \| u_{R,1} \|_{L^\infty} \| L^j \p_y^{j} u_R \langle y \rangle^{N_0 - j} \|_{L^2_{xy}} \\
\lesssim & \eps L^{-1} (\sum_{j' = 0}^{J_0 - 1} \| \p_y^{j'} u_R \|_{L^\infty})(\sum_{j' = 0}^2  \| \p_y^{j'} u_{R,1} \langle y \rangle^{N_1 } \|_{L^2_{xy}}  + \sum_{j' = 1}^{J_1} \| L^{j'}  \p_y^{2 + j'} u_{R,1} \langle y \rangle^{N_1 - j'} \|_{L^2_{xy}} ) \\
& +  \eps L^{-1} \| u_{R,1} \|_{L^\infty}\| L^j \p_y^{j} u_R \langle y \rangle^{N_0 - j} \|_{L^2_{xy}}.
\end{align*}
Finally, we have by decomposing 
\begin{align*}
&\| L^j \p_y^j \mathcal{M}_{1; \text{Quasi}} \langle y \rangle^{N_0 - j} \|_{L^2_{xy}} \\
 \lesssim & \eps L^{-1} \sum_{j' = 1}^{j-1} L^j \| \p_y^{j-j'} u_{R,1} \p_y^{j'} u_R \langle y \rangle^{N_0 - j} \|_{L^2_{xy}} + \eps L^{-1} L^j \| u_{R,1} \p_y^j u_R \langle y \rangle^{N_0 - j} \|_{L^2_{xy}} \\
&+ \eps L^{-1} L^j \| \psi_{R,1} \p_y^{j+1} u_{R} \langle y \rangle^{N_0 - j} \|_{L^2_{xy}} \\
\lesssim & \eps L^{-1} (\sum_{j' = 0}^{J_0 - 1} \| \p_y^{j'} u_R \|_{L^\infty})(\sum_{j' = 0}^2  \| \p_y^{j'} u_{R,1} \langle y \rangle^{N_1 } \|_{L^2_{xy}}  + \sum_{j' = 1}^{J_1} \| L^{j'}  \p_y^{2 + j'} u_{R,1} \langle y \rangle^{N_1 - j'} \|_{L^2_{xy}} ) \\
& +  \eps L^{-1} \| u_{R,1} \|_{L^\infty}\| L^j \p_y^{j} u_R \langle y \rangle^{N_0 - j} \|_{L^2_{xy}} + \eps L^{-1} \| \psi_{R,1} \|_{L^\infty} \|L^j \p_y^{j+1} u_R \langle y \rangle^{N_0 - j} \|_{L^2_{xy}},
\end{align*}
where we notice that the first two terms on the right-hand side above have already been estimated as part of $\| \p_y^j \mathcal{M}_{0; \text{Quasi}} \langle y \rangle^{N_0 - j} \|_{L^2_{xy}}$. All of the contributions above in the case $k = 0$ are bounded by the right-hand side of \eqref{qmr:bd:1} upon using that $k_\ast \ge 2$. This concludes the proof of the lemma. 
\end{proof}

\subsection{Proof of Proposition \ref{pro:MR}}

We are now ready to bring together the above bounds to prove the main proposition of this section. 
\begin{proof}[Proof of Proposition \ref{pro:MR}] Our starting point is estimate \eqref{mr:lin}, which we apply for each $k$ in the range $0 \le k \le k_\ast -1$, with $N = N_k$, and with $1 \le j \le J_k$:
\begin{align} \n
\| \tilde{L}^j \p_y^{2 + j} u_{R,k} \langle y \rangle^{N_k-j} \|_{L^2_{xy}} \lesssim& \| \psi_k \|_{\mathrm{Eff}, N_k} +  \| \tilde{L} \p_x \psi_{R,k} \|_{\mathrm{Eff}, N_k-1} + \tilde{L} \sum_{j' = 1}^{j-2} \| \tilde{L}^{j' + 1} \p_y^{2 + j'} \p_x u_{R,k} \langle y \rangle^{N_k-1 - j'} \|_{L^2_{xy}} \\ \n
&+ \sum_{j' = 1}^j  (\| \tilde{L}^{j'} \p_y^{j'} \mathcal{R}_k \langle y \rangle^{N_k-j'} \|_{L^2_{xy}} + \| \tilde{L}^{j'} \p_y^{j'} \mathcal{M}_{k, \text{Quasi}} \langle y \rangle^{N_k-j'} \|_{L^2_{xy}}) \\ \n
\lesssim  & \| \psi_{R,k} \|_{\mathrm{Eff}, N_k}+  \|  \psi_{R,k+1} \|_{\mathrm{Eff}, N_k-1} + \tilde{L} \sum_{j' = 1}^{j-2} \| \tilde{L}^{j' + 1} \p_y^{2 + j'} \p_x u_{R,k} \langle y \rangle^{N_k-1 - j'} \|_{L^2_{xy}} \\ \n
&+ \sum_{j' = 1}^j  (\| \tilde{L}^{j'} \p_y^{j'} \mathcal{R}_k \langle y \rangle^{N_k-j'} \|_{L^2_{xy}} + \| \tilde{L}^{j'} \p_y^{j'} \mathcal{M}_{k, \text{Quasi}} \langle y \rangle^{N_k-j'} \|_{L^2_{xy}}) \\ \n
\lesssim  & \| \psi \|_{\mathrm{Linear}}+ \tilde{L} \sum_{j' = 1}^{j-2} \| \tilde{L}^{j' } \p_y^{2 + j'}  u_{R,k+1} \langle y \rangle^{N_k-1 - j'} \|_{L^2_{xy}} \\ \label{heli:1}
&+ \sum_{j' = 1}^j  (\| \tilde{L}^{j'} \p_y^{j'} \mathcal{R}_k \langle y \rangle^{N_k-j'} \|_{L^2_{xy}} + \| \tilde{L}^{j'} \p_y^{j'} \mathcal{M}_{k, \text{Quasi}} \langle y \rangle^{N_k-j'} \|_{L^2_{xy}}).
\end{align}
Above, we have used successively that $L \p_x \psi_{R,k} = \psi_{R,k+1}$, and that $N_{k} - 1 = N_{k+1}$. 

We now want to sum in the relevant $(k, j)$ indices\footnote{Even though these are finite sums, the purpose of switching the order of summation in this argument is to extract the appropriate inductive structure, which is in the $j$ variable as opposed to the $k$ variable (see \eqref{floorbed:1}).}. To do this most efficiently, we switch the sum appearing in \eqref{MR:norm} so that the $j$ summation appears on the outside and the $k$ summation on the inside. Define 
\begin{align*}
K_j := k_\ast - \frac{\mu(j)}{2}, \qquad \mu(j) := \begin{cases} j + 1, \qquad j \text{ odd} \\ j, \qquad j \text{ even}. \end{cases}
\end{align*}
Then we have by rearranging the sum in \eqref{MR:norm} that 
\begin{align} \label{floorbed:2}
\| \psi_R \|_{\mathrm{MR}} = \sum_{j = 1}^{2k_\ast} \| \psi_R \|_{\mathrm{MR}_j}, \qquad \| \psi_R \|_{\mathrm{MR}_j} := \sum_{k = 0}^{K_j} \| L^j \p_y^{2 + j} u_{R,k} \langle y \rangle^{N_k - j} \|_{L^2_{xy}}
\end{align}
Next, we have that since $k \le K_j$, this implies that $k+1 \le K_{j-2}$ and hence the second term on the right-hand side of \eqref{heli:1} can be bounded by $L \sum_{j' = 1}^{j-2} \| \psi_R \|_{\mathrm{MR}_{j'}}$. Therefore, upon further summing \eqref{heli:1} from $k = 0$ to $k = K_j$, we have 
\begin{align} \n
\| \psi_R \|_{\mathrm{MR}_j} \lesssim& \| \psi_R \|_{\mathrm{Linear}}+ \tilde{L} \sum_{j' = 1}^{j-2}\| \psi_R \|_{\mathrm{MR}_{j'}} + \sum_{k = 0}^{K_j} \sum_{j' = 1}^{j}  (\| \tilde{L}^{j'} \p_y^{j'} \mathcal{R}_k \langle y \rangle^{N_k-j'} \|_{L^2_{xy}} \\ \label{heat:1}
&+ \| \tilde{L}^{j'} \p_y^{j'} \mathcal{M}_{k, \text{Quasi}} \langle y \rangle^{N_k-j'} \|_{L^2_{xy}}).
\end{align}
Next, as a result of \eqref{des:w:L}, we observe the following bounds for each $(k, j)$ such that $1 \le j \le J_k$, 
\begin{align*}
\| \tilde{L}^j \p_y^j \mathcal{R}_{k, \mathrm{Lin}} \langle y \rangle^{N_k-j} \|_{L^2_{xy}} \lesssim & \| \psi_R \|_{\mathrm{Linear}} +\tilde{L} \sum_{k' = 0}^{k-1} \sum_{j' = 1}^{j-1} \| \tilde{L}^{j'} \p_y^{2 + j'} u_{R, k'} \langle y \rangle^{N_{k'}-j'} \|_{L^2_{xy}} \\
\lesssim & \| \psi_R \|_{\mathrm{Linear}} + \tilde{L} \sum_{j' = 0}^{j-1} \| \psi \|_{\mathrm{MR}_{j'}}.
\end{align*}
Next, simplifying the right-hand sides of both \eqref{pluplu:1} and \eqref{qmr:bd:1}, we have for any $0 \le k \le k_\ast -1$, and all $1 \le j \le J_k$, 
\begin{align*}
&\| L^j \p_y^j \mathcal{R}_{k; \text{Semi}} \langle y \rangle^{N_k - j} \|_{L^2_{xy}} + \| L^j \p_y^j \mathcal{M}_{k; \text{Quasi}} \langle y \rangle^{N_k - j} \|_{L^2_{xy}} \\
\lesssim& \eps L^{-1} ( \sum_{k' = 0}^{k_\ast - 1} \sum_{j = 0}^{J_{k'} } \|\p_y^j \psi_{R,k'} \|_{L^\infty} +  \| \psi_R \|_{\mathrm{Linear}}  )(\| \psi_R \|_{\mathrm{MR}}  + \| \psi_R \|_{\mathrm{Linear}} ).
\end{align*}
Inserting these bounds into \eqref{heat:1}, we obtain for each $j$ (in the range $1 \le j \le 2k_\ast$), 
\begin{align} \n
\| \psi_R \|_{\mathrm{MR}_j} \lesssim&  \| \psi_R \|_{\mathrm{Linear}}+ L \sum_{j' = 1}^{j-1} \| \psi_R \|_{\mathrm{MR}_{j'}} \\ \label{floorbed:1}
& + \eps L^{-1} ( \sum_{k' = 0}^{k_\ast - 1} \sum_{j = 0}^{J_{k'} } \|\p_y^j \psi_{R,k'} \|_{L^\infty} +  \| \psi_R \|_{\mathrm{Linear}}  )(\| \psi_R \|_{\mathrm{MR}}  + \| \psi_R \|_{\mathrm{Linear}} ),
\end{align}
which inductively in $j$ implies 
\begin{align} \n
\| \psi_R \|_{\mathrm{MR}_j} \lesssim& \| \psi_R \|_{\mathrm{Linear}} + \eps L^{-1} ( \sum_{k' = 0}^{k_\ast - 1} \sum_{j = 0}^{J_{k'} } \|\p_y^j \psi_{R,k'} \|_{L^\infty} +  \| \psi_R \|_{\mathrm{Linear}}  )(\| \psi_R \|_{\mathrm{MR}}  + \| \psi_R \|_{\mathrm{Linear}} ).
\end{align}
We now sum in $j$ over $0 \le k \le 2k_\ast$ and invoke the \eqref{floorbed:2} to obtain the estimate \eqref{main:pro:MR}. The proposition is proven. 
\end{proof}

\section{Complete $Z$ Norm Estimate, $0 < L << 1$} \label{final:sec}


We are now ready to apply the results from our linearized analysis for $0 < L << 1$, in particular Proposition \ref{pro:effective}, to the specific Prandtl system at order $k$. More precisely, we will consider the system \eqref{sideLH} with specific choices of the background  
\begin{align}
u_{B, R} := u_R, \qquad u_B = u_{R; \mathrm{Quasi}}, \qquad \bar{U} := U_P = u_{FS} + \eps u_{R; \mathrm{Quasi}},
\end{align}
where $u_{R; \mathrm{Quasi}}$ is defined in \eqref{urQuasi}, $\vec{\bold{C}}$ as in \eqref{bold:C1:pcts} -- \eqref{bold:C3:pcts}, and specific source terms $\mathcal{R}_k$ defined through \eqref{defRKRK}.  We define
\begin{align} \label{Z:norm:ult}
\| \psi_R \|_{\mathcal{Z}} := \| \psi_R \|_{\mathrm{Linear}} + \| \psi_R \|_{\mathrm{MR}}.
\end{align}

\subsection{Proof of Theorem \ref{main:thm}}

\begin{proof}[Proof of Theorem \ref{main:thm}] We now bring together all of our key estimates. First, when $\bold{C}_i$ are defined through \eqref{bold:C1:pcts} -- \eqref{bold:C3:pcts}, the assumption 
\begin{align} \label{boots:pro:lin:2}
 \sum_{k = 0}^1 \sum_{j = 0}^5 \| \p_x^k \p_y^j u_R \langle y \rangle^{N_{k_\ast}+1} \|_{L^\infty} \le \eps^{-\frac{1}{3}},
\end{align}
ensures that the assumptions of Proposition \ref{pro:effective} are valid for each $0 \le k \le k_\ast$. Therefore, we may apply estimates \eqref{main:E} and \eqref{main:Ekju} for each $k$ to obtain (recall the definition \eqref{def:gamma:u:k})
\begin{align} \n
\| \psi_{R,k} \|_{\mathrm{Eff}_{N_k}} \lesssim &\mathrm{\bold{Source}}_k  +\mathrm{\bold{Dat}}_k, \\ \n
|\gamma_{u; k}| + |\gamma_{\psi; k}| \lesssim &\mathrm{\bold{Source}}_k  +\mathrm{\bold{Dat}}_{\mathrm{Below},k}
\end{align}
We therefore obtain by using $\mathrm{\bold{Dat}}_k \le \mathrm{\bold{Dat}}$, $\mathrm{\bold{Dat}}_{\mathrm{Below},k} \le \mathrm{\bold{Dat}}_{\mathrm{Below}}$, and invoking \eqref{sourke}
\begin{align} \n
\| \psi_{R,k} \|_{\mathrm{Eff}_{N_k}} \lesssim & \mathrm{\bold{Dat}} + L \sum_{k'= 0}^{k-1} \| \psi_{R, k'} \|_{\text{Eff}_{N_{k'}}} +  \eps L^{-1} (\sum_{k = 0}^{\lfloor \frac{k_\ast + 1}{2} \rfloor} \sum_{j = 0}^4 \| \p_y^j u_{R,k} \|_{L^\infty} ) \| \psi_R \|_{\mathrm{Linear}} \\
&+ \eps L^{-1} \| \psi_R \|_{\mathrm{Linear}}^2, \\ \n
|\gamma_{u; k}| + |\gamma_{\psi; k}| \lesssim & \mathrm{\bold{Dat}}_{\mathrm{Below}} + L \sum_{k'= 0}^{k-1} \| \psi_{R, k'} \|_{\text{Eff}_{N_{k'}}} +  \eps L^{-1} (\sum_{k = 0}^{\lfloor \frac{k_\ast + 1}{2} \rfloor} \sum_{j = 0}^4 \| \p_y^j u_{R,k} \|_{L^\infty} ) \| \psi_R \|_{\mathrm{Linear}} \\
&+ \eps L^{-1} \| \psi_R \|_{\mathrm{Linear}}^2.
\end{align}
Inducting in $k$, the above inequality implies 
\begin{align} \n
\| \psi_{R,k} \|_{\mathrm{Eff}_{N_k}} \lesssim & \mathrm{\bold{Dat}}  +  \eps L^{-1} (\sum_{k = 0}^{\lfloor \frac{k_\ast + 1}{2} \rfloor} \sum_{j = 0}^4 \| \p_y^j u_{R,k} \|_{L^\infty} ) \| \psi_R \|_{\mathrm{Linear}} + \eps L^{-1} \| \psi_R \|_{\mathrm{Linear}}^2, \\
|\gamma_{u; k}| + |\gamma_{\psi; k}| \lesssim & \mathrm{\bold{Dat}}_{\mathrm{Below}}  +  L\mathrm{\bold{Dat}}  +  \eps L^{-1} (\sum_{k = 0}^{\lfloor \frac{k_\ast + 1}{2} \rfloor} \sum_{j = 0}^4 \| \p_y^j u_{R,k} \|_{L^\infty} ) \| \psi_R \|_{\mathrm{Linear}} + \eps L^{-1} \| \psi_R \|_{\mathrm{Linear}}^2.
\end{align}
We now take the summation of both sides over $0 \le k \le k_\ast$. While the right-hand sides above are independent of $k$, the left-hand side of the top inequality produces the linear norm upon summation upon invoking \eqref{Linear:norm}. This closes the bounds 
\begin{align} \label{sbU:1}
\| \psi_R \|_{\mathrm{Linear}} \lesssim & \mathrm{\bold{Dat}} + \eps L^{-1} (\sum_{k = 0}^{\lfloor \frac{k_\ast + 1}{2} \rfloor} \sum_{j = 0}^4 \| \p_y^j u_{R,k} \|_{L^\infty} ) \| \psi_R \|_{\mathrm{Linear}} + \eps L^{-1} \| \psi_R \|_{\mathrm{Linear}}^2, \\ \n
\sum_{k = 0}^{k_\ast}(|\gamma_{u; k}| + |\gamma_{\psi; k}|) \lesssim &  \mathrm{\bold{Dat}}_{\mathrm{Below}} + L \mathrm{\bold{Dat}} + \eps L^{-1} (\sum_{k = 0}^{\lfloor \frac{k_\ast + 1}{2} \rfloor} \sum_{j = 0}^4 \| \p_y^j u_{R,k} \|_{L^\infty} ) \| \psi_R \|_{\mathrm{Linear}} \\ \label{shbujL2}
&+ \eps L^{-1} \| \psi_R \|_{\mathrm{Linear}}^2.
\end{align}
Upon recalling \eqref{est:data:M} and \eqref{main:pro:MR}, we therefore have the three bounds 
\begin{align} \label{sbu:ghy:1}
\| \psi_R \|_{\mathrm{Linear}} \lesssim & \mathrm{\bold{Dat}} + \eps L^{-1} (\sum_{k = 0}^{\lfloor \frac{k_\ast + 1}{2} \rfloor} \sum_{j = 0}^4 \| \p_y^j u_{R,k} \|_{L^\infty} ) \| \psi_R \|_{\mathrm{Linear}} + \eps L^{-1} \| \psi_R \|_{\mathrm{Linear}}^2, \\ \label{main:pro:MR:cons}
\| \psi_R \|_{\mathrm{MR}} \lesssim & \| \psi_R \|_{\mathrm{Linear}} + \eps L^{-1} ( \sum_{k = 0}^{k_\ast - 1} \sum_{j = 0}^{J_k } \|\p_y^j \psi_{R,k} \|_{L^\infty} +  \| \psi_R \|_{\mathrm{Linear}}  )(\| \psi_R \|_{\mathrm{MR}}  + \| \psi_R \|_{\mathrm{Linear}} ), \\ \n
\mathrm{\bold{Dat}} \lesssim & L^{\frac16} (\| \mathring{F}_{\mathrm{Left}} \langle \eta \rangle^3 \|_{H^{2 + 3k_\ast}_\eta} + \| \mathring{F}_{\mathrm{Right}} \langle \eta \rangle^3 \|_{H^{2 + 3k_\ast}_\eta}) + \|  G_{\mathrm{Right}} \|_{H^{3 + 3k_\ast}_y} + \|  G_{\mathrm{Left}} \langle y \rangle^{N_0} \|_{H^{3 + 2k_\ast}_y} \\ \label{dat:est:yea}
& + \sum_{k = 0}^{k_\ast} (|\gamma_{u;k}| + |\gamma_{\psi; k}|). 
\end{align}
We now need to bring together two of our ``more refined" estimates on just the quantities $\gamma_{u; k}, \psi_{u; k}$ due to their dependence (without a small parameter) on the right-hand side above in \eqref{dat:est:yea}. Namely, we recall \eqref{shbujL2} and \eqref{est:data:below:M} 
\begin{align} \label{two:points:1}
\sum_{k = 0}^{k_\ast} (|\gamma_{u; k}| + |\gamma_{\psi; k}|) \lesssim &\mathrm{\bold{Dat}}_{\mathrm{Below}}  + L \mathrm{\bold{Dat}} + \eps L^{-1} (\sum_{k = 0}^{\lfloor \frac{k_\ast + 1}{2} \rfloor} \sum_{j = 0}^4 \| \p_y^j u_{R,k} \|_{L^\infty} ) \| \psi_R \|_{\mathrm{Linear}} + \eps L^{-1} \| \psi_R \|_{\mathrm{Linear}}^2, \\ \n
\mathrm{\bold{Dat}}_{\mathrm{Below}} \lesssim & L^{\frac16} (\| \mathring{F}_{\mathrm{Left}} \langle \eta \rangle^3 \|_{H^{2 + 3k_\ast}_\eta} + \| \mathring{F}_{\mathrm{Right}} \langle \eta \rangle^3 \|_{H^{2 + 3k_\ast}_\eta}) + \|  G_{\mathrm{Right}} \|_{H^{3+ 3k_\ast}_y} \\ \label{two:points:2}
&+ L^{\frac13}\|  G_{\mathrm{Left}} \langle y \rangle^{N_0} \|_{H^{3 + 2k_\ast}_y}  + L^{\frac16} \sum_{k = 0}^{k_\ast} (|\gamma_{u;k}| + |\gamma_{\psi; k}|).
\end{align}
At this point, we are led to define the notation 
\begin{align} \n
\mathcal{F}[\mathring{F}_{\mathrm{Left}}, \mathring{F}_{\mathrm{Right}}, G_{\mathrm{Left}}, G_{\mathrm{Right}}]:=  &L^{\frac16} (\| \mathring{F}_{\mathrm{Left}} \langle \eta \rangle^3 \|_{H^{2 + 3k_\ast}_\eta} + \| \mathring{F}_{\mathrm{Right}} \langle \eta \rangle^3 \|_{H^{2 + 3k_\ast}_\eta}) \\
&+ \|  G_{\mathrm{Right}} \|_{H^{3 + 3k_\ast}_y} + \|  G_{\mathrm{Left}} \langle y \rangle^{N_0} \|_{H^{3 + 2k_\ast}_y}.
\end{align}
By bringing together \eqref{sbu:ghy:1}, \eqref{main:pro:MR:cons}, \eqref{dat:est:yea}, \eqref{two:points:1}, \eqref{two:points:2}, and using that $L << 1$, we obtain 
\begin{align} \n
\| \psi_R \|_{\mathrm{Linear}} \lesssim & \mathcal{F}[\mathring{F}_{\mathrm{Left}}, \mathring{F}_{\mathrm{Right}}, G_{\mathrm{Left}}, G_{\mathrm{Right}}]+ \eps L^{-1} (\sum_{k = 0}^{\lfloor \frac{k_\ast + 1}{2} \rfloor} \sum_{j = 0}^4 \| \p_y^j u_{R,k} \|_{L^\infty} ) \| \psi_R \|_{\mathrm{Linear}} \\ \label{can:haf:1} 
&+ \eps L^{-1} \| \psi_R \|_{\mathrm{Linear}}^2, \\ \label{otfs:ii:1}
\| \psi_R \|_{\mathrm{MR}} \lesssim & \| \psi_R \|_{\mathrm{Linear}} + \eps L^{-1} ( \sum_{k = 0}^{k_\ast - 1} \sum_{j = 0}^{J_k } \|\p_y^j \psi_{R,k} \|_{L^\infty} +  \| \psi_R \|_{\mathrm{Linear}}  )(\| \psi_R \|_{\mathrm{MR}}  + \| \psi_R \|_{\mathrm{Linear}} ).
\end{align}
We now have, for a factor $c(L)$ which blows up as $L \rightarrow 0$, the following bounds
\begin{align} \label{mr:b:1}
\sum_{k = 0}^1 \sum_{j = 0}^5 \| \p_x^k \p_y^j u_R \langle y \rangle^{N_{k_\ast}+1} \|_{L^\infty} \lesssim & c(L)  \| \psi_R \|_{\mathrm{MR}} \qquad \text{ if } k_\ast \ge 7, \\ \label{mr:b:2}
\sum_{k = 0}^{\lfloor \frac{k_\ast + 1}{2} \rfloor} \sum_{j = 0}^4 \| \p_y^j u_{R,k} \|_{L^\infty}  \lesssim & c(L)  \| \psi_R \|_{\mathrm{MR}} \qquad \text{ if } k_\ast \ge 7, \\ \label{mr:b:3}
\sum_{k = 0}^{k_\ast - 1} \sum_{j = 0}^{J_k } \|\p_y^j \psi_{R,k} \|_{L^\infty} \lesssim & c(L) \| \psi_R \|_{\mathrm{MR}}.
\end{align}
To obtain \eqref{mr:b:1}, we note by a standard Sobolev embedding in two-dimensions, we have 
\begin{align*}
\sum_{k = 0}^1 \sum_{j = 0}^5 \| \p_x^k \p_y^j u_R \langle y \rangle^{N_{k_\ast}+1} \|_{L^\infty}  \lesssim \sum_{k = 0}^2 \sum_{j = 0}^6 \| \p_x^k \p_y^j u_R \langle y \rangle^{N_{k_\ast}+1} \|_{L^2},
\end{align*}
and hence comparing with the definition of our maximal regularity norm, \eqref{MR:norm}, we arrive at two restrictions, one for the derivative count and one for the weight, namely $k = 2, J_k \ge 4$ (regularity), and $N_2 - 4 \ge N_{k_\ast} + 1$ (weight). The regularity constraint is satisfied if $k_\ast \ge 4$, whereas the weight constraint is satisfied if $k_\ast \ge 7$ upon invoking \eqref{def:NsubK}. The factor of $c(L)$ is used to account for the powers of $L^j$ appearing in \eqref{MR:norm}. A similar Sobolev embedding argument is used for \eqref{mr:b:2}, and also for \eqref{mr:b:3}. We now recall our other restriction on $k_\ast$, namely that $k_\ast \ge 5$ from Proposition \ref{pro:MR}. Therefore, we choose $k_\ast \ge 7$ to satisfy all constraints. 

We now recall our definition, \eqref{Z:norm:ult}. Inserting \eqref{mr:b:1} -- \eqref{mr:b:3} into \eqref{boots:pro:lin:2} and \eqref{can:haf:1} -- \eqref{otfs:ii:1} and using that $\eps << L$ gives our main \textit{a-priori} estimate. Assuming 
\begin{align}
\| \psi_R \|_{\mathcal{Z}} \le \eps^{-\frac{1}{4}},
\end{align} 
then
\begin{align} \n
\| \psi_R \|_{\mathcal{Z}} \le &  \eps^{\frac12} \| \psi_R \|_{\mathcal{Z}}^2 +C \mathcal{F}[\mathring{F}_{\mathrm{Left}}, \mathring{F}_{\mathrm{Right}}, G_{\mathrm{Left}}, G_{\mathrm{Right}}].
\end{align}
This concludes the proof of the theorem. 
\end{proof}

\section{Spectral Theory of $(-\Delta)^{\frac16} + \bold{K}_L$} \label{sec:spec}

The main task of this section is to establish Theorem \ref{thm:res}, which addresses the case when $L$ is not small. We will do this by building off of Theorem \ref{main:thm}, as well as establishing the $L$-analyticity properties of an appropriate family of linear operators. As mentioned earlier, the guiding principles can be understood through analogy with the simpler elliptic problem, $- \Delta + V_L$, where the potential $V_L$ is analytic in the parameter $L$. The analyticity will imply that there are only discretely many values of $L$ in which $-\Delta + V_L$ is not invertible. In our case, the $-\Delta$ is replaced by $(-\Delta)^{\frac16}$, however the main complication is that the potential $V_L$ is not explicit; it must be defined through a series of maps.  

In our setting, the main focus will be the inversion of $\mathcal{L}_{\mathrm{Modulation}}$, which is the only place the small $L$ enters (for example, to close both the individual estimates \eqref{nlyougo:1} -- \eqref{nlyougo:4} as well as the couplings between these estimates in order to prove Proposition \ref{pro:global}). The main goal of this section will be to show the \textit{a-priori} bound, \eqref{main:G}, on $\mathcal{L}_{\mathrm{Modulation}}$ even when $L$ is not necessarily small, but away from a discrete set of resonances. 

To make matters more precise, we need to abstract the data which we are supplementing $\mathcal{L}_{\mathrm{Prandtl}}$ with, since we will apply this result for each $k$, and for each $k$ the data on the sides changes. More precisely, we assume 
\begin{align} \label{hgjiL:1}
&u_R|_{x = 1} = A(y) \gamma_u + B(y) \gamma_\psi+ C(y), \qquad y^\ast < y < \infty, \\ \label{hgjiL:2}
&\omega_{g}|_{x = 1} = D(y), \qquad y_1(\eps) < y < y^\ast, \\ \label{hgjiL:3}
&\omega_{g}|_{x = 1 + L} = E(y), \qquad y_\ast < y < y_{1 + L}, \\ \label{hgjiL:4}
&u_R|_{x = 1 + L} = H(y), \qquad 0 < y < y_\ast, 
\end{align}
where $\gamma_u := u(1, y_1(\eps))$ and $\gamma_{\psi} := \psi(1, y_1(\eps))$. This form is motivated by \eqref{mo:by:1} and \eqref{in:bd:4}, where $A(y), B(y)$ are the fixed functions multiplying $\gamma_u, \gamma_\psi$, and $C, D, E, H$ are the remaining functions arising from the iteration scheme (that are all lower order, and hence they are not contributing linear terms that need to be inverted in this section). 

More precisely, $D(y), E(y)$ will stand for $\omega_{g, k; \iota}$, $\iota \in \{\mathrm{Left}, \mathrm{Right}\}$, which are computed from $\omega_{NL, k; \iota}$ using the results of Lemmas \ref{lem:one}, \ref{lem:four}. Similarly, $H(y)$ is a placeholder for $u_{R, k; \mathrm{Right}}$, computed using the analysis of Lemma \ref{lem:three}. Evidently, none of these contribute linear terms that we need to invert. On the other hand, examining the expression \eqref{mo:by:1}, we have a representation for $\psi_{R, k; \mathrm{Left}}$ in terms of $\psi_{R, k; \mathrm{ Left}}(y^\ast) \frac{u_{P; \mathrm{Left}}(y)}{u_{P; \mathrm{Left}}(y^\ast)}$ together with lower order terms which we group into $C(y)$. Further examining \eqref{in:bd:4}, we have 
\begin{align*}
&\psi_{R, k; \mathrm{ Left}}(y^\ast) \frac{u_{P; \mathrm{Left}}(y)}{u_{P; \mathrm{Left}}(y^\ast)} \\
=& \frac{u_{P; \mathrm{Left}}(y)}{u_{P; \mathrm{Left}}(y^\ast)} \Big( \gamma_{\psi; k} + \int_{y_1(\eps)}^{y^\ast} ( \gamma_{u; k} \frac{u_{P; \mathrm{Left}}'(y)}{u_{P; \mathrm{Left}}'(y_1(\eps))} + u_{P; \mathrm{Left}}' \int_{y_1(\eps)}^{\bar{y}} \frac{ \omega_{NL, k; \mathrm{Left}}}{u'_{P; \mathrm{Left}}} )\ud \bar{y} \Big) \\
= &  \frac{u_{P; \mathrm{Left}}(y)}{u_{P; \mathrm{Left}}(y^\ast)}  \gamma_{\psi; k} + \frac{u_{P; \mathrm{Left}}(y)}{u'_{P; \mathrm{Left}}(y_1(\eps))}  \gamma_{u; k} +  \frac{u_{P; \mathrm{Left}}(y)}{u_{P; \mathrm{Left}}(y^\ast)} \int_{y_1(\eps)}^{y^\ast}  u_{P; \mathrm{Left}}' \int_{y_1(\eps)}^{\bar{y}} \frac{ \omega_{NL, k; \mathrm{Left}}}{u'_{P; \mathrm{Left}}} \ud \bar{y}.
\end{align*}
As a result, we put
\begin{align}
A(y) :=  \frac{u_{P; \mathrm{Left}}(y)}{u'_{P; \mathrm{Left}}(y_1(\eps))}, \qquad B(y) :=  \frac{u_{P; \mathrm{Left}}(y)}{u_{P; \mathrm{Left}}(y^\ast)},
\end{align}
and the $k$-dependent term $\frac{u_{P; \mathrm{Left}}(y)}{u_{P; \mathrm{Left}}(y^\ast)} \int_{y_1(\eps)}^{y^\ast}  u_{P; \mathrm{Left}}' \int_{y_1(\eps)}^{\bar{y}} \frac{ \omega_{NL, k; \mathrm{Left}}}{u'_{P; \mathrm{Left}}} \ud \bar{y}$ into $C(y)$. 

Normally, $\zeta_{\mathrm{Left}}$ is a placeholder for $u_P u_R - u_{Py} \psi_R$, which clearly sees the contributions from $\gamma_u, \gamma_\psi$. Since we are removing these and including them as part of the linear operator, we will define $\zeta_{\mathrm{Left}}^{\mathrm{(Hom)}} = u_P C(y) - u_{Py} \int_{y_1(\eps)}^y C(\bar{y}) \ud \bar{y}$. 

The main proposition of this section will therefore be an analogue of Proposition \ref{pro:global}. 
\begin{proposition} \label{prop:spec:LjK} Assume the bootstrap assumptions of Proposition \ref{pro:global}. Fix any number $L_{\mathrm{Max}} < \infty$. Then there exist a discrete set of values $\{L_1, L_2, \dots, L_{N} \}$ $\subset (0, L_{\mathrm{Max}})$ such that for $L \notin \{L_1, L_2, \dots, L_{N} \}$, the following bounds (the analogues of \eqref{main:G}, \eqref{main:G:g}) holds
\begin{align} \n
\| \psi \|_{\mathrm{Global}_n} \lesssim & \sum_{j = 0}^3 \| \p_z^j \bold{R} \chi_{O,j} \langle z \rangle^n \|_{L^2_{sz}} + \sum_{j = 0}^1 \| \p_z^j \bold{R} \chi_I \|_{L^2_{sz}} + \| d(z) \p_z^2 \bold{R} \chi_I \|_{L^2_{sz}} \\ \n
& + \sum_{i = 0}^2 \sum_{\iota \in \mathrm{Left, Right}}  \| \langle \rho \rangle \p_\rho^i \Xi_{\iota} \|_{L^2_\rho} +  \sum_{i = 1}^2 \sum_{\iota \in \mathrm{Left, Right}} \|\rho^i \p_\rho^i \Xi_{\iota}\|_{L^2_\rho} \\  \label{main:G:LL}
& + \sum_{j = 0}^3 ( \| \p_z^j \zeta^{(\mathrm{Hom})}_{\mathrm{Left}} \chi_{O,j}^+ \langle z \rangle^n \|_{L^2_z} + \| \p_z^j \zeta_{\mathrm{Right}} \chi_{O,j}^- \|_{L^2_z} ),
\end{align}
where the implicit constant above depends poorly on $L_{\mathrm{Max}}$ and the distance between $L$ and the discrete set $\{L_1, L_2, \dots, L_N\}$. 
\end{proposition}

\subsection{Formulation of Spectral Condition}

We now formulate a linear operator, whose invertibility will be the key. 
This linear operator, which we ultimately denote as $\bold{L}[\cdot]$ (see below, \eqref{est:L:cpct}), depends on the background profile $\bar{u}$. Moreover, it will be achieved as a composition of several auxiliary linear operators. We now introduce these. First, we want to regard $\gamma^{(0)}_{\Omega}(s)$ as a prescribed function, and $(\psi, u)$ as being determined from $\gamma^{(0)}_{\Omega}(s)$.  

We define $u$ to solve the linearized version of \eqref{eq:for:v:1}, namely with $u_B = 0$, in which case $\bar{W}(s, z) = \bar{u}(s, z)$, and through \eqref{mu:pert:1}, the $z$ variable is adapted so that $\bar{u}(s, 1) = 0$. 
\begin{subequations}
\begin{align}  \label{eq:for:v:1:sp}
&\bar{u} \p_s u + v \p_z \bar{u} +  \bold{c}_1 \p_z u + \bold{c}_2 u   + \bold{c}_3 \psi - \p_z^2 u = 0, \\
&u|_{s = 1, z > 1} = 0, \\
&u|_{s = 1 + \bar{L}, 0 < z < 1} = 0, \\
&u|_{z = 0} = \psi|_{z = 0} = u|_{z \rightarrow \infty} = 0, \\
&(u_z - \frac{\bar{u}''}{\bar{u}'} u)|_{z = 1} = \gamma^{(0)}_{\Omega}(s), \qquad 1 < s < 1 + \bar{L}. 
\end{align}
\end{subequations}
That is, this $u$ represents the abstract operator that takes the boundary condition at $z = 1$ of $\gamma^{(0)}_{\Omega}(s)$ and generates a solution on $(s, z) \in (1, 1 + \bar{L}) \times \mathbb{R}_+$, without any source term. This operator is well-defined for each given $\gamma_\Omega^{(0)} \in L^2_s$ by standard parabolic theory applied to $z < 1$ and $z > 1$. For the datum, we need to extract the components of \eqref{hgjiL:1} -- \eqref{hgjiL:4} which are linearly dependent on the solution, namely, we take  
\begin{align} \label{hgjiL:1H}
&u_R|_{x = 1} = A(y) \gamma_{u; \mathrm{Dyn}} + B(y) \gamma_{\psi; \mathrm{Dyn}}, \qquad y^\ast < y < \infty, \\ \label{hgjiL:2H}
&\omega_g|_{x = 1} = 0, \qquad y_1(\eps) < y < y^\ast, \\ \label{hgjiL:3H}
&\omega_g|_{x = 1 + L} = 0, \qquad y_\ast < y < y_{1 + L}, \\ \label{hgjiL:4H}
&u_R|_{x = 1 + L} = 0, \qquad 0 < y < y_\ast.
\end{align}
We give this map the name:  
\begin{align}
\bold{K}[\gamma^{(0)}_{\Omega}; \bar{u}] := \psi(s, z).  
\end{align}
Above, we have defined:  
\begin{align}
\gamma_{u; \mathrm{Dyn}} :=& \p_z \bold{K}[\gamma^{(0)}_{\Omega}; \bar{u}] (1, y_1(\eps)), \\
\gamma_{\psi; \mathrm{Dyn}} :=&\bold{K}[\gamma^{(0)}_{\Omega}; \bar{u}] (1, y_1(\eps)), 
\end{align}
Notice that this differs from $\gamma_{u}$ and $\gamma_{\psi}$ because these quantities are determined by the full solution $\psi$ which include the data and source term.  Therefore, we will have $\gamma_u = \gamma_{u; \mathrm{Dyn}} + \gamma_{u; \text{Hom}}$ (and correspondingly for $\gamma_\psi$) where we describe $\gamma_{u; \text{Hom}}$ as follows. 

The ``homogeneous" component of the solution will not depend on $\gamma_\Omega^{(0)}$, but only the data. This component is defined as the solution to 
\begin{subequations}
\begin{align}  \label{eq:for:v:1:HOM}
&\bar{u} \p_s u + v \p_z \bar{u} +  \bold{c}_1 \p_z u + \bold{c}_2 u   + \bold{c}_3 \psi - \p_z^2 u = 0, \\
&u|_{s = 1, z > 1} = 0, \\
&u|_{s = 1 + \bar{L}, 0 < z < 1} = 0, \\
&u|_{z = 0} = \psi|_{z = 0} = u|_{z \rightarrow \infty} = 0, \\
&(u_z - \frac{\bar{u}''}{\bar{u}'} u)|_{z = 1} = 0, \qquad 1 < s < 1 + \bar{L}. 
\end{align}
\end{subequations} 
where 
\begin{align} \label{hgjiL:1H:hom}
&u_R|_{x = 1} =C(y), \qquad y^\ast < y < \infty, \\ \label{hgjiL:2H:hom}
&\omega_g|_{x = 1} = D(y), \qquad y_1(\eps) < y < y^\ast, \\ \label{hgjiL:3H:hom}
&\omega_g|_{x = 1 + L} = E(y), \qquad y_\ast < y < y_{1 + L}, \\ \label{hgjiL:4H:hom}
&u_R|_{x = 1 + L} = H(y), \qquad 0 < y < y_\ast.
\end{align}

We introduce a change of variables which is a linearized version of the $(s, Z)$ coordinate system introduced above. We note that the change of variables, \eqref{change:1} -- \eqref{change:2}, is linear, as is \eqref{Eikonal:1}, and we thus define the linear maps 
\begin{align}
\bold{K}_{\mapsto}[f] = F, \qquad \bold{K}_{\mapsfrom}[F] = f,  \qquad F(s, Z) = f(s, z), 
\end{align}
We now notice that $\tau_i$ defined in \eqref{tau:def} depend only on $\bar{u}$, and therefore contribute purely linear terms to the operator we are about to define. 
We now consider the linear part of the right-hand side of \eqref{loc:sys:I1}, namely 
\begin{align} \label{boldrl}
\bold{r}_{\mathrm{Lin}} := \chi\Big(\frac{Z}{\delta}\Big) (\tau_2 \p_Z^2 \Omega +   \tau_1 \p_Z \Omega + \tau_0 \Omega + \tau_{-1} \p_Z \Phi + \tau_{-2} \Phi)   - \frac{\chi''(\frac{Z}{\delta})}{\delta^2} \Omega - \frac{2}{\delta} \chi'(\frac{Z}{\delta})\p_Z \Omega,
\end{align}
and re-express these terms as a composition of linear operators on $\gamma^{(0)}_{\Omega}$ through the use of the operators we have defined above. In particular, we have 
\begin{align} \n
\bold{r}_{\mathrm{Lin}}[\gamma^{(0)}_{\Omega}] := &\chi\Big(\frac{Z}{\delta}\Big) (\tau_2 \p_Z^4  \bold{K}_{\mapsto} \circ \bold{K}[\gamma^{(0)}_{\Omega}]  +  \tau_1 \p_Z^3  \bold{K}_{\mapsto} \circ \bold{K}[\gamma^{(0)}_{\Omega}] +  \tau_0 \p_Z^2  \bold{K}_{\mapsto} \circ \bold{K}[\gamma^{(0)}_{\Omega}] \\ \n
&+ \tau_{-1} \p_Z  \bold{K}_{\mapsto} \circ \bold{K}[\gamma^{(0)}_{\Omega}] + \tau_{-2} \bold{K}_{\mapsto} \circ \bold{K}[\gamma^{(0)}_{\Omega}]  )   - \frac{\chi''(\frac{Z}{\delta})}{\delta^2} \p_Z^2 \bold{K}_{\mapsto} \circ \bold{K}[\gamma^{(0)}_{\Omega}] \\ \label{bold:low:r}
&  - \frac{2}{\delta} \chi'(\frac{Z}{\delta})\p_Z^3 \bold{K}_{\mapsto} \circ \bold{K}[\gamma^{(0)}_{\Omega}].
\end{align}
We now want to insert $\bold{r}_{\mathrm{Lin}}[\gamma^{(0)}_{\Omega}]$ into the Fourier integral operators, $\mathcal{T}_{-\frac13, \iota}$, where $\iota \in \pm$, which has been defined in \eqref{defn:T:op}. In reality, we think of these operators as acting on $\gamma_{Q}^{(0)}$, where $\gamma_\Omega^{(0)}$ is equivalent to $\gamma_Q^{(0)}$, up to a component that depends on the data, as shown in \eqref{gamma:Q:def:he}. Upon doing so, we define 
\begin{align} \label{est:L:cpct}
\underline{\bold{K}}_{L}[\gamma^{(0)}_{Q}] := & \sum_{\iota \in \pm} \mathcal{T}_{-\frac13, \iota} \circ \bold{r}_{\mathrm{Lin}}[\gamma^{(0)}_{Q}], \qquad \bold{L}[\gamma^{(0)}_{Q}] := (-\Delta_D)^{\frac16} \gamma^{(0)}_{Q} -\underline{ \bold{K}}_L[\gamma^{(0)}_{Q}].
\end{align}
Due to the complicated, nonlocal nature of defining $\underline{\bold{K}}_L$, it is not possible to obtain a simpler formula for the linear operator $\bold{L}[\gamma_{1,\mathcal{Q}}]$. Nevertheless, we have the bound:
\begin{align}
\| \underline{\bold{K}}_L[\gamma^{(0)}_{Q}] \|_{L^2_s} \lesssim \| \gamma^{(0)}_{Q} \|_{L^2_s} 
\end{align}
Due to \eqref{est:L:cpct}, we can regard $\underline{\bold{K}}_L$ as a compact perturbation of $(-\Delta_D)^{\frac16}$. We thus can state:
\begin{definition}[Spectral Condition] \label{defn:spectral}The operator $\bold{L}$ defined in \eqref{est:L:cpct} satisfies the \textit{Spectral Condition} if $\bold{L}$ is invertible and satisfies the following bound:
\begin{align} \label{assumption}
\| \p_s^{\frac13} \gamma^{(0)}_{Q, \mathrm{Ext}}\|_{L^2_s} \lesssim \| \bold{L}[\gamma^{(0)}_{Q}] \|_{L^2_s}.
\end{align}
\end{definition}

\subsection{Analyticity in $L$ \& structure of spectrum $\Sigma(L)$}

We rescale \eqref{eq:for:v:1:sp} by introducing
\begin{align}
t := \frac{s-1}{\bar{L}} \in (0, 1), \qquad U(t) := u(s), \qquad V(t) := - \int_0^z \p_t U \ud z', \qquad \phi(t, z) := \psi(s, z). 
\end{align}
We note that the background FS profiles are independent of $s$ and therefore of $t$ as well. In the $(t, z)$ variables, \eqref{eq:for:v:1:sp} reads
\begin{subequations}
\begin{align} \label{hguy:1}
&F(z) \p_t U + V F'(z) - L \p_z^2 U + L \sum_{j = 0}^2 A_j(z) \p_z^j \phi = 0, \\ \label{hguy:2}
&\p_z U(t, 1) - \frac{F''(z)}{F'(z)} U(t, 1)= \gamma^{(0)}_\Omega(Lt) =: \Gamma_\Omega^{(0)}(t)\\ \label{hguy:3}
&U(t, 0) = 0, \qquad U(t, \infty) = 0.
\end{align}
\end{subequations}
Above, the quantities $U, \phi$ are thought of as profiles; they depend on the parameter $L$ through the appearance of $L$ in the coefficients of \eqref{hguy:1}. To describe the data on the sides, $t = 0, t = 1$, is slightly more complicated. We essentially want to consider the $F_{\mathrm{Left}}, F_{\mathrm{Right}}, G_{\mathrm{Left}}, G_{\mathrm{Right}} = 0$ case of \eqref{PhiIota:1} -- \eqref{PhiIota:4}.
\begin{align} \label{PhiIota:1L}
U|_{t = 0}:= &  A(y) \overline{\gamma}_{u; \mathrm{Dyn}} + B(y) \overline{\gamma}_{\psi; \mathrm{Dyn}}, \qquad z > z^\ast\\ \label{PhiIota:2L}
(U_z - \frac{\bar{u}''}{\bar{u}'} U)|_{t =0} := & 0 , \qquad 1 < z < z^\ast, \\ \label{PhiIota:3L}
(U_z - \frac{\bar{u}''}{\bar{u}'} U)|_{t = 1} := & 0, \qquad z_\ast < z < 1, \\ \label{PhiIota:4L}
U|_{t = 1}(z) := &0 , \qquad 0 < z < z_{\ast}.
\end{align} 
We will define the operators 
\begin{subequations}
\begin{align}
\bold{K}^{\mathrm{lift}}_L[\Gamma_\Omega^{(0)}] :=& \phi(t, z), \\ \label{compo:proj:0}
\bold{K}^{\Pi}_{L}[F(t, Z)] := &\int_0^{\infty} ai( \frac{(i\xi)^{\frac13}}{L^{\frac13}}Z) \widehat{F}(\xi, Z) \ud Z , \\ \label{compo:1}
\bold{K}_L[\Gamma_\Omega^{(0)}(\cdot)] := & L^{\frac13} \sum_{j = 0}^4 \bold{K}^{\Pi}_{L} \circ \widehat{\tau}_{j-2} \p_Y^j \bold{K}_{\mapsto} \circ \bold{K}_L^{lift}[\Gamma^{(0)}_\Omega(\cdot)],
\end{align}
\end{subequations}
where we can obtain $\widehat{\tau}_j$ by grouping terms and collecting coefficients appropriately from \eqref{bold:low:r}. Let us define the following function spaces
\begin{align} 
\| \phi \|_{X_{\mathrm{en}}}^2 := & \sup_{0 < t < 1} \int |\phi|^2 \ud z + \int_0^1 \int F(z) |\p_z U|^2 , \\ \n
\| \phi \|_{X}^2 := & \sup_{1 \le t \le 1} \Big( \int_{\mathbb{R}} |\p_z U|^2 \ud z + \int_{\mathbb{R}} F(z)^2 |\p_z^2 U|^2 \ud z + \int_{\mathbb{R}} F(z)^4 |\p_z^3 U|^2 \ud z  \Big) \\ \label{def:Crocco:norm:spec}
& + \int_0^{1} \int_{\mathbb{R}} |F(z)| |\p_z^3 U|^2 \ud z \ud t + \int_0^{1} \int_{\mathbb{R}} |F(z)|^3 |\p_z^4 U|^2 \ud z \ud t.
\end{align}
The intuition behind \eqref{def:Crocco:norm} is: if we collect the energy-dissipation functionals from the von-Mise analysis, \eqref{yahL2} -- \eqref{vonMisenorm}, and the Crocco analysis, \eqref{def:Crocco:norm}, it is equivalent to the norm \eqref{def:Crocco:norm:spec}.

We will also need spaces of operators. For this, define 
\begin{align}
\bold{O} :=& \{ \bold{K} : L^2(0, 1) \mapsto L^2(0,1), \| \bold{K} \|_{\bold{O}} < \infty \}, \qquad \| \bold{K} \|_{\bold{O}} = \sup_{\| \gamma \|_{L^2(0, 1)} = 1} \| \bold{K}[\gamma] \|_{L^2(0,1)}, \\ 
\bold{O}_{\mathrm{lift}} :=& \{ \bold{K} : L^2(0, 1) \mapsto X, \| \bold{K} \|_{\bold{O}_{\mathrm{lift}}} < \infty \} , \qquad \| \bold{K} \|_{\bold{O}_{\mathrm{lift}}} = \sup_{\| \gamma \|_{L^2(0, 1)} = 1} \| \bold{K}[\gamma] \|_{X}, \\
\bold{O}_{\mathrm{lift, en}} :=& \{ \bold{K} : L^2(0, 1) \mapsto X_{\mathrm{en}}, \| \bold{K} \|_{\bold{O}_{\mathrm{lift, en}}} < \infty \} , \qquad \| \bold{K} \|_{\bold{O}_{\mathrm{lift, en}}} = \sup_{\| \gamma \|_{L^2(0, 1)} = 1} \| \bold{K}[\gamma] \|_{X_{\mathrm{en}}}, \\
\bold{O}_{\mathrm{proj}} :=& \{ \bold{K} : X \mapsto L^2(0, 1), \| \bold{K} \|_{\bold{O}_{\mathrm{proj}}} < \infty \} , \qquad \| \bold{K} \|_{\bold{O}_{\mathrm{proj}}} = \sup_{\| F \|_{X} = 1} \| \bold{K}[F] \|_{L^2(0,1)}. 
\end{align}

Our main proposition here will be 
\begin{proposition} \label{prop:comp} Fix any $L_0 > 0$. The operator $\bold{K}_L: L^2(0, 1) \mapsto L^2(0, 1)$ is analytic in the parameter $L$ on the domain $L \in (L_0, \infty)$. More precisely, the map $L \mapsto \bold{K}_L$ from $(L_0, \infty) \mapsto \bold{O}$ is an analytic map:
\begin{align}
\sup_{k}  \sup_{L \in (L_0, \infty)} \frac{1}{C^{k+1} k!} \| \p_L^k \bold{K}_L \|_{\bold{O}} < \infty.
\end{align}
\end{proposition}

It turns out that this proposition will follow from a weaker bound, namely
\begin{lemma}Fix any $L_0 > 0$. The lift operator $\bold{K}^{\mathrm{lift}}_L: L^2(0, 1) \mapsto X_{\mathrm{en}}$ is analytic in the parameter $L$ on the domain $L \in (L_0, \infty)$. More precisely, the map $L \mapsto \bold{K}_L^{\mathrm{lift}}$ from $(L_0, \infty) \mapsto \bold{O}_{\mathrm{lift, en}}$ is an analytic map:
\begin{align}
\sup_{k}  \sup_{L \in (L_0, \infty)} \frac{1}{C^{k+1} k!} \| \p_L^k \bold{K}_L^{\mathrm{lift}} \|_{\bold{O}_{\mathrm{lift,en}}} < \infty.
\end{align}
\end{lemma} 

Our next task will be to upgrade the above bound to the space of operators $\bold{O}_{lift}$, which is the content of the next lemma. 
\begin{lemma} \label{lemma:lift}Fix any $L_0 > 0$. The lift operator $\bold{K}^{\mathrm{lift}}_L: L^2(0, 1) \mapsto X$ is analytic in the parameter $L$ on the domain $L \in (L_0, \infty)$. More precisely, the map $L \mapsto \bold{K}_L^{\mathrm{lift}}$ from $(L_0, \infty) \mapsto \bold{O}_{\mathrm{lift}}$ is an analytic map:
\begin{align}
\sup_{k}  \sup_{L \in (L_0, \infty)} \frac{1}{C^{k+1} k!} \| \p_L^k \bold{K}_L^{\mathrm{lift}} \|_{\bold{O}_{\mathrm{lift}}} < \infty.
\end{align}
\end{lemma} 

Finally, we want to establish analyticity of the projection operators.  
\begin{lemma} \label{lemma:proj} Fix any $L_0 > 0$. The projection operator $ \bold{K}^{(\Pi)}_L \circ \p_Z: X \mapsto L^2(0, 1)$ is analytic in the parameter $L$ on the domain $L \in (L_0, \infty)$. More precisely, the map $L \mapsto \bold{K}_L^{(\Pi)} \circ \p_Z$ from $(L_0, \infty) \mapsto \bold{O}_{\mathrm{proj}}$ is an analytic map:
\begin{align} \label{an2}
\sup_{k}  \sup_{L \in (L_0, \infty)} \frac{1}{C^{k+1} k!} \| \p_L^k \bold{K}_L^{(\Pi)} \circ \p_Z \|_{\bold{O}_{\mathrm{proj}}} < \infty.
\end{align}
\end{lemma}

Given the above lemmas, we can prove Proposition \ref{prop:comp}.
\begin{proof}[Proof of Proposition \ref{prop:comp}] According to the composition, \eqref{compo:1}, we can rewrite 
\begin{align} \n
\bold{K}_L = & \bold{K}_L^{\Pi} \circ (\widehat{\tau}_{-2} \bold{K}_{\mapsto} \circ \bold{K}_L^{\mathrm{lift}}) + \sum_{j = 1}^4 (\bold{K}_L^{\Pi} \circ \widehat{\tau}_{j-2} \p_Z) \circ ( \p_Z^{j-1} \bold{K}_{\mapsto} \circ \bold{K}_L^{\mathrm{lift}}) =  \sum_{j= 0}^4 \Pi_j \circ \bold{E}_j, 
\end{align}
where we define the ``projection" and ``extension" operators via 
\begin{align}
\Pi_0 := & \bold{K}_L^{\Pi}, \qquad \bold{E}_0 := \widehat{\tau}_{-2} \bold{K}_{\mapsto} \circ \bold{K}_L^{\mathrm{lift}} \\
\Pi_j :=  & \bold{K}_L^{\Pi} \circ \widehat{\tau}_{j-2} \p_Z, \qquad \bold{E}_j :=  \p_Z^{j-1} \bold{K}_{\mapsto} \circ \bold{K}_L^{\mathrm{lift}}, \qquad j = 1, 2, 3, 4. 
\end{align}
We have 
\begin{align}
\bold{E}_j : L^2(0, 1) \mapsto X, \qquad \Pi_j: X \mapsto L^2(0, 1), \qquad \Pi_j \circ \bold{E}_j : L^2(0, 1) \mapsto L^2(0, 1). 
\end{align}
By Lemmas \ref{lemma:lift} and \ref{lemma:proj}, we have that $\bold{E}_j$ and $\Pi_j$ are individually analytic, and hence since the composition of analytic functions is analytic, we have $\bold{K}_L: L^2(0, 1) \mapsto L^2(0, 1)$ is analytic in $L$.  
\end{proof}

\subsubsection{Analyticity into $\bold{O}_{lift,en}, \bold{O}_{lift}, \bold{O}_{proj}$ Spaces} \label{sec:ana:o}

It suffices in what follows to establish the bound
\begin{align} \label{an1}
\sup_{k}  \sup_{L \in (L_0, \infty)} \frac{1}{C^{k+1} k!} \| \p_L^k \phi \|_{X_{\mathrm{en}}} \lesssim \| \Gamma_\Omega^{(0)} \|_{L^2(0, 1)}
\end{align}
To establish this bound, we differentiate \eqref{hguy:1} which results in 
\begin{subequations}
\begin{align} \label{br:1}
&\p_t \p_L^k \mathcal{U} - L \bold{D}[ \p_L^k \phi] = k \bold{D}[\p_L^{k-1}U], \\
&\bold{D}[U] := \p_z^3 \phi+ \sum_{j = 0}^2 A_j(z) \p_z^j \phi, \\ \label{int:cond:1}
&\p_L^k \p_z U(t, 1) - \frac{F''(z)}{F'(z)} \p_L^k U(t, 1) = \p_L^k \Gamma_\Omega^{(0)}(t) = 0, 
\end{align}
\end{subequations}
As usual, we need to introduce a good unknown which is transported: 
\begin{align}
\mathcal{U}  := F(z) U - F'(z) \phi. 
\end{align}
It is important to note that $L \ge L_0 > 0$ will be important to ensure that the dissipative term in \eqref{br:1} is non-degenerate. We now establish the following inductive bound:
\begin{lemma} For a constant $0 < C < \infty$ which could depend on $L_0$ and $L_{\mathrm{Max}}$, we have the following bounds:
\begin{align}
\| \p_L^k \phi \|_{X_{\mathrm{en}}} \le C k L^k \| \Gamma_\Omega^{(0)} \|_{L^2(0, 1)} + C k \| \p_L^{k-1} \phi \|_{X_{\mathrm{en}}}, \qquad k \ge 1. 
\end{align}
\end{lemma}
\begin{proof} We multiply the above system by $\p_L^k \mathcal{U}$ and integrate by parts, on the subdomains $z > 1$ and $0 < z < 1$ separately. As the proofs are nearly identical, the only exception being the data containing depending on $A(y) \gamma_{u; \mathrm{Dyn}} + B(y) \gamma_{\psi; \mathrm{Dyn}}$, we give the proof only for $0 < z < 1$. This gives the natural energy 
\begin{align} \label{C49}
\frac{\p_t}{2} \int |\p_L^k \mathcal{U}|^2 \ud z - L \int \bold{D}[\p_L^k \phi] \p_L^k \mathcal{U} = k \int \bold{D}[\p_L^{k-1} U] \p_L^k \mathcal{U}
\end{align}
We now collect an identity which captures the dissipation term $\bold{D}$:
\begin{align} \n
- L \int \bold{D}[\p_L^k \phi] \p_L^k \mathcal{U} \ud z = & L \int F(z) |\p_z \p_L^k U|^2 \ud z + L \int \alpha_1(z) |\p_L^k U|^2 \ud z + L \int \alpha_2(z) |\p_L^k \phi|^2 \ud z \\ \n
& + L \beta_1 |\p_L^k  \phi(t, 1)|^2 + L \beta_2 \p_L^k U(t, 1) \p_L^k \phi(t, 1) + L \beta_3 \p_L^k \phi(t,1) \p_z \p_L^k U(t, 1) 
\end{align}
A standard application of Young's inequality results in 
\begin{align}
- L \int \bold{D}[\p_L^k \phi] \p_L^k \mathcal{U} \ud z \ge \frac{L}{2} \int F(z) |\p_z \p_L^k U|^2 \ud z - C_0(L) L^{k} \| \Gamma_\Omega^{(0)} \|_{L^2}^2 - C_1(L) \int |\p_L^k \mathcal{U}|^2 \ud z. 
\end{align}
We now estimate the terms on the right-hand side of \eqref{C49}. We have the primary term arising from $\bold{D}$, the top order $\p_z^3 \p_L^{k-1} \phi$ term, for which an integration by parts yields 
\begin{align} \n
k |\int \p_z^3 \p_L^{k-1} \phi \p_L^k \mathcal{U}| = & |- k \int F(z)  \p_L^{k-1} \p_z U \p_L^k \p_z U +  k \int F''(z)  \p_L^{k-1} \p_z U \p_L^k \phi \\ \n
& + k \p_L^{k-1} \p_z U(t, 1) F'(1) \p_L^k \phi(t, 1)| \\
\lesssim & k \| \sqrt{F} \p_z \p_L^{k-1} U \|_{L^2_z} \| \sqrt{F} \p_z \p_L^k U \|_{L^2_z}.
\end{align}
A standard application of Cauchy-Schwartz, Young's inequality, and a Gronwall argument concludes the proof. 
\end{proof}

We want to upgrade the topology in which we prove analyticity. In particular, we want to prove Lemma \ref{lemma:lift}, which will follow from the bound:  
\begin{align} \label{b2}
\sup_{k}  \sup_{L \in (L_0, \infty)} \frac{1}{C^{k+1} k!} \| \p_L^k \phi \|_{X} \lesssim \|\Gamma_\Omega^{(0)} \|_{L^2(0, 1)}.
\end{align}
This in turn will follow from the following bound on the solution: 
\begin{lemma}For a constant $0 < C < \infty$ which could depend on $L_0$, we have the following bounds:
\begin{align} \label{b1}
\| \p_L^k \phi \|_{X} \lesssim \| \p_L^k \phi \|_{X_{\mathrm{en}}} + C k \| \p_L^{k-1} \phi \|_{X}.
\end{align}
\end{lemma}
\begin{proof} This is obtained by essentially repeating the analysis which resulted in the Crocco and von-Mise estimates, namely Propositions \ref{prop:Crocco} and \ref{pro:vonMise}. We omit repeating all the details.  
\end{proof}

Coupling \eqref{b1} now with \eqref{an1} we obtain \eqref{b2}. Finally, the analyticity \eqref{an2} follows immediately from the explicit formula, \eqref{compo:proj:0}.

\subsection{Proof of Proposition \ref{prop:spec:LjK}}

\begin{proof}[Proof of Proposition \ref{prop:spec:LjK}] We proceed in two main steps, which we delineate as follows. 

\vspace{2 mm}

\noindent \textit{Step 1: Verification of Spectral Assumption} We decompose $L^2(0, 1) = L^2_{\le N}(0, 1) \oplus L^2_{> N}(0, 1)$ where
\begin{align*}
L^2_{\le N}(0, 1) := \text{span} \{ \sin(ks) : 1 \le k \le N \}, \qquad L^2_{> N}(0, 1) := \text{span} \{ \sin(ks) : k > N \}
\end{align*}
We clearly have 
\begin{align*}
0 \notin \Sigma[ ((-\Delta)^{\frac16} + \bold{K}_L)|_{L^2_{> N}(0, 1)}],
\end{align*}
and thus it suffices to study the finite dimensional operator $((-\Delta)^{\frac16} + \bold{K}_L)|_{L^2_{\le N}(0, 1)}$. Here we have a matrix representation 
\begin{align}
A := (-\Delta)^{\frac16} |_{L^2_{\le N}(0, 1)}, \qquad B(L) := \bold{K}_L|_{L^2_{\le N}(0, 1)}. 
\end{align}
We are thus left with studying the analytically perturbed matrix $T(L) := A + B(L)$. A classical fact about analytically perturbed matrices states: by excising finitely many values of $L$, the number of eigenvalues for $T(L)$ remains constant in the interval $(L_0, L_{max})$. Let these be denoted by $\sigma_1(L), \dots, \sigma_{n_{\ast}}(L)$. These functions are analytic in $L$. We also know that $\sigma_i(L) \neq 0$ for $L < 2L_0$ according to Theorem \ref{main:thm}. Therefore, we have again at most finitely many $L$ values for which $\sigma_i(L)$ cross $0$.  (see for instance Kato, \cite{KatoBook}, p. 63-64).

To conclude the argument, we enumerate these finite values of $L$ via $\{L_1, L_2, \dots L_{N}\}$, where $L_k < L_{k+1}$. We know that $L_{k+1} - L_k > 0$, and we have the bound 
\begin{align} \label{eig:lb}
\inf_{i} |\sigma_i(L)| \gtrsim \min_k |L - L_k|.
\end{align}
The lower bound \eqref{eig:lb} immediately shows that our Spectral Condition is satisfied. 

\vspace{2 mm}

\vspace{2 mm}

\noindent \textit{Step 2: a-priori Bounds on $\| \gamma_\Omega^{(1)} \|_{L^2_s}$:} In this step, we essentially want to obtain an analogue of the bound \eqref{nlyougo:4}, but without the first three terms on the right-hand side of \eqref{nlyougo:4} appearing. Specifically, we will prove 
\begin{align} \n
\| \gamma_\Omega^{(1)} \|_{L^2_s} \lesssim & \sum_{j = 0}^3 \| \p_z^j \bold{R} \chi_{O,j} \langle z \rangle^n \|_{L^2_{sz}} + \sum_{j = 0}^1 \| \p_z^j \bold{R} \chi_I \|_{L^2_{sz}} + \| d(z) \p_z^2 \bold{R} \chi_I \|_{L^2_{sz}} \\ \n
& + \sum_{i = 0}^2 \sum_{\iota \in \mathrm{Left, Right}}  \| \langle \rho \rangle \p_\rho^i \Xi_{\iota} \|_{L^2_\rho} +  \sum_{i = 1}^2 \sum_{\iota \in \mathrm{Left, Right}} \|\rho^i \p_\rho^i \Xi_{\iota}\|_{L^2_\rho} \\  \label{mahg}
& + \sum_{j = 0}^3 ( \| \p_z^j \zeta^{(\mathrm{Hom})}_{\mathrm{Left}} \chi_{O,j}^+ \langle z \rangle^n \|_{L^2_z} + \| \p_z^j \zeta_{\mathrm{Right}} \chi_{O,j}^- \|_{L^2_z} ),
\end{align}
where the implicit constant above depends poorly on $L_{\mathrm{Max}}$ and the distance between $L$ and the discrete set $\{L_1, L_2, \dots, L_N\}$. 

Our starting point is the elliptic equation \eqref{a:Green:1}. We rewrite \eqref{a:Green:1} with $F_I = F_{\mathrm{Coupled}} + F_{\mathrm{Low}}$ as in \eqref{abs:67}, which gives the elliptic equation
\begin{align} \n  
(-\Delta_D)^{\frac16} \gamma_{Q}^{(0)} -  \sum_{\iota \in \pm} \mathcal{T}_{- \frac 1 3, \iota} [F_{\mathrm{Coupled}} 1_{1 \le s \le 1 + \bar{L}}]  =  &\sum_{\iota \in \pm} \mathcal{T}_{- \frac 1 3, \iota} [F_{\mathrm{Low}} 1_{1 \le s \le 1 + \bar{L}}] \\ \n
&+ \sum_{\iota \in \pm} \mathcal{T}_{- \frac 1 3, \iota} [F_{\mathrm{Data}, \iota}[\Xi_{\mathrm{Left}}, \Xi_{\mathrm{Right}}]].
\end{align}
Recall $F_{\mathrm{Coupled}}$ is defined in \eqref{defFCL}. We first of all split $F_{\mathrm{Coupled}} = \bold{r}_{\mathrm{Lin}}^{(\eps)} + F_{\mathrm{Coupled}, \mathrm{Hom}}$ where $\bold{r}_{\mathrm{Lin}}^{(\eps)}$ is \eqref{boldrl} but with the $\tau_j$ their $\eps$-dependent values. Secondly, $F_{\mathrm{Coupled}, \mathrm{Hom}}$ contains the components that depend on the homogeneous solutions, \eqref{eq:for:v:1:HOM}.

We then write the above equality as 
\begin{align} \n
\bold{L}_{\eps}[\gamma^{(0)}_Q] = & \sum_{\iota \in \pm} \mathcal{T}_{- \frac 1 3, \iota} [F_{\mathrm{Coupled}, \mathrm{Hom}} 1_{1 \le s \le 1 + \bar{L}}] +  \sum_{\iota \in \pm} \mathcal{T}_{- \frac 1 3, \iota} [F_{\mathrm{Low}} 1_{1 \le s \le 1 + \bar{L}}] \\ \label{ofme}
&+ \sum_{\iota \in \pm} \mathcal{T}_{- \frac 1 3, \iota} [F_{\mathrm{Data}, \iota}[\Xi_{\mathrm{Left}}, \Xi_{\mathrm{Right}}]],
\end{align}
where $\bold{L}_{\eps}$ is defined to be the version of $\bold{L}$ in which $\bold{r}_{\mathrm{Lin}}$ is replaced by $\bold{r}_{\mathrm{Lin}}^{(\eps)}$. By continuity, our Spectral Condition ensures the following bound:
\begin{align} \label{bold:L:eps:bd}
\| \p_s^{\frac13} \gamma^{(0)}_{Q, \mathrm{Ext}} \|_{L^2_s(\mathbb{R})} \lesssim \| \bold{L}_{\eps}[\gamma^{(0)}_Q] \|_{L^2_s(1, 1 + \bar{L})}.
\end{align}
For the right-hand side of \eqref{ofme}, we invoke definitions \eqref{defFCL} for the $F_{\mathrm{Coupled}, \mathrm{Hom}}$ term and \eqref{defFLOW} for the $F_{\mathrm{Low}}$ term, which are easily seen to be bounded above by the right-hand side of \eqref{main:G} upon re-applying \eqref{reapply:me}, and we invoke the bounds \eqref{spurt:1} -- \eqref{spurt:4} for the $F_{\mathrm{Data}}$ term. 

This proves that the right-hand side of \eqref{main:G:LL} controls $\| \p_s^{\frac13} \gamma^{(0)}_{Q, \mathrm{Ext}} \|_{L^2_s(\mathbb{R})}$. From here, bounding $\| \gamma_\Omega^{(1)} \|_{L^2_s}$ above by the right-hand side of \eqref{main:G:LL} follows in a straightforward manner, upon invoking the definition \eqref{DN:top} in much the same manner as in the bounds \eqref{asiun:1} and \eqref{guy:1}.

\vspace{2 mm}

\noindent \textit{Step 3: a-priori Bounds on $\| \psi \|_{\mathrm{Global}_n}$:} We now need to obtain bounds over the remaining three quantities in $\|\psi \|_{\mathrm{Global}_n}$, defined in \eqref{global:n}. To do so, we combine the energy and dissipation functionals from our Crocco norm, \eqref{def:Crocco:norm}, and our von-Mise norms, \eqref{Ejn}. We can treat the case $z < 1$ independently from $z > 1$. For simplicity, we focus on $z < 1$. We first introduce the notation 
\begin{align}
\mathcal{E}_{\mathrm{Crocco}, -}(s) :=  \int_{\mathbb{R}_-} \Omega_I^2 \ud Z + \int_{\mathbb{R}_-} Z^2 |\p_Z \Omega_I|^2 \ud Z + \int_{\mathbb{R}_-} Z^4 \p_Z^2 \Omega_{I}^2 \ud Z
\end{align}
which are the energetic quantities from \eqref{def:Crocco:norm}, and 
\begin{align}
\mathcal{D}_{\mathrm{Crocco}, -}(s) :=\int_{\mathbb{R}_-} |Z| |\p_Z^2 \Omega_{I}|^2 \ud Z \ud s + \int_1^{1 + \bar{L}} \int_{\mathbb{R}_-} |Z|^3 |\p_Z^3 \Omega_{I}|^2 \ud Z \ud s,
\end{align}
which are the dissipation terms from \eqref{def:Crocco:norm}. In a similar spirit, we introduce 
\begin{align}
\mathcal{E}_{\mathrm{von-Mise}, -}(s) := \sum_{j = 0}^3 \mathcal{E}_{O,- }^{(j)}(s), \qquad \mathcal{D}_{\mathrm{von-Mise},-}(s) :=  \sum_{j = 0}^3 \mathcal{D}_{O,- }^{(j)}(s),
\end{align}
which are the energy and dissipation functionals from \eqref{yahL1}. We then define our full energy-dissipation as 
\begin{align}
\mathcal{H}_{\mathrm{Full},-}(s) := \mathcal{H}_{\mathrm{Crocco}, -}(s)  + \mathcal{H}_{\mathrm{von-Mise}, -}(s), \qquad \mathcal{H} \in \{ \mathcal{E}, \mathcal{D} \}. 
\end{align}
Revisiting the proofs of Lemmas \ref{klem1}, \ref{klem2}, and \ref{klem3}, we obtain the energy-dissipation inequality 
\begin{align*}
\frac{\p_s}{2} \mathcal{E}_{\mathrm{Full},-}(s) - \mathcal{D}_{\mathrm{Full},-}(s) =  \mathrm{Err}_{\mathrm{Lin}}(s) + \mathrm{Err}_{\mathrm{Bdry}}(s) + \mathrm{Err}_{\mathrm{Source}}(s), 
\end{align*} 
where
\begin{align}
|\mathrm{Err}_{\mathrm{Lin}}(s)| \le &C \mathcal{E}_{\mathrm{Full},-}(s) + \frac{1}{2}\mathcal{D}_{\mathrm{Full},-}(s), \\
|\mathrm{Err}_{\mathrm{Bdry}}(s)| \lesssim & |\gamma_\Omega^{(1)}(s)|, \\
 |\mathrm{Err}_{\mathrm{Source}}(s)| \lesssim & \sum_{j = 0}^3 \| \p_z^j \bold{R} \chi_{O,j} \langle z \rangle^n \|_{L^2_{z}}^2 + \sum_{j = 0}^1 \| \p_z^j \bold{R} \chi_I \|_{L^2_{z}}^2 + \| d(z) \p_z^2 \bold{R} \chi_I \|_{L^2_{z}}^2.
\end{align}
A standard Gronwall inequality then closes the desired bounds on $\mathcal{E}_{\mathrm{Full},-}$ and $\mathcal{D}_{\mathrm{Full},-}$. A nearly identical argument is applied for the $+$ case, thereby closing the full estimate on $\| \psi \|_{\mathrm{Global}_n}$.
\end{proof}

\subsection{Proof of Theorem \ref{thm:res}}

We are now ready to establish the main theorem for $L \in (0, L_{\mathrm{Max}}$):
\begin{proof}[Proof of Theorem \ref{thm:res}] Once the main linear operator $\mathcal{L}_{\mathrm{Modulation}}$ is shown to be coercive (according to bound \eqref{main:G:LL}), we may apply the bounds on the source terms, the maximal regularity analysis, and the data (Sections \ref{ksec} -- \ref{Lpr}), since these bounds were established uniformly in $L \in (0, L_{\mathrm{Max}})$. The only exception is there is no the need for $\mathrm{\bold{Dat}}_{\mathrm{Below}}$, since the $\gamma_u, \gamma_\psi$ contributions have been moved to the inverted linear operator leaving only $\zeta_{\mathrm{Left}}^{(\mathrm{Hom})}$ in \eqref{main:G:LL}. Indeed, by using Proposition \ref{prop:spec:LjK} $k$ by $k$ (and then switching to the Effective norm in $(x,y)$ coordinates) we obtain the bounds 
\begin{align}
\| \psi_{R,k} \|_{\mathrm{Eff}_{N_k}} \lesssim &\mathrm{\bold{Source}}_k  +\mathrm{\bold{Dat}}_k,
\end{align}
where the implicit constant above depend poorly on large values of $L_{\mathrm{Max}}$ and $\min |L - L_k|$. We therefore obtain by using $\mathrm{\bold{Dat}}_k \le \mathrm{\bold{Dat}}$, and invoking \eqref{sourke}
\begin{align} \n
\| \psi_{R,k} \|_{\mathrm{Eff}_{N_k}} \lesssim & \mathrm{\bold{Dat}} + L \sum_{k'= 0}^{k-1} \| \psi_{R, k'} \|_{\text{Eff}_{N_{k'}}} +  \eps L^{-1} (\sum_{k = 0}^{\lfloor \frac{k_\ast + 1}{2} \rfloor} \sum_{j = 0}^4 \| \p_y^j u_{R,k} \|_{L^\infty} ) \| \psi_R \|_{\mathrm{Linear}} \\ \n
&+ \eps L^{-1} \| \psi_R \|_{\mathrm{Linear}}^2,
\end{align}
for an implicit constant that depends poorly on large $L_{\mathrm{Max}}$ and $\min |L - L_k|$. Inducting in $k$, the above inequality implies 
\begin{align} \n
\| \psi_{R,k} \|_{\mathrm{Eff}_{N_k}} \lesssim & \mathrm{\bold{Dat}}  +  \eps L^{-1} (\sum_{k = 0}^{\lfloor \frac{k_\ast + 1}{2} \rfloor} \sum_{j = 0}^4 \| \p_y^j u_{R,k} \|_{L^\infty} ) \| \psi_R \|_{\mathrm{Linear}} + \eps L^{-1} \| \psi_R \|_{\mathrm{Linear}}^2.
\end{align}
We now take the summation of both sides over $0 \le k \le k_\ast$, which closes the bounds 
\begin{align} \label{sbU:1}
\| \psi_R \|_{\mathrm{Linear}} \lesssim & \mathrm{\bold{Dat}} + \eps L^{-1} (\sum_{k = 0}^{\lfloor \frac{k_\ast + 1}{2} \rfloor} \sum_{j = 0}^4 \| \p_y^j u_{R,k} \|_{L^\infty} ) \| \psi_R \|_{\mathrm{Linear}} + \eps L^{-1} \| \psi_R \|_{\mathrm{Linear}}^2.
\end{align}
Now, recalling \eqref{est:data:M} and \eqref{main:pro:MR}
\begin{align} \label{sbu:ghy:1:LL}
\| \psi_R \|_{\mathrm{Linear}} \lesssim &\mathrm{\bold{Dat}} + \eps L^{-1} (\sum_{k = 0}^{\lfloor \frac{k_\ast + 1}{2} \rfloor} \sum_{j = 0}^4 \| \p_y^j u_{R,k} \|_{L^\infty} ) \| \psi_R \|_{\mathrm{Linear}} + \eps L^{-1} \| \psi_R \|_{\mathrm{Linear}}^2, \\ \n
\| \psi_R \|_{\mathrm{MR}} \lesssim & \| \psi_R \|_{\mathrm{Linear}} + \eps L^{-1} ( \sum_{k = 0}^{k_\ast - 1} \sum_{j = 0}^{J_k } \|\p_y^j \psi_{R,k} \|_{L^\infty} +  \| \psi_R \|_{\mathrm{Linear}}  ) \\ \label{main:pro:MR:cons:LL}
& \times(\| \psi_R \|_{\mathrm{MR}}  + \| \psi_R \|_{\mathrm{Linear}} ), \\ \n
\mathrm{\bold{Dat}} \lesssim & \| \mathring{F}_{\mathrm{Left}} \langle \eta \rangle^3 \|_{H^{2 + 3k_\ast}_\eta} + \| \mathring{F}_{\mathrm{Right}} \langle \eta \rangle^3 \|_{H^{2 + 3k_\ast}_\eta} + \|  G_{\mathrm{Right}} \|_{H^{3 + 3k_\ast}_y} \\ \label{dat:est:yea:LL}
&+ \|  G_{\mathrm{Left}} \langle y \rangle^{N_0} \|_{H^{3 + 2k_\ast}_y},
\end{align}
where the implicit constants above depend poorly on large values of $L_{\mathrm{Max}}$ and $\min |L - L_k|$. This then implies the result of Theorem \ref{thm:res} upon using Sobolev embedding in the same manner as \eqref{mr:b:1} -- \eqref{mr:b:3}, and upon bringing $\eps$ small relative to $L_{\mathrm{Max}}$ and $\min |L - L_k|$.
\end{proof}

\appendix

\section{Airy Functions} \label{app:Airy}

We follow largely \cite{Olver}, specifically Chapter 11, Section 8. For any $\zeta \in \mathbb{C}$, we are considering the Airy equation on $\mathbb{C}$, 
\begin{align} \label{AIRY:ODE}
\zeta F(\zeta) - F''(\zeta) = 0. 
\end{align}
We introduce two solutions that are real valued: 
\begin{align}
ai(x) := &\frac{1}{\pi} \int_0^{\infty} \cos(\frac 13 t^3 + xt) \ud t, \\
bi(x) := &\frac{1}{\pi} \int_0^{\infty} e^{- \frac 13 t^3 + xt} + \sin (\frac 13 t^3 + xt) \ud t
\end{align}
We may extend these definitions to $\mathbb{C}$ via analytic continuation, and thus define $ai(\zeta), bi(\zeta), \zeta \in \mathbb{C}$. We subsequently define the following open sets: 
\begin{subequations}
\begin{align}
\Sigma_{0} := & \{ \zeta \in \mathbb{C} : - \frac{\pi}{3} < \arg(\zeta) < \frac{\pi}{3}   \}, \\
\Sigma_{1} := & \{ \zeta \in \mathbb{C} : \frac{\pi}{3} < \arg(\zeta) < \pi \}, \\
\Sigma_{-1} := & \{ \zeta \in \mathbb{C} : -\pi < \arg(\zeta) < - \frac{\pi}{3} \},
\end{align}
\end{subequations}
and the following rays: 
\begin{subequations}
\begin{align}
&\Gamma_{\pi} := \{ \zeta \in \mathbb{C} : \arg(\zeta) = \pi \}\\
&\Gamma_{\frac{\pi}{3}} :=   \{ \zeta \in \mathbb{C} : \arg(\zeta) = \frac{\pi}{3} \}\\
&\Gamma_{-\frac{\pi}{3}} := \{ \zeta \in \mathbb{C} : \arg(\zeta) = - \frac{\pi}{3} \}
\end{align} 
\end{subequations}

We want to now record some well-known asymptotic properties of these special functions. 

\begin{lemma} \label{lemma:fund:solutions} On $\Sigma_0$, we will have $\{ ai(\zeta), bi(\zeta) \}$ as the fundamental solutions, where $bi(\cdot)$ grows exponentially in $e^{\zeta^{\frac 3 2}}$, whereas $ai(\zeta)$ decays by $e^{- \zeta^{\frac 3 2}}$. On $\Sigma_1$, we have $ai(\zeta)$ grows and $B_{+1}(\zeta) := ai(e^{- \frac 2 3 \pi i} \zeta)$ decays. On $\Sigma_{-1}$, we have $ai(\zeta)$ grows and $B_{-1}(\zeta) := ai(e^{\frac{2}{3} \pi i} \zeta)$ decays. More precisely, we have the following asymptotics:
\begin{subequations}
\begin{align} \label{airy:asy:1}
|ai(\zeta)| \lesssim |\frac{1}{\zeta^{\frac 1 4}} e^{- \frac 2 3 \zeta^{\frac 3 2}}|, \qquad |bi(\zeta)| \gtrsim |\frac{1}{\zeta^{\frac 1 4}} e^{\frac 2 3 \zeta^{\frac 3 2}}|, \qquad \zeta \in \Sigma_0, \\ \label{airy:asy:2}
|ai(\zeta)| \gtrsim |\frac{1}{\zeta^{\frac 1 4}} e^{\frac 2 3 \zeta^{\frac 3 2}}|, \qquad |B_{+1}(\zeta)| \lesssim |\frac{1}{\zeta^{\frac 1 4}} e^{-\frac 2 3 \zeta^{\frac 3 2}}|, \qquad \zeta \in \Sigma_1,  \\ \label{airy:asy:3}
|ai(\zeta)| \gtrsim |\frac{1}{\zeta^{\frac 1 4}} e^{\frac 2 3 \zeta^{\frac 3 2}}|, \qquad |B_{-1}(\zeta)| \lesssim | \frac{1}{\zeta^{\frac 1 4}} e^{-\frac 2 3 \zeta^{\frac 3 2}}|, \qquad \zeta \in \Sigma_{-1}, 
\end{align}
\end{subequations}
\end{lemma}
\begin{proof} These are standard bounds on the Airy functions, see for instance \cite{Olver}, P. 413 -- 414, for precise proofs.  
\end{proof}

\noindent \textbf{Acknowledgements:} S.I is grateful for the hospitality and inspiring work atmosphere at NYU Abu Dhabi, where this work was initiated. The work of S.I is  supported by NSF grant DMS-2306528 and a UC Davis Hellman Foundation Fellowship. The work of N.M. is supported by NSF grant DMS-1716466 and by Tamkeen under the NYU Abu Dhabi Research Institute grant of the center SITE.  The authors would like to thank Pierre-Yves Lagree for useful email exchanges regarding the self-similar Falkner-Skan reversed flow at the early stages of this project, and for pointing them towards the source \cite{Lagree}. S.I. also thanks Mateusz Kwasnicki for pointing him towards  \cite{Grubb}. Both authors thank A-L. Dalibard and F. Marbach for constructive feedback on this work. The authors thank the anonymous referee for several constructive suggestions which improved the quality and presentation of the paper.

\def\bibindent{3.5em}

\end{document}